\documentclass[a4paper,10pt]{article}
\usepackage{makeidx}
\usepackage[mathscr]{eucal}
\usepackage{amssymb, amsfonts, amsmath, amsthm, graphicx, cases, color, tikz, url} 
\usepackage{here}
\usepackage{tasks}

\newtheorem{proposition}{Proposition}[section]
\newtheorem{theorem}{Theorem}[section]

\newtheorem{lemma}[proposition]{Lemma}
\newtheorem{remark}{Remark}[section]

\def\wt#1{\widetilde{#1}}

\newcommand{\N}{\mathbb{N}}

\newcommand{\R}{\mathbb{R}}

\newcommand{\E}{\mathcal{E}}

\newcommand{\J}{\mathcal{J}}
\newcommand{\K}{\mathcal{K}}

\renewcommand{\S}{\mathcal{S}}

\numberwithin{equation}{section}
\newcommand{\e}{\varepsilon}
\newcommand{\p}{\partial}
\numberwithin{equation}{section}

\title{Classification and qualitative properties of positive solutions 
to double-power nonlinear 
stationary Schr\"{o}dinger 
equations}

\author{Takafumi Akahori, Slim Ibrahim, Hiroaki Kikuchi, 
Masataka Shibata, Juncheng Wei}
\date{}
\begin{document}
\maketitle

\begin{abstract}
In this paper, we investigate positive radial solutions to double-power nonlinear stationary Schrödinger equations in three space dimensions. It is now known that the non-uniqueness of $H^{1}$-positive solutions can occur in three dimensions when the frequency is sufficiently small. Under suitable conditions, in addition to the ground state solution 
(whose $L^{\infty}$ norm vanishes as the frequency tends to zero), there exists another positive solution that minimizes a different constrained variational problem, with an $L^{\infty}$ norm diverging as the frequency tends to zero (see Theorem \ref{thm-asy}). We classify all positive solutions with small frequency into two categories: the ground state and the Aubin–Talenti type solution. As a consequence, we establish the multiplicity of positive solutions. Finally, we also examine the non-degeneracy and Morse index of each positive solution.
\end{abstract}

\tableofcontents

\section{Introduction} 
In this paper, we consider 
positive solutions 
to the following double-power nonlinear 
stationary Schr\"{o}dinger 
equations
\begin{equation}\label{sp} 
- \Delta u + \omega u 
- |u|^{p-1}u - |u|^{2^{*}-2}u = 0
\qquad 
\mbox{in $\mathbb{R}^{d}$}, 
\end{equation}
where $d \geq 3, \omega > 0$, $2^{*} = \frac{2d}{d -2}$ 
and $1 < p < 2^{*} -1$. 

Solutions to \eqref{sp} correspond
to \textit{standing wave solutions} 
$\psi(t, x) = e^{i \omega t} u_{\omega}(x) $ to 
the following Schr\"{o}dinger equations: 
\begin{equation}\label{nls}
i \partial_{t} \psi + \Delta \psi 
+ |\psi|^{p-1} \psi + |\psi|^{2^{*} - 2}\psi 
= 0 \qquad \mbox{in $\mathbb{R} \times 
\mathbb{R}^{d}$}. 
\end{equation}
The standing waves play an important role in the 
global dynamics of \eqref{nls} (see
\cite{MR2917888, MR3090237, MR4320767, 
MR4162293, 
MR4776379, MR4404073} and references therein). 

It is known that 
for any $\psi_{0} \in H^{1}(\R^d)$, 
there exists a unique solution $\psi$ to \eqref{nls} in 
$C(I_{\max}; H^{1}(\R^d))$ with $\psi|_{t=0} 
= \psi_{0}$, where
$I_{\max} \subset \R$ denotes 
the maximal existence interval.  
Moreover, the solution $\psi$ satisfies the following 
conservation laws of the mass and 
the energy in 
this order: 
\begin{equation} \label{cv-law}
\|\psi(t)\|_{L^{2}} = \|\psi_{0}\|_{L^{2}}, \qquad 
\mathcal{E}(\psi(t)) = \mathcal{E}(\psi_{0})
\qquad 
\mbox{for all $t \in I_{\max}$}. 
\end{equation}
Here
\begin{equation} \label{eq-E}
\mathcal{E}(u) := 
\frac{1}{2} \|\nabla u\|_{L^{2}}^{2} 
- \frac{1}{p+1} \|u\|_{L^{p+1}}^{p+1} 
- \frac{1}{2^{*}} \|u\|_{L^{2^{*}}}^{2^{*}}. 
\end{equation}
If in addition $\psi|_{t = 0} = 
\psi_{0} \in L^{2}(\R^d, |x|^{2} \, dx)$, 
then the solution 
$\psi$ belongs to $C(I_{\max}; L^{2}
(\R^{d}), |x|^{2} \, dx))$ and satisfies the so-called 
virial identity:
\begin{equation} \label{virial}
\begin{split}
\int_{\R^{d}} |x|^{2} |\psi(t, x)|^{2} \, dx 
& = \int_{\R^d} |x|^{2} |\psi_{0}(x)|^{2} \, dx 
+ 4 t \textrm{Im} \int_{\R^{d}} x \cdot \nabla \psi_{0}(x)
\overline{\psi_{0}(x)} \, dx \\[6pt]
& \quad 
+ 8
\int_{0}^{t} \int_{0}^{t^{\prime}} 
\mathcal{K}(\psi(t^{\prime \prime})) \, dt^{\prime \prime} 
dt^{\prime} \qquad \mbox{for any $t \in I_{\max}$}, 
\end{split}
\end{equation}
where 
\begin{equation} \label{eq-K}
\mathcal{K}(u):= 
\|\nabla u\|_{L^{2}}^{2} 
- \frac{d(p-1)}{2(p+1)}\|u\|_{L^{p+1}}^{p+1} 
- \|u\|_{L^{2^{*}}}^{2^{*}}. 
\end{equation}

Recently, equation \eqref{sp}
has been studied intensively
(see 
\cite{MR3964275, MR4445670, MR4483547, MR4572464, MR4162293, MR4476243, MR4578398, MR4096725, MR4433054, MR4638619} 
and references therein). 
In this paper, we focus on positive solutions to \eqref{sp}. 
From the result of Gidas, Ni and Nirenberg~\cite{MR634248}, 
positive solutions to \eqref{sp} are automatically radially symmetric and monotone decreasing 
in $|x| > 0$. 
Moreover, 
Wei and Wu~\cite{MR4638619} 
showed that 
when $d = 3$, $1 < p < 3$ and the frequency $\omega > 0$ 
is sufficiently small, 
equation \eqref{sp} has two 
distinct positive solutions, implying 
the non-uniqueness of positive 
solutions. 
One of the positive solutions 
is the ground state and the other one is 
obtained by a certain minimization problem 
(see \eqref{eq-mini-e} and \eqref{mini-sec-mp} below). 
This result coincides with the numerical computation by 
D\'{a}vila, del Pino and I. Guerra
~\cite{MR3021464}. 
Our aim is to study the multiplicity, the non-degeneracy 
and the Morse index of positive solutions to \eqref{sp}.

Recall that the ground state $Q_{\omega}$ 
is a minimizer of the action functional 
$\S_{\omega}$ 
defined on the Nehari manifold 
\[
\left\{u \in H^{1}(\mathbb{R}^{d}) \setminus \{0\} \colon \; 
\mathcal{N}_{\omega}(u) = 0 \right\} 
\]
with
\begin{align*}
& 
\mathcal{S}_{\omega}(u) := 
\frac{1}{2}\|\nabla u \|_{L^{2}}^{2}
+ \frac{\omega}{2}\|u\|_{L^{2}}^{2}
- \frac{1}{p+1}\|u\|_{L^{p+1}}^{p+1}
- \frac{1}{2^{*}}\|u\|_{L^{2^{*}}}^{2^{*}}, 
\\[6pt]
& \mathcal{N}_{\omega}(u) 
:= \langle \mathcal{S}_{\omega}^{\prime}(u), u \rangle 
= \|\nabla u \|_{L^{2}}^{2}
+ \omega\|u\|_{L^{2}}^{2}
- \|u\|_{L^{p+1}}^{p+1}
- \|u\|_{L^{2^{*}}}^{2^{*}},  
\end{align*} 
where $\mathcal{S}_{\omega}^{\prime}$ denotes the 
Frech\'{e}t derivative of $\mathcal{S}_{\omega}$. 
More precisely, if we consider the following minimization problem
\begin{equation} \label{mini}
m_{\omega} := 
\inf\left\{\mathcal{S}_{\omega}(u) \colon 
u \in H^{1}(\mathbb{R}^{d}) \setminus \{0\}, \; 
\mathcal{N}_{\omega}(u) = 0
\right\}, 
\end{equation}
then the ground state $Q_{\omega}$ 
satisfies $S_{\omega}(Q_{\omega}) = m_{\omega}$ and $\mathcal{N}_{\omega}(Q_{\omega}) = 0$.  
Note that by a standard argument, 
$Q_{\omega}$ can be taken 
a positive radial ground state. 
\begin{remark}
Although several definitions of the ground state can be found in the literature, the one adopted above appears to be the most suitable 
in this paper.
\end{remark}
We can verify that if the ground state $Q_{\omega}$ exists, 
then 
it minimizes $\S_{\omega}$ over 
all non-trivial $H^{1}$-solutions 
to \eqref{sp}, that is, 
$\mathcal{S}_{\omega}(Q_{\omega}) \leq \mathcal{S}_{\omega}(u)$ 
for all non-trivial solutions $u \in H^{1}(\R^{d})$ to \eqref{sp}. 
We now recall the results concerning 
the existence and non-existence of the ground state: 
\begin{theorem}[\cite{MR4445670, MR4638619, MR3166237}]
\label{thm-ex0}
\begin{enumerate}
\item[\textrm{(i)}] 
Let $d \geq 4$ 
and $1 < p < 2^{*} -1$ or $d = 3$ and $3 < p < 2^{*} -1$. 
For any $\omega > 0$, 
a ground state $Q_{\omega}$ exists.
\item[\textrm{(ii)}]
Let $d = 3$ and $1 < p \leq 3$. 
There exists $\omega_{c} > 0$ such that 
the ground state $Q_{\omega}$ exists for any 
$\omega \in (0, \omega_{c})$ and does not exist for 
any $\omega \in (\omega_{c}, \infty)$.
Finally, when $\omega = \omega_{c}$, 
a ground state $Q_{\omega_{c}}$ 
exists for $1 < p < 3 \; (p \neq 3)$. 
\end{enumerate} 
\end{theorem}
Part \textrm{(i)} of Theorem 
\ref{thm-ex0} was proved by 
Zhang and Zou~\cite[Theorem 1.1]{MR3166237}
while part \textrm{(ii)} 
was proved independently
by \cite[Theorem 1.1]{MR4445670} 
and \cite[Theorem 1.2.]{MR4638619}. 
\begin{remark}
\begin{enumerate}
\item[\textrm{(i)}] 
In the case of $d = 3$, $p = 3$ 
and $\omega = \omega_{c}$, 
the existence of the ground state 
is still unknown. 
We refer to Remark 3.1 of 
\cite{MR4638619} for further 
discussion. 
\item[\textrm{(ii)}] 
It is known that the limit of the ground state under a certain 
rescaling is the Aubin-Talenti function
\begin{equation} \label{talenti}
W(x) := \left(1 + \frac{|x|^{2}}{d(d-2)} \right)^{-\frac{d-2}{2}}, 
\end{equation}
which solves
\begin{equation} \label{eq-at}
- \Delta W = W^{2^{*} -1} \qquad 
\mbox{in $\mathbb{R}^{d}$}. 
\end{equation}
A main reason why the case 
$d = 3$ is different from $d \geq 4$ 
arises from the integrability of the 
Aubin-Talenti function. 
We can easily check that 
$W \in L^{2}(\R^{d})$ 
when $d \geq 5$ but $W \not\in L^{2}(\R^{d})$ when $d = 3, 4$
\footnote{When $d = 4$, $W \in L^{2 + \varepsilon}
(\R^{d})$ for any $\varepsilon > 0$.}. 
These are the so-called \lq\lq resonance" when $d = 3, 4$
(see \cite{MR4445670} and \cite{MR4638619} 
for the existence 
and non-existence of the ground state 
in the case of $d = 3$ and $1 < p \leq 3$ in detail). 
\end{enumerate}
\end{remark}

Next, we pay our attention to a positive solution 
which is different from the ground state 
and characterized by the following minimization problem: 
\begin{equation} \label{eq-mini-e}
E_{\min}(m) := \inf\left\{
\mathcal{E} (u) \colon 
u \in \mathcal{P}(m) \right\}, 
\end{equation} 
where $m > 0$, 
\[
S(m) := \left\{ 
u \in H^{1}(\R^{d}) \colon \|u\|_{L^{2}}^{2} = m
\right\} 
\]
and 
\[
\mathcal{P}(m) := 
\left\{ 
u \in S(m) \colon \mathcal{K}(u) = 0
\right\} 
= \left\{u \in H^{1}(\R^{d}) \colon 
\|u\|_{L^{2}}^{2} = m, \; \mathcal{K}(u) = 0
\right\}. 
\] 

We can easily verify that for each $m>0$, 
$\mathcal{P}(m) \neq \emptyset$. 
Since this kind of the minimization problem is 
closely related to the global dynamics of solutions to 
the evolution equations \eqref{nls}, 
it is often used as well as $m_{\omega}$ 
(see \cite{MR4476243, MR3634031, MR4096725, MR4433054} and references 
therein). 
Here, we note that for 
the $L^{2}$-constrained minimization problem as \eqref{eq-mini-e}, 
the $L^{2}$-critical exponent 
$2_* = 1 + 4/d$ plays an important role. 
To explain this, let us consider the 
following single power nonlinear Schr\"{o}dinger equations: 
\begin{equation} \label{nls-p}
i \p_{t} \psi + \Delta \psi + |\psi|^{p-1} \psi = 0
\qquad \mbox{in $\R \times\R^{d}$}, 
\end{equation}
where 
$d \geq 1$ and $1 < p < 2^{*} - 1$. 
Equation \eqref{nls-p} is invariant under 
the scaling: 
\begin{equation} \label{eq-scale}
\psi_{\lambda}(t, x) 
:= e^{\frac{2}{p-1} \lambda} \psi(e^{2 \lambda} t, 
e^{\lambda} x) 
\qquad 
(\lambda \in \R), 
\end{equation}
which preserves 
the $L^{2}$-norm when $2_* = 1 + 4/d$. 
Thus, the exponent 
$2_* = 1 + 4/d$
is referred to as \lq\lq $L^{2}$-critical'' or \lq\lq mass critical". 
The $L^{2}$-critical exponent $2_{*} = 1 + 4/d$ 
is also known as the threshold 
for the orbital stability/instability of standing waves of \eqref{nls-p}. 
The corresponding stationary equation is 
\begin{equation} \label{eq-sc}
- \Delta u + \omega u - |u|^{p-1}u = 0 
\qquad \mbox{in $\mathbb{R}^{d}$}, 
\end{equation}
whose unique positive solution
can be written as 
$U_{\omega}(\cdot) = 
\omega^{1/(p-1)} U^{\dagger} 
(\sqrt{\omega} \cdot)$, where 
$U^{\dagger}$ solves \eqref{eq-sc} 
with $\omega = 1$. 
We refer to \cite{MR695535} for the existence
and to \cite{MR969899} for the uniqueness. 
Cazenave and P. L. Lions~\cite{MR677997} proved that 
in the $L^{2}$-subcritical case $1 < p < 1 + 4/d$, $U^{\dagger}$ 
minimizes the corresponding energy under the $L^{2}$-norm constrained, 
that is, 
\begin{equation} \label{U-vari}
\E^{\dagger}(U^{\dagger}) = \inf\left\{\E^{\dagger}(u) 
\colon u \in S(\|U^{\dagger}\|_{L^{2}}^{2}) 
\right\}, 
\end{equation}
where 
\[
\E^{\dagger}(u) = 
\frac{1}{2} \|\nabla u\|_{L^{2}}^{2} - 
\frac{1}{p+1} \|u\|_{L^{p+1}}^{p+1}. 
\]
Using the variational characterization \eqref{U-vari} and the concentration-compactness method, 
Cazenave and P. L. Lions~\cite{MR677997} showed that the standing wave $e^{i t} U^{\dagger}$ 
is orbitally stable.~\footnote{
Roughly speaking, we say that 
the standing wave is orbitally stable if we take 
the initial data sufficiently close to the standing wave, 
then the corresponding solution stays close to its orbit. 
Otherwise, we say that the standing wave is orbitally unstable. 
} 
On the other hand, in the $L^{2}$-supercritical case $1 + 4/d < p < 2^{*} -1$,
the situation becomes different. 
Let us explain briefly. 
For each $u \in H^{1}(\R^{d})$ and $\lambda \in \R$, 
we put
\begin{equation} \label{l2-scale}
T_{\lambda} u(\cdot) := e^{\frac{d}{2} \lambda} 
u(e^{\lambda} \cdot) \qquad (\lambda \in \R). 
\end{equation}
Then, we can easily find that 
$\|T_{\lambda} U^{\dagger}\|_{L^{2}} 
= \|U^{\dagger}\|_{L^{2}}$, 
while $\inf_{\lambda \in \R} \E^{\dagger} (T_{\lambda} U^{\dagger}) 
= -\infty$, that is, 
the energy $\E^{\dagger}$ 
is not bounded from below on the manifold 
$S(\|U^{\dagger}
\|_{L^{2}}^{2})$ when $1 + 4/d < p < 
2^{*} -1$. 
Therefore, 
\eqref{U-vari} does not hold 
in the $L^{2}$-supercritical case $1 + 4/d < p < 2^{*} -1$. 
In addition, Berestycki and Cazenave~\cite{MR646873} proved that the 
standing wave $e^{i t} U^{\dagger}$ 
is orbitally unstable. 
This explains that 
the $L^{2}$-critical exponent $p = 1 + 4/d$ 
plays a key role in 
the $L^{2}$-constrained minimization problem. 

Now, let us return to our equations \eqref{sp} again. 
We now summarize the result for the $L^{2}$-critical and supercritical cases. 
Soave~\cite{MR4096725} 
obtained the following (see also \cite[Theorem 1,3]{MR4706031}): 
\begin{theorem}[Theorem 1.1 and Proposition 1.5 of Soave~\cite{MR4096725}]
\label{thm-ex-1}
\begin{enumerate}
\item[\textrm{(i)}]
Let
$d \geq 3$ 
and $1 + 4/d < p < 2^{*} -1$.
For each $m > 0$, $E_{\min}(m) > 0$ and 
there exists a minimizer for $E_{\min}(m)$, which 
is a positive solution to \eqref{sp} 
with the frequency 
$\omega = \omega(m)$ for some 
$\omega(m) > 0$.
\item[\textrm{(ii)}]
Let $d \geq 3$ 
and $p = 1 + 4/d$. 
For each $m \in (0, \|U^{\dagger}\|_{L^{2}}^{2})$, 
the same conclusion as in \textrm{(i)} holds. 
\end{enumerate} 
\end{theorem}
%

Next, we turn our attention to the $L^{2}$-subcritical case 
$1 < p < 1 + 4/d$. 
In fact, the situation becomes 
more complex and 
the functional $\E|_{S_{m}}$ has a second critical point. 
To explain this more precisely, we prepare 
several notations. 
Jeanjean~\cite{MR1430506} introduced 
the following \textit{fiber maps} to study the general 
$L^{2}$-constrained minimization problem: 
\[
\Psi(u, \lambda) 
:= \E(T_{\lambda} u) 
= \frac{e^{2 \lambda}}{2} \|\nabla u\|_{L^{2}}^{2} 
- \frac{e^{\frac{d}{2} (p-1) \lambda}}{p+1} 
\|u\|_{L^{p+1}}^{p+1} 
- \frac{e^{2^{*}\lambda}}{2^{*}} 
\|u\|_{L^{2^{*}}}^{2^{*}}. 
\]
We can easily verify that 
$\K(u) = \frac{d \Psi}{d \lambda}(u, \lambda)\big|_{\lambda = 0}$ and 
$\mathcal{P}(m)$ 
can be rewritten by 
\[
\mathcal{P}(m) = \left\{u \in S(m) \colon 
\frac{d \Psi}{d \lambda}(u, \lambda)\big|_{\lambda = 0} = 0 \right\}.
\]
Note that $d(p-1)/2 < 2$ when 
$1 < p < 1 + 4/d$, so that the 
graph of $\Psi(u, \lambda)$ may have 
local maximum and minimum points 
if $m = \|u\|_{L^{2}}^{2} > 0$ 
is sufficiently small (see \cite[Lemma 4.1]{MR4096725}). 
Taking this into consideration, 
we decompose $\mathcal{P}(m)$ into the disjoint union 
$\mathcal{P}(m) = \mathcal{P}_{+}(m) \cup 
\mathcal{P}_{0}(m) \cup \mathcal{P}_{-}(m)$, 
where 
\[
\begin{split} 
\mathcal{P}_{+}(m) 
& := 
\left\{ 
u \in S(m) \colon 
\frac{d \Psi}{d \lambda}(u, \lambda)\biggl|_{\lambda = 0} = 0, \; 
\frac{d^{2} \Psi}{d \lambda^{2}}(u, \lambda)\biggl|_{\lambda = 0} > 0
\right\} \\[6pt]
& = \left\{
u \in \mathcal{P}(m) \colon 
2\|\nabla u\|_{L^{2}}^{2} 
> \frac{d^{2}(p-1)^{2}}{4 (p+1)} 
\|u\|_{L^{p+1}}^{p+1}
+ 2^{*} \|u\|_{L^{2^{*}}}^{2^{*}}
\right\}, \\[6pt]
\mathcal{P}_{0}(m) 
& := 
\left\{ 
u \in S(m) \colon 
\frac{d \Psi}{d \lambda}(u, \lambda)\biggl|_{\lambda = 0} = 0, \; 
\frac{d^{2} \Psi}{d \lambda^{2}}(u, \lambda)\biggl|_{\lambda = 0} = 0
\right\} \\[6pt]
& = \left\{
u \in \mathcal{P}(m) \colon 
2\|\nabla u\|_{L^{2}}^{2} 
= \frac{d^{2}(p-1)^{2}}{4 (p+1)} 
\|u\|_{L^{p+1}}^{p+1}
+ 2^{*} \|u\|_{L^{2^{*}}}^{2^{*}}
\right\}, \\[6pt]
\mathcal{P}_{-}(m) 
& := 
\left\{ 
u \in S(m) \colon 
\frac{d \Psi}{d \lambda}(u, \lambda)\big|_{\lambda = 0} = 0, \; 
\frac{d^{2} \Psi}{d \lambda^{2}}(u, \lambda)\biggl|_{\lambda = 0} < 0
\right\} \\[6pt]
& = \left\{
u \in \mathcal{P}(m) \colon 
2\|\nabla u\|_{L^{2}}^{2} 
< \frac{d^{2}(p-1)^{2}}{4 (p+1)} 
\|u\|_{L^{p+1}}^{p+1}
+ 2^{*} \|u\|_{L^{2^{*}}}^{2^{*}}
\right\}. 
\end{split}
\] 
We put 
\begin{equation} \label{mini-sec-mp}
E_{\min, \pm}(m) := \inf\left\{ 
\E(u) \colon u \in \mathcal{P}_{\pm}(m)
\right\}. 
\end{equation} 
Then, Soave~\cite{MR4096725}, Jeanjean and Lu~\cite{MR4476243}, 
Wei and Wu~\cite{MR4433054} obtained the following:
\begin{theorem}[\cite{MR4476243, MR4096725, MR4433054}]\label{thm-ex2}
Let $d \geq 3$ and $1 < p < 1 + 4/d$. 
\begin{enumerate}
\item[\textrm{(i)}]
There exists $m_{*} > 0$ such that 
for each $m \in (0, m_{*})$, 
$E_{\min, -}(m)$ is positive and has a minimizer which 
is a positive solution to \eqref{sp} 
with $\omega = \omega_{-}(m)$ for some 
$\omega_{-}(m) > 0$.
\item[\textrm(ii)]
There exists $m_{**} > 0$ such that 
for each $m \in (0, m_{**})$, 
\[
E_{\min}(m) = E_{\min, +}(m)< 0. 
\]
In particular, 
$E_{\min}(m)$ has a minimizer, which 
is a positive solution to 
\eqref{sp} with $\omega = 
\omega_{+}(m)$ for some 
$\omega_{+}(m) > 0$.
\end{enumerate}
\end{theorem}
\begin{remark}
It follows from Theorem \ref{thm-ex2} that 
$E_{\min, -}(m) > 0 > E_{\min}(m)$. 
\end{remark}
Part \textrm{(i)} of 
Theorem \ref{thm-ex2} 
was proved by Jeanjean and Lu~\cite[Theorem 1.6]{MR4476243} in the case of $d \geq 4$ and
by Wei and Wu~\cite[Theorem 1.4]{MR4433054} 
in the case of $d = 3$
and Part \textrm{(ii)} 
was proved by  
Soave~\cite[Theorem 1.1 and Proposition 1.5 ]{MR4096725}.
In what follows, 
we denote the minimizers for $E_{\min}(m)$ and 
$E_{\min, -}(m)$
by 
$R_{\omega(m)}$ and 
$R_{\omega_{-}(m), -}$ 
which are solutions to \eqref{sp} with $\omega = \omega(m)$ 
and $\omega = \omega_{-}(m)$, respectively.~\footnote{
When we would like to stress the dependence on 
the frequency $\omega$, we denote 
just by $R_{\omega}$ and $R_{\omega, -}$. 
} 
Wei and Wu~\cite{MR4638619} investigated the case where 
$d = 3$ and $m > 0$ is sufficiently small.
The following theorem 
summarizes their results: 

\begin{theorem}[Theorem 1.2 (3) of Wei and Wu~\cite{MR4638619}] \label{thm-asy}
\begin{enumerate}
\item[\textrm{(i)}]
Let $d = 3$, $1 + 4/d \leq 
p < 3$ and 
$\omega(m) > 0$ be the number given in 
Theorem \ref{thm-ex-1} 
and \ref{thm-ex2} \textrm{(i)}. 
Then, we have $\lim_{m \to 0} \omega(m) = 0$ 
and $\lim_{m \to 0} \|R_{\omega(m)}\|_{L^{\infty}} = \infty$. 
\item[\textrm{(ii)}]
Let $d = 3$, $1 < p < 1 + 4/d$ and 
$\omega_{-}(m) > 0$ be the number given in 
Theorem \ref{thm-ex2} \textrm{(ii)}. 
Then, we have $\lim_{m \to 0} \omega_{-}(m) = 0$ 
and $\lim_{m \to 0} \|R_{\omega_{-}(m), -}
\|_{L^{\infty}} = \infty$. 
\end{enumerate}
\end{theorem} 
\begin{remark}
\begin{enumerate}
\item[\textrm{(i)}]

Let $Q_{\omega(m)}$ denote the ground state 
to \eqref{sp} when $\omega = \omega(m)$. 
We find 
that $R_{\omega(m)} \neq Q_{\omega(m)}$ 
for sufficiently small $m>0$ and 
$1 + 4/d \leq p < 2^{*} -1$.
Indeed, 
    \[
    \limsup_{m \to 0} \|Q_{\omega(m)}\|_{L^{\infty}} 
< \infty, \qquad  
\lim_{m \to 0} \|R_{\omega(m)}\|_{L^{\infty}} = \infty.
    \]
Hence, the uniqueness of positive solution to 
\eqref{sp} fails 
for sufficiently small $\omega > 0$. 
Similarly, the non-uniqueness also occurs 
when $1 < p < 1 + 4/d$.
\item[\textrm{(ii)}]
Contrary to the case where $d = 3$ and $1 + 4/d < p < 3$, 
in \cite[Proposition C.1]{MR3964275}, it is shown that 
\eqref{sp} admits 
a unique positive solution for any 
$\omega > 0$ when $3 \leq d \leq 6$ and $\frac{4}{d-2} < p < 2^{*} -1$. 
\end{enumerate}
\end{remark}

\section{Main results}
In this section, we state our main results. 
One of our aim is to classify all positive  solutions 
to \eqref{sp} for sufficiently small $\omega > 0$ 
when $d = 3$. 
Let 
$\{u_{1, \omega}\}$ and $\{u_{2, \omega}\}$
be two families of positive solutions
to \eqref{sp} in $H^{1}(\R^{d})$ 
satisfying $\limsup_{\omega \to 0} \|u_{1, \omega} \|_{L^{\infty}} < \infty$ 
and $\lim_{\omega \to 0} \|u_{2, \omega} \|_{L^{\infty}} = \infty$, 
respectively. 
We shall show that 
when $d = 3$, 
$u_{1, \omega}$ must be the ground state $Q_{\omega}$ while 
$u_{2, \omega}$ be 
the minimizer $R_{\omega}$ of 
$E_{\min}(m)$ when $1 + \frac{4}{d} 
< p < 3$ and the minimizer $R_{\omega_{-}, -}$ of $E_{\min, -}$ 
when $1 < p < 1 + \frac{4}{d}$
for sufficiently small $\omega > 0$. 
As a result, we can verify 
\eqref{sp} has 
exactly two positive solutions. 

The second our aim is to study the non-degeneracy and the Morse index 
of positive solutions to \eqref{sp}
for sufficiently small $\omega > 0$. 
Here, we recall that 
$u_{\omega}$ is 
\textit{non-degenerate} in $H^{1}_{\text{rad}}(\R^{d})$ if 
the linearized equation 
\[
L_{\omega, +} z = 0 
\qquad \mbox{in $\R^{d}$}, 
\]
has no non-trivial solution 
in $H^{1}_{\text{rad}}(\R^{d})$, 
where 
\[
L_{\omega, +} = 
- \Delta + \omega - p 
u_{\omega}^{p-1} 
- (2^{*} - 1) 
u_{\omega}^{2^{*} - 2}. 
\]
The Morse index of 
the solution $u_{\omega}$ is defined by 
\[
\begin{split}
\max \biggl\{
\dim H \colon 
\mbox{$H$ is a subspace 
of $H^{1}(\R^{d})$ satisfying $\langle L_{\omega, +} h, h \rangle 
< 0$ for all $h \in H \setminus 
\{0\}$}
\biggl\}. 
\end{split}
\]
The study of 
the non-degeneracy 
and the Morse index of 
solutions to \eqref{sp} becomes important
in the analysis of 
the dynamics of solutions 
to \eqref{nls}. 
In fact, for each $m > 0$, 
if the energy $\E$ of an 
initial data $\psi_{0}$ with 
$\|\psi_{0}\|_{L^{2}}^{2} = m$
is less than $E_{\min}(m)$, 
then the corresponding
solution $\psi$ to 
\eqref{nls} blows up or scatters forward in time 
(see e.g. \cite{MR2917888, MR3090237, MR3634031, MR4096725}). 
On the other hand, when the
energy of initial data is equal to or slightly greater than the ground state energy, 
the dynamics of solutions becomes more complicated. 
To study this case, the non-degeneracy and the Morse index are needed. 
We refer to \cite{MR4320767, MR2491692, MR2898769} for more details. 
This is one of our main motivations for the present
paper.
We note that non-degeneracy of the solution to \eqref{sp} is also key to apply the Lyapunov-Schmidt reduction method
(see e.g. \cite{MR2186962, MR867665, MR970154} and references therein). 

We obtain the following classification result of 
$H^{1}$-positive solutions to \eqref{sp}: 
\begin{theorem}\label{class}
Let $d = 3$. 
\begin{enumerate}
\item[\textrm{(i)}]
Let $7/3 \leq p < 3$. 
There exists $\omega_{1} > 0$ such that 
for any $\omega \in (0, \omega_{1})$, 
\eqref{sp} has exactly two positive solutions 
in $H^{1}(\R^{d})$ up to spatial translations; 
one is the ground state $Q_{\omega}$ and 
the other is the minimizer $R_{\omega}$ of 
$E_{\min}(m)$ for some $m > 0$. 
\item[\textrm{(ii)}]
Let $1 < p < 7/3$. 
There exists $\omega_{2} > 0$ such that 
for any $\omega \in (0, \omega_{2})$, 
\eqref{sp} has exactly two positive solutions 
in $H^{1}(\R^{d})$ up to spatial translations; 
one is the ground state $Q_{\omega}$ and 
the other is the minimizer $R_{\omega_{-}(m), -}$ of 
$E_{\min}(m)$ for some $m > 0$. 
\end{enumerate}
\end{theorem}
To prove Theorem \ref{class}, we consider two types of families of positive solutions 
$\{u_{1, \omega}\}$ 
and $\{u_{2, \omega}\}$
to \eqref{sp} in $H^{1}(\R^{d})$ 
satisfying 
$\limsup_{\omega \to 0} 
\|u_{1, \omega}\|_{L^{\infty}} < \infty$ 
and $\lim_{\omega \to 0} 
\|u_{2, \omega}\|_{L^{\infty}} = \infty$, respectively. 

Concerning the family of positive solutions 
$\{u_{1, \omega}\}$, 
we shall show the following: 
\begin{theorem}\label{thm-0-1}
Let $d = 3$ and $1 < p < 2^{*}-1$. 
Let $\{u_{1, \omega}\}$ be a family of 
positive solutions to \eqref{sp} satisfying 
$\limsup_{\omega \to 0} 
\|u_{1, \omega}\|_{L^{\infty}} < \infty$. 
Putting 
\begin{equation} \label{main-eq1}
\widetilde{u}_{1, \omega}(\cdot) = 
\omega^{- \frac{1}{p-1}}
u_{1, \omega}
(\omega^{- \frac{1}{2}} \cdot), 
\end{equation}
we have 
\begin{equation*}
\widetilde{u}_{1, \omega} \to U^{\dagger} 
\qquad \mbox{strongly in 
$H^{1}(\R^{d})$ as $\omega 
\to 0$}.
\end{equation*} 
\end{theorem}
\begin{theorem}\label{thm-0}
Let $d = 3$ and $1 < p < 2^{*}-1$.

\begin{enumerate}
\item[\textrm{(i)}]
There exists $\omega_{3} > 0$ 
with the following properties: let 
$\{u_{1, \omega}\}$ 
denote a family of positive solutions to 
\eqref{sp} satisfying 
$\limsup_{\omega \to 0} 
\|u_{1, \omega}\|_{L^{\infty}} < \infty$. Then the solution is unique for all 
$\omega \in (0, \omega_{3})$,
namely, if there exists another 
family of positive solutions $\{v_{1, \omega}\}$
to \eqref{sp} in $H^{1}(\R^{d})$ 
satisfying 
$\limsup_{\omega \to 0} 
\|v_{1, \omega}\|_{L^{\infty}} < \infty$, 
then 
$u_{1, \omega} = v_{1, \omega}$ 
for all $\omega \in (0, \omega_{3})$. 
\item[\textrm{(ii)}] 
There exists $\omega_{4} > 0$ such that 
for any $\omega \in (0, \omega_{4})$, 
the Morse index of $u_{1, \omega}$ equals 
$1$. 
\end{enumerate}
\end{theorem}
\begin{remark}
\begin{enumerate}
\item[\textrm{(i)}] 
The main difficulty for the proof of Theorem \ref{thm-0-1} is 
to obtain a boundedness of $H^{1}$ norm of $\{\widetilde{u}_{1, \omega}\}$. 
For example, when $Q_{\omega}$ is 
the ground state, 
one can easily obtain a uniform boundedness of $\{\|Q_{\omega}\|_{H^{1}}\}$ 
from the fact that $Q_{\omega}$ is a minimizer for $m_{\omega}$. 
However, since we consider a 
general positive solution to \eqref{sp}, 
we cannot employ a variational argument 
like the ground state. 
To overcome this difficulty, 
we consider another rescaling different from 
\eqref{main-eq1} (see \eqref{S3-5} below).
and employ a kind of 
Pohozaev identity of a limit equation. 
\item[\textrm{(ii)}] 
It follows from Theorem \ref{thm-0} \textrm{(i)} that 
$u_{1, \omega} = Q_{\omega}$ for 
sufficiently small $\omega >0$. 
\item[\textrm{(iii)}] 
We do not assert the 
novelty of Theorem \ref{thm-0} 
\textrm{(i)} and \textrm{(ii)}.
These assertions can be derived by 
rather standard arguments, 
specifically from
\cite[Proposition 2.0.4]{MR4320767} in 
the case of Theorem \ref{thm-0} 
\textrm{(i)} and
\cite[Lemma 2.3]{MR2756065} in the case of Theorem \ref{thm-0} \textrm{(ii)}.
Nevertheless, we include them here in order to enable a direct comparison with the corresponding results for $u_{2,\omega}$
(see Theorem \ref{thm-bl} \textrm{(ii)} and \textrm{(iii)} below).
\end{enumerate}
\end{remark}
Next, we study a family of positive solutions 
$\{u_{2, \omega}\}$ to \eqref{sp} in $H^{1}(\R^{d})$ 
satisfying 
$\lim_{\omega \to 0} 
\|u_{2, \omega}\|_{L^{\infty}} = \infty$. 
Concerning this, we obtain the following: 
\begin{theorem}\label{thm-bl-0}
Let $d = 3$ and $1 < p < 3$.  
Let $\{u_{2, \omega}\}$ be a 
family of positive solutions to \eqref{sp} in $H^{1}(\R^{d})$ 
satisfying 
$\lim_{\omega \to 0} 
\|u_{2, \omega}\|_{L^{\infty}} = \infty$. 
Putting 
\begin{equation} \label{main-eq2}
M_{2, \omega} = \|u_{2, \omega}\|_{L^{\infty}}, 
\qquad 
\widetilde{u}_{2, \omega}(\cdot)
= M_{2, \omega}^{-1} 
u_{2, \omega} (M_{2, \omega}^{-\frac{2}{d-2}} \cdot), 
\end{equation}
we have 
\begin{equation}\label{main-eq3}
\widetilde{u}_{2, \omega} \to W 
\qquad \mbox{strongly in 
$\dot{H}^{1} \cap L^{q}(\R^{d})\; 
(q > \frac{d}{d - 2})$ as $\omega 
\to 0$}, 
\end{equation}
where $W$ is the Aubin-Talenti function defined by 
\eqref{talenti}. 
\end{theorem}
\begin{theorem}\label{thm-bl}
Let $d = 3$ and $1 < p < 3$.  
\begin{enumerate}
\item[\textrm{(i)}]
There exists $\omega_{5} > 0$ 
with the following properties: let 
$\{u_{2, \omega}\}$ 
denote a family of positive solutions to 
\eqref{sp} satisfying 
$\lim_{\omega \to 0} 
\|u_{2, \omega}\|_{L^{\infty}} = \infty$. Then the solution is unique for all 
$\omega \in (0, \omega_{5})$,
namely, if there exists another 
family of positive solutions $\{v_{2, \omega}\}$
to \eqref{sp} in $H^{1}(\R^{d})$ 
satisfying 
$\lim_{\omega \to 0} 
\|v_{2, \omega}\|_{L^{\infty}} = \infty$, 
then 
$u_{2, \omega} = v_{2, \omega}$ 
for all $\omega \in (0, \omega_{5})$. 
\item[\textrm{(ii)}]
There exists $\omega_{6} > 0$ such that 
when $1 < p < 3$, 
the solution $u_{2, \omega}$ 
is non-degenerate for all $\omega \in (0, \omega_{6})$.
\item[\textrm{(iii}]
Let $d = 3$ and $1 
+ \frac{4}{d}
< p < 3$, there exists $\omega_{7} > 0$ such that for all $\omega \in (0, \omega_{7})$,
the Morse index of $u_{2,\omega}$ equals $2$.
\end{enumerate}
\end{theorem}

\begin{remark}\label{rem-thm-bl}
\begin{enumerate}

\item[\textrm{(i)}]
As in Theorem \ref{thm-bl} \textrm{(i)}, 
we also meet a difficulty to obtain a boundedness of 
$\dot{H}^{1}$ norm of $\{\widetilde{u}_{2, \omega}\}$. 
Following \cite{MR1912264, MR1022990, MR4638619}, 
which is originated in \cite{MR851145}, 
we prove a pointwise estimate of $\widetilde{u}_{2, \omega}$, 
which is a key role to obtain the boundedness. 
\item[\textrm{(ii)}]
For the proof of Theorem \ref{thm-bl} \textrm{(ii)},
we follow the arguments in \cite{MR3964275} and \cite{MR4572464} (see also \cite{HKW}).
In \cite{MR3964275}, we established the uniqueness and non-degeneracy of the ground state $Q_{\omega}$
of \eqref{sp} when $d \geq 5$ and $\omega > 0$ is sufficiently large.
The restriction $d \geq 5$ was required there in order to ensure that the Aubin–Talenti function $W$ belongs to $L^{2}(\mathbb{R}^{d})$.
Subsequently, Akahori and Murata~\cite{MR4572464} extended the result of \cite{MR3964275},
proving the non-degeneracy of the ground state $Q_{\omega}$ of \eqref{sp} when $d=3$ and $3 < p < 2^{*}-1$ for sufficiently large $\omega > 0$.
However, the proof in \cite{MR4572464} essentially relies on the condition $p > 3$ and does not apply, at least directly, when $d=3$ and $1 < p < 3$,
so that additional ingredients are required.
In the present work, we employ the pointwise estimate of $\widetilde{u}_{2,\omega}$ mentioned in Remark \ref{rem-thm-bl} \textrm{(i)}.
\item[\textrm{(iii)}] 
It follows from Theorem \ref{thm-bl} \textrm{(ii)}, 
$\lim_{m \to 0} \omega(m) = 0$ and $\lim_{m \to 0} \omega_{-}(m) = 0$
that $u_{2, \omega(m)} = R_{\omega(m)}$ in the 
$L^{2}$-critical and supercritical case 
$1 + 4/d \leq p < 3$
and 
$u_{2, \omega(m)} = R_{\omega_{-}(m), -}$ 
in the $L^{2}$-subcritical case 
$1 < p < 1 + 4/d$
for sufficiently small $m>0$. 
\item[\textrm{(iv)}]
When $d = 3$ and $1 + \tfrac{4}{d} \leq p < 3$,
it follows from the uniqueness result (Theorem \ref{thm-bl} \textrm{(ii)}) that $u_{2,\omega(m)} = R_{\omega(m)}$.
Hence, by the variational characterization \eqref{eq-mini-e},
the Morse index of $u_{2,\omega}$ is either $1$ or $2$ (see Lemma \ref{mor-lem1}).
Therefore, it remains to show that the Morse index of $u_{2,\omega}$ is not equal to $1$.
To this end, we invoke the abstract theory of Grillakis, Shatah, and Strauss~\cite{MR901236},
which provides sufficient conditions for the stability and instability of the standing wave solution $e^{i \omega t} u_{\omega}$ under certain assumptions.
One of these assumptions is that the Morse index of $u_{\omega}$ equals $1$.
In fact, applying the result of \cite{MR901236}, we obtain a contradiction if we assume that 
the Morse index of 
$u_{\omega}$ equals $1$.
This line of reasoning appears to be new in the study of the Morse index of solutions to nonlinear elliptic equations.
\end{enumerate}
\end{remark}
\begin{center}
\begin{tabular}{|c|c|c|}
\hline
& $1 < p < \frac{7}{3}$
& $\frac{7}{3} \leq p < 3$
\\
\hline
$u_{1, \omega}$ 
& \multicolumn{2}{c|}{$Q_{\omega}$} 
\\
\hline
$u_{2, \omega}$ 
& $R_{\omega_{-}, -}$ 
& $R_{\omega}$
\\
\hline
\end{tabular}
\end{center}

We shall prove Theorem \ref{thm-bl} 
by the blowup analysis. 
Note that 
$\widetilde{u}_{2, \omega}$ defined by \eqref{main-eq2} 
satisfies 
    \begin{equation} \label{main-eq41}
    - \Delta \widetilde{u}_{2, \omega}  
    + \alpha_{2, \omega} \widetilde{u}_{2, \omega}
- \beta_{2, \omega} \widetilde{u}_{2, \omega}^{p} 
- \widetilde{u}_{2, \omega}^{5}
= 0, 
\qquad 
\|\widetilde{u}_{2, \omega}\|_{L^{\infty}} 
= \widetilde{u}_{2, \omega}(0) =1,  
    \end{equation}
where 
\begin{equation}\label{main-eq5}
\alpha_{2, \omega} 
:= \omega M_{2, \omega}^{-4}, \qquad 
\beta_{2, \omega} 
:= M_{2, \omega}^{p-5}.
\end{equation}
Since $\widetilde{u}_{2, \omega}$ 
is radially symmetric, \eqref{main-eq41} 
can be written by the following: 
\begin{equation} \label{main-eq412}
\begin{cases}
- \widetilde{u}_{2, \omega}^{\prime \prime} 
- \frac{2}{r} \widetilde{u}_{2, \omega}^{\prime }
+ \alpha_{2, \omega} \widetilde{u}_{2, \omega}
- \beta_{2, \omega} \widetilde{u}_{2, \omega}^{p} 
- \widetilde{u}_{2, \omega}^{5}
= 0 
\qquad 
\text{on $(0, \infty)$}, \\[6pt]
\widetilde{u}_{2, \omega}(0) =1, 
\qquad 
\widetilde{u}_{2, \omega}^{\prime}(0) =0.  
& 
\end{cases}
\end{equation}
It follows from the definition of \eqref{main-eq5} 
and $\lim_{\omega \to 0} M_{2, \omega} = \infty$ that 
\begin{equation} \label{main-eq6}
\lim_{\omega \to 0} \alpha_{2, \omega} =
\lim_{\omega \to 0} \beta_{2, \omega} = 0.
\end{equation} 
From this, we can find that 
the limit equation of \eqref{main-eq41} 
is \eqref{eq-at} which the Aubin-Talenti 
function \eqref{talenti} satisfies. 
However, 
as in Theorem \ref{thm-0-1}, 
we also face a difficulty obtaining the boundedness of the
$H^{1}$ norm of $\{\widetilde{u}_{2, \omega}\}$ to prove the convergence 
\eqref{main-eq3}. 
Following \cite{MR1912264, MR1022990, MR4638619}, 
which is originated in \cite{MR851145}, 
we prove a pointwise estimate of $\widetilde{u}_{2, \omega}$ (see \eqref{EqL-1} below), 
which is a key role in obtaining the boundedness. 

Contrary to the solution $U^{\dagger}$ to 
\eqref{eq-sc} with $\omega = 1$, 
the Aubin-Talenti function is degenerate. 
Thus, we cannot obtain the uniqueness result 
(Theorem \ref{thm-bl} \textrm{(i)}) directly. 
In \cite{MR3964275}, in order to 
get over the difficulty, 
we employed the 
Pohozaev identity, that is, 
the solution 
$\widetilde{u}_{2, \omega}$ to \eqref{main-eq41} 
satisfy the following: 
\begin{equation} \label{main-eq7}
\alpha_{2, \omega}\|\widetilde{u}_{2, \omega}\|_{L^{2}}^{2} 
= \frac{5 - p}{2 (p+1)} \beta_{2, \omega} 
\|\widetilde{u}_{2, \omega}\|_{L^{p+1}}^{p+1}. 
\end{equation}
In \cite{MR3964275}, we considered 
the case of $d \geq 5$, so that 
$W \in L^{2}(\R^{d})$. 
From this and \eqref{main-eq3}, 
we can take a limit of $\omega$ 
in \eqref{main-eq7} directly and showed that 
\begin{equation*} 
\lim_{\omega \to 0} 
\frac{\beta_{2, \omega}}{\alpha_{2, \omega}} 
= \frac{2(p + 1)}{4 - (d - 2)(p - 1)} 
\frac{\|W\|_{L^{2}}^{2}}{
\|W\|_{L^{p + 1}}^{p + 1}} 
\end{equation*}
when $d \geq 5$. 
This relation between 
$\alpha_{2, \omega}$ and 
$\beta_{2, \omega}$ plays an important 
role for the proof of the uniqueness in \cite{MR3964275}. 
However, when we consider three 
dimensional case, we see that 
$W \not\in L^{2}(\R^{3})$. 
Thus, the relation between 
$\alpha_{2, \omega}$ and 
$\beta_{2, \omega}$ 
becomes different. 
For this reason, the uniqueness result cannot be obtained directly in the case $d = 3$.
Moreover, we note that $W \notin L^{p+1}(\mathbb{R}^{3})$ when $1 < p \leq 2$.
Thus, the situation becomes more complex 
in the case of $d = 3$ and $1 < p \leq 2$. 

Coles and Gustafson~\cite{MR4162293} 
obtained the uniqueness of 
the ground state $Q_{\omega}$ 
to \eqref{sp}
if $\omega > 0$ is sufficiently large 
when $d = 3$ and $3 < p < 2^{*} - 1$. 
They used the resolvent expansion 
by Jensen and Kato~\cite{MR544248}. 
As Coles and Gustafson~\cite{MR4162293} did, 
we use the resolvent expansion and 
show the following: 
\begin{align} 
& \lim_{\omega \to 0} 
\frac{\beta_{2, \omega}} {\sqrt{\alpha_{2, \omega}}}
= \frac{12 \pi (p+1)}{(5 - p)\|W\|_{L^{p+1}}^{p+1}} 
\qquad \mbox{when $2 < p < 3$}, 
\label{main-eq9}\\
& 
\lim_{\omega \to 0} 
\frac{\beta_{2, \omega} |\log \alpha_{2, \omega}|}
{\sqrt{\alpha_{2, \omega}}} = 
\frac{2}{\sqrt{3}}
\qquad \mbox{when $p = 2$}. 
\label{main-eq10} 
\end{align}
See Proposition \ref{PropR-1} \textrm{(i)} and \textrm{(ii)} below.
Since $W \notin L^{p}(\mathbb{R}^{3})$ for $1 < p \leq 2$,
it appears natural from \eqref{main-eq3} and \eqref{main-eq7} that
$p = 2$ serves as the threshold for the relation between $\alpha_{2,\omega}$ and $\beta_{2,\omega}$.
Using \eqref{main-eq9}, \eqref{main-eq10} 
and the pointwise estimate 
of $\widetilde{u}_{2,\omega}$ (see 
\eqref{EqL-1} below, 
we show the uniqueness and non-degeneracy of 
$\widetilde{u}_{2, \omega}$ for sufficiently large 
$\omega > 0$.

Next, we pay our attention to the case of 
$1 < p < 2$. 
As we mentioned above, since 
$W \not\in L^{p + 1}(\R^{3})$ when 
$1 < p \leq 2$, 
the situation becomes more complex. 

Let $\{\omega_{n}\}$ be 
a sequence in $(0, \infty)$ 
with $\lim_{n \to \infty} \omega_{n} = 0$. 
We see that 
there exists a constant 
$\theta_{0} > 0$ such that 
\[
\lim_{n \to \infty} 
\frac{\beta_{2, \omega_{n}}}
{\alpha_{2, \omega_{n}}
^{\frac{3 - p}{2}}} = \theta_{0} 
\qquad \mbox{when $1 < p < 2$}.  
\]
See Lemmas \ref{LemR-15} and 
\ref{LemR-15-1} below. 
We remark that the limit $\theta_{0} > 0$ depends 
on the sequence 
$\{\omega_{n}\}$ at this stage. 
However, we can find that $\theta_{0}$ is a universal constant. 
To prove this, we consider the following 
rescaling: 
\[
\widehat{u}_{2, \omega_{n}} 
(s)
= \alpha_{2, \omega_{n}}^{-\frac{1}{2}} 
\widetilde{u}_{2, \omega_{n}}(r), 
\qquad s = \alpha_{\omega_{n}}^{\frac{1}{2}} r. 
\]
Then, we see that $\widehat{u}_{2, \omega_{n}}$ 
satisfies the following: 
	\begin{equation*}
	- \frac{d^{2} \widehat{u}_{2, \omega_{n}}}
    {d s^{2}} - 
	\frac{2}{s} \frac{d \widehat{u}_{2, \omega_{n}}
    }{d s} 
	- \widehat{u}_{2, \omega_{n}}  - \frac{\beta_{2, \omega_{n}}}{\alpha_{2, 
    \omega_{n}}^{\frac{3 - p}{2}}}
	\widehat{u}_{2, \omega_{n}} ^{p} 
    - \alpha_{2, \omega_{n}} 
    \widehat{u}_{2, \omega_{n}}^{5} = 0 
	\qquad \mbox{in $(0, \infty)$}. 
	\end{equation*}
For any finite interval $I \subset (0, \infty)$, 
we have 
\[
\lim_{n \to \infty} 
\widehat{u}_{2, \omega_{n}}(s) 
= U_{\theta_{0}, \infty}(s) \qquad \text{uniformly in $I$}, 
\]
where $U_{\theta_{0}, \infty}$ is a solution to 
\begin{equation}\label{main-eq11}
- \frac{d^{2 } U}{d s^{2}} - \frac{2}{s} 
\frac{d U}{d s} + U - \theta_{0} U^{p} = 0 
\qquad \mbox{in $(0, \infty)$}
\end{equation} 
satisfying 
\begin{equation*} 
\lim_{s \to 0} s U_{\theta_{0}, \infty}(s) 
= \sqrt{3}. 
\end{equation*}
See Proposition \ref{conv-sing} below. 
It is known that 
the equation \eqref{main-eq11} 
has infinitely many singular solutions 
(see \cite[Theorem 1]{MR1313805}).
However, we can obtain a uniqueness 
of the singular solution $U_{\theta_{0}, \infty}$. 
Indeed, putting 
$\overline{u}_{2, \omega}(s) 
= s \widehat{u}_{2, \omega}(s)$. 
we find that 
$\overline{u}_{2, \omega}$ 
satisfies 
\begin{equation*}
\begin{cases}
\frac{d^2 \overline{u}_{2, \omega}}{d s^{2}} 
- \overline{u}_{2, \omega} 
+ \beta_{2, \omega} 
\alpha_{2, \omega}
^{\frac{p - 3}{2}} 
s^{1 - p} \overline{u}_{2, \omega}^{p} 
+ \alpha_{2, \omega} s^{-4} 
\overline{u}_{2, \omega}^{5} = 0 
\qquad \mbox{in $(0, \infty)$}, \\
\overline{u}_{2, \omega}(s) > 0
\qquad \mbox{in $(0, \infty)$}, \\
\overline{u}_{2, \omega}(0) = 0, 
\qquad \frac{d \overline{u}_{2, \omega}}{d s}(0) 
= \frac{1}
{\sqrt{\alpha_{2, \omega}}}, 
\qquad 
\lim_{s \to \infty} \overline{u}_{2, \omega}(s) 
= 0. &
\end{cases}
\end{equation*}
Then, we can prove that 
for any finite interval 
$I$ in $(0, \infty)$, one has 
$\lim_{n \to \infty} 
\theta_{0}^{\frac{1}{p-1}} 
\overline{u}_{\omega_{n}}(s) 
= \overline{v}_{0}(s)$ for $s \in I$, 
where $\overline{v}_{0}(s)$ 
satisfies 
\begin{equation}\label{main-eq12}
\begin{cases}
\frac{d^2 \overline{v}_{0}}{d s^2} - 
\overline{v}_{0} 
+ 
s^{1 - p} \overline{v}_{0}^{p} = 0 
& \qquad \mbox{on $(0, \infty)$}, \\
\frac{d \overline{v}_{0}}{d s}(0) = 0, \qquad 
\lim_{s \to \infty} \overline{v}_{0}(s) = 0 & 
\end{cases}
\end{equation}
and 
\begin{equation}\label{main-eq13}
\overline{v}_{0}(0) = 
\sqrt{3} \theta_{0}^{\frac{1}{p - 1}}, 
\end{equation}
See Lemmas \ref{lem-conv2} and 
\ref{LemR-20} below. 
Concerning \eqref{main-eq12}, the following result 
is obtained by Genoud~\cite{MR2656687} and 
Toland~\cite{MR751198}
\begin{proposition}[Genoud~\cite{MR2656687} 
and Toland~\cite{MR751198}]\label{thm-scud}
Let $1 < p < 2$. 
\begin{enumerate}
\item[\textrm{(i)}]
The equation \eqref{main-eq12} 
has a unique positive 
solution $V$. 
\item[\textrm{(ii)}]
The positive solution $V$ to 
the equation \eqref{main-eq12}
is non-degenerate. 
\end{enumerate}
\end{proposition}
See Toland~\cite[Page 262]{MR751198} 
for the proof of Proposition 
\ref{thm-scud} \textrm{(i)}
and Genoud~\cite[Proposition 2.1]{MR2656687}
for Proposition \ref{thm-scud} \textrm{(ii)}, respectively.  
It follows from \eqref{main-eq13} and 
Theorem \ref{thm-scud} that 
we obtain 
$V(0) = \sqrt{3} \theta_{0}^{\frac{1}{p - 1}}$, 
that is, 
$\theta_{0} = 3^{- \frac{p - 1}{2}} V^{p - 1}(0)$, 
which implies that $\theta_{0}$ does not depend on the 
sequence $\{\omega_{n}\}$ and is a universal constant. 
As a result, we can show the following: 
\begin{equation} \label{main-eq14}
\lim_{\omega \to 0} \frac{\beta_{2, \omega}}
{\alpha_{2, \omega}^{\frac{3 - p}{2}}} 
= 3^{- \frac{p-1}{2}} V^{p-1}(0) 
\qquad \mbox{when $1 < p < 2$}. 
\end{equation}
Using \eqref{main-eq14} and Proposition 
\ref{thm-scud} \textrm{(ii)}, we can prove 
Theorem \ref{thm-bl} \textrm{(i)} when 
$1 < p < 2$. 
We remark that 
Wei and Wu~\cite[Proposition 4.2]{MR4433054} 
have already 
obtained a result concerning
a relation between 
$\alpha_{2, \omega}$ and 
$\beta_{2, \omega}$. 
However, to prove the uniqueness, 
we need the exact values of the limits 
as in 
\eqref{main-eq9}, \eqref{main-eq10} and 
\eqref{main-eq14}. 

Using the argument in the proof of Theorem \ref{thm-bl-0}, 
we can show the non-existence of positive solution to \eqref{sp} 
for sufficiently large $\omega > 0$. 

\begin{theorem}\label{thm-none}
Let $d = 3$ and $1 < p < 3$. 
There exists $\omega_{8} > 0$ such that 
the equation \eqref{sp} does not have any positive 
solution in $H^{1}(\R^{3})$ for $\omega \in (\omega_{8}, \infty)$.
\end{theorem}

The rest of this paper is 
organized as follows: 
The proofs of Theorems 
\ref{thm-0-1} and \ref{thm-0} 
are presented in Section 
\ref{sec-bdd}. 
We prove the convergence result to the Aubin-Talenti function 
(Theorem \ref{thm-bl-0}) in Section \ref{sec-cat}. 
We study the relation between the coefficients 
$\alpha_{2, \omega}$ and $\beta_{2, \omega}$ (see \eqref{main-eq5} for the 
definition of the coefficients)
in Section \ref{section-R}. 
Using the result obtained in Section \ref{section-R}, 
we give proofs of the uniqueness and non-degeneracy 
of $u_{2, \omega}$ (Theorem \ref{thm-bl} \textrm{(i)} and \textrm{(ii)}) 
in Section \ref{sec-uni2} and 
Section \ref{sec-nd2}, respectively. 
We also give the proof of Theorem \ref{class} 
in Section \ref{sec-uni2}. 
In Section \ref{sec-uniR}, we show the uniqueness of 
the minimizer $R_{\omega(m)}$ for each $m > 0$. 
From the uniqueness of the minimizer $R_{\omega(m)}$, 
we obtain a regularity of $R_{\omega(m)}$ in Section 
\ref{sec-reg}. 
In Section \ref{sec-mi}, we shall 
study the Morse index of $R_{\omega(m)}$ and give a proof of 
Theorem \ref{thm-bl} \textrm{(iii)}. 
We give the proof of Theorem 
\ref{thm-none} in Appendix \ref{sec-a}. 
In Appendix \ref{section:B}, 
we obtain a boundedness of the $L^{\infty}$ 
norm of solutions to some elliptic equations, 
which is needed in Section \ref{sec-uniR}. 
In Apendix \ref{section:C}, 
we list a table of symbols and their 
descriptions used in this papers. 

In what follows, we will always fix $d = 3$ unless otherwise noted. 
Furthermore, we consider only the real-valued functions.

\subsection{Notation}
\begin{enumerate}
\item[\textrm{(i)}]
~We use the symbol $( \cdot, \, \cdot)_{L^{2}}$ 
to denote the inner product 
in $L^{2}(\mathbb{R}^{3})$, namely, 
if $f, g \in L^{2}(\mathbb{R}^{3})$, then 
\begin{equation*}
( f, \, g)_{L^{2}} = 
\int_{\mathbb{R}^{3}} f(x) g(x) \, dx.
\end{equation*} 
We also use the same symbol to denote the pairing between $H^{1}(\mathbb{R}^{3})$ and $H^{-1}(\mathbb{R}^{3})$, namely if $f \in H^{1}(\mathbb{R}^{3})$ and $g \in H^{-1}(\mathbb{R}^{3})$, then 
\begin{equation*}
\langle f, \, g \rangle 
= 
\int_{\mathbb{R}^{3}} \langle \nabla \rangle f(x) 
\langle \nabla \rangle^{-1} g(x) \, dx.
\end{equation*} 
\item[\textrm{(ii)}]
Throughout the paper, $C$ 
denotes a positive constant, 
that does not depend on the
parameters, unless otherwise noted and may change from line to line.
\item[\textrm{(iii)}]
For given positive quantities $a$ and $b$, 
the notation $a \lesssim b$ means the inequality 
$a \leq C b$ for
some positive constant $C$. We
also use the notation 
$a \sim b$ when $a \lesssim b$ and $b \lesssim a$. 
\end{enumerate}


\section{Proofs of Theorems 
\ref{thm-0-1} and 
\ref{thm-0}} 
\label{sec-bdd}
This section focuses on 
the family of solutions $\{u_{1, \omega}\}$ 
satisfying $\limsup_{\omega \to 0} \|u_{1, \omega}\|_{L^{\infty}} < \infty$. 
We will prove Theorems 
\ref{thm-0-1} and 
\ref{thm-0}. 
\subsection{Relation between $L^{\infty}$-norm of 
$u_{1, \omega}$ and $\omega$}
In this subsection, we shall study the 
relation between the $L^{\infty}$-norm of 
the solution $u_{1, \omega}$ and $\omega$. 
Put $M_{1, \omega} := \|u_{1, \omega}\|_{L^{\infty}}$
and 
\begin{equation}\label{S3-1}
\alpha_{1, \omega} 
:= \omega M_{1, \omega}^{- (p - 1)}. 
\end{equation}
We shall show the following: 
\begin{proposition}\label{propS3-1} 
Let $1 < p < 5$ and $\{u_{1, \omega}\}$ 
be a family of positive solutions to \eqref{sp} satisfying 
$\limsup_{\omega \to 0} M_{1, \omega} < \infty$. 
We have 
\begin{equation*}
\alpha_{1, \omega} \sim 1 
\qquad \mbox{as $\omega \to 0$}. 
\end{equation*}
\end{proposition} 
We prove Proposition \ref{propS3-1}, we need 
several preparations. 
First, we obtain an upper bound of $\alpha_{1, \omega}$: 
\begin{lemma}\label{lemS3-2} 
Let $1 < p < 5$ and 
$u_{1, \omega} \in H^{1}(\R^{3})$ be a positive solution 
to \eqref{sp}.
We obtain 
\begin{equation} \label{S3-2}
\alpha_{1, \omega} < \frac{5 - p}{2(p+1)}. 
\end{equation}
\end{lemma}
\begin{proof}
Multiplying \eqref{sp} by 
$u_{1, \omega}$ and integrating 
the resulting equation, we have 
\begin{equation}\label{S3-3}
\|\nabla u_{1, \omega}\|_{L^{2}}^{2} 
+ \omega \|u_{1, \omega}\|_{L^{2}}^{2}
= 
\|u_{1, \omega}\|_{L^{p+1}}^{p+1} 
+ \|u_{1, \omega}\|_{L^{6}}^{6}. 
\end{equation}
From the Pohozaev identity of \eqref{sp}, 
we obtain 
\begin{equation}\label{S3-4}
\frac{1}{2} \|\nabla u_{1, \omega}\|_{L^{2}}^{2} 
+ \frac{3}{2} \omega 
\|u_{1, \omega}\|_{L^{2}}^{2}
= \frac{3}{p+1} \|u_{1, \omega}\|_{L^{p+1}}^{p+1} 
+ \frac{1}{2} \|u_{1, \omega}\|_{L^{6}}^{6}. 
\end{equation}
It follows from \eqref{S3-3}, \eqref{S3-4} and 
$M_{1, \omega} = \|u_{1, \omega}\|_{L^{\infty}}$ that 
\begin{equation*}
\omega\|u_{1, \omega}\|_{L^{2}}^{2} 
= \frac{5 - p}{2 (p+1)} 
\|u_{1, \omega}\|_{L^{p+1}}^{p+1} 
< \frac{5 - p}{2(p+1)} M_{1, \omega}^{p-1} 
\|u_{1, \omega}\|_{L^{2}}^{2} 
\end{equation*}
leading to \eqref{S3-2}. 
\end{proof}
To obtain a lower bound of 
$\alpha_{1, \omega}$, 
we consider the following rescaling: 
\begin{equation}\label{S3-5}
\widehat{u}_{1, \omega}(\cdot) 
= M_{1, \omega}^{-1} u_{1, \omega}
(M_{1, \omega}^{- \frac{p - 1}{2}} \cdot). 
\end{equation}
Then, we see that 
$\widehat{u}_{1, \omega}$ 
satisfies 
\begin{equation}\label{S3-6}
- \Delta \widehat{u}_{1, \omega} 
+ \alpha_{1, \omega} 
\widehat{u}_{1, \omega}
- \widehat{u}_{1, \omega}^{p} 
- M_{1, \omega}^{5 - p} 
\widehat{u}_{1, \omega}^{5}
= 0, 
\qquad 
\widehat{u}_{1, \omega}(0) = 
\|\widehat{u}_{1, \omega}\|_{L^{\infty}} =1. 
\end{equation}
Recalling that it follows 
from the result of Gidas, Ni and Nirenberg~\cite{MR634248} 
that $\widehat{u}_{1, \omega}$ 
is radially symmetric and decreasing 
in $r = |x|$. 
Thus, there exists some 
function 
$\widehat{v}_{1, \omega}$ on $[0, \infty)$ 
such that $\widehat{v}_{1, \omega}(r) = 
\widehat{u}_{1, \omega}(x)$ for 
$r = |x|$. 
By abuse of notation, 
we identify $\widehat{v}_{1, \omega}$ by 
$\widehat{u}_{1, \omega}$. 
Since $\widehat{u}_{1, \omega}$ is radially symmetric 
with $\widehat{u}_{1, \omega} (0) = 
\|\widehat{u}_{1, \omega}\|_{L^{\infty}} = 1$, 
\eqref{S3-6} can be transformed 
to the following ordinary differential 
equations: 
\begin{equation}\label{S3-7}
\begin{cases}
- \widehat{u}_{1, \omega}^{\prime \prime} 
- \frac{2}{r} 
\widehat{u}_{1, \omega}^{\prime }
+ \alpha_{1, \omega} 
\widehat{u}_{1, \omega}
- \widehat{u}_{1, \omega}^{p} 
- M_{1, \omega}^{5 - p} 
\widehat{u}_{1, \omega}^{5}
= 0 
\qquad 
\text{on $(0, \infty)$}, \\[6pt]
\widehat{u}_{1, \omega}(0) =1, 
\qquad 
\widehat{u}_{1, \omega}^{\prime}(0) =0, 
& 
\end{cases}
\end{equation}
where the prime mark denotes the differentiation 
with respect to $r$.
Here, we remark that throughout this section, 
we always assume that $\limsup_{\omega \to 0} M_{1, \omega} 
< +\infty$.

We shall show that $\liminf_{\omega \to 0} 
\alpha_{1, \omega} > 0$. 
Suppose to the contrary that 
there exists a sequence $\{\omega_{n}\} \subset (0, \infty)$ 
with $\lim_{n \to \infty} \omega_{n} = 0$ such that
$\lim_{n \to \infty} \alpha_{1, \omega_{n}} 
= 0$. 
Then, there exist a subsequence of $\{\omega_{n}\}$ 
(we still denote by the same symbol) and $M_{1, 0} \in [0, \infty)$ 
such that $\lim_{n \to \infty} M_{1, \omega_{n}} = M_{1, 0}$. 
We introduce 
the following \lq\lq limit'' equation: 
\begin{equation}\label{S3-8}
\begin{cases}
- \widehat{u}_{1, 0}^{\prime \prime} 
- \frac{2}{r} 
\widehat{u}_{1, 0}^{\prime }
- \widehat{u}_{1, 0}^{p} 
- M_{1, 0}^{5 - p} 
\widehat{u}_{1, 0}^{5}
= 0 
\qquad 
\text{on $(0, \infty)$}, \\[6pt]
\widehat{u}_{1, 0}(0) =1, 
\qquad 
\widehat{u}_{1, 0}^{\prime}(0) =0. 
& 
\end{cases}
\end{equation}
Concerning the solution to \eqref{S3-8}, 
we shall show the following: 
\begin{lemma}\label{lemS3-3} 
Let $1 < p < 5$ and $\widehat{u}_{1, 0}$ be 
the unique solution to 
\eqref{S3-8}. 
Then, $\widehat{u}_{1, 0}$ 
must change the sign on $[0, \infty)$. 
\end{lemma} 

\begin{proof}
We shall prove this lemma by contradiction. 
Suppose to the contrary that 
$\widehat{u}_{1, 0}(r) > 0$ 
for any $r > 0$. 
It follows from the equation in \eqref{S3-8} that 
\begin{equation} \label{S3-9}
- r^{2} \widehat{u}_{1, 0}^{\prime}(r) 
= \int_{0}^{r} (\widehat{u}_{1, 0}^{p}(s) 
+ M_{1, 0}^{5 - p} \widehat{u}_{1, 0}^{5}(s)) s^{2} \, ds > 0. 
\end{equation}
Thus, $\widehat{u}_{1, 0}(r)$ is monotone 
decreasing in $r>0$ and has a limit 
$\ell_{1, 0} \geq 0$,  
$\ell_{1, 0} = \lim_{r \to \infty} 
\widehat{u}_{1, 0}(r)$. 
We define an energy $E_{1, 0}$ by 
\begin{equation*}
E_{1, 0}(r) := 
\frac{(
\widehat{u}_{1, 0}^{\prime}(r))^{2}}{2} 
+ \frac{\widehat{u}_{1, 0}^{p+1}
(r)}{p+1} 
+ M_{1, 0}^{5 - p}
\frac{\widehat{u}_{1, 0}^{6}(r)}{6} 
\end{equation*}
From \eqref{S3-8} and \eqref{S3-9}, we can easily verify that 
\begin{equation*} 
\begin{split}
E_{1, 0}^{\prime}(r) 
= \widehat{u}_{1, 0}^{\prime}(r)
\widehat{u}_{1, 0}^{\prime \prime}(r) 
+ \widehat{u}_{1, 0}^{p}(r)
\widehat{u}_{1, 0}^{\prime}(r)
+ M_{1, 0}^{5 - p}
\widehat{u}_{1, 0}^{5}(r)
\widehat{u}_{1, 0}^{\prime}(r) 
= - \frac{2}{r} (\widehat{u}_{1, 0}^{\prime}(r))^{2} < 0. 
\end{split}
\end{equation*}
This together with the initial conditions in 
\eqref{S3-8} implies that 
\[
E_{1, 0}(r) \leq E_{1, 0}(0) = \frac{1}{p+1} + 
\frac{M_{1, 0}^{5 - p}}{6} < \infty.
\]
Thus, we see that 
$\widehat{u}_{1, 0}$ and $\widehat{u}_{1, 0}^{\prime}$ 
are bounded in $r>0$. 

It follows from the boundedness 
of $\widehat{u}_{1, 0}^{\prime}(r)$, 
$\ell_{1, 0} = \lim_{r \to \infty} 
\widehat{u}_{1, 0}(r)$
and \eqref{S3-8} 
that 
\[
- \lim_{r \to \infty} 
\widehat{u}_{1, 0}^{\prime \prime}(r)
= \ell_{1, 0}^{p} + M_{1, 0} \ell_{1, 0}^{5} (\geq 0). 
\]
Since $\widehat{u}_{1, 0}^{\prime}(r)$ is bounded, 
we have $\ell_{1, 0} = 0$. 
Then, from \eqref{S3-9}, the monotonicity 
and positivity of 
$\widehat{u}_{1, 0}$ , we obtain 
\[
- r^{2} \widehat{u}_{1, 0}^{\prime}(r)
= \int_{0}^{r} 
(\widehat{u}_{1, 0}^{p}(s) 
+ M_{1, 0}^{5 - p} \widehat{u}_{1, 0}
^{5}(s)) s^{2} \, ds
\geq \widehat{u}_{1, 0}^{p}(r)
\int_{0}^{r} s^{2} \, ds 
= \frac{r^{3} \widehat{u}_{1, 0}^{p}(r)}{3}. 
\]
This implies that 
\[
\left( 
\frac{1}{p - 1}
\widehat{u}_{1, 0}^{- (p - 1)}(r)
\right)^{\prime} 
= - \widehat{u}_{1, 0}^{\prime}(r) 
\widehat{u}_{1, 0}^{- p}(r)
> \frac{r}{3} 
= \left(\frac{r^{2}}{6} \right)^{\prime}.
\]
Thus, integrating the above from $0$ to $r$, 
we have by the initial condition 
$\widehat{u}_{1, 0}(0) = 1$ that 
\begin{equation} \label{S3-10}
\frac{\widehat{u}_{1, 0}^{- (p - 1)}(r)}{p-1} 
> 
\frac{\widehat{u}_{1, 0}^{- (p - 1)}(r)}{p-1} 
- \frac{\widehat{u}_{1, 0}^{-(p - 1)}(0)}{p-1} 
> 
\frac{r^{2}}{6}
- \frac{0^{2}}{6}
= \frac{r^{2}}{6} > 0. 
\end{equation}
From \eqref{S3-10}, we conclude that 
\[
\widehat{u}_{1, 0}(r)
\leq \left(\frac{6}{p-1}
\right)^{\frac{1}{p-1}} r^{- \frac{2}{p-1}}. 
\]
This together with the condition $p < 5$ implies that 
\begin{equation} \label{S3-11}
\int_{0}^{\infty} r^{2} \widehat{u}_{1, 0}^{p+1}
(r) \, dr + 
\int_{0}^{\infty} r^{2} \widehat{u}_{1, 0}^{6}(r) \, dr < \infty
\end{equation}
Observe from the equation of \eqref{S3-8} that 
\begin{equation} \label{S3-12}
\begin{split}
\left(\frac{r^{3}}{2} (\widehat{u}_{1, 0}^{\prime}(r))^{2} 
+ \frac{r^{3}}{p+1} \widehat{u}_{1, 0}^{p+1}(r) 
+ \frac{M_{1, 0}^{5 - p}}{6} 
r^{3} \widehat{u}_{1, 0}^{6}(r)\right)^{\prime} 
+ \frac{1}{2} r^{2} (\widehat{u}_{1, 0}^{\prime}(r))^{2}
= \frac{3}{p+1} r^{2} \widehat{u}_{1, 0}^{p+1}(r) 
+ \frac{M_{1, 0}^{5 - p}}{2} r^{2} \widehat{u}_{1, 0}^{6}(r). 
\end{split}
\end{equation}
Integrating \eqref{S3-12} from $0$
to $r$, we have 
\begin{equation} \label{S3-13}
\begin{split}
& \quad 
\frac{r^{3}}{2} (\widehat{u}_{1, 0}^{\prime}(r))^{2} 
+ \frac{r^{3}}{p+1} \widehat{u}_{1, 0}^{p+1}(r) 
+ \frac{M_{1, 0}^{5 - p}}{6} r^{3} 
\widehat{u}_{1, 0}^{6}(r) 
+ \frac{1}{2} \int_{0}^{r} 
s^{2} (\widehat{u}_{1, 0}^{\prime}(s))^{2} \, ds \\[6pt]
& 
= \frac{3}{p+1} \int_{0}^{r} s^{2} 
\widehat{u}_{1, 0}^{p+1}(s) \, ds
+ \frac{M_{1, 0}^{5 - p}}{2} 
\int_{0}^{r} s^{2} 
\widehat{u}_{1, 0}^{6}(s) \, ds. 
\end{split}
\end{equation}
This yields that 
\[
\int_{0}^{r} 
s^{2} (\widehat{u}_{1, 0}^{\prime}(s))^{2} \, ds 
\leq \frac{6}{p+1} \int_{0}^{r} s^{2} 
\widehat{u}_{1, 0}^{p+1}(s) \, ds
+ M_{1, 0}^{5 - p} 
\int_{0}^{r} s^{2} 
\widehat{u}_{1, 0}^{6}(s) \, ds
\] 
Letting $r \to \infty$, we have 
by \eqref{S3-11} that 
\begin{equation} \label{S3-14}
\int_{0}^{\infty} 
s^{2} (\widehat{u}_{1, 0}^{\prime}(s))^{2} \, ds < \infty. 
\end{equation}
Thus, it follows from \eqref{S3-11} and 
\eqref{S3-14} that 
\begin{equation} \label{S3-15}
\int_{0}^{\infty} 
s^{2} \left\{ (\widehat{u}_{1, 0}^{\prime}(s))^{2} 
+ \widehat{u}_{1, 0}^{p+1}(s) + \widehat{u}_{1, 0}^{6}(s) 
\right\}\, ds < \infty.
\end{equation}
We see from \eqref{S3-15} that 
there exists a sequence $\{r_{n}\}$ 
with $\lim_{n \to \infty} r_{n} =\infty$ 
such that 
\begin{equation} \label{S3-16}
r_{n}^{3} ((\widehat{u}_{1, 0}^{\prime}(r_{n}))^{2} 
+ (\widehat{u}_{1, 0}(r_{n}))^{p+1} 
+ (\widehat{u}_{1, 0}(r_{n}))^{6}) \to 0 
\qquad \mbox{as $n \to \infty$}. 
\end{equation}
Indeed, if not, there exist  
$C_{0} > 0$ and $r_{0} > 0$ such that 
\[
r^{3} ((\widehat{u}_{1, 0}^{\prime}(r))^{2} 
+ (\widehat{u}_{1, 0}(r))^{p+1} 
+ (\widehat{u}_{1, 0}(r))^{6}) 
\geq C_{0}
\qquad (r \in (r_{0}, \infty)). 
\] 
This implies that 
\begin{equation*}
\int_{0}^{\infty} 
s^{2} \left\{ (\widehat{u}_{1, 0}^{\prime}(s))^{2} 
+ \widehat{u}_{1, 0}^{p+1}(s) + \widehat{u}_{1, 0}^{6}(s) 
\right\}\, ds 
\geq \int_{r_{0}}^{\infty} 
s^{2} \left\{ (\widehat{u}_{1, 0}^{\prime}(s))^{2} 
+ \widehat{u}_{1, 0}^{p+1}(s) + \widehat{u}_{1, 0}^{6}(s) \right\}\, ds 
\geq C_{0} \int_{r_{0}}^{\infty} s^{-1} \, ds 
= \infty, 
\end{equation*} 
which contradicts \eqref{S3-15}. 
Thus, \eqref{S3-16} holds. 
Substituting $r = r_{n}$ in 
\eqref{S3-13} and letting $n$ go to infinity, 
we see from \eqref{S3-16} that 
\begin{equation} \label{S3-17}
\begin{split}
\frac{1}{2} \int_{0}^{\infty} 
s^{2} (\widehat{u}_{1, 0}^{\prime}(s))^{2} \, ds 
= \frac{3}{p+1} \int_{0}^{\infty} s^{2} 
\widehat{u}_{1, 0}^{p+1}(s) \, ds
+ \frac{M_{1, 0}^{5 - p}}{2} 
\int_{0}^{\infty} 
s^{2} \widehat{u}_{1, 0}^{6}(s) \, ds. 
\end{split}
\end{equation}
On the other hand, 
from the equation in \eqref{S3-8}, 
we have 
\[
(r^{2} \widehat{u}_{1, 0}(r) \widehat{u}_{1, 0}^{\prime}(r))^{\prime} 
+ r^{2} \widehat{u}_{1, 0}^{p+1}(r) + M_{1, 0}^{5 - p} 
r^{2} \widehat{u}_{1, 0}^{6}(r)
= r^{2}(\widehat{u}_{1, 0}^{\prime}(r))^{2}. 
\]
Integrating the above equality from 
$0$ to $r$, one has 
\begin{equation}\label{S3-18} 
r^{2} \widehat{u}_{1, 0}(r) \widehat{u}_{1, 0}^{\prime}(r)
+ \int_{0}^{r}s^{2} \widehat{u}_{1, 0}^{p+1}(s) \, ds
+ M_{1, 0}^{5 - p} \int_{0}^{r}s ^{2} \widehat{u}_{1, 0}^{6}(s) \, ds
= \int_{0}^{r} s^{2}(\widehat{u}_{1, 0}^{\prime}(s))^{2} \, ds. 
\end{equation}
It follows from \eqref{S3-16} 
that there exists a constant 
$C_{d} > 0$ such that 
$|\widehat{u}_{1, 0}^{\prime}(r_{n})| \leq C_{d} 
r_{n}^{- \frac{3}{2}}$ and 
$\widehat{u}_{1, 0}(r_{n}) \leq C_{d} 
r_{n}^{- \frac{3}{p+1}}$. 
This together with the condition $p < 5$ implies that 
$\lim_{n \to \infty}
r_{n}^{2} \widehat{u}_{1, 0}(r_{n}) \widehat{u}_{1, 0}^{\prime}(r_{n}) 
= 0$. Thus, 
substituting $r = r_{n}$ in 
\eqref{S3-18} and letting $n$ go to infinity, 
we obtain 
\begin{equation}\label{S3-19}
\int_{0}^{\infty} 
s^{2}(\widehat{u}_{1, 0}^{\prime}(s)) \, ds = 
\int_{0}^{\infty}s^{2} \widehat{u}_{1, 0}^{p+1}(s) \, ds
+ M_{1, 0}^{5-p} \int_{0}^{\infty} 
s ^{2} \widehat{u}_{1, 0}^{6}(s) \, ds. 
\end{equation}
From \eqref{S3-17}, 
\eqref{S3-19} and $p < 5$, we obtain 
\[
\int_{0}^{\infty}s^{2} \widehat{u}_{1, 0}^{p+1}(s) \, ds 
= 0, 
\]
which contradicts the positivity of 
$\widehat{u}_{1, 0}(r)$. 
This completes the proof. 
\end{proof}

We are now in a position to prove Proposition \ref{propS3-1}. 
\begin{proof}[Proof of Proposition \ref{propS3-1}]
By Lemma \ref{lemS3-2}, it suffices to 
show that $\liminf_{\omega \to 0} 
\alpha_{1, \omega} > 0$. 
Suppose to the contrary that 
there exists a sequence $\{\omega_{n}\}$ in $(0, \infty)$ 
with $\lim_{n \to \infty} \omega_{n} = 0$ such that
\begin{equation*}
\lim_{n \to \infty} \alpha_{1, \omega_{n}} 
= 0. 
\end{equation*}
From the assumptions $\limsup_{\omega \to 0} \|u_{1, \omega}\|_{L^{\infty}} < \infty$ and $M_{1, \omega} = \|u_{1, \omega}\|_{L^{\infty}}$, it follows that
up to a subsequence, 
we may assume that the sequence
$\{M_{1, \omega_n}\}$ has a limit $M_{1, 0}$, i.e.
$\lim_{n \to \infty} M_{1, \omega_{n}} 
= M_{1, 0}$. 
Observe from the equation in \eqref{S3-7} that 
\begin{equation} \label{S3-36}
- r^{2} \widehat{u}_{1, \omega_{n}}^{\prime}(r) 
= \int_{0}^{r} (\widehat{u}_{1, \omega_{n}}^{p}(s) 
+ M_{1, \omega_{n}}^{5 - p} \widehat{u}_{1, \omega_{n}}^{5}(s)
- \alpha_{1, \omega_{n}} \widehat{u}_{1, \omega_{n}}) 
s^{2} \, ds. 
\end{equation}
It follows from $\widehat{u}_{1, \omega_{n}} (0) = 
\|\widehat{u}_{1, \omega_{n}}\|_{L^{\infty}} = 1$
for all $n \in \N$ that 
\[
|r^{2} \widehat{u}_{1, \omega_{n}}^{\prime}(r)|
\leq 
(1 + M_{1, \omega_{n}}^{5 - p} 
+ \alpha_{1, \omega_{n}}) 
\int_{0}^{r} s^{2} \, ds
= \frac{(1 + M_{1, \omega_{n}}^{5 - p} 
+ \alpha_{1, \omega_{n}})}{3} r^{3}. 
\]
Thus, by the assumption $\limsup_{n \to \infty} 
M_{1, \omega_{n}} < \infty$ and 
\eqref{S3-2}, 
there exists a constant $C_{1} > 0$ 
independent of $n \in \N$, such that 
\[
|\widehat{u}_{1, \omega_{n}}^{\prime}(r)| 
\leq C_{1} r. 
\]
This together with the equation in \eqref{S3-7} 
yields that 
\[
\sup_{n \in \N} 
|\widehat{u}_{1, \omega_{n}}^{\prime \prime}(r)|
\leq 2
\sup_{n \in \N} 
\frac{|\widehat{u}_{1, \omega_{n}}^{\prime }(r)|}{r}
+ \sup_{n \in \N} \alpha_{1, \omega_{n}} 
|\widehat{u}_{1, \omega_{n}}(r)|
+ \sup_{n \in \N} 
|\widehat{u}_{1, \omega_{n}}^{p}(r)| 
+ \sup_{n \in \N} 
M_{1, \omega_{n}}^{5 - p} 
|\widehat{u}_{1, \omega_{n}}^{5}(r)| 
< \infty 
\]
for any $r > 0$. 
Then, by the Ascoli-Arzela theorem, 
there exists a subsequence of $\{\widehat{u}_{1, \omega_{n}}\}$ 
(we denote it by the same symbol) and a function 
$v_{0} \in C^{1}([0, \infty))$ such that 
for any bounded interval $I$ in $[0, \infty)$, we have
\[ 
\widehat{u}_{1, \omega_{n}}(r) \to v_{0} (r), \qquad
\widehat{u}_{1, \omega_{n}}^{\prime}(r) \to 
v_{0}^{\prime}(r)
\qquad \mbox{
uniformly in $r \in I$}. 
\] 
In addition, taking the limit of \eqref{S3-36}, 
we see from $\lim_{n \to \infty} \alpha_{1, \omega_{n}} 
= 0$
that $v_{0}$ satisfies 
\[
- r^{2} v_{0}^{\prime}(r) 
= \int_{0}^{r} (v_{0}^{p}(s) 
+ M_{1, 0}^{5 - p} v_{0}^{5}(s))
s^{2} \, ds 
\]
for any $r \in I$. 
Then, $v_{0}$ satisfies \eqref{S3-8}. 
It follows from the uniqueness of the solution that 
$v_{0} = \widehat{u}_{1, 0}$. 

Note that $\widehat{u}_{1, \omega_{n}}$ 
is positive and converges to 
$\widehat{u}_{1, 0}$ uniformly on $I$ 
as $n$ tends to infinity. 
However, 
$\widehat{u}_{1, 0}$ 
changes sign (see Lemma \ref{lemS3-3}), leading 
to a contradiction. 
Thus, we see that $\liminf_{\omega \to 0} 
\alpha_{1, \omega} > 0$. 
\end{proof}

\subsection{Conclusion}
In this subsection, we complete the proofs 
of Theorems \ref{thm-0-1} and \ref{thm-0}. 
To prove Theorem \ref{thm-0-1}, 
we will obtain a boundedness of 
$\{\|\widehat{u}_{1, \omega}\|_{H^{1}}\}$. 
First, we shall show the following: 
\begin{lemma}\label{lem2-4}
Let $1 < p < 5$ and $\widehat{u}_{1, \omega}$ 
be the positive solution to \eqref{S3-7}.
There exist $C_{\omega} > 0$
such that 
\begin{equation}\label{S3-20}
\widehat{u}_{1, \omega}(r), \; 
|\widehat{u}_{1, \omega}^{\prime}(r)| 
\leq C_{\omega} \exp\left[- \sqrt{\frac{\alpha_{1, \omega} }{2}} r\right]
\qquad \mbox{for all $r > 0$}. 
\end{equation}
\end{lemma}
\begin{proof}
Since 
$\widehat{u}_{1, \omega}^{\prime}(r) < 0$ ($r>0$) and 
$\lim_{r \to \infty} 
\widehat{u}_{1, \omega}(r) = 0$, 
there exists $R_{\omega} > 0$ such that 
\[
\begin{split}
\widehat{u}_{1, \omega}^{\prime \prime}(r) 
& 
= - \frac{2}{r}\widehat{u}_{1, \omega}^{\prime}(r) 
+ \alpha_{1, \omega} \widehat{u}_{1, \omega}(r) 
- \widehat{u}_{1, \omega}^{p}(r) 
- M_{1, \omega}^{5 - p} \widehat{u}_{1, \omega}^{5}(r) 
\\[6pt]
& > (\alpha_{1, \omega} 
- \widehat{u}_{1, \omega}^{p-1}(r) 
- M_{1, \omega}^{5 - p} \widehat{u}_{1, \omega}^{4}(r))
\widehat{u}_{1, \omega}(r) 
> \frac{\alpha_{1, \omega} }{2} 
\widehat{u}_{1, \omega}(r) 
\qquad \mbox{for all $r \geq R_{\omega}$}. 
\end{split}
\]
This yields that 
\[
\p_{r} \left( 
e^{\sqrt{\frac{\alpha_{1, \omega} }{2}} r}
\left(- \p_{r} + \sqrt{\frac{\alpha_{1, \omega} }{2}}
\right)\widehat{u}_{1, \omega} 
\right) 
= e^{\sqrt{\frac{\alpha_{1, \omega} }{2}} r}
\left(
- \widehat{u}_{1, \omega}^{\prime \prime} + 
\frac{\alpha_{1, \omega} }{2} 
\widehat{u}_{1, \omega}(r)
\right)
< 0. 
\] 
Thus, we see that 
$e^{\sqrt{\frac{\alpha_{1, \omega} }{2}} r}
\left(- \p_{r} + 
\sqrt{\frac{\alpha_{1, \omega} }{2}} \right)
\widehat{u}_{1, \omega}$ 
is decreasing on $[R_{\omega}, \infty)$, which implies that 
\begin{equation} \label{S3-22}
e^{\sqrt{\frac{\alpha_{1, \omega} }{2}} r}
\left(- \p_{r} + 
\sqrt{\frac{\alpha_{1, \omega} }{2}} 
\right)\widehat{u}_{1, \omega}(r) 
\leq 
e^{\sqrt{\frac{\alpha_{1, \omega} }{2}} 
R_{\omega}}
\left(- \p_{r} + 
\sqrt{\frac{\alpha_{1, \omega} }{2}} 
\right)\widehat{u}_{1, \omega}
(R_{\omega}) =: \widetilde{C}_{\omega}
\end{equation}
Since $\widehat{u}_{1, \omega}^{\prime}(r) 
< 0$, we have by \eqref{S3-22} that 
\[
\widehat{u}_{1, \omega}(r)
\leq \widetilde{C}_{\omega} \sqrt{\frac{2}{\alpha_{1, \omega}}}
e^{- \sqrt{\frac{\alpha_{1, \omega} }{2}} r}, 
\qquad 
| \widehat{u}_{1, \omega}^{\prime}(r)| 
= - \widehat{u}_{1, \omega}^{\prime}(r) 
\leq \widetilde{C}_{\omega} 
e^{- \sqrt{\frac{\alpha_{1, \omega} }{2}} r}
\]
Thus, putting
\[
C_{\omega} := 
\widetilde{C}_{\omega} \max\left\{ 
\sqrt{\frac{2}{\alpha_{1, \omega}}}, 1 \right\}, 
\]
we find that \eqref{S3-20} holds. 
\end{proof}
Next, following \cite{MR1912264, MR1022990, MR4638619}, 
which are originated in \cite{MR851145}, 
we prove the following pointwise estimate of $\widehat{u}_{1, \omega}$:
\begin{lemma}\label{lem2-5}
Let $1 < p < 5$ and $\widehat{u}_{1, \omega}$ 
be the positive solution to \eqref{S3-7}.
There exists $\omega_{1,1} > 0$ 
such that for $\omega 
\in (0, \omega_{1,1})$, we have 
\begin{equation}\label{S3-23}
\widehat{u}_{1, \omega}(r) 
\leq \left( 
1 + \frac{\gamma_{1, \omega}}{3} r^{2}
\right)^{-\frac{1}{2}}
\qquad \mbox{for all $r > 0$}, 
\end{equation}
where 
\begin{equation}\label{S3-24}
\gamma_{1, \omega} 
:= - \alpha_{1, \omega} 
+ 1 + M_{1, \omega}^{5 - p}. 
\end{equation}
\end{lemma}
\begin{remark}
We see from \eqref{S3-2} that 
\[
\gamma_{1, \omega} > \frac{3(p -1)}{2(p+1)}
\] 
for any 
$\omega > 0$.
\end{remark} 
\begin{proof}[Proof of 
Lemma \ref{lem2-5}]
We put $f_{1, \omega}(s) 
:= \alpha_{1, \omega} s - 
s^{p} - M_{1, \omega}^{5 - p} s^{5}$ 
for $s > 0$. 
Then, the equation in \eqref{S3-7} 
can be written as: 
\begin{equation}\label{S3-25}
\widehat{u}_{1, \omega}^{\prime \prime} 
= - 2 
\frac{\widehat{u}_{1, \omega}^{\prime}}{r} 
+ f_{1, \omega}(\widehat{u}_{1, \omega}). 
\end{equation}
Putting
\begin{equation} \label{S3-26}
\Psi_{1, \omega}(r) := 
\frac{- \widehat{u}_{1, \omega}^{\prime}(r)}
{r \widehat{u}_{1, \omega}^{3}(r)}, 
\end{equation}
we have by \eqref{S3-25} that 
\begin{equation} \label{S3-27}
\begin{split}
\Psi_{1, \omega}^{\prime}(r)
= \dfrac{- 
\widehat{u}_{1, \omega}^{\prime \prime} 
r \widehat{u}_{1, \omega}^{3} +
\widehat{u}_{1, \omega}^{\prime} 
(\widehat{u}_{1, \omega}^{3} 
+ 3 r \widehat{u}_{1, \omega}^{2} 
\widehat{u}_{1, \omega}^{\prime})}
{r^{2} \widehat{u}_{1, \omega}^{6}} 
& = \frac{3\widehat{u}_{1, \omega}^{3} 
\widehat{u}_{1, \omega}^{\prime} 
- f_{1, \omega} 
(\widehat{u}_{1, \omega}) r \widehat{u}_{1, \omega}^{3} 
+ 3 r\widehat{u}_{1, \omega}^{2} 
(\widehat{u}_{1, \omega}^{\prime})^{2}}{
r^{2} \widehat{u}_{1, \omega}^{6}} \\[6pt]
& = 3 \frac{H_{1, \omega}(r)}{r^{4} 
\widehat{u}_{1, \omega}^{4}}, 
\end{split}
\end{equation}
where 
\[
H_{1, \omega}(r) := r^{3} 
(\widehat{u}_{1, \omega}^{\prime})^{2} 
+ r^{2} \widehat{u}_{1, \omega} 
\widehat{u}_{1, \omega}^{\prime}
- \frac{r^{3}f_{1, \omega}
(\widehat{u}_{1, \omega}) 
\widehat{u}_{1, \omega}}{3}. 
\]
With a little bit of effort, 
we see from \eqref{S3-25} and 
$f_{1, \omega}(s) = \alpha_{1, \omega} s - 
s^{p} - M_{1, \omega}^{5 - p} s^{5}$ 
that 
\begin{equation} \label{S3-28}
\begin{split}
H_{1, \omega}^{\prime}(r)
& = 3r^{2} (\widehat{u}_{2, 
\omega}^{\prime})^{2} 
+ 2 r^{3} \widehat{u}_{1, \omega}^{\prime} 
\widehat{u}_{1, \omega}^{\prime \prime}
+ 2r \widehat{u}_{1, \omega} 
\widehat{u}_{1, \omega}^{\prime}
+ r^{2} 
(\widehat{u}_{1, \omega}^{\prime})^{2} 
+ r^{2} \widehat{u}_{1, \omega} 
\widehat{u}_{1, \omega}^{\prime \prime} \\[6pt] 
& \quad 
- r^{2} \widehat{u}_{1, \omega} f_{1, \omega} (\widehat{u}_{1, \omega}) 
- \frac{r^{3} \widehat{u}_{1, \omega} 
\widehat{u}_{1, \omega}^{\prime} 
f_{1, \omega}^{\prime} 
(\widehat{u}_{1, \omega})}{3}
- \frac{r^{3} 
\widehat{u}_{1, \omega}^{\prime} 
f_{1, \omega} (\widehat{u}_{1, \omega})}{3} 
\\[6pt]
& = \frac{r^{3} \widehat{u}_{1, \omega} 
\widehat{u}_{1, \omega}^{\prime}}{3}
(4 \alpha_{1, \omega} 
- (5 - p) 
\widehat{u}_{1, \omega}^{p - 1}). 
\end{split}
\end{equation}
By \eqref{S3-2}, 
$\widehat{u}_{1, \omega} (0) = 
\|\widehat{u}_{1, \omega}\|_{L^{\infty}} = 1$ and 
$p > 1$, one has 
\begin{equation}\label{S3-29}
4 \alpha_{1, \omega} 
- (5 - p) 
\widehat{u}_{1, \omega}^{p-1}(0) 
< \frac{2(5 - p)}{p+1}
- (5 - p) 
< 0
\end{equation}
for any $\omega > 0$. 
By $\widehat{u}_{1, \omega}^{\prime}(r) < 0$, 
$\lim_{r \to \infty} \widehat{u}_{1, \omega}(r) = 0$ 
(see Lemma \ref{lem2-4}), 
\eqref{S3-28} and \eqref{S3-29}, 
there exists $r_{1, \omega} >0$ such that 
\[
H_{1, \omega}^{\prime}(r) 
\begin{cases}
> 0 & \qquad (0 < r < r_{1, \omega}), \\[6pt]
= 0 & \qquad (r = r_{1, \omega}), \\[6pt]
< 0 & \qquad (r_{1, \omega} < r).
\end{cases}
\]
In addition, since the functions 
$\widehat{u}_{1, \omega}, 
- \widehat{u}_{1, \omega}^{\prime}$ 
decay exponentially 
(see Lemma \ref{lem2-4}), 
we have 
$\lim_{r \to \infty} H_{1, \omega}(r) 
= 0$ for each $\omega > 0$. 
These together with $H_{1, \omega}(0) = 0$ 
imply that 
$H_{1, \omega}(r) > 0$ for all $r > 0$. 
It follows from \eqref{S3-27} that 
$\Psi_{1, \omega}^{\prime}(r) > 0$ for all $r > 0$. 
This together with the l'Hopital rule, 
$\widehat{u}_{1, \omega}(0) = 1$ and 
$\widehat{u}_{1, \omega}^{\prime}(0) = 0$ yields that  
\[
\Psi_{1, \omega}(r) > \Psi_{1, \omega}(0) 
= \lim_{r \to 0} 
\frac{- \widehat{u}_{1, \omega}^{\prime}(r)}
{r \widehat{u}_{1, \omega}^{3}(r)} 
= \lim_{r \to 0} 
\frac{- \widehat{u}_{1, \omega}^{\prime 
\prime}(r)}{\widehat{u}_{1, \omega}^{3}(r) 
+ 3 r \widehat{u}_{1, \omega}^{2}(r) 
\widehat{u}_{1, \omega}^{\prime}(r)} 
= - \widehat{u}_{1, \omega}^{\prime \prime}(0). 
\]
Using the l'Hopital rule again, 
we have 
$\lim_{r \to 0} 
\frac{\widehat{u}_{1, \omega}^{\prime}(r)}{r} 
= \widehat{u}_{1, \omega}^{\prime \prime}(0)$.
Then, it follows from \eqref{S3-25} that 
\begin{equation} \label{S3-30}
\Psi_{1, \omega}(r) > 
- \widehat{u}_{1, \omega}^{\prime \prime}(0) 
= - \frac{f_{1, \omega}
(\widehat{u}_{1, \omega}(0))}{3}
= - \frac{f_{1, \omega}(1)}{3}
= \frac{\gamma_{1, \omega}}{3} 
\qquad \mbox{for all $r > 0$}, 
\end{equation}
where $\gamma_{1, \omega}$ is 
the positive number defined by 
\eqref{S3-24}. 
We put 
\[
Z_{1, \omega}(r) := 
\left( 
1 + \frac{\gamma_{1, \omega}}{3} r^{2}
\right)^{-\frac{1}{2}}. 
\] 
We can easily verify that 
\[
Z_{1, \omega}^{3}(r) 
= \left( 
1 + \frac{\gamma_{1, \omega}}{3} r^{2}
\right)^{-\frac{3}{2}}, 
\qquad 
Z_{1, \omega}^{\prime}(r) 
= - \frac{\gamma_{1, \omega}}{3} 
\left( 
1 + \frac{\gamma_{1, \omega}}{3} r^{2}
\right)^{-\frac{3}{2}} r 
= - \frac{\gamma_{1, \omega}}{3} 
Z_{1, \omega}^{3}(r) r. 
\]
Thus, we have by \eqref{S3-30} and \eqref{S3-26}, 
that 
\[
(Z_{1, \omega}^{-2}(r))^{\prime} = 
- 2 \frac{Z_{1, \omega}^{\prime}(r)}
{Z_{1, \omega}^{3}(r)} 
= \frac{2}{3}\gamma_{1, \omega}r 
< 2 \Psi_{1, \omega}(r) r = - 
2 \frac{\widehat{u}_{1, \omega}^{\prime}(r)}
{\widehat{u}_{1, \omega}^{3}(r)} 
= 
(\widehat{u}_{1, \omega}^{-2}(r))^{\prime}. 
\]
Integrating the above inequality from 
$0$ to $r$, we have by
$\widehat{u}_{1, \omega}(0) 
= Z_{1, \omega}(0) = 1$ that 
\[
Z_{1, \omega}^{-2}(r) - 1 = 
Z_{1, \omega}^{-2}(r) - Z_{1, \omega}^{-2}(0) 
< \widehat{u}_{1, \omega}^{-2}(r) - 
\widehat{u}_{1, \omega}^{-2}(0) 
= \widehat{u}_{1, \omega}^{-2}(r) - 1. 
\]
This yields that 
$\widehat{u}_{1, \omega}(r) 
< Z_{1, \omega}(r)$ for all $r > 0$. 
This concludes the proof. 
\end{proof}
The following result gives a 
decay estimate of the positive solution $\widehat{u}_{1, \omega}$ 
to \eqref{S3-6} that 
is better than the one given 
by Lemma \ref{lem2-4}: 
\begin{lemma}\label{lem2-6}
Let $1 < p < 5$.
There exists $C_{0} > 0$, which does not depend on 
$\omega > 0$, and a sufficiently small $\omega_{1,2} > 0$ 
such that for $\omega \in (0, \omega_{1,2})$ 
and positive solution $\widehat{u}_{1, \omega}$ 
to \eqref{S3-6}, we obtain 
\begin{equation} \label{S3-31}
\widehat{u}_{1, \omega}(x) \leq 
C_{0} |x|^{-1} \exp(- \frac{\sqrt{\alpha_{1, \omega}}}{2} x) 
\qquad (x \in \R^{3}).
\end{equation} 
\end{lemma}
\begin{proof}[Proof of Lemma 
\ref{lem2-6}]
From \eqref{S3-23} and 
$\liminf_{\omega \to 0} \alpha_{1, \omega} > 0$ 
(see Proposition \ref{propS3-1}), 
there exist $\omega_{1, 2}$ and $R_{0} > 0$, which are independent of 
$\omega > 0$, such that for $\omega \in (0, \omega_{1, 2})$, 
we have 
\[
- \Delta \widehat{u}_{1, \omega} 
+ \frac{\alpha_{1, \omega}}{4} 
\widehat{u}_{1, \omega} \leq 0 
\qquad (|x| \geq R_{0}). 
\]
For such $R_{0}>0$, by \eqref{S3-2}, we can 
take a constant $C_{0} >0$ sufficiently large, which does not depend on 
$\omega > 0$, such that 
\[
\widehat{u}_{1, \omega}(x) \leq 
\left(1 + \frac{\gamma_{1, 2\omega}}{3} R_{0}^{2} \right)^{-\frac{1}{2}}
\leq 
C_{0} R_{0}^{-1} \exp\left(- \frac{\sqrt{\alpha_{1, \omega}}}{2} R_{0} 
\right) \qquad 
(|x| = R_{0}). 
\]
Observe that 
$Y_{\omega}(x) := 
C_{0} |x|^{-1} \exp(- \frac{\sqrt{\alpha_{1, \omega}}}{2} 
|x|)$ satisfies $- \Delta Y_{\omega} 
+ \frac{\alpha_{1, \omega}}{4} Y_{\omega} = 0$.
Then, it follows from the 
maximum principle that 
\begin{equation} \label{S3-32}
\widehat{u}_{1, \omega}(x) \leq 
Y_{\omega}(x) = 
C_{0} |x|^{-1} \exp(- \frac{\sqrt{\alpha_{1, \omega}}}{2} 
|x|) \qquad (|x| \geq R_{0}).
\end{equation} 
Taking $C_{0} > 0$ sufficiently large, if necessary, 
we have by \eqref{S3-23} that 
\begin{equation} \label{S3-33}
\widehat{u}_{1, \omega}(x) \leq 
\left(1 + \frac{\gamma_{1, 2\omega}}{3} |x|^{2} \right)^{-\frac{1}{2}}
\leq
C_{0} |x|^{-1} \exp(- \frac{\sqrt{\alpha_{1, \omega}}}{2} 
|x|) \qquad (|x| < R_{0}).
\end{equation} 
Thus, from \eqref{S3-32} and \eqref{S3-33}, 
we conclude that \eqref{S3-31} holds. 
\end{proof}
Now, we are in a position to estimate 
$H^{1}$-norm of $\widehat{u}_{1, \omega}$. 
\begin{lemma}\label{lem2-7}
Let $1 < p < 5$ and $\{u_{1, \omega}\}$ 
be a family of positive solutions to \eqref{sp} satisfying 
$\limsup_{\omega \to 0} 
\|u_{1, \omega}\|_{L^{\infty}} < \infty$. 
We have 
$\limsup_{\omega \to 0} 
\| \widehat{u}_{1, \omega}\|_{H^{1}} 
< + \infty$, where 
$\widehat{u}_{1, \omega}$ is defined by 
\eqref{S3-5}. 
\end{lemma}
\begin{proof}
We see from Lemmas 
\ref{lem2-5} and \ref{lem2-6} 
that 
\begin{equation}\label{S3-34}
\limsup_{\omega \to 0} 
\|\widehat{u}_{1, \omega}\|_{L^{q}} 
< + \infty
\end{equation}
for any $q \geq 1$. 
Multiplying \eqref{S3-6} by 
$\widehat{u}_{1, \omega}$ and integrating 
the resulting equation, one has 
\begin{equation*}
\|\nabla \widehat{u}_{1, \omega}\|_{L^{2}}^{2} 
+ \alpha_{1, \omega} 
\|\widehat{u}_{1, \omega}\|_{L^{2}}^{2}
= 
\|\widehat{u}_{1, \omega}\|_{L^{p + 1}}^{p + 1} 
+ M_{1, \omega}^{5 - p} 
\|\widehat{u}_{1, \omega}\|_{L^{6}}^{6}. 
\end{equation*}
This together with the assumption 
$\limsup_{\omega \to 0} M_{1, \omega} 
< \infty$ (recall that 
$M_{1, \omega} = \|u_{1, \omega}\|_{L^{\infty}}$)
and \eqref{S3-34} implies that
\[
\limsup_{\omega \to 0} 
\left\{
\|\nabla \widehat{u}_{1, \omega}\|_{L^{2}} 
+ \alpha_{1, \omega} 
\|\widehat{u}_{1, \omega}\|_{L^{2}}^{2}
\right\}
\leq \limsup_{\omega \to 0}
\left\{ 
\|\widehat{u}_{1, \omega}\|_{L^{p + 1}}^{p + 1} 
+ M_{1, \omega}^{5 - p} 
\|\widehat{u}_{1, \omega}\|_{L^{6}}^{6}
\right\} \lesssim 1. 
\]
Since $\alpha_{1, \omega} \sim 1$ as $\omega \to 0$ 
(see Proposition \ref{propS3-1}), 
we obtain the desired result.
\end{proof}
We are now in a position to prove 
Theorem \ref{thm-0-1}. 
\begin{proof}[Proof of Theorem \ref{thm-0-1}]
Note that 
\begin{equation} \label{S3-35}
\begin{split}
u_{1, \omega}(r) 
& = M_{1, \omega} \widehat{u}_{1, \omega}
(M_{1, \omega}^{\frac{p - 1}{2}} r). 
\end{split}
\end{equation}
Since $\widetilde{u}_{1, \omega}(\cdot) = 
\omega^{- \frac{1}{p-1}}
u_{1, \omega}
(- \omega^{- \frac{1}{2}} \cdot)$ 
and $\alpha_{1, \omega} = \omega M_{1, \omega}^{1-p}$ 
(see \eqref{main-eq1} and \eqref{S3-1}), 
we have by \eqref{S3-35} that 
\begin{equation*} 
\begin{split}
\widetilde{u}_{1, \omega}(r) 
= \omega^{- \frac{1}{p-1}}
u_{1, \omega}
(\omega^{- \frac{1}{2}} r) 
= \omega^{- \frac{1}{p-1}}
M_{1, \omega}
\widehat{u}_{1, \omega}
(\omega^{- \frac{1}{2}} 
M_{1, \omega}^{\frac{p-1}{2}}r)
= \alpha_{1, \omega}^{- \frac{1}{p-1}} 
\widehat{u}_{1, \omega}
(\alpha_{1, \omega}^{- \frac{1}{2}} r). 
\end{split}
\end{equation*}
Since $\alpha_{1, \omega} \sim 1$ 
as $\omega \to 0$, 
we see from Lemma \ref{lem2-7} that 
\begin{equation*}
\limsup_{\omega \to 0} 
\|\widetilde{u}_{1, \omega}\|_{H^{1}} 
< + \infty. 
\end{equation*}
Since $\limsup_{\omega \to 0} 
\|\widetilde{u}_{1, \omega}\|_{H^{1}} < \infty$, 
there exist a subsequence of 
$\{\widetilde{u}_{1, \omega}\}$ and 
$u_{1, \infty} \in H^{1}(\R^{3})$ such that 
$\lim_{\omega \to 0} 
\widetilde{u}_{1, \omega} = u_{1, \infty}$ 
weakly in $H^{1}(\R^{3})$. 
In addition, by Proposition \ref{propS3-1}, 
we obtain 
\[
\widetilde{u}_{1, \omega}(0) 
= \omega^{- \frac{1}{p-1}} M_{1, \omega}
= (\omega M_{1, \omega}^{1 - p})
^{- \frac{1}{p-1}} 
= \alpha_{1, \omega}^{- \frac{1}{p-1}}
\gtrsim 1 \qquad \mbox{for any sufficiently small $\omega > 0$}, 
\] 
where the implicit constant does not depend on $\omega > 0$. 
This yields that $u_{1, \infty} \neq 0$. 

Observe that 
$\widetilde{u}_{1, \omega}$ 
satisfies 
\begin{equation*}
- \Delta \widetilde{u}_{1, \omega} 
+ \widetilde{u}_{1, \omega}
- \widetilde{u}_{1, \omega}^{p} 
- \omega^{\frac{5 - p}{p - 1}} 
\widetilde{u}_{1, \omega}^{5}
= 0. 
\end{equation*}
Then, we see that 
the weak limit $u_{1, \infty}$ is a 
positive solution to \eqref{eq-sc} with $\omega = 1$. 
From the uniqueness of the positive 
radial solution to \eqref{eq-sc}, 
we have $u_{1, \infty} = U^{\dagger}$. 
In addition, since $\lim_{\omega \to 0} 
\widetilde{u}_{1, \omega} = u_{1, \infty} = U^{\dagger}$ 
strongly in $L^{q}(\R^{3})$ for 
$2 < q < 6$ by the radial compactness (see e.g. 
Berestycki and P. L. Lions~\cite[Lemma A.2]{MR695535}) 
and the family 
$\{\widetilde{u}_{1, \omega}\}$ is bounded in 
$H^{1}(\R^{3})$, we have 
\[
\|\widetilde{u}_{1, \omega}\|_{H^{1}}^{2} 
= \|\widetilde{u}_{1, \omega}
\|_{L^{p + 1}}^{p+1} 
+ \omega^{\frac{5 - p}{p - 1}} 
\|\widetilde{u}_{1, \omega}
\|_{L^{6}}^{6} 
\to \|U^{\dagger}\|_{L^{p+1}}^{p+1} 
= \|U^{\dagger}\|_{H^{1}}^{2}
\]
This together with 
$\lim_{\omega \to 0} 
\widetilde{u}_{1, \omega} = u_{1, \infty} 
= U^{\dagger}$ 
weakly in $H^{1}(\R^{3})$ implies that 
$\lim_{\omega \to 0} 
\widetilde{u}_{1, \omega} = u_{1, \infty} = U^{\dagger}$ 
strongly in $H^{1}(\R^{3})$. 
Thus, we see that 
Theorem \ref{thm-0-1} holds. 
\end{proof}
\begin{proof}[Proof of Theorem \ref{thm-0}]
We can show Theorem \ref{thm-0} \textrm{(i)}
by using the similar argument of 
\cite[Proposition 2.0.4]{MR4320767}. 
In addition, we can prove Theorem 
\ref{thm-0} \textrm{(ii)} as 
in the proof of Lemma 2.3 of \cite{MR2756065}. 
\end{proof} 
\section{Convergence to the Aubin-Talenti function} 
\label{sec-cat}
In what follows, we consider only the family of positive solutions to \eqref{sp} satisfying condition $\lim_{\omega \to 0} 
\|u_{2, \omega}\| = \infty$.
Thus, when no confusion arises, we denote
$u_{2, \omega}, 
\alpha_{2, \omega}, \beta_{2, \omega}$ simply by
$\widetilde{u}_{\omega}, 
\alpha_{\omega}, \beta_{\omega}$, respectively.

This section is devoted to the proof of Theorem \ref{thm-bl-0}. 
Let $\{u_{\omega}\}$ be a family of 
positive solutions to \eqref{sp} satisfying $\lim_{\omega \to 0} 
\|u_{\omega}\|_{L^{\infty}} = \infty$. 
We define 
\[
M_{\omega} := \|u_{\omega}\|_{L^{\infty}}, 
\qquad 
\widetilde{u}_{\omega}(\cdot) 
:= M_{\omega}^{-1}u_{\omega}(M_{\omega}^{- 2}\cdot). 
\]
Then, we see that 
for each $\omega > 0$, 
$\widetilde{u}_{\omega}$ 
satisfies \eqref{main-eq41}. 
We can obtain the following decay 
estimate of 
$\widetilde{u}_{\omega}$: 
\begin{lemma}\label{lem3-3}
Let $1 < p < 3$ and 
$\widetilde{u}_{\omega}$ be a positive solution to 
\eqref{main-eq41}.
There exists a sufficiently small 
$\omega_{1} > 0$ 
such that for $\omega 
\in (0, \omega_{1})$, we have 
\begin{equation}\label{EqL-1}
\widetilde{u}_{\omega}(r) \leq \left( 
1 + \frac{\gamma_{\omega} }{3} r^{2}
\right)^{-\frac{1}{2}}
\qquad \mbox{for all $r > 0$}, 
\end{equation}
where 
\begin{equation}\label{EqL-2}
\gamma_{\omega}:= - \alpha_{\omega} + \beta_{\omega} + 1
\end{equation}
\end{lemma}
We can prove Lemma \ref{lem3-3} by 
a similar argument to the proof of 
Lemma \ref{lem2-5}. 
Thus, we omit the proof. 
Next, we claim the following identity holds: 
\begin{lemma}\label{lem-poho1}
Let $1 < p < 3$ and 
$\widetilde{u}_{\omega}$ be a positive solution to 
\eqref{main-eq41}.
Then, the identity \eqref{main-eq7} holds, that is, 
\begin{equation} \label{EqL-3} 
\alpha_{\omega}\|\widetilde{u}_{\omega}\|_{L^{2}}^{2} 
= \frac{5 - p}{2 (p+1)} \beta_{\omega} 
\|\widetilde{u}_{\omega}\|_{L^{p+1}}^{p+1}.
\end{equation}
\end{lemma}
\begin{proof}
Multiplying \eqref{main-eq41} by 
$\widetilde{u}_{\omega}$ and integrating 
the resulting equation, we have 
\begin{equation}\label{EqL-4}
\|\nabla \widetilde{u}_{\omega}\|_{L^{2}}^{2} 
+ \alpha_{\omega} 
\|\widetilde{u}_{\omega}\|_{L^{2}}^{2}
= 
\beta_{\omega}
\|\widetilde{u}_{\omega}\|_{L^{p+1}}^{p+1} 
+ \|\widetilde{u}_{\omega}\|_{L^{6}}^{6}. 
\end{equation}
From the Pohozaev identity 
of \eqref{main-eq41}, we obtain 
\begin{equation}\label{EqL-5}
\frac{1}{2} \|\nabla \widetilde{u}_{\omega}
\|_{L^{2}}^{2} 
+ \frac{3}{2} \alpha_{\omega} 
\|\widetilde{u}_{\omega}\|_{L^{2}}^{2}
= \frac{3}{p+1} 
\beta_{\omega} 
\|\widetilde{u}_{\omega}\|_{L^{p+1}}^{p+1} 
+ \frac{1}{2} 
\|\widetilde{u}_{\omega}\|_{L^{6}}^{6}. 
\end{equation}
From \eqref{EqL-4} and 
\eqref{EqL-5}, we see that \eqref{main-eq7} holds.
\end{proof}
From Lemmas \ref{lem3-3} and 
\ref{lem-poho1}, we can obtain a 
uniform boundedness of 
\{$\|\nabla \widetilde{u}_{\omega}\|_{L^{2}}$\}: 
\begin{lemma}\label{lem3-5}
Let $1 < p < 3$ and 
$\widetilde{u}_{\omega}$ be a positive solution to 
\eqref{main-eq41}.
Then, the following hold: 
\begin{enumerate}
\item[\textrm{(i)}] 
\begin{equation}\label{EqL-6}
\limsup_{\omega \to 0} 
\sqrt{\alpha_{\omega}} 
\|\widetilde{u}_{\omega}\|_{L^{2}}^{2} 
< \infty. 
\end{equation}
\item[\textrm{(ii)}] 
\begin{equation} \label{EqL-7}
\limsup_{\omega \to 0} 
\beta_{\omega}\| \widetilde{u}_{\omega}\|_{L^{p+1}}^{p+1} 
= 0, 
\end{equation}
\item[\textrm{(iii)}] 
\begin{equation*}
\limsup_{\omega \to 0} 
\|\nabla \widetilde{u}_{\omega}\|_{L^{2}} 
< + \infty. 
\end{equation*}
\end{enumerate}
\end{lemma}
\begin{proof}
\textrm{(i)}\; 
We divide the proof of \textrm{(i)} 
into 2 steps. 

\textbf{(Step 1). 
}
Define
\begin{equation} 
\label{EqL-8}
L_{\omega} := 
\inf\left\{ 
R > 0 \colon 
\frac{p+3}{2(p+1)}
\beta_{\omega} 
\widetilde{u}_{\omega}^{p-1}(r) 
\leq \alpha_{\omega}
\quad (r \in [R, \infty))
\right\}. 
\end{equation}
Since $\widetilde{u}_{\omega}$ is decreasing in $r > 0$ 
and $\lim_{r \to \infty} \widetilde{u}_{\omega}(r) = 0$, 
we see that 
\[
\frac{p+3}{2(p+1)}
\beta_{\omega} 
\widetilde{u}_{\omega}^{p-1}(r) 
\begin{cases}
> \alpha_{\omega} & \qquad r \in [0, L_{\omega}), \\
= \alpha_{\omega} & \qquad r = L_{\omega}, \\
< \alpha_{\omega} & \qquad r \in (L_{\omega}, \infty). 
\end{cases}
\]
We claim that 
\begin{equation}\label{EqL-9}
\sqrt{\alpha_{\omega}}
L_{\omega} \leq 
\frac{\sqrt{6(p + 3)}}{p-1}. 
\end{equation} 

If $L_{\omega} = 0$, then 
\eqref{EqL-9} is trivial. 
Thus, we may assume that $L_{\omega} > 0$. 
By the definition of $L_{\omega}$, 
we have 
$\frac{p+3}{2(p+1)} 
\beta_{\omega} 
\widetilde{u}_{\omega}^{p-1}(r) 
\geq \alpha_{\omega}$ for $r \leq L_{\omega}$, 
so that $\alpha_{\omega}
\beta_{\omega}^{-1}
(\widetilde{u}_{\omega}(r))^{- (p-1)} 
\leq \frac{p+3}{2(p+1)}$. 
Thus, we obtain 
\begin{equation} \label{EqL-10}
\begin{split}
\beta_{\omega}\widetilde{u}_{\omega}^{p}(r) 
- \alpha_{\omega} \widetilde{u}_{\omega}(r) 
\geq \beta_{\omega} 
\widetilde{u}_{\omega}^{p}(r) 
\left(1 - \alpha_{\omega} 
\beta_{\omega}^{-1} 
(\widetilde{u}_{\omega}(r))^{- (p-1)} \right) 
& \geq \beta_{\omega} 
\widetilde{u}_{\omega}^{p}(r) 
\left(1 - \frac{p+3}{2(p+1)} \right) \\[6pt]
& = \frac{p-1}{2(p + 1)} \beta_{\omega} \widetilde{u}_{\omega}^{p}(r). 
\end{split}
\end{equation}
Then, from \eqref{main-eq41}, \eqref{EqL-10}, 
the monotonicity 
and positivity of 
$\widetilde{u}_{\omega}$, one has 
\[
\begin{split}
- r^{2} \widetilde{u}_{\omega}^{\prime}(r)
= - \int_{0}^{r} \frac{d}{d s} (s^{2} 
\widetilde{u}^{\prime}(s)) \, ds 
& = - \int_{0}^{r} (s^{2} 
\widetilde{u}^{\prime \prime}(s) 
+ 2s \widetilde{u}^{\prime}(s)) \, ds \\
& = \int_{0}^{r} s^{2} 
(\beta_{\omega} 
\widetilde{u}_{\omega}^{p}(s) 
+ \widetilde{u}_{\omega}^{5}(s) 
- \alpha_{\omega} 
\widetilde{u}_{\omega}(s)) \, ds \\
& \geq \frac{p-1}{2(p + 1)}
\int_{0}^{r} s^{2} \beta_{\omega} 
\widetilde{u}_{\omega}^{p}(s) \, ds \\[6pt]
& \geq \frac{(p-1)}{2(p+1)}
\beta_{\omega} 
\widetilde{u}_{\omega}^{p}(r)
\int_{0}^{r} s^{2} \, ds \\[6pt]
& = 
\frac{p-1}{6(p + 1)}
\beta_{\omega} 
\widetilde{u}_{\omega}^{p}(r) r^{3}
\end{split}
\]
for $r \leq L_{\omega}$. 
This implies that 
\[
\left(\frac{(\widetilde{u}_{\omega}
(r))^{- (p - 1)}}{p-1} 
\right)^{\prime} 
= - \widetilde{u}_{\omega}^{\prime}(r) 
(\widetilde{u}_{\omega}(r))^{- p}
\geq \frac{p-1}{6(p + 1)}\beta_{\omega} r 
= \left(
\frac{p-1}{12(p + 1)} \beta_{\omega} 
r^{2} \right)^{\prime}
\]
for $r \leq L_{\omega}$. 
Integrating from $0$ to $r$, we obtain 
\[
\frac{(\widetilde{u}_{\omega}
(r))^{- (p - 1)}}{p-1} 
- \frac{(\widetilde{u}_{\omega}(0))^{- (p - 1)}}
{p-1} 
\geq 
\frac{p-1}{12(p + 1)} \beta_{\omega} r^{2}
- 
\frac{p-1}{12(p + 1)} \beta_{\omega} 
\times 0^{2}
= 
\frac{p-1}{12(p + 1)} \beta_{\omega} 
r^{2}. 
\]
It follows from the initial condition 
$\widetilde{u}_{1, \omega}(0) = 1$ 
(see \eqref{main-eq41}) that 
\[
\frac{(\widetilde{u}_{\omega}
(r))^{- (p - 1)}}{p-1}
\geq 
\frac{p-1}{12(p + 1)} \beta_{\omega} 
r^{2}. 
\]
for $r \leq L_{\omega}$. 
From this, we conclude that 
\[
\widetilde{u}_{\omega}(r)
\leq \left(\frac{12(p+1)}{(p-1)^{2}} \right)^{\frac{1}{p-1}} \beta_{\omega}^{- \frac{1}{p-1}} 
r^{- \frac{2}{p-1}} 
\]
for $r \leq L_{\omega}$. 
From this inequality and 
the definition of $L_{\omega}$, we have 
\[
\left(\frac{2(p + 1)}{p + 3} 
\alpha_{\omega} \beta_{\omega}^{-1}
\right)^{\frac{1}{p-1}} 
= \widetilde{u}_{\omega}
(L_{\omega})
\leq 
\left(\frac{12(p+1)}{(p-1)^{2}} \right)^{\frac{1}{p-1}}
\beta_{\omega}^{- \frac{1}{p-1}} 
L_{\omega}^{- \frac{2}{p-1}}. 
\]
Thus, we see that \eqref{EqL-9} 
holds. 

\textbf{(Step 2). 
}
We finish the proof of \textrm{(i)}. 
We put 
\[
E(r):= 
\frac{(\widetilde{u}_{\omega}^{\prime}(r))^{2}}{2} 
+ \beta_{\omega} 
\frac{(\widetilde{u}_{\omega}(r))^{p+1}}{p+1} 
+ \frac{(\widetilde{u}_{\omega}(r))^{6}}{6} 
- \alpha_{\omega} 
\frac{(\widetilde{u}_{\omega}(r))^{2}}{2}. 
\]
Then, it follows from \eqref{main-eq41} that 
\[
\begin{split}
E^{\prime}(r) 
= 
\widetilde{u}_{\omega}^{\prime}(r)\widetilde{u}_{\omega}^{\prime \prime}(r)
+ \beta_{\omega} 
\widetilde{u}_{\omega}^{p}(r) \widetilde{u}_{\omega}^{\prime}(r)
+ \widetilde{u}_{\omega}^{5}(r) \widetilde{u}_{\omega}^{\prime}(r)
- \alpha_{\omega} \widetilde{u}_{\omega}(r)\widetilde{u}_{\omega}^{\prime}(r) 
= - \frac{2}{r} (\widetilde{u}_{\omega}^{\prime}(r))^{2} < 0. 
\end{split}
\]
This together with $\lim_{r \to \infty} \widetilde{u}_{\omega}(r) 
= 0$, we have 
\[
E(r) \geq \liminf_{r \to \infty}E(r) 
= \liminf_{r \to \infty}
\frac{(\widetilde{u}_{\omega}^{\prime}(r))^{2}}{2} 
\geq 0
\qquad (r > 0).
\]
Using $E(r) \geq 0, \frac{p+3}{2(p+1)}
\beta_{\omega} 
\widetilde{u}_{\omega}^{p-1}(r) 
\leq \alpha_{\omega}$
for $r \geq L_{\omega}$, \eqref{EqL-1} and \eqref{EqL-9}, 
we obtain
\begin{equation*}
\begin{split}
\frac{
(\widetilde{u}_{\omega}^{\prime}(r))^{2}}{2} 
& \geq \alpha_{\omega} 
\frac{(\widetilde{u}_{\omega}(r))^{2}}{2} 
- \beta_{\omega} 
\frac{(\widetilde{u}_{\omega}(r))^{p+1}}{p+1} 
- \frac{(\widetilde{u}_{\omega}(r))^{6}}{6} \\[6pt]
& \geq 
\alpha_{\omega} 
\frac{(\widetilde{u}_{\omega}(r))^{2}}{2} 
- \frac{2}{p+3}\alpha_{\omega} 
(\widetilde{u}_{\omega}(r))^{2}
- C \alpha_{\omega}^{4}
(\widetilde{u}_{\omega}(r))^{2} \\[6pt]
& = \frac{p - 1}{2(p+3)}\alpha_{\omega} 
(\widetilde{u}_{\omega}(r))^{2}
- C \alpha_{\omega}^{4}
(\widetilde{u}_{\omega}(r))^{2}. 
\end{split}
\end{equation*}
for $r \geq \max\{L_{\omega}, \alpha_{\omega}^{- \frac{1}{2}}\}$, where $C$ is a constant independent of $\omega$. 
Since $\lim_{\omega \to 0} 
\alpha_{\omega} = 0$, we have 
\begin{equation} \label{EqL-11}
(\widetilde{u}_{\omega}^{\prime}(r))^{2} 
\geq 
2 \left\{\frac{p - 1}{2(p+3)}\alpha_{\omega} 
(\widetilde{u}_{\omega}(r))^{2}
- C \alpha_{\omega}^{4} 
(\widetilde{u}_{\omega}(r))^{2} \right\}
\geq \frac{p-1}{2(p+3)} \alpha_{\omega} 
(\widetilde{u}_{\omega}(r))^{2}.
\end{equation}
Putting 
\[
C_{p} := \frac{p-1}{2(p+3)}, 
\]
we have by \eqref{EqL-11} that 
\[
0 \leq (\widetilde{u}_{\omega}^{\prime}(r))^{2} 
- C_{p} \alpha_{\omega} 
(\widetilde{u}_{\omega}(r))^{2} 
= (\widetilde{u}_{\omega}^{\prime}(r) 
+ (C_{p} \alpha_{\omega})^{\frac{1}{2}}
\widetilde{u}_{\omega}(r)) 
(\widetilde{u}_{\omega}^{\prime}(r) 
- (C_{p} \alpha_{\omega})^{\frac{1}{2}}
\widetilde{u}_{\omega}(r))
\]
for $r \geq \max\{L_{\omega}, \alpha_{\omega}^{- \frac{1}{2}}\}$. 
This together with 
$\widetilde{u}_{\omega}^{\prime}(r) < 0$
and $\widetilde{u}_{\omega}(r) > 0$
implies that 
\[
\widetilde{u}_{\omega}^{\prime}(r) 
+ (C_{p} \alpha_{\omega})^{\frac{1}{2}}
\widetilde{u}_{\omega}(r) 
\leq 0 
\]
for $r \geq \max\{L_{\omega}, 
\alpha_{\omega}^{- \frac{1}{2}}\}$. 
It follows that 
\[
\left(e^{(C_{p} \alpha_{\omega})^{\frac{1}{2}}r} \widetilde{u}_{\omega}(r)\right)^{\prime} 
= e^{(C_{p} \alpha_{\omega})^{\frac{1}{2}}r} 
(\widetilde{u}_{\omega}^{\prime}(r) 
+ (C_{p} \alpha_{\omega})^{\frac{1}{2}}
\widetilde{u}_{\omega}(r)) \leq 0 
\]
for $r \geq \max\{L_{\omega}, 
\alpha_{\omega}^{- \frac{1}{2}}\}$.
Thus, we obtain 
\begin{equation} \label{EqL-12}
\widetilde{u}_{\omega}(r) 
\leq e^{- (C_{p} \alpha_{\omega})^{\frac{1}{2}}r} 
e^{(C_{p} \alpha_{\omega})^{\frac{1}{2}}
\max\{L_{\omega}, 
\alpha_{\omega}^{- \frac{1}{2}}\}} 
\widetilde{u}_{\omega}(\max\{L_{\omega}, 
\alpha_{\omega}^{- \frac{1}{2}}\})
\end{equation}
We see from \eqref{EqL-1}, 
\eqref{EqL-9} and 
\eqref{EqL-12} that 
\begin{equation} \label{EqL-13}
\widetilde{u}_{\omega}(r) 
\lesssim 
e^{- (C_{p} \alpha_{\omega})^{\frac{1}{2}}r} 
\left(1 + \frac{\gamma_{\omega} }{3} 
\max\{L_{\omega}^{2}, 
\alpha_{\omega}^{-1}\} 
\right)^{- \frac{1}{2}}
\lesssim
\sqrt{\alpha_{\omega}}
e^{- (C_{p} \alpha_{\omega})^{\frac{1}{2}}r} 
\end{equation}
for $r \geq \max\{L_{\omega}, 
\alpha_{\omega}^{- \frac{1}{2}}\}$. 
Using 
$\limsup_{\omega \to 0} 
L_{\omega} \sqrt{\alpha_{\omega}} < + 
\infty$, \eqref{EqL-1} and \eqref{EqL-13}, 
we obtain 
\begin{equation*} 
\begin{split}
\|\widetilde{u}_{\omega}\|_{L^{2}}^{2} 
& 
\lesssim \int_{0}^{\max\{L_{\omega}, 
\alpha_{\omega}^{- \frac{1}{2}}\}} 
\widetilde{u}_{\omega}^{2}(r) r^{2} \, dr 
+ \int_{\max\{L_{\omega}, \alpha_{\omega}^{- \frac{1}{2}}\}}^{\infty} 
\widetilde{u}_{\omega}^{2} r^{2} \, dr \\[6pt]
& \lesssim 
\int_{0}^{\max\{L_{\omega}, 
\alpha_{\omega}^{- \frac{1}{2}}\}} 
r^{-2} r^{2} \, dr 
+ 
\int_{\max\{L_{\omega}, 
\alpha_{\omega}^{- \frac{1}{2}}\}}^{\infty} 
\alpha_{\omega}
e^{- 2 (C_{p} \alpha_{\omega})^{\frac{1}{2}}r} r^{2} \, dr \\[6pt]
& \lesssim 
\max\{L_{\omega}, 
\alpha_{\omega}^{- \frac{1}{2}}\}
+ \alpha_{\omega}^{- \frac{1}{2}} 
\lesssim \alpha_{\omega}^{- \frac{1}{2}}. 
\end{split}
\end{equation*} 
Thus, we find that \eqref{EqL-6} holds. 

\textrm{(ii)}
From the definition of $\alpha_{\omega}$ (see \eqref{main-eq5}), 
we see that 
\begin{equation}\label{EqL-14}
\lim_{\omega \to 0}\alpha_{\omega} =0. 
\end{equation}
It follows from \eqref{EqL-6} and \eqref{EqL-14} that 
\begin{equation} \label{EqL-15}
\limsup_{\omega \to 0} 
\alpha_{\omega} \|\widetilde{u}_{\omega}
\|_{L^{2}}^{2} = 0. 
\end{equation}
\eqref{EqL-15} together with 
\eqref{EqL-3} yields that 
\[
\limsup_{\omega \to 0} 
\beta_{\omega} \|\widetilde{u}_{\omega}
\|_{L^{p+1}}^{p+1} = 0.
\]
Thus, we see that \eqref{EqL-7} holds.

\textrm{(iii)}
Observe from \eqref{EqL-1} that 
\begin{equation}\label{EqL-16}
\limsup_{\omega \to 0} 
\|\widetilde{u}_{\omega}\|_{L^{6}}
< + \infty.
\end{equation}
By \eqref{EqL-4}, \eqref{EqL-7} 
and \eqref{EqL-16}, 
we obtain that
\[
\limsup_{\omega \to 0} 
\|\nabla \widetilde{u}_{\omega}\|_{L^{2}}^{2} 
\leq \limsup_{\omega \to 0}
\left\{ 
\beta_{\omega} 
\|\widetilde{u}_{\omega}\|_{L^{p+1}}^{p+1} 
+ \|\widetilde{u}_{\omega}\|_{L^{6}}^{6} 
\right\} < \infty. 
\]
This completes the proof. 
\end{proof}
Here, in order to prove 
Theorem \ref{thm-bl-0}, 
we recall the following compactness result by 
Brezis and Lieb~\cite{MR699419}. 
\begin{lemma}[Brezis and Lieb~\cite{MR699419}]\label{lemm-bl}
Let $\{u_{n}\}$ be a bounded sequence in $H^{1}(\R^{3})$ such that 
\begin{equation*}
\lim_{n\to \infty}u_{n}(x)=u(x) 
\qquad 
\mbox{almost all $x\in \R^{3}$}
\end{equation*}
for some function $u \in H^{1}(\R^{3})$. Then, 
for any $2 \le r \le 6$, up to some subsequence, 
\begin{equation*}
\lim_{n\to \infty}\left\{
\|u_{n}\|_{L^{r}}^{r} - \|u_{n}-u\|_{L^{r}}^{r} 
- \|u\|_{L^{r}}^{r} 
\right\} =0, 
\end{equation*}
\begin{equation*}
\lim_{n\to \infty} 
\left\{\|\nabla u_{n}\|_{L^{2}}^{2} - 
\|\nabla (u_{n}-u)\|_{L^{2}}^{2} 
- \|\nabla u\|_{L^{2}}^{2}\right\} =0.
\end{equation*}
\end{lemma}

We are now in a position to prove 
Theorem \ref{thm-bl-0}. 
\begin{proof}[Proof of Theorem 
\ref{thm-bl-0}]
Let $\{\omega_{n}\}$ be any sequence 
in $(0, +\infty)$ satisfying 
$\lim_{n \to \infty} \omega_{n} = 0$. 
By the boundedness of 
$\|\nabla \widetilde{u}_{\omega_{n}}\|_{L^{2}}^{2}$ 
(see Lemma \ref{lem3-5} \textrm{(iii)}), 
$\|\widetilde{u}_{\omega_{n}}\|_{L^{\infty}}=1$, the $W_{loc}^{2,q}$ estimate, 
and Schauder's estimate 
(see e.g. \cite[Theorems 6.2 and 8.17]{MR1814364}), 
we verify that there exists 
$\widetilde{u}_{\infty} \in \dot{H}^{1}(\mathbb{R}^{3})$ such that, passing to some subsequence, 
\begin{equation}\label{EqL-17}
\lim_{n\to \infty}\widetilde{u}_{\omega_{n}} 
= \widetilde{u}_{\infty} 
\quad 
\text{weakly in $\dot{H}^1(\mathbb{R}^{3})$, and strongly in 
$C^2_{\rm loc}(\mathbb{R}^{3})$}.
\end{equation}
In addition, we see 
from \eqref{main-eq41}, 
\eqref{main-eq6} and \eqref{EqL-17} that 
$\widetilde{u}_{\infty}$ is a solution to 
\begin{equation*}
- \Delta \widetilde{u}_{\infty} - 
\widetilde{u}_{\infty}^{5} = 0 \qquad \mbox{in $\R^{3}$}. 
\end{equation*}
From \eqref{EqL-17} and 
$\widetilde{u}_{\omega_{n}}(0) 
= \|\widetilde{u}_{\omega_{n}}\|_{L^{\infty}} = 1$, 
one has 
$1 = \lim_{n \to \infty} \widetilde{u}_{\omega_{n}}(0) 
= \widetilde{u}_{\infty}(0)$.
Thus, $\widetilde{u}_{\infty}$ satisfies 
\begin{equation*}
\left\{ \begin{array}{l}
-\Delta \widetilde{u}_{\infty} 
= \widetilde{u}_{\infty}^{5} \quad {\rm in} \ \mathbb{R}^{3},
\\[6pt]
\widetilde{u}_{\infty}(0)=1.
\end{array} \right. 
\end{equation*}
Then, by the results in 
\cite[Corollary 8.2]{MR982351}, 
we infer that 
$\widetilde{u}_{\infty} = W$, 
where $W$ is the Aubin-Talenti function defined by \eqref{talenti}. 
We put 
\[
\begin{split}
\mathcal{N}^{\ddagger}(u) :=
\|\nabla W\|_{L^{2}}^{2} 
- \|W\|_{L^{6}}^{6}, \qquad 
\mathcal{E}^{\ddagger}(u) 
:= \frac{1}{2} \|\nabla u\|_{L^{2}}^{2}
- \frac{1}{6} \|u\|_{L^{6}}^{6}. 
\end{split}
\]
It is easy to verify that 
$\mathcal{N}^{\ddagger}
(W) = 0$.
Note that $W$ also satisfies 
$\|\nabla W\|_{L^{2}}^{2} 
= \sigma \|W\|_{L^{6}}^{2}$, 
where 
\begin{equation} \label{EqL-18}
\sigma:= 
\inf\left\{
\|\nabla u\|_{L^{2}}^{2} 
\colon u \in \dot{H}^{1}(\R^{3}) 
\; \mbox{with $\|u\|_{L^{6}} = 1$}
\right\}. 
\end{equation}
By $\mathcal{N}^{\ddagger}
(W) = 0$ and $\|\nabla W\|_{L^{2}}^{2} 
= \sigma \|W\|_{L^{6}}^{2}$, 
we see that $\|W\|_{L^{6}}^{6} = 
\sigma^{\frac{3}{2}}$ and 
\begin{equation}\label{EqL-19}
\mathcal{E}^{\ddagger}(W) 
= 
\frac{1}{2} \|\nabla W\|_{L^{2}}^{2}
-
\frac{1}{6} \|W\|_{L^{6}}^{6} 
=
\frac{1}{3} \|W\|_{L^{6}}^{6} 
= 
\frac{1}{3}\sigma^{\frac{3}{2}}.
\end{equation}
This together with \eqref{EqL-17}, 
$\widetilde{u}_{\infty} = W$ 
and \eqref{EqL-1} 
yields that 
\begin{equation}\label{EqL-20}
\lim_{n \to 0} \|\widetilde{u}_{\omega_{n}}\|_{L^{6}}^{6} 
= \|W\|_{L^{6}}^{6} 
= \sigma^{\frac{3}{2}}. 
\end{equation}
Observe from \eqref{main-eq41} 
that for each $n \in \N$, we have 
$\widetilde{\mathcal{N}}_{\omega_{n}}
(\widetilde{u}_{\omega_{n}})=0$, 
where 
\[
\widetilde{\mathcal{N}}_{\omega}(u) := 
\|\nabla u\|_{L^{2}}^{2} 
+ \alpha_{\omega} \|u\|_{L^{2}}^{2} 
- \beta_{\omega} \|u\|_{L^{p+1}}^{p+1} 
- \|u\|_{L^{6}}^{6}. 
\] 
It follows from 
$\widetilde{\mathcal{N}}_{\omega_{n}}
(\widetilde{u}_{\omega_{n}})=0$, 
\eqref{main-eq6}, \eqref{EqL-20}, 
\eqref{EqL-6} and 
\eqref{EqL-7} that 
\begin{equation} \label{EqL-21}
\begin{split}
\lim_{n \to \infty}
\mathcal{E}^{\ddagger}(\widetilde{u}_{\omega_{n}}) 
= 
\lim_{n \to \infty} 
\Bigm\{ 
\frac{1}{3} \|\widetilde{u}_{\omega_{n}}\|_{L^{6}}^{6} 
- \frac{\alpha_{\omega_{n}}}{2} 
\|\widetilde{u}_{\omega_{n}}\|_{L^{2}}^{2} 
+ 
\frac{\beta_{\omega_{n}}}{2} 
\|\widetilde{u}_{\omega_{n}}\|_{L^{p+1}}^{p+1}
\Bigm\} 
= \frac{1}{3}\sigma^{\frac{3}{2}}. 
\end{split}
\end{equation}
Similarly, we can verify that 
\begin{equation} \label{EqL-22}
\begin{split}
\lim_{n \to \infty}
\mathcal{N}^{\ddagger}(\widetilde{u}_{\omega_{n}}) 
& = 
\lim_{n \to \infty} 
\bigm\{ 
-\alpha_{\omega_{n}} 
\|\widetilde{u}_{\omega_{n}}\|_{L^{2}}^{2} 
+ \beta_{\omega_{n}}
\|\widetilde{u}_{\omega_{n}}\|_{L^{p+1}}^{p+1}
\bigm\} = 0.
\end{split}
\end{equation}
Using the MR699419 lemma (see Lemma \ref{lemm-bl}), we have
\begin{align}
\label{EqL-23}
\mathcal{N}^{\ddagger}(\widetilde{u}_{\omega_{n}})
- \mathcal{N}^{\ddagger}(\widetilde{u}_{\omega_{n}} - W)
-\mathcal{N}^{\ddagger}(W)
& = o_{n}(1), \\[6pt]
\mathcal{E}^{\ddagger}(\widetilde{u}_{\omega_{n}})
- \mathcal{E}^{\ddagger}(\widetilde{u}_{\omega_{n}} - W)
-\mathcal{E}^{\ddagger}(W)
& = o_{n}(1). 
\label{EqL-24}
\end{align} 
Putting \eqref{EqL-22}, \eqref{EqL-23} and 
$\mathcal{N}^{\ddagger}
(W) = 0$ together, we obtain 
\begin{equation}\label{EqL-25}
0 = \lim_{n \to \infty} \mathcal{N}^{\ddagger}
(\widetilde{u}_{\omega_{n}} - W) 
= 
\lim_{n \to \infty}
\bigm\{ 
\|\nabla (\widetilde{u}_{\omega_{n}} - W)\|_{L^{2}}^{2}
- \|\widetilde{u}_{\omega_{n}} 
- W\|_{L^{6}}^{6}
\bigm\}.
\end{equation} 
Furthermore, we see from \eqref{EqL-25} that 
\begin{equation} \label{EqL-26}
\mathcal{E}^{\ddagger}
(\widetilde{u}_{\omega_{n}} - W) 
= 
\frac{1}{3} \| \nabla (\widetilde{u}_{\omega_{n}} - \widetilde{u}_{\infty})\|_{L^{2}}^{2} +
o_{n}(1). 
\end{equation}
Now, we find from \eqref{EqL-24}, 
\eqref{EqL-21} and \eqref{EqL-19} that 
\begin{equation}\label{EqL-27}
\begin{split} 
\mathcal{E}^{\ddagger}
(\widetilde{u}_{\omega_{n}} - W)
=
\mathcal{E}^{\ddagger}(\widetilde{u}_{\omega_{n}}) 
- \mathcal{E}^{\ddagger}(W) 
+ o_{n}(1) 
= \frac{1}{3} \sigma^{\frac{3}{2}} + o_{n}(1) - \frac{1}{3} \sigma^{\frac{3}{2}} 
= o_{n}(1). 
\end{split} 
\end{equation} 
It follows from \eqref{EqL-26} and 
\eqref{EqL-27} that 
\begin{equation} \label{EqL-28}
\lim_{n \to \infty}
\|\nabla (\widetilde{u}_{\omega_{n}} - W) \|_{L^{2}}^{2} = 0. 
\end{equation}
In addition, we see from \eqref{EqL-1} 
and \eqref{EqL-28} that 
\[
\lim_{n \to \infty}
\|\widetilde{u}_{\omega_{n}} - W
\|_{L^{q}} = 0 
\qquad (q > 3). 
\]
Since $\{\omega_{n}\}$ is arbitrary sequence in 
$(0, \infty)$ satisfying 
$\lim_{n \to \infty} \omega_{n} = 0$, we have obtained the 
desired result. 
\end{proof}

\section{Relation between
$\alpha_{\omega}$ and $\beta_{\omega}$}
\label{section-R}
In this section, 
in order to prove Theorem \ref{thm-bl} \textrm{(i)}, 
we shall study the 
relation between $\alpha_{\omega}$ and $\beta_{\omega}$
(see \eqref{main-eq5} for the definitions of 
$\alpha_{\omega}$ and $\beta_{\omega}$). 
More precisely, we shall show the following: 
\begin{proposition}\label{PropR-1}
Let $\alpha_{\omega}$ and $\beta_{\omega}$ 
be the number given by \eqref{main-eq5}. 
We have the followings: 
\begin{enumerate}
\item[\textrm{(i)}]
Let $2 < p < 3$. 
\begin{equation} \label{EqB-1} 
\lim_{\omega \to 0} 
\frac{\beta_{\omega}} {\sqrt{\alpha_{\omega}}}
= \frac{12 \pi (p+1)}{(5 - p)\|W\|_{L^{p+1}}^{p+1}}, 
\end{equation} 
where $W$ is the Aubin-Talenti function defined by 
\eqref{talenti}. 
\item[\textrm{(ii)}]
Let $p = 2$. 
\begin{equation} \label{EqB-2}
\lim_{\omega \to 0}
\frac{\beta_{\omega} |\log \alpha_{\omega}|}
{\sqrt{\alpha_{\omega}}} = 
\frac{2}{\sqrt{3}}. 
\end{equation}
\item[\textrm{(iii)}]
Let $1 < p < 2$. 
\begin{equation} \label{EqB-3} 
\lim_{\omega \to 0} \frac{\beta_{\omega}}
{\alpha_{\omega}^{\frac{3 - p}{2}}} 
= 3^{- \frac{p-1}{2}} V^{p-1}(0), 
\end{equation}
where $V$ is the unique positive solution to 
\begin{equation}\label{EqB-4}
\begin{cases}
v_{ss} - v + s^{1 - p} v^{p} = 0, 
& \qquad \mbox{on $(0, \infty)$}, \\
v_{s}(0) = 0, \qquad \lim_{s \to \infty} 
v(s) = 0. & 
\end{cases}
\end{equation}
\end{enumerate}
\end{proposition}
\begin{remark}
\begin{enumerate}
\item[\textrm{(i)}]
Observe that 
\[
W 
\begin{cases}
\notin L^{p+1}(\R^{3}) 
& \qquad \mbox{when $1 < p \leq 2$}, \\
\in L^{p+1}(\R^{3}) 
& \qquad \mbox{when $p > 2$}.
\end{cases}
\]
Therefore, it is natural that the situation becomes 
different at $p = 2$. 
\item[\textrm{(ii)}]
Note that Wei and Wu~\cite{MR4433054} already showed 
the following:
\[
\begin{cases}
\frac{\beta_{\omega}}{\sqrt{\alpha_{\omega}}} 
\sim 1 & \qquad \mbox{when $2 < p < 3$}, \\
\frac{\beta_{\omega}|\log \alpha_{\omega}|}{\sqrt{\alpha_{\omega}}} 
\sim 1 & \qquad \mbox{when $p = 2$}, \\
\frac{\beta_{\omega}}
{\alpha_{\omega}^{\frac{3-p}{2}}} \sim 1 
& \qquad \mbox{when $1 < p < 2$}. 
\end{cases}
\]
See \cite[Proposition 4.2]{MR4433054} in detail. 
However, in order to prove Theorem \ref{thm-bl} \textrm{(i)},
we need to show that the limits \eqref{EqB-1}--\eqref{EqB-3} 
exists. 
\item[\textrm{(iii)}] 
The existence, 
the uniqueness and non-degeneracy of the positive 
solution to \eqref{EqB-4} are obtained by 
Tolland and Genoud. 
\end{enumerate}
\end{remark}
\subsection{Upper and lower bound of 
$\widetilde{u}_{\omega}$}

To prove Proposition \ref{LemR-6}, 
we obtain the following lower bound of the solution 
$\widetilde{u}_{\omega}$: 
\begin{lemma}
[Lower estimate 
of $\widetilde{u}_{\omega}$] 
\label{LemR-10-1}
Let $1 < p < 3$ and 
$Y_{\omega}(r) := \sqrt{3} 
\frac{e^{- \sqrt{\alpha_{\omega}}|r|}}{|r|}$. 
Then, 
for any $R_{0} > 0$,  
one has
\begin{equation} \label{EqB-35}
\widetilde{u}_{\omega}(r) \geq 
\frac{\widetilde{u}_{\omega}(R_{0})}{
Y_{\omega}(R_{0})} Y_{ \omega}(r)
\qquad \mbox{for all $r \geq R_{0}$}. 
\end{equation} 
\end{lemma}
\begin{proof}
Clearly, we have 
    \[
    \widetilde{u}_{\omega}(R_{0}) =
    \frac{\widetilde{u}_{\omega}(R_{0})}
    {Y_{\omega}(R_{0})} Y_{ \omega}(R_{0}). 
    \]
Observe that $Y_{\omega}$ satisfies 
\begin{equation} \label{EqB-34}
- \Delta Y_{\omega} + \alpha_{\omega} Y_\omega = 0 
\qquad \mbox{in $\R^{3} \setminus \{0\}$}.
\end{equation}
In addition, it follows from \eqref{main-eq41} that 
\begin{equation} \label{EqB-88}
- \Delta \widetilde{u}_{\omega} 
+ \alpha_{\omega} \widetilde{u}_{\omega}
= \beta_{\omega} \widetilde{u}_{\omega}^{p} + 
\widetilde{u}_{\omega}^{5} > 0. 
\end{equation}
Then, we see from \eqref{EqB-34}, 
\eqref{EqB-88} and the maximal principle that 
\eqref{EqB-35} holds. 
\end{proof}

\begin{remark}

Observe from Theorem \ref{thm-bl-0} 
that for any $R>0$, we see that 
\[
\lim_{\omega \to 0}\widetilde{u}_{\omega}(r) = W(r) 
\qquad \mbox{uniformly in $B(0, R)$}, 
\]
Thus, there exist 
$R_{2} > 0$ and 
$\omega_2 > 0$ such that 
for $0 < \omega < \omega_2$, 
we obtain 
\[
\widetilde{u}_{\omega}(r) 
\geq \frac{\sqrt{3}}{2 r} 
\geq \frac{Y_{\omega}(r)}{2} 
\qquad \mbox{for $r = R_{2}$}. 
\]
This together with \eqref{EqB-35} yields that 
for $1 < p < 3$, we obtain 
    \begin{equation} \label{EqB-89}
    \widetilde{u}_{\omega}(r) 
    \geq \frac{Y_{\omega}(r)}{2}. 
    \qquad \mbox{for $r \geq R_{2}$}.
    \end{equation}    
\end{remark}
We also need the following 
upper bound of the solution $\widetilde{u}_{\omega}$: 
\begin{lemma}[Upper estimate 
of $\widetilde{u}_{\omega}$] 
\label{nd-lem9-1}
Let $1 < p < 3$ and 
$\widetilde{u}_{\omega}$ 
be a positive solution to 
\eqref{main-eq41}.
For any $\e > 0$, put 
\[
\widetilde{Y}_{\omega, \e}(r) := 
\frac{e^{- \sqrt{(1 - \e) 
\alpha_{\omega}}|r|}}{|r|}. 
\]
Take $R_{\omega} >0$ such that 
    \begin{equation} \label{EqB-90}
    \varepsilon \geq 
    \frac{\beta_{\omega}}{\alpha_{\omega}} 
    \widetilde{u}_{\omega}^{p-1}(R_{\omega}) 
    + \frac{1}{\alpha_{\omega}} 
    \widetilde{u}_{\omega}^{4}(R_{\omega}). 
    \end{equation}
Then, for any $\omega > 0$, 
one has
\begin{equation} \label{eqU-6-3-1}
\widetilde{u}_{\omega}(r) \leq 
\frac{\widetilde{u}_{\omega}(R_{\omega})}
{\widetilde{Y}_{\omega, \e}(R_{\omega})}
\widetilde{Y}_{\omega, \e}(r)
\qquad \mbox{for $r \geq R_{\omega}$}. 
\end{equation} 
\end{lemma}
\begin{proof}
Immediately, one has  
    \[
    \widetilde{u}_{\omega}(R_{\omega}) =
    \frac{\widetilde{u}_{\omega}(R_{\omega})}
    {Y_{\omega, \e}(R_{\omega})} Y_{\omega, \e}(R_{\omega}). 
    \]
Note that $\widetilde{Y}_{\omega, \e}$ 
satisfies
\[
- \Delta \widetilde{Y}_{\omega, \e} 
+ (1 - \e) \alpha_{\omega} \widetilde{Y}_{\omega, \e} = 0 
\qquad \mbox{in $\R^{3} \setminus \{0\}$}. 
\]
Observe from \eqref{EqB-90} that 
\begin{equation*} 
- \Delta \widetilde{u}_{\omega} 
+ (1- \e)\alpha_{\omega} \widetilde{u}_{\omega}
= -\e \alpha_{\omega} \widetilde{u}_{\omega}
\beta_{\omega} \widetilde{u}_{\omega}^{p} + 
\widetilde{u}_{\omega}^{5} 
= (-\e \alpha_{\omega} 
\beta_{\omega} \widetilde{u}_{\omega}^{p - 1} 
+ \widetilde{u}_{\omega}^{4}) 
\widetilde{u}_{\omega}
< 0. 
\end{equation*}

Then, from the maximum principle theorem, 
we see that \eqref{eqU-6-3-1} holds.  
\end{proof}

\subsection{(Case 1) $2 < p < 3$ 
and (Case 2) $p = 2$}
\subsubsection{(Case 1) $2 < p < 3$}
In this subsection, we shall give the proof of 
Proposition \ref{PropR-1} \textrm{(i)}. 
Theorem \ref{thm-bl-0}, which has been proved in Section \ref{sec-cat}, 
implies that the limit equation of \eqref{main-eq41} 
is \eqref{eq-at}. 
This suggests that 
the terms $\alpha_{\omega} \widetilde{u}_{\omega}$
and $\beta_{\omega} \widetilde{u}_{\omega}^{p}$ 
involving in \eqref{main-eq41} are error terms. 
Thus, to study the coefficients $\alpha_{\omega}$ and 
$\beta_{\omega}$, we need to 
consider the error $\eta_{\omega}:= 
\widetilde{u}_{\omega} - W$.
Then, we can easily verify that $\eta_{\omega}$ satisfies 
\begin{equation} \label{EqB-5}
\begin{split}
(- \Delta - 5 W^{4} + \alpha_{\omega}) 
\eta_{\omega} 
& = F_{\omega}(\eta_{\omega}), 
\end{split}
\end{equation}
where 
\[
\begin{split}
F_{\omega}(\eta_{\omega}) 
:= - \alpha_{\omega} W 
+ \beta_{\omega} W^{p} 
+ (W + \eta_{\omega})^{5} 
- W^{5} - 5 W^{4} \eta_{\omega} 
+ \beta_{\omega}\big\{(W + \eta_{\omega})^{p} 
- W^{p} \big\}.
\end{split}
\]
It follows that 
\begin{equation}\label{EqB-6}
|F_{\omega} (\eta_{\omega})| 
\lesssim \alpha_{\omega} W 
+ \beta_{\omega} W^{p} 
+ W^{3}\eta_{\omega}^{2} 
+ |\eta_{\omega}|^{5} 
+ \beta_{\omega} W^{p-1} |\eta_{\omega}| 
+ \beta_{\omega} |\eta_{\omega}|^{p}. 
\end{equation}
We can rewrite \eqref{EqB-5} by 
\begin{equation} \label{EqB-7}
\eta_{\omega} 
= (1 - 5 (- \Delta + \alpha_{\omega})^{-1}W^{4})^{-1}
(- \Delta + \alpha_{\omega})^{-1} F_{\omega}(\eta_{\omega}).
\end{equation}
It follows from Theorem \ref{thm-bl-0} 
that 
\begin{equation}\label{EqB-8}
\lim_{\omega \to 0} \eta_{\omega} = 0
\qquad \mbox{strongly in $\dot{H}^{1} \cap L^{q}(\R^{3})\; (q>3)$}.
\end{equation}
Put 
\begin{equation} \label{EqB-9}
\Lambda W := \frac{1}{2} W + x \cdot \nabla W. 
\end{equation}
Then, we see that $\Lambda W 
\in \dot{H}^{1}(\R^{3})$ satisfies 
\begin{equation*} 
(- \Delta - 5 W^{4})\Lambda W = 0 
\qquad \mbox{in $\R^{3}$}.
\end{equation*}
To prove Proposition \ref{PropR-1}, 
we recall the following resolvent expansion 
by Jensen and Kato~\cite{MR544248} 
(see also Remark 2.3 of Coles and Gustafson~\cite{MR4162293}): 
\begin{lemma}[Lemma 4.3 of 
Jensen and Kato~\cite{MR544248}]
\label{LemR-2}
Let $s$ satisfy $3/2 < s < 5/2$. 
Then, we have the following expansion 
\begin{equation} \label{EqB-10} 
\begin{split} 
(1 - 5 (- \Delta + \alpha_{\omega})^{-1}W^{4})^{-1} 
& = 
\frac{5}{3\pi}
\alpha_{\omega}^{-\frac{1}{2}} 
\langle W^{4} \Lambda W, \cdot \rangle \Lambda W 
+ C_{0}^1 + o(1)
\end{split}
\end{equation}
as $\omega \to 0$
in $B(H_{\text{rad}, -s}^{1}
(\mathbb{R}^{3}), 
H_{\text{rad}, -s}^{1}(\mathbb{R}^{3}))$ 
for $\frac{3}{2} < s < \frac{5}{2}$.
Here, 
$C_0^1$ is an explicit operator 
and $\|u\|_{H^{1}_{-s}} = 
\|(1 + |x|^{2})^{- \frac{s}{2}} u\|
_{H^{1}}$. 
\end{lemma}
It follows from \eqref{EqB-5}, \eqref{EqB-7} and 
Lemma \ref{LemR-2} that 
\begin{equation} \label{EqB-11}
\begin{split}
\eta_{\omega} 
& = (1 - 5 (- \Delta + \alpha_{\omega})^{-1}W^{4})^{-1} 
(- \Delta + \alpha_{\omega})^{-1} F_{\omega}
(\eta_{\omega}) \\[6pt]
& = \frac{5}{3\pi}
\alpha_{\omega}^{-\frac{1}{2}} 
\langle W^{4} \Lambda W, (- \Delta + \alpha_{\omega})^{-1} 
F_{\omega}(\eta_{\omega}) \rangle \Lambda W 
\\[6pt]
& \quad 
+ C_{0}^{1}(- \Delta + \alpha_{\omega})^{-1} 
F_{\omega}(\eta_{\omega}) 
+ o(1) \qquad \mbox{in $H_{\text{rad}, -s}^{1}(\R^{3})$ 
as $\omega \to 0$}. 
\end{split}
\end{equation}
At least formally, we see from 
\eqref{EqB-6}, \eqref{EqB-8}, \eqref{EqB-11} and $\lim_{\omega \to 0} 
\alpha_{\omega}^{- \frac{1}{2}} = \infty$ 
(see \eqref{main-eq6})
that $(- \Delta + \alpha_{\omega})^{-1} F_{\omega}$ 
satisfies the orthogonal condition 
\begin{equation} \label{EqB-12}
\lim_{\omega \to 0} 
\langle W^{4} \Lambda W, (- \Delta + \alpha_{\omega})^{-1} F_{\omega} \rangle = 0.
\end{equation}
In fact, we can show this rigorously. 
To explain this more precisely, 
we need several notations. 
For a given $\nu > 0$ and a function $f$ on $\R^{3}$, we define 
$S_{\nu} f$ by 
\begin{equation*}
S_{\nu} f: = \nu^{-1}f (\nu^{-2} \cdot).
\end{equation*} 
Observe that 
\[
\|\nabla (S_{\nu} f)\|_{L^{2}} = \|\nabla f\|_{L^{2}}, 
\qquad 
\|S_{\nu} f\|_{L^{6}} = \|f\|_{L^{6}} 
\qquad \mbox{for all $f \in \dot{H}^{1}(\R^{3})$}. 
\]
We also define $\widetilde{u}_{\omega,\nu}:= 
S_{\nu} \widetilde{u}_{\omega}$ for 
$\nu > 0$. 
We see from \eqref{main-eq41} that 
$\widetilde{u}_{\omega,\nu}$ satisfies 
\begin{equation*}
- \Delta \widetilde{u}_{\omega,\nu} 
+ 
\alpha_{\omega,\nu} \widetilde{u}_{\omega,\nu}
- \beta_{\omega,\nu} \widetilde{u}_{\omega,\nu}^{p} 
-\widetilde{u}_{\omega,\nu}^{5} = 0 
\qquad \mbox{in $\mathbb{R}^{3}$}, 
\end{equation*}
where $\alpha_{\omega,\nu} 
= \alpha_{\omega} \nu^{- 4}$ and 
$\beta_{\omega,\nu} = \beta_{\omega} \nu^{p - 5}$.
We decompose $\widetilde{u}_{\omega,\nu}$ as 
\begin{equation*} 
\widetilde{u}_{\omega,\nu}= W+ \eta_{\omega, \nu}. 
\end{equation*}
Then, we can verify that $\eta_{\omega, \nu}$ satisfies 
\begin{equation}\label{EqB-13}
\begin{split}
(- \Delta - 5 W^{4} + \alpha_{\omega,\nu}) 
\eta_{\omega, \nu} 
= F_{\omega, \nu}, 
\end{split}
\end{equation}
where
\[
F_{\omega, \nu}
:=
- \alpha_{\omega,\nu} W 
+ \beta_{\omega,\nu} W^{p} 
+ N_{\omega, \nu}, 
\]
and 
\begin{equation*}
N_{\omega, \nu}
:= 
(W + \eta_{\omega, \nu})^{5} 
- W^{5} - 5 W^{4} \eta_{\omega, \nu} 
+ 
\beta_{\omega,\nu}\big\{(W + \eta_{\omega, \nu})^{p} 
- W^{p} \big\}. 
\end{equation*}
We can rewrite \eqref{EqB-13} as 
\begin{equation} \label{EqB-14}
\eta_{\omega, \nu} 
= (1 - 5 (- \Delta + \alpha_{\omega,\nu})^{-1}W^{4})^{-1}
(- \Delta + \alpha_{\omega,\nu})^{-1} F_{\omega, \nu}.
\end{equation}
Then, by an argument 
similar to \cite[Lemma 3.8]{MR4162293}, 
we can obtain the following result: 
\begin{lemma} \label{LemR-3}
Let $1 < p < 3$. 
For any $\omega > 0$, 
there exists $\nu_{\omega} >0$ 
with $\lim_{\omega \to 0} \nu_{\omega} 
= 1$
such that 
\begin{equation} \label{EqB-15}
\big\langle W^{4} \Lambda W, 
\ (- \Delta + \alpha_{\omega,\nu_{\omega}})^{-1}
F_{\omega,\nu_{\omega}}
\big\rangle= 0. 
\end{equation}
\end{lemma}
\begin{remark}
Since $\lim_{\omega \to 0} \nu_{\omega} = 1$, 
we see from \eqref{EqB-15} that \eqref{EqB-12} holds. 
\end{remark}
Define
\[
\widetilde{u}_{2, [\omega]} := 
\widetilde{u}_{\omega,\nu_{\omega}}, \qquad
\eta_{[\omega]} :=
\widetilde{u}_{2, [\omega]} - W, \qquad 
\alpha_{2, [\omega]} := \alpha_{\omega,\nu_{\omega}}, \qquad 
\beta_{2, [\omega]} := \beta_{\omega,\nu_{\omega}}, \qquad 
F_{[\omega]} := F_{\omega, \nu_{\omega}}. 
\] 
where $\nu_{\omega} > 0$ is the number given in 
Lemma \ref{LemR-3}. 
Note that since $\lim_{\omega \to 0} 
\widetilde{u}_{\omega} = W$ strongly in 
$\dot{H}^{1} \cap L^{q}(\R^{3})\; (q > 3)$ 
(see Theorem \ref{thm-bl-0})
and $\lim_{\omega \to 0} \nu_{\omega} = 1$, 
we have 
\begin{equation}\label{EqB-16}
\lim_{\omega \to 0} \eta_{[\omega]} 
= 0 \quad \mbox{strongly 
in $\dot{H}^{1} \cap L^{q}(\R^{3})\; (q > 3)$}. 
\end{equation}
Furthermore, 
one can see from \eqref{EqB-13} that $\eta_{[\omega]}$ satisfies 
\begin{equation}\label{EqB-17}
\begin{split}
(- \Delta - 5 W^{4} + \alpha_{[\omega]}) \eta_{[\omega]} 
= F_{[\omega]}, 
\end{split}
\end{equation}
where 
$F_{[\omega]} 
= -\alpha_{[\omega]} W + \beta_{[\omega]} W^{p} 
+ N_{[\omega]}$ 
and 
\begin{equation} \label{EqB-18}
\begin{split}
N_{[\omega]}
&:= 
(W + \eta_{[\omega]})^{5} - W^{5} 
- 5 W^{4} \eta_{[\omega]} 
+ \beta_{[\omega]} \bigm\{ 
(W + \eta_{[\omega]})^{p} - W^{p} 
\bigm\}. 
\end{split} 
\end{equation}
We can rewrite \eqref{EqB-15} by 
\begin{align}
\label{EqB-19}
\langle W^{4}\Lambda W, 
(-\Delta + \alpha_{[\omega]})^{-1} F_{[\omega]} \rangle 
= 0.
\end{align}
It follows from \eqref{EqB-19} and $F_{[\omega]} 
= -\alpha_{[\omega]} W + \beta_{[\omega]} W^{p} 
+ N_{[\omega]}$ 
that 
\begin{equation} \label{EqB-20}
\begin{split}
0 
= 
-\alpha_{[\omega]} \langle W^{4}\Lambda W, 
(-\Delta + \alpha_{[\omega]})^{-1}W
\rangle 
+ \beta_{[\omega]} \langle 
W^{4} \Lambda W, 
(-\Delta + \alpha_{[\omega]})^{-1}W^{p} \rangle 
+ \langle W^{4} \Lambda W, 
(-\Delta + \alpha_{[\omega]})^{-1} N_{[\omega]} \rangle. 
\end{split}
\end{equation}
Dividing \eqref{EqB-20} by 
$\sqrt{\alpha_{[\omega]}}$, 
we obtain 
\begin{equation} \label{EqB-21}
\begin{split}
0 
& = 
- \sqrt{\alpha_{[\omega]}} 
\langle W^{4}\Lambda W, 
(-\Delta + \alpha_{[\omega]})^{-1}W \rangle 
+ \beta_{[\omega]}\alpha_{[\omega]}^{- \frac{1}{2}} 
\langle 
W^{4} 
\Lambda W, (-\Delta + \alpha_{[\omega]})^{-1}W^{p} \rangle 
\\[6pt] 
& \quad 
+ \alpha_{[\omega]}^{- \frac{1}{2}}
\langle W^{4} \Lambda W, 
(-\Delta + \alpha_{[\omega]})^{-1} N_{[\omega]} \rangle. 
\end{split}
\end{equation}
We will use the 
following result obtained by Coles and Gustafson~\cite{MR4162293}: 
\begin{proposition}[Lemma 2.7 of Coles and Gustafson~\cite{MR4162293}]
\label{LemR-4}
\begin{enumerate}
\item[\textrm{(i)}]
For any $2< p < 5$ and any $0< \delta_{1}< p-2$, we have 
\begin{equation}\label{EqB-22}
\bigm\langle 
W^{4}\Lambda W, \, (-\Delta + \alpha_{[\omega]})^{-1} W^{p} 
\bigm\rangle 
= -\frac{5-p}{10(p+1)}
\|W\|_{L^{p+1}}^{p+1} + O(\alpha_{[\omega]}^{\frac{\delta_{1}}{2}}).
\end{equation}
\item[\textrm{(ii)}]
For any $\delta_{1} > 0$, we have 
\begin{equation}\label{EqB-23}
\sqrt{\alpha_{[\omega]}} 
\bigm\langle W^{4}\Lambda W, \, (-\Delta + \alpha_{[\omega]})^{-1} W 
\bigm\rangle 
= - \frac{6}{5} \pi
+ O(\alpha_{[\omega]}^{\frac{\delta_{1}}{2}}).
\end{equation}
\end{enumerate}
\end{proposition}
We now estimate the third term of the right-hand side of 
\eqref{EqB-21}. 
Concerning this, 
we claim the following:
\begin{lemma}\label{LemR-5}
Let $2 < p < 3$. 
Then, the following estimate holds 
for any sufficiently small 
$\varepsilon>0$ and $\omega > 0$: 
\begin{align}\label{EqB-24} 
\|(-\Delta + \alpha_{[\omega]})^{-1}
N_{[\omega]}\|_{L^{\infty}} 
\lesssim \alpha_{[\omega]}^{\frac{p-1 - (p+2) \varepsilon}{2}} . 
\end{align}
\end{lemma} 
The proof of Lemma \ref{LemR-5} 
is similar to that of Lemma 2.6 of \cite{MR4445670}. 
However, in \cite{MR4445670}, we consider the case where 
$\omega > 0$ is large, so that the details are different. 
For this reason, we shall give the proof of 
Lemma \ref{LemR-5} in Subsection \ref{ssec-nonest} below. 
Admitting Lemma \ref{LemR-5}, 
we will give a proof of 
Proposition \ref{PropR-1} \textrm{(i)}: 
\begin{proof}[Proof of Proposition \ref{PropR-1} \textrm{(i)}]
By the H\"older inequality and Lemma \ref{LemR-5}, 
one can estimate the last term on the right-hand side of 
\eqref{EqB-21} as follows: for any 
sufficiently small $\varepsilon > 0$, 
\begin{equation}\label{EqB-25}
\begin{split}
\big|\langle W^{4} \Lambda W, \ 
(-\Delta + \alpha_{[\omega]})^{-1}N_{[\omega]} \rangle \big| 
\le 
\|W^{4} \Lambda W \|_{L^{1}}
\| (-\Delta + \alpha_{[\omega]})^{-1} N_{[\omega]} \|_{L^{\infty}} 
\lesssim 
\alpha_{[\omega]}^{\frac{p - 1 - (p+2) \varepsilon}{2}}. 
\end{split}
\end{equation} 
By \eqref{EqB-25} and $p > 2$, we have 
\begin{equation} \label{EqB-26}
\begin{split}
\biggl|\alpha_{[\omega]}^{-\frac{1}{2}}
\langle W^{4} \Lambda W, 
(-\Delta + \alpha_{[\omega]})^{-1} 
N_{[\omega]} \rangle \biggl| 
\lesssim \alpha_{[\omega]}^{\frac{p - 2 - (p+2) \varepsilon}{2}} 
\to 0 \qquad 
\mbox{as $\omega \to 0$}. 
\end{split}
\end{equation}
From \eqref{EqB-21}, \eqref{EqB-23}, 
\eqref{EqB-22} and \eqref{EqB-26}, 
we see that 
\[
\begin{split}
\lim_{\omega \to 0} 
\frac{\beta_{[\omega]}}{\sqrt{\alpha_{[\omega]}}}
& 
= \lim_{\omega \to 0} \frac
{\sqrt{\alpha_{[\omega]}} 
\langle (-\Delta + \alpha_{[\omega]})^{-1}W, W^{4}\Lambda W 
\rangle - \alpha_{[\omega]}^{-\frac{1}{2}}
\langle W^{4} \Lambda W, 
(-\Delta + \alpha_{[\omega]})^{-1} 
N_{[\omega]} \rangle} 
{\langle (-\Delta + \alpha_{[\omega]})^{-1}W^{p}, W^{4} \Lambda W \rangle} \\[6pt] 
& 
= \lim_{\omega \to 0} \frac
{\sqrt{\alpha_{[\omega]}} 
\langle (-\Delta + \alpha_{[\omega]})^{-1}W, W^{4}\Lambda W 
\rangle} 
{\langle (-\Delta + \alpha_{[\omega]})^{-1}W^{p}, W^{4} \Lambda W \rangle} \\[6pt] 
& = - \frac{6}{5} \pi \times 
\left( - \frac{10(p+1)}{(5 - p) \|W\|_{L^{p+1}}^{p+1}}
\right)
= \frac{12\pi (p+1)}{(5 - p) \|W\|_{L^{p+1}}^{p+1}}. 
\end{split}
\]
Since $\lim_{\omega \to 0} \nu_{\omega} = 1$, 
we see that \eqref{EqB-1} holds. 
This completes the proof of Proposition 
\ref{PropR-1}.
\end{proof} 
\subsubsection{(Case 2) $p = 2$. 
}
In this subsection, we shall give the proof of 
Proposition \ref{PropR-1} \textrm{(ii)}.
We remark that when $p = 2$, 
the decay of $W^{p}$ is slow, so that 
we cannot regard the term $u_{2, [\omega]}^{p}$ 
as a perturbation of $W^{p}$. 
Thus, we rewrite the equation \eqref{EqB-17} by 
\begin{equation}\label{EqB-27}
\begin{split}
(- \Delta - 5 W^{4} + \alpha_{[\omega]}) \eta_{[\omega]} 
= F_{[\omega]}
= -\alpha_{[\omega]} W + \beta_{[\omega]} 
\widetilde{u}_{2, [\omega]}^{p} + \widehat{N}_{[\omega]}, 
\end{split}
\end{equation}
where 
\begin{equation}\label{EqB-28}
\begin{split}
\widehat{N}_{[\omega]}
&:= 
(W + \eta_{[\omega]})^{5} - W^{5} 
- 5 W^{4} \eta_{[\omega]}. 
\end{split} 
\end{equation}
It follows from \eqref{EqB-15} and 
\eqref{EqB-27} that 
\begin{equation} \label{EqB-29}
\begin{split}
0 
= 
-\alpha_{[\omega]} \langle W^{4}\Lambda W, 
(-\Delta + \alpha_{[\omega]})^{-1}W
\rangle 
+ \beta_{[\omega]} \langle 
W^{4} \Lambda W, 
(-\Delta + \alpha_{[\omega]})^{-1}
\widetilde{u}_{2, [\omega]}^{p} \rangle 
+ \langle W^{4} \Lambda W, 
(-\Delta + \alpha_{[\omega]})^{-1} 
\widehat{N}_{[\omega]} \rangle. 
\end{split}
\end{equation}
We will use the 
following result: 
\begin{proposition}\label{LemR-6}
Let $p = 2$. 
We have 
\begin{equation} \label{EqB-30}
\lim_{\omega \to 0} 
|\log \alpha_{[\omega]}|^{-1} 
\langle W^{4} \Lambda W, 
(- \Delta + \alpha_{[\omega]})^{-1}
\widetilde{u}_{2, [\omega]}^{p} \rangle 
= - \frac{3 \sqrt{3}}{5} \pi. 
\end{equation}
\end{proposition}
We will show Proposition \ref{LemR-6} in Subsection 
\ref{sub-est-pri} below. 
Next, we consider the third term of the right-hand side of 
\eqref{EqB-27}. 
Concerning this, 
we claim the following:
\begin{lemma}\label{LemR-9}
Let $p = 2$. 
Then, the following estimate holds 
for any sufficiently small 
$\varepsilon>0$ and $\omega > 0$: 
\begin{align*}
\|(-\Delta + \alpha_{[\omega]})^{-1}
\widehat{N}_{[\omega]}\|_{L^{\infty}} 
\lesssim \alpha_{[\omega]}^{1 - 2 \e}.
\end{align*}
\end{lemma}

We shall give a proof of Lemma \ref{LemR-9} 
in Subsection \ref{ssec-nonest} below. 
We are now in a position to give the proof of 
Proposition \ref{PropR-1} \textrm{(ii)}: 
\begin{proof}[Proof of Proposition \ref{PropR-1} \textrm{(ii)}]
By the H\"older inequality and Lemma \ref{LemR-9}, 
one can estimate the last term on the right-hand side of 
\eqref{EqB-21} as follows: for any 
sufficiently small $\varepsilon > 0$, 
\begin{equation}\label{EqB-31}
\begin{split}
\big|\langle W^{4} \Lambda W, \ 
(-\Delta + \alpha_{[\omega]})^{-1}N_{[\omega]} \rangle \big| 
\le 
\|W^{4} \Lambda W \|_{L^{1}}
\| (-\Delta + \alpha_{[\omega]})^{-1} N_{[\omega]} \|_{L^{\infty}} 
\lesssim 
\alpha_{[\omega]}^{1 - p\e}. 
\end{split}
\end{equation} 
By \eqref{EqB-31}, we have 
\begin{equation} \label{EqB-32} 
\begin{split}
\biggl|\alpha_{[\omega]}^{-\frac{1}{2}}
\langle W^{4} \Lambda W, 
(-\Delta + \alpha_{[\omega]})^{-1} 
N_{[\omega]} \rangle \biggl| 
\lesssim \alpha_{[\omega]}^{\frac{1}{2} - p \varepsilon}
\to 0 \qquad 
\mbox{as $\omega \to 0$}. 
\end{split}
\end{equation}
From \eqref{EqB-29}, 
\eqref{EqB-23}, \eqref{EqB-30} and \eqref{EqB-32}, 
we see that 
\[
\begin{split}
& \quad \lim_{\omega \to 0}
\frac{\beta_{[\omega]}
|\log \alpha_{[\omega]}|}
{\sqrt{\alpha_{[\omega]}}} \\[6pt]
& = \lim_{\omega \to 0} 
\frac{|\log \alpha_{[\omega]}|}
{\sqrt{\alpha_{[\omega]}}}
\left\{
\frac{\alpha_{[\omega]} \langle W^{4}\Lambda W, 
(-\Delta + \alpha_{[\omega]})^{-1}W
\rangle - \langle W^{4} \Lambda W, 
(-\Delta + \alpha_{[\omega]})^{-1} 
\widehat{N}_{[\omega]} \rangle}{ 
\langle W^{4} \Lambda W, (-\Delta + \alpha_{[\omega]})^{-1}
\widetilde{u}_{\omega,\nu_{\omega}}^{p} \rangle} 
\right\} \\[6pt]
& = \lim_{\omega \to 0} 
\frac{\sqrt{\alpha_{[\omega]}}
\langle W^{4}\Lambda W, 
(-\Delta + \alpha_{[\omega]})^{-1}W
\rangle - \alpha_{[\omega]}^{- \frac{1}{2}}\langle W^{4} \Lambda W, 
(-\Delta + \alpha_{[\omega]})^{-1} 
\widehat{N}_{[\omega]} \rangle}{ 
|\log \alpha_{[\omega]}|^{-1}
\langle W^{4} \Lambda W, (-\Delta + \alpha_{[\omega]})^{-1}
\widetilde{u}_{\omega,\nu_{\omega}}^{p} \rangle} 
\\[6pt]
& = 
\dfrac{- \frac{6}{5} \pi}{- \frac{3 \sqrt{3}}{5} \pi} 
= \frac{2}{\sqrt{3}}. 
\end{split}
\]
Since $\lim_{\omega \to 0} \nu_{\omega} = 1$, 
we see that \eqref{EqB-2} holds. 
This completes the proof. 
\end{proof} 
\subsubsection{Estimate of the 
principle term 
(proof of Proposition \ref{LemR-6})}
\label{sub-est-pri}
In this subsection, we will prove Proposition \ref{LemR-6}. 
We shall use the following inequalities 
which we can obtain from 
the Young inequality 
(see, e.g., (9) in Section 4.3 of \cite{MR1817225})
\begin{lemma}\label{LemR-11}
Let $\alpha >0$, and let 
$1\le s \le q \le \infty$ and 
$3(\frac{1}{s}-\frac{1}{q})<2$. Then, we have 
\begin{equation}\label{EqB-33}
\|(-\Delta+\alpha)^{-1} \|_{L^{s}(\mathbb{R}^{3}) \to L^{q}(\mathbb{R}^{3})}
\lesssim 
\alpha^{\frac{3}{2}
(\frac{1}{s}-\frac{1}{q})-1}. 
\end{equation}
\end{lemma}

We are now in a position to prove Proposition \ref{LemR-6}. 
\begin{proof}[Proof of Proposition \ref{LemR-6}]
We shall divide the proof into $3$ steps. 

\textbf{(Step 1). 
} 
Note that 
\[
\begin{split}
& 
W(x) = \left(1 + \frac{|x|^{2}}{3}\right)^{- \frac{1}{2}} 
= \sqrt{3} |x|^{-1} \left(1 + 3|x|^{-2}\right)^{- \frac{1}{2}} 
= \sqrt{3} |x|^{-1} + O(|x|^{-3}) \qquad 
\mbox{as $|x| \to \infty$}, \\[6pt]
& x \cdot \nabla W 
= - \frac{|x|^{2}}{3} 
\left(1 + \frac{|x|^{2}}{3}\right)^{- \frac{3}{2}} 
= - \sqrt{3} |x|^{-1} 
\left(1 + 3 |x|^{-2}\right)^{- \frac{3}{2}} 
= - \sqrt{3} |x|^{-1} + O(|x|^{-3}) 
\qquad 
\mbox{as $|x| \to \infty$} . 
\end{split}
\]
Thus, 
if we put $Z:= \Lambda W - 
\left( - \frac{\sqrt{3}}{2} |x|^{-1}\right)$, 
we find from 
$\Lambda W = \frac{1}{2} W + x \cdot \nabla W$
that $Z = O(|x|^{-3})$ as $|x| 
\to \infty$. 
In addition, since
$
- \Delta \Lambda W 
= 5 W^{4} 
\Lambda W $ and $\Lambda W = 
- \frac{\sqrt{3}}{2} |x|^{-1} + Z$, we have 
\begin{equation} \label{EqB-36}
\begin{split}
\langle 
(- \Delta + 
\alpha_{\omega})^{-1}
\widetilde{u}_{\omega}^{p}, 
W^{4} \Lambda W \rangle 
& = 
\frac{1}{5}
\langle 
(- \Delta + 
\alpha_{\omega})^{-1}
\widetilde{u}_{\omega}^{p}, -\Delta \Lambda W \rangle \\[6pt]
& = 
- \frac{\sqrt{3}}{10}
\langle 
(- \Delta + 
\alpha_{\omega})^{-1}
\widetilde{u}_{\omega}^{p}, 
- \Delta |x|^{-1} \rangle 
+ \frac{1}{5}
\langle 
(- \Delta + 
\alpha_{\omega})^{-1}
\widetilde{u}_{\omega}^{p}, 
- \Delta Z \rangle. 
\end{split}
\end{equation}
Recall (see, e.g., Section 6.23 of \cite{MR1817225}) 
that for any $\alpha >0$ and any function $f$ on $\mathbb{R}^{3}$, 
\begin{equation} \label{EqB-37}
(- \Delta + \alpha)^{-1} f(x) 
= \int_{\mathbb{R}^{3}} \frac{e^{- 
\sqrt{\alpha} |x-y|}}{4\pi |x - y|} f(y) dy.
\end{equation}
In addition, note that $- \Delta |x|^{-1} = 
4 \pi \delta_{0}$ in 
$\mathcal{D}^{\prime}
(\mathbb{R}^{3})$ (see e.g. \cite[Page 156]{MR1817225}). 
Thus, it follows that 
\begin{equation} \label{EqB-38}
\begin{split}
- \frac{\sqrt{3}}{10}
\langle 
(- \Delta + 
\alpha_{\omega})^{-1}
\widetilde{u}_{\omega}^{p}, 
(- \Delta)|x|^{-1} \rangle 
& = 
- \frac{2\sqrt{3} \pi}{5} 
\left\{(- \Delta + 
\alpha_{\omega})^{-1} 
\widetilde{u}_{\omega}^{p}
\right\}(0) \\
& = - \frac{\sqrt{3}}{10} 
\int_{\mathbb{R}^{3}} 
\frac{e^{- \sqrt{\alpha_{\omega}}|x|}}{|x|} 
\widetilde{u}_{\omega}^{p}(x) \, dx. 
\end{split}
\end{equation}
By the H\"older inequality and Lemma \ref{LemR-11}, 
$\widetilde{u}_{\omega}^{p} = O(|x|^{-p})$ and 
$Z = O(|x|^{-3})$ as $|x| \to \infty$, we have 
\begin{equation} \label{EqB-39}
\begin{split}
& \quad 
\biggl|
\langle (- \Delta + \alpha_{\omega})^{-1} 
\widetilde{u}_{\omega}^{p}, \Delta Z \rangle
\biggl| = 
\biggl|
\langle (- \Delta + \alpha_{\omega})^{-1} 
\widetilde{u}_{\omega}^{p}, \left\{\alpha_{\omega} 
- (- \Delta + \alpha_{\omega}) \right\}Z \rangle
\biggl| \\[6pt]
& \leq \alpha_{\omega}
\|(- \Delta + \alpha_{\omega})^{-1} \widetilde{u}_{\omega}^{p}\|
_{L^{6}}\|Z\|_{L^{\frac{6}{5}}} 
+ \| \widetilde{u}_{\omega}^{p}
\|_{L^{6}}\|Z\|_{L^{\frac{6}{5}}} 
\lesssim \alpha_{\omega} 
\times \alpha_{\omega}^{\frac{3}{8}p - \frac{5}{4}}
\|\widetilde{u}_{\omega}^{p}\|
_{L^{\frac{4}{p}}}\|Z\|_{L^{\frac{6}{5}}} 
+ 1
\lesssim 1. 
\end{split}
\end{equation}
From \eqref{EqB-36}, \eqref{EqB-38} and \eqref{EqB-39}, 
we obtain 
\begin{equation}\label{EqB-40} 
\langle (- \Delta + \alpha_{\omega})^{-1}
\widetilde{u}_{\omega}^{p}, 
W^{4} \Lambda W \rangle
= 
- \frac{\sqrt{3}}{10} 
\int_{\mathbb{R}^{3}} 
\frac{e^{- \sqrt{\alpha_{\omega}}|x|}}{|x|} 
\widetilde{u}_{\omega}^{p}(x) \, dx + O(1). 
\end{equation}
\textbf{(Step 2). 
} 
From \eqref{EqB-40} 
and \eqref{EqB-35}, we obtain 
\begin{equation*}
\begin{split}
& \quad 
\langle 
(- \Delta + \alpha_{\omega})^{-1}
\widetilde{u}_{\omega}^{p}, 
W^{4} \Lambda W\rangle
\\[6pt]
& 
= 
- \frac{\sqrt{3}}{10} 
\int_{\mathbb{R}^{3}} 
\frac{e^{- \sqrt{\alpha_{\omega}}|x|}}{|x|} 
\widetilde{u}_{\omega}^{p}(x) \, dx + O(1) \\[6pt]
& \leq 
- \frac{\sqrt{3}(1 - \varepsilon)^{p}}{10} \times 
4\pi 
\int_{R_{\e}}^{\infty} 
\frac{e^{- \sqrt{\alpha_{\omega}}r}}{r} 
Y_{\omega}^{p}(r)\times r^{2}\, dr + O(1) \\[6pt]
& = - \frac{2(1 - \varepsilon)^{p} \pi}{5} 3^{\frac{p + 1}{2}}
\int_{R_{\e}}^{\infty} 
\frac{e^{- (p + 1) \sqrt{\alpha_{\omega}}r}}{r^{p -1}} \, dr 
+ O(1)
\\[6pt]
& = 
- \frac{2(1 - \varepsilon)^{p} \pi}{5} 3^{\frac{p + 1}{2}}
\left\{\int_{R_{\e}}
^{\alpha_{\omega}^{-\frac{1}{2}}}
\frac{e^{- (p + 1) \sqrt{\alpha_{\omega}}r}}{r^{p -1}} \, dr
+ \int_{\alpha_{\omega}^{-\frac{1}{2}}}^{\infty} 
\frac{e^{- (p + 1) \sqrt{\alpha_{\omega}}r}}{r^{p -1}} \, dr\right\} 
+ O(1)\\[6pt]
& = 
- \frac{2(1 - \varepsilon)^{p} \pi}{5} 3^{\frac{p + 1}{2}}
\left\{\int_{R_{\e}}
^{\alpha_{\omega}^{-\frac{1}{2}}}
\frac{1}{r^{p-1}} \, dr
+ \int_{R_{\e}}^{\alpha_{\omega}^{-\frac{1}{2}}}
\frac{e^{- (p + 1) \sqrt{\alpha_{\omega}}r} - 1}{r^{p-1}} \, dr
+ \int_{\alpha_{\omega}^{-\frac{1}{2}}}^{\infty} 
\frac{e^{- (p + 1) \sqrt{\alpha_{\omega}}r}}{r^{p-1}} \, dr\right\} 
+ O(1). 
\end{split}
\end{equation*}
Then, since $p = 2$, we have 
\[
\begin{split}
& \int_{R_{\e}}
^{\alpha_{\omega}^{-\frac{1}{2}}}
\frac{dr}{r^{p-1}} 
= \int_{R_{\e}}
^{\alpha_{\omega}^{-\frac{1}{2}}}
\frac{dr}{r} 
= \frac{|\log \alpha_{\omega}|}{2} 
- \log R_{\e}, 
\qquad 
\biggl|\int_{1}^{\alpha_{\omega}^{-\frac{1}{2}}}
\frac{e^{- 3 \sqrt{\alpha_{\omega}}r} - 1}{r^{p-1}} \, dr \biggl|
\leq \int_{1}^{\alpha_{\omega}^{-\frac{1}{2}}} 
\frac{\sqrt{\alpha_{\omega}}r}{r} \, dr \lesssim 1, 
\end{split}
\]
where we have used the elementary inequality 
$e^{-3 \sqrt{\alpha_{\omega}} r} - 1 - \sqrt{\alpha_{\omega}}r < 0$
for all $r > 0$ in the second inequality. 
In addition, since $p = 2$, we obtain 
\[
\begin{split}
\biggl|\int_{\alpha_{\omega}^{-\frac{1}{2}}}^{\infty} 
\frac{e^{- (p+1) \sqrt{\alpha_{\omega}}r}}{r^{p-1}} \, dr \biggl|
\leq \int_{1}^{\infty} 
\frac{e^{- 3 s}}{s} \, ds \lesssim 1. 
\end{split}
\]
These together with $p = 2$ imply that 
\begin{equation}\label{EqB-41}
\limsup_{\omega \to 0} 
|\log \alpha_{\omega}|^{-1} 
\langle 
(- \Delta + \alpha_{\omega})^{-1}
\widetilde{u}_{\omega}^{p}, 
W^{4} \Lambda W\rangle 
\leq 
- \frac{2(1 - \varepsilon)^{2} \pi}{5} 3^{\frac{3}{2}}
\times \frac{1}{2}
= - \frac{(1 - \varepsilon)^{2} }{5} 3^{\frac{3}{2}} \pi. 
\end{equation}
We see from \eqref{EqL-1} that 
\begin{equation*} 
\widetilde{u}_{\omega}(x) \leq 
3^{\frac{1}{2}}\gamma_{\omega}^{- \frac{1}{2}} r^{-1}
\qquad \mbox{for all $r>1$}, 
\end{equation*}
where $\gamma_{\omega} = - \alpha_{\omega} 
+ \beta_{n} + 1$. 
This together with 
$0 < \widetilde{u}_{\omega}(x) \leq \widetilde{u}_{\omega}(0) = 1$
and $p = 2$ implies that 
\begin{equation}\label{EqB-42}
\begin{split}
\int_{\mathbb{R}^{3}} 
\frac{e^{- \sqrt{\alpha_{\omega}}|x|}}{|x|} 
\widetilde{u}_{\omega}^{p} \, dx 
& \leq 4\pi \int_{0}^{1} 
e^{- \sqrt{\alpha_{\omega}}r} r \, dr 
+ 4 \pi \times 3^{\frac{p}{2}} \gamma_{\omega}^{-\frac{p}{2}}
\int_{1}^{\alpha_{\omega}^{- \frac{1}{2}}} 
\frac{e^{- \sqrt{\alpha_{\omega}} r}}{r^{p-1}} \, dr
+ 4 \pi \times 3^{\frac{p}{2}} \gamma_{\omega}^{-\frac{p}{2}}
\int_{\alpha_{\omega}^{- \frac{1}{2}}}^{\infty} 
\frac{e^{- \sqrt{\alpha_{\omega}} r}}{r^{p-1}} \, dr \\[6pt]
& 
\leq O(1)
+ 4 \pi \times 3 \gamma_{\omega}^{-1}
\int_{1}^{\alpha_{\omega}^{- \frac{1}{2}}} 
\frac{1}{r} \, dr
+ 4 \pi \times 3 \gamma_{\omega}^{-1}
\int_{\alpha_{\omega}^{- \frac{1}{2}}}^{\infty} 
\frac{e^{- \sqrt{\alpha_{\omega}} r}}{r} \, dr \\[6pt]
& = 2 \pi \times 3 \gamma_{\omega}^{-1} |\log \alpha_{\omega}| + O(1). 
\end{split}
\end{equation}
It follows from \eqref{EqB-40}, \eqref{EqB-42} and 
$\lim_{\omega \to 0} \gamma_{\omega} = 1$ that 
\begin{equation} \label{EqB-43}
\liminf_{\omega \to 0} 
|\log \alpha_{\omega}|^{-1} 
\langle 
(- \Delta + \alpha_{\omega})^{-1}
\widetilde{u}_{\omega}^{p}, 
W^{4} \Lambda W\rangle 
\geq - \frac{\sqrt{3}}{10} \times 2 \pi \times 3 = 
- \frac{ 3^{\frac{3}{2}}}{5} \pi. 
\end{equation}
Since $\e > 0$ is arbitrary, we see from \eqref{EqB-41} 
and \eqref{EqB-43} that 
\[
\lim_{\omega \to 0} 
|\log \alpha_{\omega}|^{-1} 
\langle (- \Delta + \alpha_{\omega})^{-1}
\widetilde{u}_{\omega}^{p}, 
W^{4} \Lambda W\rangle 
= - \frac{ 3^{\frac{3}{2}}}{5} \pi. 
\]
Thus, \eqref{EqB-30} holds. 
This concludes the proof. 
\end{proof}
\subsubsection{Estimate of the remainder term 
(Proof of Lemmas \ref{LemR-5} and \ref{LemR-9})}
\label{ssec-nonest}
In this subsection, we give the proof of Lemmas \ref{LemR-5} and 
\ref{LemR-9}. 
To this end, we need several preparations. 
We will use the following inequalities 
(see, e.g., \cite[(2.10)]{MR4162293})
%

\begin{lemma}\label{LemR-12}
Let $\alpha >0$. 
For any $\e>0$, we obtain 
\begin{equation}\label{EqB-44}
\|(-\Delta+\alpha)^{-1} \|
_{L^{\frac{3}{2} + \e}(\mathbb{R}
^{3}) \cap L^{\frac{3}{2} - \e} 
(\mathbb{R}^{3})\to L^{\infty}(\mathbb{R}^{3})}
\lesssim 1. 
\end{equation}
\end{lemma}
Lemma \ref{LemR-11} does not deal with the case 
where $3(\frac{1}{s}-\frac{1}{q})=2$. In that case, we use the following:
\begin{lemma}\label{LemR-13}
Let $\alpha >0$, and 
let $3 < q < \infty$. 
Then, we have 
\begin{equation*}
\|(-\Delta+\alpha)^{-1} \|
_{L^{\frac{3q}{3+2q}}(\mathbb{R}
^{3}) \to L^{q}(\mathbb{R}^{3})}
\lesssim 1. 
\end{equation*}
\end{lemma}
We also need the following $L^{\infty}$-estimate of 
$(- \Delta + \alpha)^{-1} |x|^{-1}$: 
\begin{lemma}[Lemma 3.13 of \cite{HKW}] \label{LemR-14}
The following estimate holds for all 
$\alpha >0$: 
\begin{equation*}
\|(- \Delta + \alpha)^{-1} |x|^{-1}\|_{L^{\infty}} 
\lesssim \alpha^{-\frac{1}{2}},
\end{equation*}
where the implicit constant 
is independent of $\alpha$. 
\end{lemma}
From the resolvent expansion in Lemma \ref{LemR-2}, 
Coles and Gustafson~\cite[Lemma 2.4]{MR4162293} obtained the 
following resolvent estimate: 
\begin{lemma}[Coles and Gustafson~\cite{MR4162293}]
\label{LemR-16}
Let $3 < r \leq \infty$. 
\begin{equation} \label{EqB-45}
\|(1 - 5(- \Delta + \alpha_{\omega})^{-1}W^{4})^{-1}f\|_{L^{r}} 
\lesssim \|f\|_{L^{r}}
\end{equation}
for any $f \in L_{\text{rad}}^{r}(\mathbb{R}^{3})$ with 
$\langle W^{4} \Lambda W, f \rangle = 0$. 
\end{lemma}

\begin{remark}
Note that the resolvent expansion \eqref{EqB-10} 
has a singularity $\alpha_{\omega}^{-\frac{1}{2}}$. 
However, we see from Lemma \ref{LemR-16} 
that if we have the orthogonal condition 
$\langle W^{4} \Lambda W, f \rangle = 0$, 
we can remove the singularity in 
the resolvent estimate \eqref{EqB-45}. 
\end{remark}
Next, we estimate $\beta_{\omega}$ by $\alpha_{\omega}$. 
\begin{lemma}\label{LemR-15}
Let $\widetilde{u}_{\omega}$ 
be the positive solution to \eqref{main-eq41}.
\begin{enumerate}
\item[\textrm{(i)}]
Let $2 < p < 3$. 
There exists a constant 
$C_{1} > 0$ such that 
\begin{equation}\label{EqB-46}
\beta_{\omega}
\leq C_{1} \sqrt{\alpha_{\omega}}. 
\end{equation}
\item[\textrm{(ii)}]
Let $p = 2$. 
There exists a constant 
$C_{2} > 0$ such that 
\begin{equation}\label{EqB-47}
\beta_{\omega}
\leq C_{2} \frac{\sqrt{\alpha_{\omega}}}
{|\log \alpha_{\omega}|}. 
\end{equation}
\item[\textrm{(iii)}]
Let $1 < p < 2$. 
There exists a constant 
$C_{3} > 0$ such that 
\begin{equation}\label{EqB-48}
\beta_{\omega}
\leq C_{3} \alpha_{\omega}^{\frac{3-p}{2}}. 
\end{equation}
\end{enumerate}
\end{lemma}
We now give the proof of Lemma 
\ref{LemR-15}. 
\begin{proof}[Proof of Lemma 
\ref{LemR-15}]
\textrm{(i)} 
We first consider the case of $2 < p < 3$. 
It follows from Theorem \ref{thm-bl-0} 
and $W \in L^{p + 1}(\R^{3})$ that 
there exists $C_{p}>0$ such that 
\begin{equation*}
\|\widetilde{u}_{\omega}\|_{L^{p+1}}^{p+1} 
\geq C_{p}
\end{equation*}
for $2 < p < 3$. 
In addition, by Lemma \ref{lem3-5} \textrm{(i)}, we have 
\begin{equation} \label{EqB-49}
\|\widetilde{u}_{\omega}\|_{L^{2}}^{2} 
\leq C \alpha_{\omega}^{- \frac{1}{2}}
\end{equation}
for some $C > 0$. 
These together with \eqref{main-eq7}
yields that 
\[
\frac{5 - p}{2 (p+1)} C_{p} \beta_{\omega} \leq 
\frac{5 - p}{2 (p+1)} \beta_{\omega} 
\|\widetilde{u}_{\omega}\|_{L^{p+1}}^{p+1}
= 
\alpha_{\omega}\|\widetilde{u}_{\omega}\|_{L^{2}}^{2} 
\leq C \sqrt{\alpha_{\omega}}. 
\]
Thus, \eqref{EqB-46} holds. 

\textrm{(ii)}
We now consider the case of $p = 2$.     
By \eqref{EqB-35} 
and $p = 2$, we have 
\[
\|\widetilde{u}_{\omega}\|_{L^{p + 1}}^{p + 1} 
\geq C
\int_{R_{2}}^{\alpha_{\omega}^{- \frac{1}{2}}} 
\frac{e^{-(p+1) \sqrt{\alpha_{\omega}} r}}{r^{p-1}} \, dr 
\geq C 
\int_{R_{2}}^{\alpha_{2, 
\omega}^{- \frac{1}{2}}} 
\frac{1}{r} \, dr \geq 
C (|\log \alpha_{\omega}| - \log R_2). 
\]
It follows from \eqref{main-eq7} and \eqref{EqL-6} that 
\[
C \beta_{\omega} 
|\log \alpha_{\omega}| 
\leq \frac{\beta_{\omega}}{2} 
\|\widetilde{u}_{\omega}\|_{L^{p + 1}}^{p + 1}
= 
\alpha_{\omega}\|\widetilde{u}_{\omega}\|_{L^{2}}^{2} 
\leq C \sqrt{\alpha_{\omega}}
\]
This yields \eqref{EqB-47}. 

\textrm{(iii)} 
Finally, we consider the case of $1 < p < 2$. 
Putting $s = \sqrt{\alpha_{\omega}} r$, 
we see from \eqref{EqB-89} that 
\[
\|\widetilde{u}_{\omega}\|_{L^{p+1}}^{p+1} 
\geq C
\int_{\alpha_{\omega}^{-\frac{1}{2}}}^{\infty} 
\frac{e^{- (p+1) \sqrt{\alpha_{\omega} r}}}{r^{p-1}} \, dr 
= C \alpha_{\omega}^{\frac{p - 2}{2}} 
\int_{1}^{\infty} \frac{e^{- (p+1) s}}{s^{p-1}} \, ds 
\geq C \alpha_{\omega}^{\frac{p - 2}{2}} 
\]
for some $C>0$. 
By \eqref{main-eq7} and \eqref{EqL-6}, one has 
\[
C \beta_{\omega} 
\alpha_{\omega}^{\frac{p - 2}{2}}
\leq 
\frac{5 - p}{2 (p+1)} \beta_{\omega} 
\|\widetilde{u}_{\omega}\|_{L^{p+1}}^{p+1}
= 
\alpha_{\omega}\|\widetilde{u}_{\omega}\|_{L^{2}}^{2} 
\leq C \sqrt{\alpha_{\omega}}. 
\]
Therefore, \eqref{EqB-48} holds. 
\end{proof}
Next, we obtain the following estimate of $\|\eta_{[\omega]}\|_{L^{\infty}}$. 
\begin{lemma}[c.f. Lemma 2.5 of \cite{MR4445670}]\label{LemR-17}
Let $2 < p < 3$. 
For any small $\omega > 0$, 
the following estimate holds:
\begin{equation}\label{EqB-50}
\|\eta_{[\omega]}\|_{L^{\infty}} 
\lesssim 
\sqrt{\alpha_{[\omega]}}. 
\end{equation}
\end{lemma}

\begin{proof}
We shall divide the proof into 2 steps. 

\textbf{(Step 1).}~By \eqref{EqB-14}, 
Lemma \ref{LemR-16} with 
$f = (- \Delta + \alpha_{[\omega]})^{-1} F_{[\omega]}$, \eqref{EqB-19} 
and $F_{[\omega]} 
= -\alpha_{[\omega]} W + \beta_{[\omega]} W^{p} 
+ N_{[\omega]}$, one can see that 
\begin{equation} \label{EqB-51}
\begin{split}
\|\eta_{[\omega]}\|_{L^{\infty}} 
& = 
\|(1 - 5 (-\Delta + \alpha_{[\omega]})^{-1} W^{4})^{-1}
(-\Delta + \alpha_{[\omega]})^{-1}F_{[\omega]}\|_{L^{\infty}}
\\[6pt]
&\leq 
\alpha_{[\omega]}
\|(-\Delta + \alpha_{[\omega]})^{-1} W\|_{L^{\infty}} 
\hspace{-3pt} + 
\beta_{[\omega]} 
\|(-\Delta + \alpha_{[\omega]})^{-1} W^{p}\|_{L^{\infty}} 
+ 
\|(-\Delta + \alpha_{[\omega]})^{-1} N_{[\omega]} \|_{L^{\infty}}.
\end{split} 
\end{equation}
We consider the first two terms on the right-hand side of \eqref{EqB-51}. 
It follows from 
Lemmas \ref{LemR-11} and \ref{LemR-14}, \eqref{talenti}, 
\eqref{EqB-46} and $p > 2$ that 
\begin{equation}\label{EqB-52}
\begin{split}
\alpha_{[\omega]}
\|(-\Delta + \alpha_{[\omega]})^{-1} W\|_{L^{\infty}} 
+ 
\beta_{[\omega]} 
\|(-\Delta + \alpha_{[\omega]})^{-1} W^{p}\|_{L^{\infty}}
&\lesssim 
\alpha_{[\omega]}\alpha_{[\omega]}^{-\frac{1}{2}}
+ 
\sqrt{\alpha_{[\omega]}} 
\alpha_{[\omega]}
^{\frac{p-2}{2} - \frac{p}{2} \varepsilon} 
\|W^{p}\|_{L^{\frac{3}{p(1-\varepsilon)}}}
\\[6pt]
& \lesssim 
\sqrt{\alpha_{[\omega]}} 
+ \alpha_{[\omega]}^{\frac{p-1}{2} - \frac{p}{2} \varepsilon} 
\lesssim \sqrt{\alpha_{[\omega]}}. 
\end{split}
\end{equation}

Concerning the last term on the right-hand side of \eqref{EqB-51}, 
we claim that 
\begin{equation}\label{EqB-53}
\|(-\Delta + \alpha_{[\omega]})^{-1} 
N_{[\omega]} \|_{L^{\infty}} 
\le o_{n}(\|\eta_{[\omega]}\|_{L^{\infty}}).
\end{equation}
Once we obtain \eqref{EqB-53}, 
we see that the last term 
on the right-hand side of \eqref{EqB-51} can be absorbed into the left-hand side. 
Hence, the claim \eqref{EqB-50} follows 
from \eqref{EqB-51}--\eqref{EqB-53}.

\textbf{(Step 2).}~
It remains to prove \eqref{EqB-53}. 
It follows from 
\eqref{EqB-18} and 
the triangle inequality 
that 
\begin{equation} \label{EqB-54}
\begin{split}
\|(-\Delta + \alpha_{[\omega]})^{-1} N_{[\omega]} \|_{L^{\infty}} 
&\le 
\|(-\Delta + \alpha_{[\omega]})^{-1} 
\bigm\{ 
(W + \eta_{[\omega]})^{5}
- W^{5} - 5 W^{4} \eta_{[\omega]}
\bigm\}
\|_{L^{\infty}} 
\\[6pt]
&\quad + 
\beta_{[\omega]} \|(-\Delta + \alpha_{[\omega]})^{-1} 
\big\{(W + \eta_{[\omega]})^{p} - W^{p} \big\} \|_{L^{\infty}}.
\end{split}
\end{equation}
We consider the first term 
on the right-hand side of \eqref{EqB-54}. 
By Lemma \ref{LemR-12}, 
elementary computations, 
the H\"older inequality and \eqref{EqB-16}, 
one can see that 
\begin{equation} \label{EqB-55}
\begin{split}
&\quad \|(-\Delta + \alpha_{[\omega]})^{-1} 
\bigm\{ 
(W + \eta_{[\omega]})^{5} - W^{5} - 5 W^{4} \eta_{[\omega]}
\bigm\} \|_{L^{\infty}} 
\\[6pt]
&\lesssim 
\bigm\|(W + \eta_{[\omega]})^{5} - W^{5} - 5 W^{4} \eta_{[\omega]}
\bigm\|_{L^{\frac{3 + \varepsilon}{2}} \cap L^{\frac{3 - \varepsilon}{2}}}
\\[6pt]
&\lesssim 
\|W^{3} |\eta_{[\omega]}|^{2} \|_{L^{\frac{3 + \varepsilon}{2}} \cap L^{\frac{3 - \varepsilon}{2}}}
+
\|\eta_{[\omega]}^{5} \|_{L^{\frac{3 + \varepsilon}{2}} \cap L^{\frac{3 - \varepsilon}{2}}}
\\[6pt]
&\le \|W^{3}\|_{L^{\frac{6(3 + \varepsilon)}{9-\varepsilon}} \cap 
L^{\frac{6(3 - \varepsilon)}{9 + \varepsilon}}} 
\|\eta_{[\omega]}\|_{L^{6}} \|\eta_{[\omega]}\|_{L^{\infty}}
+\|\eta_{[\omega]}^{4}\|_{L^{\frac{3 + \varepsilon}{2}} \cap L^{\frac{3 - \varepsilon}{2}}}
\|\eta_{[\omega]}\|_{L^{\infty}} \\[6pt]
& = 
o_{n}(1) \|\eta_{[\omega]}\|_{L^{\infty}}. 
\end{split} 
\end{equation} 
Next, we consider the second term 
on the right-hand side of \eqref{EqB-54}. 
It follows from elementary computations 
and Lemma \ref{LemR-11} that 
\begin{equation}\label{EqB-56}
\begin{split} 
& \quad 
\beta_{[\omega]} \|(-\Delta + \alpha_{[\omega]})^{-1}
\big\{ (W + \eta_{[\omega]})^{p} 
- W^{p}
\big\} 
\|_{L^{\infty}} 
\\[6pt]
&\lesssim 
\beta_{[\omega]} 
\|(-\Delta + \alpha_{[\omega]})^{-1}
\bigm\{
W^{p-1} |\eta_{[\omega]}|
\bigm\} 
\|_{L^{\infty}}
+ \beta_{[\omega]} 
\|(-\Delta + \alpha_{[\omega]})^{-1}
\bigm\{
|\eta_{[\omega]}|^{p}
\bigm\} 
\|_{L^{\infty}}
\\[6pt]
&\lesssim 
\beta_{[\omega]} 
\alpha_{[\omega]}^{\frac{p-3 - \varepsilon}{2}} 
\|W^{p-1} |\eta_{[\omega]}|\|_{
L^{\frac{3}{p - 1 - \varepsilon}}} 
+ \beta_{[\omega]} 
\alpha_{[\omega]}^{\frac{p - 3 - \varepsilon}{2}} 
\||\eta_{[\omega]}|^{p}\|_{
L^{\frac{3}{p - 1 - \varepsilon}}}. 
\end{split}
\end{equation} 
Consider the first term on the right-hand side of \eqref{EqB-56}. 
By the H\"older inequality and \eqref{EqB-46}, 
one can see that 
\begin{equation*}
\begin{split}
\beta_{[\omega]} 
\alpha_{[\omega]}^{\frac{p-3 - \varepsilon}{2}} 
\|W^{p-1} |\eta_{[\omega]}|\|_{
L^{\frac{3}{p - 1 - \varepsilon}}} 
\lesssim
\alpha_{[\omega]}^{\frac{p - 2 - \varepsilon}{2}} 
\|W\|_{L^{\frac{3(p-1)}{p - 1 - \varepsilon}}}^{p-1}
\|\eta_{[\omega]}\|_{L^{\infty}}
\lesssim 
\alpha_{[\omega]}^{\frac{p-2 - \varepsilon}{2}} 
\|\eta_{[\omega]}\|_{L^{\infty}}. 
\end{split}
\end{equation*}
Similarly, we obtain 
\[
\begin{split}
\beta_{[\omega]} 
\alpha_{[\omega]}^{\frac{p - 3 - \varepsilon}{2}} 
\||\eta_{[\omega]}|^{p}\|_{
L^{\frac{3}{p - 1 - \varepsilon}}} 
\lesssim 
\alpha_{[\omega]}^{\frac{p-2 - \varepsilon}{2}} 
\|\eta_{[\omega]}\|_{L^{\frac{3(p-1)}{p - 1 - \varepsilon}}}^{p-1}
\|\eta_{[\omega]}\|_{L^{\infty}} 
\lesssim 
\alpha_{[\omega]}^{\frac{p-2 - \varepsilon}{2}} 
\|\eta_{[\omega]}\|_{L^{\infty}}. 
\end{split}
\]
Therefore, we obtain 
\begin{equation} \label{EqB-57}
\beta_{[\omega]} \|(-\Delta + \alpha_{[\omega]})^{-1}
\big\{ (W + \eta_{[\omega]})^{p} 
- W^{p}
\big\} 
\|_{L^{\infty}} \lesssim 
\alpha_{[\omega]}^{\frac{p- 2 - \varepsilon}{2}} 
\|\eta_{[\omega]}\|_{L^{\infty}} 
= o_{n}(1) \|\eta_{[\omega]}\|_{L^{\infty}}. 
\end{equation}
Putting the estimates \eqref{EqB-54}, 
\eqref{EqB-55} and \eqref{EqB-57} together, 
we obtain the desired estimate \eqref{EqB-53}. 
This completes the proof. 
\end{proof}
We are now in a position to prove Lemma \ref{LemR-5}. 
\begin{proof}[Proof of Lemma \ref{LemR-5}]
As in \eqref{EqB-54}, we see that
\begin{equation}\label{EqB-58}
\begin{split}
\|(-\Delta + \alpha_{[\omega]})^{-1} N_{[\omega]}
\|_{L^{\infty}}
&\le 
\|(-\Delta + \alpha_{[\omega]})^{-1} 
\bigm\{ 
(W + \eta_{[\omega]})^{5} - W^{5} - 5 W^{4} \eta_{n}
\bigm\}
\|_{L^{\infty}} 
\\[6pt]
&\quad +
\beta_{[\omega]} \|(-\Delta + \alpha_{[\omega]})^{-1} 
\big\{ |W + \eta_{[\omega]}|^{p-1}(W+\eta_{[\omega]}) 
- W^{p}
\big\} 
\|_{L^{\infty}}.
\end{split}
\end{equation}
We consider the first term on the right-hand side of \eqref{EqB-58}. 
By Lemma \ref{LemR-13}, \eqref{EqB-33}, the H\"older inequality 
and Lemma \ref{LemR-17}, one can see that 
\begin{equation}\label{EqB-59}
\begin{split}
& \quad \|(-\Delta + \alpha_{[\omega]})^{-1}
\bigm\{ 
(W + \eta_{[\omega]})^{5} - W^{5} - 5 W^{4} \eta_{n}
\bigm\}
\|_{L^{\infty}} \\[6pt]
&\lesssim
\| (-\Delta + \alpha_{[\omega]})^{-1}\big\{ W^{3}|\eta_{[\omega]}|^{2} \big\} 
\|_{L^{\infty}}
+ 
\| (-\Delta + \alpha_{[\omega]})^{-1} |\eta_{[\omega]}|^{5} 
\|_{L^{\infty}}
\\[6pt]
&\lesssim
\|W^{3} |\eta_{[\omega]}|^{2} \|_{L^{\frac{3}{2} + \e} \cap L^{\frac{3}{2} - \e}}
+ 
\| |\eta_{[\omega]}|^{5} 
\|_{L^{\frac{3}{2} + \e} \cap L^{\frac{3}{2} - \e}}
\\[6pt]
&\le 
\|W\|_{L^{\frac{9}{2} + 3\e} \cap L^{\frac{9}{2} - 3\e}}^{3}
\| \eta_{[\omega]} \|_{L^{\infty}}^{2}
+
\|\eta_{[\omega]}\|_{L^{\frac{9}{2} + 3\e} \cap L^{\frac{9}{2} - 3\e}}^{3}
\| \eta_{[\omega]} 
\|_{L^{\infty}}^{2}
\lesssim 
\alpha_{[\omega]}.
\end{split} 
\end{equation}
Next, we consider the second term on the right-hand 
side of \eqref{EqB-58}. It follows from 
\eqref{EqB-33}, the H\"older inequality, 
\eqref{EqB-47} and 
Lemma \ref{LemR-17} that if $\varepsilon >0$ 
is sufficiently small, 
we have 
\begin{equation}\label{EqB-60}
\begin{split}
& \quad 
\beta_{[\omega]}
\|
(-\Delta + \alpha_{[\omega]})^{-1}
\bigm\{ 
(W + \eta_{[\omega]})^{p} - W^{p} \bigm\}
\|_{L^{\infty}}
\\[6pt]
&\lesssim
\beta_{[\omega]}
\| (-\Delta + \alpha_{[\omega]})^{-1}\big\{ W^{p-1} 
\eta_{[\omega]} \big\} \|_{L^{\infty}}
+ 
\beta_{[\omega]}
\| (-\Delta + \alpha_{[\omega]})^{-1}|\eta_{[\omega]}|^{p} \|_{L^{\infty}}
\\[6pt]
&\lesssim
\beta_{[\omega]} \alpha_{[\omega]}^{\frac{p-3 - (p+2) \varepsilon}{2}}
\|W^{p-1} \eta_{[\omega]} \|_{L^{\frac{3}{p-1 - (p-1)\varepsilon}}}
+ 
\beta_{[\omega]} 
\alpha_{[\omega]}^{\frac{p-3 - (p+2) \varepsilon}{2}}
\| |\eta_{[\omega]}|^{p} \|_{L^{\frac{3}{p-1 - (p-1)\varepsilon}}}
\\[6pt]
&\lesssim 
\alpha_{[\omega]}^{\frac{p-2 - (p+2) \varepsilon}{2}}
\|W\|_{L^{\frac{3(p-1)}{p-1- (p-1)
\varepsilon}}}^{p-1} \| \eta_{[\omega]} \|_{L^{\infty}}
+ \alpha_{[\omega]}^{\frac{p-2 - (p+2) \varepsilon}{2}}
\|\eta_{[\omega]}\|_{L^{\frac{3(p-1)}{p-1- (p - 1)
\varepsilon}}}^{p-1} \| \eta_{[\omega]} \|_{L^{\infty}}
\\[6pt]
&\lesssim 
\alpha_{[\omega]}^{\frac{p-1 - (p+2) \varepsilon}{2}} 
\end{split} 
\end{equation}
Putting \eqref{EqB-58}, \eqref{EqB-59} and \eqref{EqB-60}
together, we obtain the desired estimate \eqref{EqB-24}. 
\end{proof}

Next, we shall give the proof of Lemma \ref{LemR-9}, which is 
similar to that of Lemma \ref{LemR-5}. 
\begin{lemma}\label{LemR-18}
Let $p = 2$. 
For sufficiently small 
$\varepsilon > 0$ and $\omega > 0$, the following estimate holds:
\begin{equation}\label{EqB-61}
\|\eta_{[\omega]}\|_{L^{\infty}} 
\lesssim 
\alpha_{[\omega]}^{\frac{1}{2} - \varepsilon}. 
\end{equation}
\end{lemma}
\begin{proof}
Similarly to \eqref{EqB-51}, 
by \eqref{EqB-14}, 
Lemma \ref{LemR-16} with 
$f = (- \Delta + \alpha_{[\omega]})^{-1} F_{[\omega]}$, 
and \eqref{EqB-15}, 
one can see that 
\begin{equation} \label{EqB-62}
\begin{split}
\|\eta_{[\omega]}\|_{L^{\infty}} 
& \lesssim 
\|(-\Delta + \alpha_{[\omega]}
- 5 W^{4})^{-1}F_{[\omega]}\|_{L^{\infty}} \\[6pt]
& = 
\|(1 - 5 (-\Delta + \alpha_{[\omega]})^{-1} W^{4})^{-1}
(-\Delta + \alpha_{[\omega]})^{-1}F_{[\omega]}\|_{L^{\infty}}
\\[6pt]
&\leq 
\alpha_{[\omega]}
\|(-\Delta + \alpha_{[\omega]})^{-1} W\|_{L^{\infty}} 
\hspace{-3pt} + 
\beta_{[\omega]} 
\|(-\Delta + \alpha_{[\omega]})^{-1} \widetilde{u}_{\omega,\nu_{\omega}}^{p}
\|_{L^{\infty}} 
+ \|(-\Delta + \alpha_{[\omega]})^{-1} \widehat{N}_{[\omega]} \|_{L^{\infty}}.
\end{split} 
\end{equation}
We consider the first two terms on the right-hand side of 
\eqref{EqB-62}. 
Since $p = 2$, we have by 
Lemmas \ref{LemR-11} and \ref{LemR-14}, \eqref{EqB-47} that 
\begin{equation}\label{EqB-63}
\begin{split}
& \quad \alpha_{[\omega]}
\|(-\Delta + \alpha_{[\omega]})^{-1} W\|_{L^{\infty}} 
+ 
\beta_{[\omega]} 
\|(-\Delta + \alpha_{[\omega]})^{-1} \widetilde{u}_{\omega,\nu_{\omega}}^{p}\|_{L^{\infty}}
\\[6pt]
&\lesssim 
\alpha_{[\omega]}\alpha_{[\omega]}^{-\frac{1}{2}}
+ 
\frac{\sqrt{\alpha_{[\omega]}}}{|\log \alpha_{[\omega]}|} 
\alpha_{[\omega]}
^{- \varepsilon} 
\|\widetilde{u}_{\omega,\nu_{\omega}}^{p}\|_{L^{\frac{3}{2(1-\varepsilon)}}}
\\[6pt]
& \lesssim 
\sqrt{\alpha_{[\omega]}} 
+ \alpha_{[\omega]}^{\frac{1}{2} - \varepsilon} 
\lesssim \alpha_{[\omega]}^{\frac{1}{2} - \varepsilon}. 
\end{split}
\end{equation}
Concerning the last term on the right-hand side of 
\eqref{EqB-62}. 
We claim that 
\begin{equation}\label{EqB-64}
\|(-\Delta + \alpha_{[\omega]})^{-1} 
\widehat{N}_{[\omega]} \|_{L^{\infty}} 
\le o_{n}(\|\eta_{[\omega]}\|_{L^{\infty}}).
\end{equation}
Indeed, 
as in \eqref{EqB-55}, 
by \eqref{EqB-28}, Lemma \ref{LemR-12}, 
elementary computations, 
the H\"older inequality and \eqref{EqB-16}, 
one can see that 
\begin{equation*} 
\begin{split}
\|(-\Delta + \alpha_{[\omega]})^{-1} \widehat{N}_{[\omega]} \|_{L^{\infty}} 
&\le 
\|(-\Delta + \alpha_{[\omega]})^{-1} 
\bigm\{ 
(W + \eta_{[\omega]})^{5}
- W^{5} - 5 W^{4} \eta_{[\omega]}
\bigm\}
\|_{L^{\infty}} \\[6pt]
&\lesssim 
\bigm\|(W + \eta_{[\omega]})^{5} - W^{5} - 5 W^{4} \eta_{[\omega]}
\bigm\|_{L^{\frac{3 + \varepsilon}{2}} \cap L^{\frac{3 - \varepsilon}{2}}}
\\[6pt]
&\lesssim 
\|W^{3} \eta_{[\omega]}^{2} \|_{L^{\frac{3 + \varepsilon}{2}} \cap L^{\frac{3 - \varepsilon}{2}}}
+
\|\eta_{[\omega]}^{5} \|_{L^{\frac{3 + \varepsilon}{2}} \cap L^{\frac{3 - \varepsilon}{2}}}
\\[6pt]
&\le \|W^{3}\|_{L^{\frac{6(3 + \varepsilon)}{9-\varepsilon}} \cap 
L^{\frac{6(3 - \varepsilon)}{9 + \varepsilon}}} 
\|\eta_{[\omega]}\|_{L^{6}} \|\eta_{[\omega]}\|_{L^{\infty}}
+\|\eta_{[\omega]}^{4}\|_{L^{\frac{3 + \varepsilon}{2}} \cap L^{\frac{3 - \varepsilon}{2}}}
\|\eta_{[\omega]}\|_{L^{\infty}} 
= 
o_{n}(1) \|\eta_{[\omega]}\|_{L^{\infty}}. 
\end{split} 
\end{equation*} 
Thus, we obtain the desired estimate \eqref{EqB-64}. 

Notice that by \eqref{EqB-64}, the last term 
on the right-hand side of \eqref{EqB-62} can be absorbed into the left-hand side. 
Hence, \eqref{EqB-61} holds from 
\eqref{EqB-62}, \eqref{EqB-63} and 
\eqref{EqB-64}. 
\end{proof}
We now give the proof of Lemma \ref{LemR-9}. 
\begin{proof}[Proof of Lemma \ref{LemR-9}]
As in \eqref{EqB-59}, we see 
from \eqref{EqB-61} that
\begin{equation*}
\begin{split}
\|(-\Delta + \alpha_{[\omega]})^{-1} 
\widehat{N}_{[\omega]}
\|_{L^{\infty}}
&
\le 
\|(-\Delta + \alpha_{[\omega]})^{-1} 
\bigm\{ 
(W + \eta_{[\omega]})^{5} - W^{5} - 5 W^{4} \eta_{n}
\bigm\}
\|_{L^{\infty}} 
\\[6pt]
&\lesssim
\| (-\Delta + \alpha_{[\omega]})^{-1}\big\{ W^{3}|\eta_{[\omega]}|^{2} \big\} 
\|_{L^{\infty}}
+ 
\| (-\Delta + \alpha_{[\omega]})^{-1} |\eta_{[\omega]}|^{5} 
\|_{L^{\infty}}
\\[6pt]
&\lesssim 
\|W^{3} |\eta_{[\omega]}|^{2} \|_{L^{\frac{3 + \varepsilon}{2}} \cap L^{\frac{3 - \varepsilon}{2}}}
+
\|\eta_{[\omega]}^{5} \|_{L^{\frac{3 + \varepsilon}{2}} \cap L^{\frac{3 - \varepsilon}{2}}}
\\[6pt]
&\le 
\|W\|_{L^{\frac{9}{2} + 3\e} \cap L^{\frac{9}{2} - 3\e}}\| \eta_{[\omega]} 
\|_{L^{\infty}}^{2}
+
\|\eta_{[\omega]}\|_{L^{\frac{9}{2} + 3\e} \cap L^{\frac{9}{2} - 3\e}}
\| \eta_{[\omega]} 
\|_{L^{\infty}}^{2}
\lesssim 
\alpha_{[\omega]}^{1 - 2 \e}. 
\end{split}
\end{equation*}
Thus, we have obtained the desired result.
\end{proof}

\subsection{(Case 3) $1 < p < 2$. 
}
In this section, we consider the case of 
$1 < p < 2$. 
We shall give the proof of Proposition 
\ref{PropR-1} \textrm{(iii)}. 
\subsubsection{Convergence to a 
singular solution}
We rescale 
the solution $\widetilde{u}_{\omega}$ 
as follows:
\begin{equation} \label{EqB-65}
\widehat{u}_{\omega} 
(s)
= \alpha_{\omega}^{-\frac{1}{2}} 
\widetilde{u}_{\omega}(r), 
\qquad 
s= \alpha_{\omega}^{\frac{1}{2}}r. 
\end{equation}
Then, $\widehat{u}_{\omega_{n}}$ 
satisfies the following: 
	\begin{equation} \label{EqB-66}
	- \frac{d^{2} \widehat{u}}{d s^{2}} - 
	\frac{2}{s} \frac{d \widehat{u}}{d s} 
	- \widehat{u} - \frac{\beta_{\omega_{n}}}{\alpha_{\omega_{n}}^{\frac{3 - p}{2}}}
	\widehat{u}^{p} - \alpha_{\omega_{n}} \widehat{u}^{5} = 0 
	\qquad \mbox{in $(0, \infty)$}. 
	\end{equation}
In this subsection, 
we will show that the rescaled 
function $\widehat{u}_{\omega}$
converges to a singular solution 
$U_{\theta_{0}, \infty}$ to some scalar field 
equation as $\omega$ tends to zero 
(see Proposition \ref{conv-sing} below 
more precisely). 
To this end, we prepare the following 
lemma. 
\begin{lemma}\label{LemR-15-1}
Let $1 < p < 2$ and 
$\widetilde{u}_{\omega}$ 
be the positive solution to \eqref{main-eq41}.
There exists a constant 
$C_{3}^{\prime} > 0$ such that 
\begin{equation}\label{EqB-67}
\beta_{\omega}
\geq C_{3}^{\prime} 
\alpha_{\omega}^{\frac{3-p}{2}}. 
\end{equation}
\end{lemma}
\begin{proof}[Proof of Lemma \ref{LemR-15-1}]
Putting $s = \sqrt{\alpha_{\omega}} r$, 
we see from \eqref{EqB-89} that 
\begin{equation}\label{EqB-69}
\|\widetilde{u}_{\omega}\|_{L^{2}}^{2} 
\geq C \int_{\alpha_{\omega}^{-\frac{1}{2}}}^{\infty} 
e^{- 2 \sqrt{\alpha_{\omega} r}} \, dr 
= C \alpha_{\omega}^{- \frac{1}{2}} 
\int_{1}^{\infty} e^{- 2 s} \, ds 
\geq C \alpha_{\omega}^{- \frac{1}{2}} 
\end{equation}
for some $C>0$. 
Next, we obtain an upper bound of 
$\widetilde{u}_{\omega}$ by 
applying Lemma \ref{nd-lem9-1}. 
By \eqref{EqB-48} and \eqref{EqL-1}, 
we have 
\[
\frac{\beta_{\omega}}{\alpha_{\omega}}      
\widetilde{u}_{\omega}^{p-1}(L \alpha_{\omega}^{\frac{1}{2}}) 
+ \frac{1}{\alpha_{\omega}}
\widetilde{u}_{\omega}^{4}(L \alpha_{\omega}^{\frac{1}{2}}) 
\leq \frac{C}{L^{p-1}} 
\alpha_{\omega}^{\frac{3 - p}{2} -1 
+ \frac{p - 1}{2}} + 
C \frac{\alpha_{\omega}^{3}}{L^{4}}
= \frac{C}{L^{p-1}} + 
C \frac{\alpha_{\omega}^{3}}{L^{4}}
< \frac{1}{2}. 
\]
for sufficiently large $L > 0$. 
We take $L > 0$ satisfying above and fix it. 
Then, it follows from Lemma \ref{nd-lem9-1}
with $R_{\omega} = L \alpha_{\omega}^{\frac{1}{2}}$ and \eqref{EqL-1} that 
    \begin{equation} \label{Eq-B91}
    \widetilde{u}_{\omega}(r) 
    \leq \frac{\widetilde{u}_{\omega}
    (L \alpha_{\omega}^{\frac{1}{2}})}
    {\widetilde{Y}_{\omega, 1/2}(L \alpha_{\omega}^{\frac{1}{2}})} 
    \widetilde{Y}_{\omega, 1/2}(r) 
    \leq 
    \dfrac{\frac{C}{L 
    \alpha_{\omega}^{1/2}}}{
    \frac{e^{- 
    \sqrt{\frac{\alpha_{\omega}}{2}} 
    L \alpha_{\omega}
    }}{L \alpha_{\omega}}} 
    \frac{e^{- 
    \sqrt{\frac{\alpha_{\omega}}{2}} r}}{r}
    = C L e^{\frac{L}{\sqrt{2}}}
    \frac{e^{- \sqrt{\frac{\alpha_{\omega}}
    {2}} r}}{r} 
    \end{equation}
for $r \geq L \alpha_{\omega}^{\frac{1}{2}}$. 
Observe that 
\[
\|\widetilde{u}_{\omega}
\|_{L^{p+1}}^{p + 1} 
= 4 \pi\int_{0}^{L_{1/2}
\alpha_{\omega}^{-\frac{1}{2}}} 
\widetilde{u}_{\omega}^{p+1} 
r^{2} \, dr 
+ 4 \pi\int_{L_{1/2} 
\alpha_{\omega}^{-\frac{1}{2}}}
^{\infty} 
\widetilde{u}_{\omega}^{p + 1} 
r^{2} \, dr. 
\]
By \eqref{EqL-1}, we obtain 
\[
\int_{0}^{L_{1/2} 
\alpha_{\omega}^{-\frac{1}{2}}} 
\widetilde{u}_{\omega}^{p + 1} 
r^{2} \, dr 
\leq C \int_{0}^{L_{1/2} 
\alpha_{\omega}^{-\frac{1}{2}}} 
r^{- p + 1} \, dr 
\leq C \alpha_{\omega}
^{\frac{p - 2}{2}}. 
\]
In addition, by \eqref{Eq-B91}, 
we have 
\[
\int_{L_{1/2} 
\alpha_{\omega}^{-\frac{1}{2}}}
^{\infty} 
\widetilde{u}_{\omega}^{p + 1} 
r^{2} \, dr 
\leq C \int_{L_{1/2} 
\alpha_{\omega}^{-\frac{1}{2}}}
^{\infty} 
e^{- (p +1) (1 - \e) 
\sqrt{\alpha_{\omega}} 
r} r^{- p + 1} \, dr \leq 
C \alpha_{\omega}^{\frac{p - 2}{2}}
\int_{L_{1/2} }
^{\infty} 
e^{- (p +1) (1 - \e) s} \, ds 
\leq C 
\alpha_{\omega}^{\frac{p - 2}{2}}. 
\] 
These yields that 
\begin{equation}\label{EqB-70}
\|\widetilde{u}_{\omega}
\|_{L^{p+1}}^{p + 1} 
\leq C \alpha_{\omega}
^{\frac{p - 2}{2}}. 
\end{equation} 
By \eqref{main-eq7}, \eqref{EqB-69} 
and \eqref{EqB-70}, one has 
\[
C \sqrt{\alpha_{\omega}} 
\leq \alpha_{\omega}\|\widetilde{u}_{\omega}\|_{L^{2}}^{2} 
= \frac{5 - p}{2 (p+1)} \beta_{\omega} 
\|\widetilde{u}_{\omega}\|_{L^{p+1}}^{p+1} 
\leq 
C \beta_{\omega} 
\alpha_{\omega}^{\frac{p - 2}{2}}. 
\]
Therefore, \eqref{EqB-67} holds. 
\end{proof}
Let $\{\omega_{n}\}$ be 
a sequence in $(0, \infty)$ 
with $\lim_{n \to \infty} \omega_{n} = 0$. 
By \eqref{EqB-48} and 
\eqref{EqB-67}, we see that 
there exists a constant 
$\theta_{0} > 0$ such that 
\[
\lim_{n \to \infty} 
\frac{\beta_{\omega_{n}}}
{\alpha_{\omega_{n}}
^{\frac{3 - p}{2}}} = \theta_{0}. 
\]
Note that the limit $\theta_{0}$ 
depends on the sequence 
$\{\omega_{n}\}$ in general. 
However, we will find that $\theta_{0}$ 
is a universal number, that is, it 
does not depend on the choice of 
the sequence $\{\omega_{n}\}$. 
To this end, we first show the following:
\begin{proposition}\label{conv-sing}
Let $1 < p < 2$ and 
$\widehat{u}_{\omega_{n}}$ 
be a function defined by \eqref{EqB-65}. 
For any finite interval $I \subset (0, \infty)$, we have 
\[
\lim_{n \to \infty} 
\widehat{u}_{\omega_{n}}(r) 
= U_{\theta_{0}, \infty}(r) \qquad \text{uniformly in $I$}, 
\]
where $U_{\theta_{0}, \infty}$ is a solution to 
\begin{equation}\label{EqB-71}
- U_{rr} - \frac{2}{r} U_{r} + U - \theta_{0} U^{p} = 0 
\qquad \mbox{in $(0, \infty)$}
\end{equation} 
satisfying 
\begin{equation} \label{EqB-72}
\lim_{r \to 0} r U_{\theta_{0}, \infty}(r) 
= \sqrt{3}. 
\end{equation}
\end{proposition}

\begin{remark}
We see from the result of 
\cite[Theorem 2]{MR1313805} that 
the equation \eqref{EqB-71} has infinitely many 
positive singular solutions $U$ satisfying
\[
\lim_{s \to 0} s U(s) 
= C
\] 
for some $C > 0$. 
However, we can specify 
the singular solution $U_{\theta_{0}, 
\infty}$ which is obtained in 
Proposition \ref{conv-sing} 
by using the uniqueness of 
positive solution to \eqref{main-eq11}. 
\end{remark}

\begin{lemma}\label{lem1-p12}
For any $L > 0$ 
and $\e > 0$, 
there exists 
and $R_{\e} > 0$
$n_{0} > 0$ such that 
\begin{equation} \label{EqB-73}
\sqrt{3} \frac{1- \e}{r} e^{-L} < 
\widetilde{u}_{\omega_{n}}(r) 
< \sqrt{3} \frac{1 + \e}{r} \qquad 
(R_{\e} 
< r < \frac{L}{\sqrt{\alpha_{\omega_{n}}}}). 
\end{equation}
for $n \geq n_{0}$. 
\end{lemma} 

\begin{proof}
Note that 
\[
\widetilde{u}_{\omega_{n}}(r) 
\leq \left( 
1 + \frac{
\gamma_{\omega_{n}} }{3} r^{2}
\right)^{-\frac{1}{2}}
\qquad \mbox{for all $r > 0$}, 
\]
This together with 
$\lim_{n \to \infty} \gamma_{\omega_{n}} 
= 1$ (see \eqref{EqL-2} for the definition 
of $\gamma_{\omega}$)
implies that $\widetilde{u}_{\omega_{n}}(r) 
< \sqrt{3} \frac{1 + \e}{r}$ 
for sufficiently large $n \in \mathbb{N}$. 

On the other hand, it follows from \eqref{EqB-35} that 
for any $\e>0$, we obtain 
\[
\widetilde{u}_{\omega_{n}}(r) 
\geq 
\sqrt{3} (1 - \e) \frac{e^{- L}}{r} 
\qquad \mbox{for $R_{\e} 
< r< \frac{L}{\sqrt{\alpha_{\omega_{n}}}}$}. 
\]
This completes the proof. 
\end{proof}
Putting 
$\overline{u}_{\omega_{n}}(s) 
= s 
\widehat{u}_{\omega_{n}}(s)$, we see that 
$\overline{u}_{\omega_{n}}$ 
satisfies the following: 
\begin{equation}
\label{EqB-74}
\begin{cases}
\overline{u}_{ss} - \overline{u} 
+ \beta_{\omega_{n}} \alpha_{\omega_{n}}
^{\frac{p - 3}{2}} 
s^{1 - p} \overline{u}^{p} 
+ \alpha_{\omega_{n}} s^{-4} \overline{u}^{5} = 0 
\qquad \mbox{in $(0, \infty)$}, \\
\overline{u}(s) > 0
\qquad \mbox{in $(0, \infty)$}, \\
\overline{u}(0) = 0, 
\qquad \overline{u}_{s}(0) 
= \frac{1}
{\sqrt{\alpha_{\omega_{n}}}}, 
\qquad 
\lim_{s \to \infty} \overline{u}(s) = 0. &
\end{cases}
\end{equation}
We obtain the following: 
\begin{lemma}\label{lem-conv2}
For any finite interval 
$I$ in $(0, \infty)$, one has 
$\lim_{n \to \infty} 
\overline{u}_{\omega_{n}}(s) 
= \overline{u}_{0}(s)$ for $s \in I$, 
where $\overline{u}_{0}(s)$ 
satisfies 
\begin{equation}\label{EqB-75}
\begin{cases}
\frac{d^{2} \overline{u}_{0}}{d s^{2}} - \overline{u}_{0} 
+ \theta_{0} 
s^{1 - p} \overline{u}_{0}^{p} = 0 
& \qquad \mbox{on $(0, \infty)$}, \\
\overline{u}_{0}(0) = \sqrt{3}, 
\qquad 
\lim_{s \to \infty} \overline{u}_{0}(s) = 0. & 
\end{cases}
\end{equation}
\end{lemma}
\begin{proof}
Observe from Lemma \ref{lem1-p12} that for any $L > 0$ and 
$\delta > 0$, there exists $C_{1} > 0$ 
and $C_{2} > 0$ such that 
$C_{1} < r \widetilde{u}_{\omega_{n}}(r) < C_{2}$ 
for $r \in [\frac{\delta}{\sqrt{\alpha_{\omega_{n}}}}, \frac{L}{\sqrt{\alpha_{\omega_{n}}}}]$. 
This together with 
$\overline{u}_{\omega_{n}}(s) 
= s \widehat{u}_{\omega_{n}}(s) 
= \alpha_{\omega_{n}}^{- \frac{1}{2}} s 
\widetilde{u}_{\omega_{n}}(r)$ 
and $s = \sqrt{\alpha_{\omega_{n}}} r$
implies that 
$C_{1} < \overline{u}_{\omega_{n}}(s) 
< C_{2}$ 
for $s \in [\delta, L]$. 
In addition, we see from \eqref{EqB-74} that 
\[
\biggl|\frac{d^{2} \overline{u}_{\omega_{n}}}
{d s^{2}}(s)\biggl| 
\leq 
\overline{u}_{\omega_{n}}(s) + 
\beta_{\omega_{n}} \alpha_{\omega_{n}}
^{\frac{p - 3}{2}} 
s^{1 - p} \overline{u}_{\omega(m_{n})}^{p}(s) 
+ \alpha_{\omega_{n}} s^{-4} 
\overline{u}_{\omega(m_{n})}^{5}(s) 
\lesssim_{\delta} 1 
\qquad \mbox{for any $s \in [\delta, L]$}. 
\]
Then, by the Ascoli-Arzela theorem, there exists 
$\overline{u}_{0}$ such that 
\begin{equation}\label{EqB-76}
\lim_{n \to \infty} 
\overline{u}_{\omega_{n}}(s) 
= \overline{u}_{0}(s), 
\qquad 
\lim_{n \to \infty} 
\frac{d \overline{u}_{\omega_{n}}}{d s}(s) 
= \frac{d \overline{u}_{0}}{d s}(s) 
\qquad 
\mbox{uniformly in $[\delta, L]$}.
\end{equation}
Observe that 
$\overline{u}_{\omega_{n}}$ satisfies 
\begin{equation}\label{EqB-77}
\frac{d \overline{u}_{\omega_{n}}}{d s}
(s) - 
\frac{d \overline{u}_{\omega_{n}}}{d s}
(\delta) 
= \int_{\delta}^{s} 
(\overline{u}_{\omega_{n}}(\sigma) 
- \beta_{\omega_{n}} 
\alpha_{\omega_{n}}^{\frac{p - 3}{2}} 
\sigma^{1 - p} 
\overline{u}_{\omega_{n}}^{p}(\sigma) 
+ \alpha_{\omega_{n}} 
\sigma^{-4} \overline{u}_{\omega_{n}}^{5} 
(\sigma)) d \sigma. 
\end{equation}
It follows from \eqref{EqB-76}, 
\eqref{EqB-77} and $\lim_{n \to \infty} 
\alpha_{\omega_{n}} = 0$ that 
\begin{equation} \label{EqB-78}
\frac{d \overline{u}_{0}}{d s}(s) - 
\frac{d \overline{u}_{0}}{d s}(\delta) 
= \int_{\delta}^{s} 
(\overline{u}_{0}(\sigma) 
- \theta_{0} 
\sigma^{1 - p} 
\overline{u}_{0}^{p}(\sigma)) 
d \sigma. 
\end{equation}
Since $\delta > 0$ and $L > 0$ are 
arbitrary, 
we see from \eqref{EqB-78} that 
$\overline{u}_{0}$ satisfies 
the equation in \eqref{EqB-75}. 

Next, we shall show that 
$\overline{u}_{0}(0) 
= \lim_{s \to 0}\overline{u}_{0}(s) = \sqrt{3}$. 
For any $s_{0} > 0$, we see that 
\[
\begin{split}
\overline{u}_{0}(s_{0}) 
= \liminf_{n \to \infty} 
\overline{u}_{\omega_{n}}(s_{0})
& = \liminf_{n \to \infty} 
s_{0}\widehat{u}_{\omega_{n}} 
\left(s_{0}\right) \\
& = \liminf_{n \to \infty} 
\frac{s_{0}}{\sqrt{\alpha_{\omega_{n}}}}
\widetilde{u}_{\omega_{n}}
\left(\frac{s_{0}}{\sqrt{\alpha_{\omega_{n}}}} 
\right) 
\\
& \geq (1 - \e)
\liminf_{n \to \infty} 
\frac{s_{0}}{\sqrt{\alpha_{\omega_{n}}}}
Y_{\omega_{n}} 
\left(\frac{s_{0}}{\sqrt{\alpha_{\omega_{n}}}} 
\right) \\
& = \sqrt{3} (1 - \e) 
\liminf_{n \to \infty} 
\frac{s_{0}}{\sqrt{\alpha_{\omega_{n}}}} 
\times
\frac{\sqrt{\alpha_{\omega_{n}}}}{s_{0}} 
e^{- \sqrt{\alpha_{\omega_{n}}} \times \frac{s_{0}}
{\sqrt{\alpha_{\omega_{n}}}}} \\
& = \sqrt{3} (1 - \e)e^{- s_{0}}. 
\end{split}
\]
This implies that 
$\overline{u}_{0}(s_{0}) \geq \sqrt{3}(1 - \e)e^{- s_{0}}$. 

On the other hand, 
by \eqref{EqL-1}, 
$\lim_{n \to \infty} 
\alpha_{\omega_{n}} = 0$ and 
$\lim_{n \to \infty} 
\gamma_{\omega_{n}} = 1$ 
we obtain 
\[
\begin{split}
\overline{u}_{0}(s_{0}) 
= \limsup_{n \to \infty} 
\overline{u}_{\omega_{n}}
(s_{0})
& = \limsup_{n \to \infty} 
s_{0} \widehat{u}_{\omega_{n}} (s_{0})\\
& = \limsup_{n \to \infty}
\frac{s_{0}}{\sqrt{
\alpha_{\omega_{n}}}}
\widetilde{u}_{\omega_{n}}
\left(\frac{s_{0}}{\sqrt{\alpha_{\omega_{n}}}} 
\right) \\
& \leq \limsup_{n \to \infty} 
\frac{s_{0}}{\sqrt{\alpha_{\omega_{n}}}}
\left(1 + 
\frac{\gamma_{\omega_{n}}}{3} \left(\frac{s_{0}^{2}}
{\alpha_{\omega_{n}}} \right)
\right)^{-\frac{1}{2}} \\
& = \limsup_{n \to \infty}
\frac{s_{0}}{\sqrt{
\alpha_{\omega_{n}}}}
\frac{\sqrt{3 \alpha_{\omega_{n}}}}{s_{0} 
\sqrt{\gamma_{\omega_{n}}}} 
\left(\frac{3}
{\gamma_{\omega_{n}}} 
\left(\frac{\alpha_{\omega_{n}}}
{s_{0}^{2}} \right) 
+ 1\right)^{-\frac{1}{2}} \\
& = \limsup_{n \to \infty} 
\frac{\sqrt{3}}{\sqrt{
\gamma_{\omega_{n}}}}
\left(\frac{3}
{\gamma_{\omega_{n}}} 
\left(\frac{\alpha_{\omega_{n}}}
{s_{0}^{2}}\right) 
+ 1\right)^{-\frac{1}{2}} 
= \sqrt{3}. 
\end{split}
\]
Thus, we have $\sqrt{3}(1 - \e)e^{- s_{0}} 
\leq \overline{u}_{0}(s_{0}) \leq \sqrt{3}$ for any $s_{0} > 0$. 
Since $\e > 0$ is arbitrary, 
we see that 
$\overline{u}_{0}(0) 
= \lim_{s_{0} \to 0}\overline{u}_{0}(s_{0}) = \sqrt{3}$. 
This completes the proof. 
\end{proof}
\begin{proof}[Proof of Proposition 
\ref{conv-sing}]
We put $U_{\theta_{0}, \infty}(s) := 
\overline{u}_{0}(s) s^{-1}$. 
Then we see that 
$U_{\theta_{0}, \infty}(s)$ is a 
solution to \eqref{EqB-71} satisfying 
\eqref{EqB-73}. 
Then, from Lemma \ref{lem-conv2} and 
$r = \alpha_{\omega_{n}}^{-\frac{1}{2}} s$, 
we obtain 
\[
\lim_{n \to \infty}
\alpha_{\omega_{n}}^{-\frac{1}{2}} \widetilde{u}_{\omega_{n}}
(\alpha_{\omega_{n}}^{-\frac{1}{2}}s) = 
\lim_{n \to \infty}
\overline{u}_{\omega_{n}}(s)s^{-1}
= \overline{u}_{0}(s) s^{-1} 
= U_{\theta_{0}, \infty}(s)
\qquad 
\text{uniformly in $s \in I$}.
\]
This concludes the proof. 
\end{proof}
\subsubsection{Proof of 
Proposition \ref{PropR-1} \textrm{(iii)}}

In this subsection, we prove Proposition \ref{PropR-1} \textrm{(iii)}. 
To this end, we need several preparations. 
\begin{lemma}\label{LemR-21}
Let $1 < p < 2$. 
Then, for any $q \in [2, 3)$, we obtain
\begin{equation} \label{EqB-79}
\lim_{n \to \infty}
\alpha_{\omega_{n}}
^{\frac{3 - q}{2}} 
\|\widetilde{u}_{\omega_{n}}
\|_{L^{q}}^{q} 
= \|U_{\theta_{0}, \infty}\|_{L^{q}}^{q}. 
\end{equation}
\end{lemma}
\begin{proof}
For any $\e > 0$, we have 
\begin{equation} \label{EqB-80}
\begin{split}
\|\widetilde{u}_{\omega_{n}}
\|_{L^{q}}^{q} 
& = 4 \pi \int_{0}^{\e \alpha_{\omega_{n}}^{- \frac{1}{2}}} 
\widetilde{u}_{\omega_{n}}^{q} 
(r) r^{2} \, dr + 
4 \pi 
\int_{\e \alpha_{\omega_{n}}^{- 
\frac{1}{2}}}
^{\e^{-1} \alpha_{\omega_{n}}^{- \frac{1}{2}}} 
\widetilde{u}_{\omega_{n}}^{q} 
(r) r^{2} \, dr \\
& \quad 
+ 4 \pi 
\int_{\e^{-1} 
\alpha_{\omega_{n}}^{- \frac{1}{2}}}
^{\infty} 
\widetilde{u}_{\omega_{n}}^{q} 
(r) r^{2} \, dr \\
& =: I_{1, n} + I_{2, n} + I_{3, n}. 
\end{split}
\end{equation}
By \eqref{EqL-1}, one has 
\[
I_{1, n} \lesssim 
\int_{0}^{\e \alpha_{\omega_{n}}^{- \frac{1}{2}}} r^{- q + 2} \, dr 
\lesssim \e ^{3 - q} 
\alpha_{\omega_{n}}^{
\frac{q - 3}{2}}. 
\]
By \eqref{EqB-48} and \eqref{EqL-1}, 
we have 
\[
\frac{\beta_{\omega}}{\alpha_{\omega}}      
\widetilde{u}_{\omega}^{p-1}
(\e^{-1} \alpha_{\omega}^{\frac{1}{2}}) 
+ \frac{1}{\alpha_{\omega}}
\widetilde{u}_{\omega}^{4}(\e^{-1} 
\alpha_{\omega}^{\frac{1}{2}}) 
\leq C \e^{p-1} 
\alpha_{\omega}^{\frac{3 - p}{2} -1 
+ \frac{p - 1}{2}} + 
C \e^{4} \alpha_{\omega}^{3}
= C \e^{p-1} + 
C \e^{4} \alpha_{\omega}^{3}
<  C \e^{p-1}. 
\] 
Applying Lemma \ref{nd-lem9-1}
with $R_{\omega} = \e^{-1} 
\alpha_{\omega}^{\frac{1}{2}}$, 
we have by \eqref{EqL-1} that 
    \begin{equation} \label{Eq-B91-2}
    \widetilde{u}_{\omega}(r) 
    \leq \frac{\widetilde{u}_{\omega}
    (\e^{-1} \alpha_{\omega}^{\frac{1}{2}})}
    {\widetilde{Y}_{\omega, C \e^{p-1}}
    (\e^-1 \alpha_{\omega}^{\frac{1}{2}})} 
    \widetilde{Y}_{\omega, C \e^{p-1}}(r) 
    \leq 
    C e^{- 
    \sqrt{(1 - C \e^{p-1})} 
    \e^{-1}} 
    \frac{e^{- \sqrt{\frac{\alpha_{\omega}}
    {2}} r}}{r} 
    \end{equation}
for $r \geq \e^{-1} 
\alpha_{\omega}^{\frac{1}{2}}$.
It follows from  \eqref{Eq-B91-2}
that 
\[
I_{3, n} \lesssim 
\int_{\e^{-1} 
\alpha_{\omega_{n}}^{- \frac{1}{2}}}
^{\infty}
\frac{e^{- q 
\sqrt{\frac{\alpha
_{\omega_{n}}}{2}} r}}{r^{q -2 }} \, dr 
= \alpha_{\omega_{n}}
^{\frac{q - 3}{2}} 
\int_{\e^{-1}}^{\infty} 
\frac{e^{- \frac{q}{\sqrt{2}} s}}
{s^{q - 2}} \, ds 
\lesssim 
\alpha_{\omega_{n}}
^{\frac{q - 3}{2}} 
\int_{\e^{-1}}^{\infty}
e^{- \frac{q s}{2}}\, ds 
\lesssim \e \alpha_{\omega_{n}}
^{\frac{q - 3}{2}} 
\]
for any small $\e > 0$. 
These imply that 
\begin{equation} \label{EqB-81}
\alpha_{\omega_{n}}
^{\frac{3 - q}{2}} I_{1, n}, \; 
\alpha_{\omega_{n}}
^{\frac{3 - q}{2}} I_{3, n} 
\lesssim \e^{3 - q}. 
\end{equation}
Observe from 
$\widetilde{u}_{\omega_{n}} 
(r)
= \alpha_{\omega_{n}}^{\frac{1}{2}} 
\widehat{u}_{\omega_{n}}(\alpha_{\omega_{n}}^{\frac{1}{2}} r)$ for $r > 0$ that 
\[
\begin{split}
I_{2, n} = 4 \pi 
\int_{\e \alpha_{\omega_{n}}^{- 
\frac{1}{2}}}
^{\e^{-1} \alpha_{\omega_{n}}^{- \frac{1}{2}}} 
\widetilde{u}_{\omega_{n}}^{q} 
(r) r^{2} \, dr 
= 4 \pi 
\int_{\e \alpha_{\omega_{n}}^{- 
\frac{1}{2}}}
^{\e^{-1} \alpha_{\omega_{n}}^{- \frac{1}{2}}} 
\alpha_{\omega_{n}}^{\frac{q}{2}}
\widehat{u}_{\omega_{n}}^{q} 
(\alpha_{\omega_{n}}^{
\frac{1}{2}}r) r^{2} \, dr 
= 4 \pi 
\alpha_{\omega_{n}}
^{\frac{q - 3}{2}}
\int_{\e}
^{\e^{-1}} 
\widehat{u}_{\omega_{n}}^{q} 
(s) s^{2} \, ds. 
\end{split}
\]
In addition, it follows from Proposition 
\ref{conv-sing} that 
\begin{equation*} 
\alpha_{\omega_{n}}
^{\frac{3 - q}{2}}
I_{2, n}
\to 4 \pi \int_{\e}
^{\e^{-1}} 
U_{\theta_{0}, \infty}^{q} 
(s) s^{2} \, ds 
\qquad 
\mbox{as $n \to \infty$}. 
\end{equation*} 
Since $\e > 0$ is arbitrary, we see 
from \eqref{EqB-80}, \eqref{EqB-81} and 
\eqref{EqB-51} that \eqref{EqB-79} holds. 
\end{proof}

\begin{lemma}\label{LemR-20}
Let $1 < p < 2$ and 
$\overline{u}_{0}$ be the solution 
to \eqref{EqB-75}, 
obtained in Lemma \ref{lem-conv2}. 
Then, we have 
$\frac{d \overline{u}_{0}}{d s}(0) = 0$. 
\end{lemma}
\begin{proof}
It follows from \eqref{main-eq7} that 
\[
\alpha_{\omega_{n}} 
\|\widetilde{u}_{\omega_{n}}\|_{L^{2}}^{2} 
= \frac{5 - p}{2(p + 1)} \beta_{\omega_{n}} 
\|\widetilde{u}_{\omega_{n}}
\|_{L^{p + 1}}^{p + 1}. 
\]
This yields that 
\[
\alpha_{\omega_{n}}^{\frac{1}{2}} 
\|\widetilde{u}_{\omega_{n}}\|_{L^{2}}^{2} 
= \frac{5 - p}{2(p + 1)} \beta_{\omega_{n}} 
\alpha_{\omega_{n}}^{- \frac{1}{2}}
\|\widetilde{u}_{\omega_{n}}
\|_{L^{p + 1}}^{p + 1}
= \frac{5 - p}{2(p + 1)} \beta_{\omega_{n}} 
\alpha_{\omega_{n}}^{\frac{p - 3}{2}}
\times \alpha_{\omega_{n}}^{\frac{2 - p}{2}}
\|\widetilde{u}_{\omega_{n}}\|_{L^{p + 1}}^{p + 1}. 
\]
This together with \eqref{EqB-79} 
and $U_{\theta_{0}, \infty}(s) 
= \overline{u}_{0}(s) s^{-1}$
yields that 
\begin{equation} \label{EqB-82}
\begin{split}
\theta_{0} = 
\lim_{n \to \infty}
\beta_{\omega_{n}} 
\alpha_{\omega_{n}}^{\frac{p - 3}{2}} 
= 
\lim_{n \to \infty}
\frac{2(p + 1)}{5 - p} 
\dfrac{ \alpha_{\omega_{n}}^{\frac{1}{2}} 
\|\widetilde{u}_{\omega_{n}}
\|_{L^{2}}^{2}}
{ \alpha_{\omega_{n}}^{\frac{2 - p}{2}}
\|\widetilde{u}_{\omega_{n}}
\|_{L^{p + 1}}^{p + 1}} 
& = \frac{2(p + 1)}{5 - p} 
\dfrac{\| U_{\theta_{0}, \infty}\|_{L^{2}}^{2}}
{\|U_{\theta_{0}, \infty}\|_{L^{p + 1}}^{p + 1}} \\
& = \frac{2(p + 1)}{5 - p} 
\dfrac{\int_{0}^{\infty} 
\overline{u}_{0}^{2}(s) \, ds }
{\int_{0}^{\infty} 
\overline{u}_{0}^{p + 1}(s) s^{- p + 1} \, ds}. 
\end{split}
\end{equation}
Multiplying the equation in 
\eqref{EqB-75} by $\overline{u}_{0}$ 
and integrating the resulting equation 
from $0$ to $\infty$, 
we obtain 
\begin{equation}\label{EqB-83}
\frac{d \overline{u}_{0}}{d s}(0) 
\overline{u}_{0}(0)
= - 
\int_{0}^{\infty} 
\left\{
\left(\frac{d \overline{u}_{0}}{d s} \right)^{2} 
+ (\overline{u}_{0})^{2}
- \theta_{0} 
(\overline{u}_{0})^{p+1}s^{1 - p}
\right\}
\, ds. 
\end{equation}
This implies that 
$\big|\frac{d \overline{u}_{0}}{d s}(0) 
\big| < \infty$. 
Moreover, observe from \eqref{EqB-75} that 
\begin{equation}\label{EqB-84}
\frac{d}{d s}\left(
\frac{1}{2} 
\left(\frac{d \overline{u}_{0}}{d s}\right)^{2}
- \frac{(\overline{u}_{0})^{2}}{2}
+ \theta_{0}
\frac{s^{1 - p}(\overline{u}_{0})^{p+1}}
{p+1}
\right)
= - \frac{p-1}{p+1} \theta_{0} 
s^{-p} (\overline{u}_{0})^{p+1}. 
\end{equation}
In addition, 
multiplying the equation in \eqref{EqB-84} by $s$ and 
integrating the resulting equation 
from $0$ to $\infty$, we get
\begin{equation}\label{EqB-85}
\int_{0}^{\infty} 
\left\{
\frac{1}{2} 
\left(\frac{d \overline{u}_{0}}{d s} \right)^{2}
- \frac{(\overline{u}_{0})^{2}}{2}
+ \theta_{0}
\frac{(\overline{u}_{0})^{p+1}s^{1 - p}}
{p+1}
\right\} \, ds = 
\frac{p-1}{p+1} \theta_{0}
\int_{0}^{\infty} 
(\overline{u}_{0})^{p+1}s^{1-p} \, ds. 
\end{equation}
This together with \eqref{EqB-82}
implies that 
\begin{equation*}
\int_{0}^{\infty} 
\left(\frac{d \overline{u}_{0}}{d s} \right)^{2} \, ds
= 
\int_{0}^{\infty} (\overline{u}_{0})^{2} 
\, ds + 
\frac{2(p-2)}
{p + 1} 
\theta_{0}
\int_{0}^{\infty} 
(\overline{u}_{0})^{p+1}s^{1-p} \, ds 
= \frac{3p - 3}{5 - p} 
\int_{0}^{\infty} (\overline{u}_{0})^{2} 
\, ds.
\end{equation*}
It follows from \eqref{EqB-83}, 
\eqref{EqB-84} and \eqref{EqB-82} that 
\begin{equation} \label{EqB-86}
\begin{split}
\frac{d \overline{u}_{0}}{d s}(0) 
\overline{u}_{0}(0)
& = - \int_{0}^{\infty} 
\left\{
\left(\frac{d \overline{u}_{0}}{d s} 
\right)^{2} 
+ (\overline{u}_{0})^{2}
- \theta_{0} 
(\overline{u}_{0})^{p+1}s^{1 - p}
\right\}
\, ds \\[6pt]
& = - \left\{ 
\frac{3p -3}{5 - p} + 1 
- \frac{2(p + 1)}{5 - p}
\right\} \int_{0}^{\infty} 
\overline{u}_{0}^{2}(s) s^{2} \, ds 
\\[6pt]
& = 0. 
\end{split}
\end{equation}
Since $\overline{u}_{0}(0) = \sqrt{3}$, 
we see from \eqref{EqB-86} 
that $\frac{d \overline{u}_{0}}{d s}
(0) = 0$. 
Thus, we have obtained the desired result. 
\end{proof}
We are now in a position to prove 
Proposition \ref{PropR-1} \textrm{(iii)}.
\begin{proof}[Proof of 
Proposition \ref{PropR-1} \textrm{(iii)}]
We put $\overline{v}_{0} = 
\theta_{0}^{\frac{1}{p-1}} 
\overline{u}_{0}$. 
Then, $\overline{v}_{0}$ satisfies 
\begin{equation}\label{EqB-87}
\begin{cases}
\frac{d^{2} \overline{v}}{d s^{2}} 
- \overline{v} 
+ s^{1 - p} \overline{v}^{p} = 0 
& \qquad \mbox{on $(0, \infty)$}, \\
\overline{v}(0) = \sqrt{3} 
\theta_{0}^{\frac{1}{p - 1}}, 
\qquad
\frac{d \overline{v}}{d s}(0) = 0, 
\qquad
\lim_{s \to \infty} \overline{v}_{0}(s) 
= 0. & 
\end{cases}
\end{equation}
Note that the equation in 
\eqref{EqB-87} has a unique positive 
solution $V$ in $H^{1}(0, \infty)$ (see 
Proposition \ref{thm-scud}). 
Thus, we see that $\overline{v}_{0} = V$, 
so that 
we have $V(0) = \overline{v}_{0}(0) 
= \sqrt{3} \theta_{0}^{\frac{1}{p- 1}}$. 
This implies that $\theta_{0} 
= 3^{- \frac{p - 1}{2}} V^{p - 1}(0)$. 
\end{proof}
\section{Uniqueness of the large solution $u_{\omega}$ and 
proof of Theorem \ref{class}}
\label{sec-uni2}
In this section, following \cite{MR3964275}, 
we shall give 
the proof of Theorem \ref{thm-bl} \textrm{(i)} 
(uniqueness of 
the positive solution $u_{\omega}$ to \eqref{sp} 
for sufficiently small $\omega > 0$) by contradiction. 
Suppose to the contrary that 
there exists a sequence $\{\omega_{n}\}$ in $(0, \infty)$ 
with $\lim_{n \to \infty} \omega_{n} = 0$ such that 
for each $n \in \mathbb{N}$, 
$u_{n}^{(i)}\; (i = 1, 2)$ is a solution 
to \eqref{sp} with $\omega = \omega_{n}$ satisfying 
$u_{n}^{(1)} \neq u_{n}^{(2)}$ and 
$\lim_{n \to \infty} \|u_{n}^{(i)}\|_{L^{\infty}} 
= \infty$. 
For $i = 1, 2$, we set 
\begin{align*} 
& M_{n}^{(i)} := \|u_{n}^{(i)}\|_{L^{\infty}}, 
\qquad 
\widetilde{u}_{n}^{(i)}(\cdot) = (M_{n}^{(i)})^{-1} 
u_{n}^{(i)}((M_{n}^{(i)})^{-2} \cdot), 
\\[6pt]
& 
\alpha_{n}^{(i)} := \omega_{n} 
(M_{n}^{(i)})^{-4}, 
\qquad 
\beta_{n}^{(i)}:= (M_{n}^{(i)})^{p - 5} 
\end{align*}
and 
\[
\nu_{n} = \frac{M_{n}^{(2)}}{M_{n}^{(1)}}. 
\]
It follows from \eqref{main-eq5}, 
\eqref{EqB-1}, \eqref{EqB-2} 
and \eqref{EqB-3} 
that 
\begin{equation} \label{eqU-1}
\begin{split}
& \lim_{n \to \infty} 
\omega_{n}^{-\frac{1}{2}} 
(M_{n}^{(i)})^{p - 3} 
= \frac
{12 \pi (p+1)}{(5 - p)\|W\|_{L^{p+1}}^{p+1}} \qquad 
\mbox{when $2 < p < 3$}, \\[6pt]
& \lim_{n \to \infty} 
\omega_{n}^{-\frac{1}{2}} 
(M_{n}^{(i)})^{p - 3} 
|\log(\alpha_{n}^{(i)})| 
= \frac{2}{\sqrt{3}}
\qquad \mbox{when $p = 2$}, \\
& \lim_{n \to \infty} 
\omega_{n}^{\frac{p - 3}{2}} 
(M_{n}^{(i)})^{- p + 1} = 
3^{- \frac{p - 1}{2}} V^{p - 1}(0) 
\qquad \mbox{when $1 < p < 2$}
\end{split}
\end{equation}
for $i = 1, 2$.
This implies the following lemma:
\begin{lemma}\label{lem-uni0}
For $1 < p < 3$, up to subsequence, we have 
\begin{equation} \label{eqU-2}
\lim_{n \to \infty} \nu_{n} = 1. 
\end{equation}
\end{lemma}
\begin{proof}
The case of $p \neq 2$ is clear. 
Thus, we give the proof only for the case of $p = 2$. 

Changing the order of $M_{n}^{(1)}$ 
and $M_{n}^{(2)}$ if necessary, 
we may assume that 
$\limsup_{n \to \infty} \nu_{n} \geq 1$. 
It follows from \eqref{eqU-1}, $p = 2$ and 
$\alpha_{n}^{(2)} = \omega_{n} 
(M_{n}^{(2)})^{-4} = \omega_{n} 
(\nu_{n}M_{n}^{(1)})^{-4} = \nu_{n}^{-4} \alpha_{n}^{(1)}$ that 
\begin{equation}\label{eqU-2-2}
\begin{split}
1 
= \lim_{n \to \infty} 
\nu_{n} \frac{|\log \alpha_{n}^{(1)}|}
{|\log \alpha_{n}^{(2)}|} 
= \lim_{n \to \infty} \nu_{n} 
\biggl|\dfrac{\log \alpha_{n}^{(1)}}{-4 \log \nu_{n} 
+ \log \alpha_{n}^{(1)}} \biggl|. 
\end{split}
\end{equation}
We claim that up to subsequence, one has 
\begin{equation} \label{eqU-2-3}
\lim_{n \to \infty} \frac{\log \nu_{n}}{(- \log \alpha_{n}^{(1)})} 
= 0. 
\end{equation}
If \eqref{eqU-2-3} holds, we have by \eqref{eqU-2-2} that 
\begin{equation*} 
1 = \lim_{n \to \infty} \nu_{n} 
\biggl|\dfrac{\log \alpha_{n}^{(1)}}{-4 \log \nu_{n} 
+ \log \alpha_{n}^{(1)}} \biggl|
= \lim_{n \to \infty} \nu_{n} 
\biggl|\dfrac{1}{4 \frac{\log \nu_{n}}{(- \log \alpha_{n}^{(1)})} 
+ 1} \biggl| = \lim_{n \to \infty} \nu_{n}. 
\end{equation*}
Thus, we see that \eqref{eqU-2} holds. 
Suppose to the contrary that \eqref{eqU-2-3} does not hold. 
Then, since $\limsup_{n \to \infty} \nu_{n} \geq 1$, up to subsequence, 
we see that one of the following occurs 
\[
\limsup_{n \to \infty} \frac{\log \nu_{n}}{(- \log \alpha_{n}^{(1)})} 
= \kappa \quad \mbox{for some $\kappa \in (0, \infty]$}. 
\]
Since $\lim_{n \to \infty} (- \log \alpha_{n}^{(1)}) 
= \infty$, 
we obtain $\lim_{n \to \infty} \nu_{n} = \infty$, 
which implies that $\lim_{n \to \infty} \frac{\nu_{n}}{\log \nu_{n}} 
= \infty$. 
Then, up to subsequence, 
we see from \eqref{eqU-2-2} that 
\[
1 = \lim_{n \to \infty} \frac{\nu_{n}}{\log \nu_{n}} 
\biggl|\dfrac{\log \alpha_{n}^{(1)}}
{4 + \frac{( - \log \alpha_{n}^{(1)})}{\log \nu_{n}}} \biggl| 
= \lim_{n \to \infty} \frac{\nu_{n}}{\log \nu_{n}} 
\frac{|\log \alpha_{n}^{(1)}|}{4 + \frac{1}{\kappa}} = \infty, 
\]
which is a contradiction. 
Thus, we see that \eqref{eqU-2-3} holds. 
This completes the proof. 
\end{proof}
We see from 
Theorem \ref{thm-bl-0}
and \eqref{eqU-2} that 
\begin{equation} \label{eqU-50}
\lim_{n \to \infty} 
\|\widetilde{u}_{n}^{(1)} - W\|_{\dot{H}^{1}} 
= \lim_{n \to \infty} 
\|
\nu_{n} \widetilde{u}_{n}^{(2)}
(\nu_{n}^{2} \cdot) - W\|_{\dot{H}^{1}} = 0. 
\end{equation}
Next, we define 
\begin{equation*}
\Psi_{n}(x) 
= u_{n}^{(1)} ((M_{n}^{(1)})^{-2} x) - 
u_{n}^{(2)}((M_{n}^{(2)})^{-2} x) 
= M_{n}^{(1)} 
\left\{\widetilde{u}_{n}^{(1)}(x) - 
\nu_{n} \widetilde{u}_{n}^{(2)} 
(\nu_{n}^{2} x)
\right\} 
\end{equation*}
and 
\begin{equation*}
\widetilde{z}_{n}(x) 
:= \frac{\Psi_{n}(x)}{
\|\nabla \Psi_{n}\|_{L^{2}}}. 
\end{equation*}
Since $\widetilde{u}_{n}^{(1)}$ 
and $\nu_{n} \widetilde{u}_{n}^{(2)}(\nu_{n}^{2} x)$ are 
positive solutions to the following equation: 
\begin{equation*} 
- \Delta \widetilde{u} + \alpha_{n}^{(1)} \widetilde{u}
= \beta_{n}^{(1)} \widetilde{u}^{p} + \widetilde{u}^{5}, 
\end{equation*} 
we can verify that $\widetilde{z}_{n}$ satisfies 
\begin{equation} \label{eqU-3}
- \Delta \widetilde{z}_{n} 
+ \alpha_{n}^{(1)} \widetilde{z}_{n} 
= p \beta_{n}^{(1)} \int_{0}^{1} 
V_{n}^{p-1}(x, \theta) d \theta 
\widetilde{z}_{n} 
+ 5 \int_{0}^{1} V_{n}^{4} (x, \theta) 
d\theta \widetilde{z}_{n}, 
\end{equation}
where 
\begin{equation} \label{eqU-4}
V_{n}(x, \theta) := 
\theta \widetilde{u}_{n}^{(1)}(x) 
+ (1 - \theta) 
\nu_{n} \widetilde{u}_{n}^{(2)}(
\nu_{n}^{2} x). 
\end{equation}
We prepare the following lemma, which is 
needed later. 
\begin{lemma}[Lemma 4.3 of \cite{MR3964275}]\label{lem-uni2}
Let $1 < p < 3$. 
Then, there exists $C_{1}>0$ such that for any 
$x \in \R^{3}$ and $n\in \mathbb{N}$,
\begin{equation} \label{eqU-4-1}
|\widetilde{z}_{n}(x)| \leq C_{1} |x|^{-1}. 
\end{equation} 
\end{lemma}
\begin{remark}
Although we assumed the condition $d \geq 5$ in 
\cite[Lemma 4.3]{MR3964275}, 
the proof still works for $d = 3$. 
\end{remark}
Since $\|\nabla \widetilde{z}_{n}\|_{L^{2}} = 1$ 
for all $n \in \mathbb{N}$, we find that 
there exists a subsequence of 
$\{\widetilde{z}_{n}\}$ 
(we still denote it by the same letter) 
and a function $\widetilde{z}_{\infty} 
\in \dot{H}^{1}_{\text{rad}}(\R^{3})$ 
such that 
\[
\lim_{n \to \infty} 
\widetilde{z}_{n} = 
\widetilde{z}_{\infty} 
\qquad 
\mbox{weakly in $\dot{H}^{1}
(\R^{3})$ and almost everywhere 
in $\R^{3}$}. 
\]
By the elliptic regularity, we can see that 
\[
\lim_{n \to \infty} \widetilde{z}_{n} = \widetilde{z}_{\infty} 
\qquad \mbox{in $C_{\text{loc}}^{2}(\R^{3})$}. 
\] 
Here, we shall drive a contradiction dividing 
into 2 cases (case of $2 \leq p < 3$ and case of $1 < p < 2$).  

\subsection{Case of $2 \leq p < 3$}
In this subsection, we consider the case of $2 \leq p < 3$. 
We first obtain the following:
\begin{lemma}
\label{lem-uni7}
Let $2 \leq p < 3$. 
$\widetilde{z}_{\infty} \neq 0$. 
\end{lemma}
\begin{proof}[Proof of Lemma \ref{lem-uni7}]
Suppose to the contrary that 
$\widetilde{z}_{\infty} = 0$. 
Multiplying \eqref{eqU-3} by $\widetilde{z}_{n}$ and 
integrating the resulting equation, we have 
\begin{equation} \label{eqU-7}
\|\nabla \widetilde{z}_{n}\|_{L^{2}}^{2} 
+ \alpha_{n}^{(1)} \|\widetilde{z}_{n}\|_{L^{2}}^{2} 
= p \beta_{n}^{(1)} \int_{\R^{3}} \int_{0}^{1} 
V_{n}^{p-1}(x, \theta) d \theta 
\widetilde{z}_{n}^{2} \, dx 
+ 5 \int_{\R^{3}} \int_{0}^{1} V_{n}^{4} (x, \theta) 
d\theta \widetilde{z}_{n}^{2} \, dx. 
\end{equation}
We claim that 
\begin{equation} \label{eqU-11}
\lim_{n \to \infty} 
\beta_{n}^{(1)} 
\sup_{\theta \in [0, 1]}
\int_{\R^{3}} \int_{0}^{1} 
V_{n}^{p-1} (x, \theta) d \theta 
\widetilde{z}_{n}^{2} \, dx = 0 
\end{equation}
for $2 \leq p < 3$. 
First, we will show that 
\begin{equation}\label{eqU-8}
\lim_{n \to \infty} \sup_{\theta \in [0, 1]} 
\int_{\R^{3}} \int_{0}^{1} 
V_{n}^{q -1}(x, \theta) d \theta 
\; \widetilde{z}_{n}^{2} \, dx = 0. 
\end{equation}
for any $q > 2$. 
It follows from \eqref{eqU-4} that for each $q \geq 2, x\in \R^{3}$, 
\begin{equation}\label{eqU-56}
\sup_{\theta \in [0, 1]} 
|V_{n}^{q - 1}(x, \theta)| \leq 
C_{q} \left( 
(\widetilde{u}_{n}^{(1)}(x))^{q - 1} 
+ 
(\widetilde{u}_{n}^{(2)}(x))^{q - 1} \right), 
\end{equation}
where $C_{q} > 0$ is a constant 
depending only on $q > 2$. 
For any $R > 0$, 
using \eqref{eqU-56}, \eqref{EqL-1}, 
Lemma \ref{lem-uni2} and $q > 2$, we have 
\begin{equation} \label{eqU-9} 
\begin{split}
\sup_{x \in [0, 1]} \int_{|x| \geq R} \int_{0}^{1} 
V_{n}^{q - 1}(x, \theta) d \theta 
\; \widetilde{z}_{n}^{2} \, dx 
\lesssim
\int_{|x|\geq R} 
\left((\widetilde{u}_{n}^{(1)})^{q - 1}
+ (\widetilde{u}_{n}^{(2)})^{q - 1}
\right) 
\widetilde{z}_{n}^{2} \, dx 
& \lesssim 
\int_{R}^{\infty} r^{-q + 1} 
r^{-2} r^{2} \, dr \lesssim R^{- q + 2}. 
\end{split}
\end{equation}
Since $\lim_{n \to \infty} \widetilde{z}_{n} 
= 0$ 
weakly in $\dot{H}^{1}(\R^{3})$, 
we see from $\|\widetilde{u}_{n}^{(i)}\|
_{L^{\infty}} = 1$, the H\"{o}lder inequality 
and the Rellich--Kondrachov theorem 
that for any sufficiently small $\e > 0$, 
one has 
\begin{equation} \label{eqU-10} 
\begin{split}
\sup_{\theta \in [0, 1]} 
\int_{|x| < R} \int_{0}^{1} 
V_{n}^{q - 1}(x, \theta) d \theta 
\; \widetilde{z}_{n}^{2} \, dx 
& \lesssim 
\int_{|x| < R} \left(
(\widetilde{u}_{n}^{(1)})^{q - 1}
+ (\widetilde{u}_{n}^{(2)})^{q - 1}
\right)
\; \widetilde{z}_{n}^{2} \, dx \\
& \lesssim 
\left(\|\widetilde{u}_{n}^{(1)}
\|_{L^{\frac{3 - \e}{2 - \e}(q - 1)}
(B_{R})}^{q - 1} 
+ \|\widetilde{u}_{n}^{(2)}
\|_{L^{\frac{3 - \e}{2 - \e}(q - 1)}
(B_{R})}^{q - 1}
\right) \|\widetilde{z}_{n}\|_{L^{6 - 2 \e}
(B_{R})}^{2} \\
& \to 0 \qquad \mbox{as $n \to \infty$}. 
\end{split}
\end{equation}
From \eqref{eqU-9} and 
\eqref{eqU-10}, 
we deduce that \eqref{eqU-8} holds 
for $q > 2$. 

Next, we obtain the upper bound of 
$\widetilde{u}_{n}^{(i)}$. 
We take $\e>0$ arbitrary and fix it. 
Observe from \eqref{EqL-1}, 
\eqref{EqB-1}, \eqref{EqB-2} and $2 
\leq p < 3$ that
    \[
\frac{\beta_{n}^{(i)}}{\alpha_{n}^{(i)}}      
(\widetilde{u}_{n}^{(i)})^{p-1}(\e 
(\alpha_{n}^{(i)})^{- \frac{1}{2}}) 
+ \frac{1}{\alpha_{n}^{(i)}}
(\widetilde{u}_{n}^{(i)})^{4}(\e 
(\alpha_{n}^{(i)})^{- \frac{1}{2}}) 
\leq \frac{\sqrt{\alpha_{n}^{(i)}}}
{|\log \alpha_{n}^{(i)}|} 
(\alpha_{n}^{(i)})^{-1} \e^{- (p-1)} 
(\alpha_{n}^{(i)})^{\frac{p-1}{2}}
+ \e^{-4} \alpha_{n}^{(i)} 
< \e
    \]
Then, applying Lemma \ref{nd-lem9-1}, 
we obtain 
    \begin{equation} \label{eqU-6-3}
    \widetilde{u}_{n}^{(i)}(r) 
    \leq \frac{\widetilde{u}_{n}^{(i)} 
    (\e (\alpha_{n}^{(i)})^{- \frac{1}{2}})}
    {\widetilde{Y}_{\omega_{n}, \e, i}(\e (\alpha_{n}^{(i)})^{- \frac{1}{2}})}
    \widetilde{Y}_{\omega_{n}, \e, i}(r)
    \leq C e^{\e \sqrt{1 - \e}} 
    \frac{e^{- 
    \sqrt{(1 - \e) \alpha_{n}^{(i)}} r}}{r} 
    \qquad \mbox{for $r \geq 
    \e (\alpha_{n}^{(i)})^{-\frac{1}{2}}$},  
    \end{equation}
where 
    \[
    \widetilde{Y}_{\omega, \e, i}(r) 
    = \frac{e^{- (1 - \e) 
    \sqrt{\alpha_{n}^{(i)}} r}}{r}. 
    \]
    
Now, we pay our attention to the case of 
$p = 2$.   
Using 
$\|\widetilde{u}_{n}^{(i)}\|_{L^{\infty}} = 1$, 
Lemma
\ref{lem-uni2} and 
\eqref{EqL-1} for 
$1 \leq |x| \leq (\alpha_{n}^{(1)})^{-\frac{1}{2}}$, 
\eqref{eqU-6-3} 
for $|x| \geq (\alpha_{n}^{(1)})^{-\frac{1}{2}}$ and 
$p = 2$, we have 
\begin{equation} \label{eqU-56-2} 
\begin{split}
\biggl| 
\int_{\R^{3}} (\widetilde{u}_{n}^{(i)}
)^{p - 1} 
\widetilde{z}_{n}^{2} \, dx 
\biggl| 
& = 
\int_{|x|\leq 1} 
(\widetilde{u}_{n}^{(i)})^{p - 1} 
\widetilde{z}_{n}^{2} \, dx + 
\int_{1 \leq |x|\leq (\alpha_{n}^{(1)})^{-\frac{1}{2}}} 
(\widetilde{u}_{n}^{(i)})^{p - 1} 
\widetilde{z}_{n}^{2} \, dx 
+ \int_{|x|\geq (\alpha_{n}^{(1)})^{-\frac{1}{2}}} 
(\widetilde{u}_{n}^{(i)})^{p - 1} 
\widetilde{z}_{n}^{2} \, dx \\[6pt]
& \lesssim 
\int_{0}^{1} \, dr
+ \int_{1}^{(\alpha_{n}^{(1)})^{-\frac{1}{2}}} r^{-(p-1)} 
r^{-2} r^{2} \, dr 
+ \int_{(\alpha_{n}^{(1)})^{-\frac{1}{2}}}^{\infty} 
e^{- (1 - \e) (p-1)\sqrt{\alpha_{n}^{(1)}} r}r^{- (p-1)} 
r^{-2} r^{2} \, dr. 
\\[6pt]
& \lesssim 1 + 
\int_{1}^{(\alpha_{n}^{(1)})^{-\frac{1}{2}}} 
r^{-1} \, dr + \int_{(\alpha_{n}^{(1)})^{-\frac{1}{2}}}^{\infty} 
e^{- (1 - \e) (p-1)\sqrt{\alpha_{n}^{(1)}} r}r^{- 1} \, dr
\\[6pt]
& \lesssim 
1 + |\log \alpha_{n}^{(1)}| + \int_{1}^{\infty} 
e^{- (1 - \e) (p-1) s} s^{- 1} \, ds \\[6pt] 
& \lesssim |\log \alpha_{n}^{(1)}|. 
\end{split}
\end{equation}
Then, by \eqref{EqB-2}, \eqref{eqU-56} and 
\eqref{eqU-56-2}, we obtain 
\begin{equation} \label{eqU-57}
\beta_{n}^{(1)} \int_{\R^{3}} \int_{0}^{1} 
V_{n}^{p-1}(x, \theta) d \theta 
\widetilde{z}_{n}^{2} \, dx 
\lesssim 
\frac{(\alpha_{n}^{(1)})^{\frac{1}{2}}}
{|\log \alpha_{n}^{(1)}|}
|\log \alpha_{n}^{(1)}|
= (\alpha_{n}^{(1)})^{\frac{1}{2}} \to 0 
\qquad \mbox{as $n \to \infty$} 
\end{equation}
when $p = 2$. 
We see from \eqref{eqU-8} and 
\eqref{eqU-57} that 
\eqref{eqU-11} holds when 
$2 \leq p < 3$. 

We now derive a conclusion. 
It follows from $\|\nabla 
\widetilde{z}_{n}\|_{L^{2}} = 1$, \eqref{eqU-7}, \eqref{eqU-11} and 
\eqref{eqU-8} with $q = 5$ that 
\[
\begin{split}
1 
& \leq \|\nabla \widetilde{z}_{n}\|_{L^{2}}^{2} 
+ \alpha_{n}^{(1)} \|\widetilde{z}_{n}\|_{L^{2}}^{2} \\[6pt]
& 
= p \beta_{n}^{(1)} \int_{\R^{3}} \int_{0}^{1} 
V_{n}^{p-1}(x, \theta) d \theta 
\widetilde{z}_{n}^{2} \, dx 
+ 5 \int_{\R^{3}} \int_{0}^{1} V_{n}^{4} (x, \theta) 
d\theta \widetilde{z}_{n}^{2} \, dx 
\to 0 \qquad \mbox{as $n \to \infty$}, 
\end{split}
\] 
which is a contradiction. 
Thus, we see that $\widetilde{z}_{\infty} \neq 0$. 
\end{proof}
For any $\phi \in C_{0}^{\infty}
(\R^{3})$, we see from \eqref{eqU-3}, 
$\|\nabla \widetilde{z}_{n}\|_{L^{2}} = 1$ 
and \eqref{main-eq3} that 
\begin{equation*}
\begin{split}
& \quad 
\langle (-\Delta - 5 W^{4}) 
\widetilde{z}_{\infty}, \phi \rangle \\[6pt]
& = 
\langle (-\Delta - 5 W^{4}) 
\widetilde{z}_{\infty}, \phi \rangle 
- \biggl\langle
\left(- \Delta + \alpha_{n}^{(1)} 
- \beta_{n}^{(1)} \int_{0}^{1} 
V_{n}^{p-1}(x, \theta) d \theta 
- 5 \int_{0}^{1} V_{n}^{4} (x, \theta) 
d\theta\right) 
\widetilde{z}_{n}, \phi
\biggl\rangle \\[6pt]
& = \langle (-\Delta - 5 W^{4}) 
\widetilde{z}_{\infty}, \phi \rangle
- \langle
(- \Delta - 5 W^{4}) \widetilde{z}_{n}, \phi
\rangle \\[6pt]
& \quad 
+ 5 \langle (W^{4} - \int_{0}^{1} V_{n}^{4} (x, \theta) d \theta) 
\widetilde{z}_{n}, \phi \rangle
+ \langle 
(\alpha_{\omega_{n}} 
- \beta_{n}^{(1)} \int_{0}^{1} 
V_{n}^{p-1}(x, \theta) d \theta) \widetilde{z}_{n}, \phi
\rangle \\[6pt]
& \to 0 \qquad \mbox{as $n \to \infty$}. 
\end{split}
\end{equation*}
This implies that 
\[ 
(-\Delta - 5 W^{4}) 
\widetilde{z}_{\infty} = 0. 
\]
This together with
$\widetilde{z}_{\infty} \neq 0$ yields that 
$\widetilde{z}_{\infty} = 
\kappa \Lambda W$ for some 
$\kappa \in \R \setminus \{0\}$. 
Here, $\Lambda W$ is the function 
defined by \eqref{EqB-9}. 
Replacing $\widetilde{z}_{n}$ by $-\widetilde{z}_{n}$ if 
necessary, we may assume that $\kappa > 0$.
Similarly to \eqref{main-eq7}, 
we can obtain 
\begin{equation}\label{eqU-12}
\omega \|u_{n}^{(i)}\|_{L^{2}}^{2} 
= \frac{5 - p}{2 (p+1)} 
\|u_{n}^{(i)}\|_{L^{p+1}}^{p+1} 
\qquad \mbox{for $i = 1, 2$}. 
\end{equation}
It follows from \eqref{eqU-12} that 
\begin{equation} \label{eqU-13}
\frac{M_{n}^{(1)}}{\|\nabla \Psi_{n}\|_{L^{2}}} 
\omega_{n} (\|u_{n}^{(1)}\|_{L^{2}}^{2} - \|u_{n}^{(2)}\|_{L^{2}}^{2}) 
= \frac{5 - p}{2(p+1)}
\frac{M_{n}^{(1)}}{\|\nabla \Psi_{n}\|_{L^{2}}} 
(\|u_{n}^{(1)}\|_{L^{p+1}}^{p+1} 
- \|u_{n}^{(2)}\|_{L^{p+1}}^{p+1})
\end{equation} 
For the left-hand side of \eqref{eqU-13}, we observe that 
\begin{equation*} 
\begin{split}
\frac{M_{n}^{(1)}}{\|\nabla \Psi_{n}\|_{L^{2}} } 
\omega_{n}(\|u_{n}^{(1)}\|_{L^{2}}^{2} - \|u_{n}^{(2)}\|_{L^{2}}^{2}) 
& = M_{n}^{(1)}\omega_{n} 
\int_{\R^{3}} \frac{u_{n}^{(1)}(x) - u_{n}^{(2)}(x)}{\|\nabla \Psi_{n}\|_{L^{2}}}
(u_{n}^{(1)}(x) + u_{n}^{(2)}(x)) \, dx 
\\[6pt]
& = \alpha_{n}^{(1)}
\langle \widetilde{u}_{n}^{(1)} + 
\nu_{n} \widetilde{u}_{n}^{(2)}
(\nu_{n}^{2} \cdot), \widetilde{z}_{n} \rangle. 
\end{split}
\end{equation*}
Similarly, for the right-hand side of \eqref{eqU-13}, we obtain 
\begin{equation} \label{eqU-14}
\begin{split}
\frac{5 - p}{2(p+1)} 
\frac{M_{n}^{(1)}}{\|\nabla \Psi_{n}\|_{L^{2}}} 
(\|u_{n}^{(1)}\|_{L^{p+1}}^{p+1} 
- \|u_{n}^{(2)}\|_{L^{p+1}}^{p+1}) 
& = \frac{5 - p}{2} M_{n}^{(1)} 
\int_{\R^{3}} 
\left(\int_{0}^{1} 
\left[\theta u_{n}^{(1)} + (1 - \theta) u_{n}^{(2)}\right]^{p} d\theta \right)
\frac{u_{n}^{(1)} - u_{n}^{(2)}}{\|\nabla \Psi_{n}\|_{L^{2}}} \, dx \\[6pt]
& = 
\frac{5 - p}{2}
\beta_{n}^{(1)} 
\left\langle \int_{0}^{1}
V_{n}^{p}(x, \theta) d\theta, 
\widetilde{z}_{n} \right\rangle,  
\end{split}
\end{equation}
where $V_{n}$ is the function defined by \eqref{eqU-4}. 
These imply that 
\begin{equation}\label{eqU-51}
\alpha_{n}^{(1)}
\langle \widetilde{u}_{n}^{(1)} + 
\nu_{n} \widetilde{u}_{n}^{(2)}
(\nu_{n}^{2} \cdot), \widetilde{z}_{n} 
\rangle 
= 
\frac{5 - p}{2}
\beta_{n}^{(1)} 
\left\langle \int_{0}^{1}
V_{n}^{p}(x, \theta) d\theta, 
\widetilde{z}_{n} \right\rangle.
\end{equation}
Here, we claim the following:
\begin{proposition} \label{propR-2} 
For $2 \leq p < 3$,  
we have the following: 
\begin{equation} \label{eqU-15}
\liminf_{n \to \infty} 
\sqrt{\alpha_{n}^{(1)}} 
\langle \widetilde{u}_{n}^{(1)}(x) + 
\nu_{n} \widetilde{u}_{n}^{(2)}
(\nu_{n}^{2} \cdot), \widetilde{z}_{n} \rangle
\geq - 6 \pi \kappa. 
\end{equation}
\end{proposition}
\begin{proposition}\label{propR-3} 
For $2 \leq p < 3$. 
The following  holds: 
\begin{equation} \label{eqU-16}
\limsup_{n \to \infty}
\frac{\beta_{n}^{(1)}}{\sqrt{\alpha_{n}^{(1)}}} 
\left\langle \int_{0}^{1} 
V_{n}^{p}(x, \theta) d\theta, 
\widetilde{z}_{n} \right\rangle 
= - 6 \pi \kappa. 
\end{equation}
\end{proposition}
Let us admit Propositions \ref{propR-2} and 
\ref{propR-3} for a moment. 
Then, we can give the proof of Theorem 
\ref{thm-bl} \textrm{(i)} 
in the case of $2 \leq p < 3$. 
\begin{proof}[Proof of Theorem 
\ref{thm-bl} \textrm{(i)} 
in the case of $2 \leq p < 3$]
It follows from \eqref{eqU-51}--\eqref{eqU-16} that
\begin{equation} \label{eqU-17}
\begin{split}
- 3 \pi \kappa (5 - p) 
& = 
\frac{5 - p}{2} \limsup_{n \to \infty}
\frac{\beta_{n}^{(1)}}{\sqrt{\alpha_{n}^{(1)}}} 
\left\langle \int_{0}^{1} 
V_{n}^{p}(x, \theta) d\theta, 
\widetilde{z}_{n} \right\rangle \\[6pt]
& = \liminf_{n \to \infty} 
\sqrt{\alpha_{n}^{(1)}} 
\langle \widetilde{u}_{n}^{(1)} + 
\nu_{n} \widetilde{u}_{n}^{(2)}
(\nu_{n}^{2} \cdot), \widetilde{z}_{n} \rangle 
\geq - 6 \pi \kappa. 
\end{split}
\end{equation}
It follows from \eqref{eqU-17} 
and $\kappa > 0$ that 
$p \geq 3$, which contradicts 
the condition $p < 3$. 
Thus, we conclude that 
the solution $u_{\omega}$ to \eqref{sp} 
satisfying $\lim_{\omega \to 0} \|u_{\omega}\|_{L^{\infty}} = \infty$ 
is unique for sufficiently 
small $\omega > 0$. 
\end{proof}
\begin{remark}
Note that $W \in L^{2}(\R^{d})$ when $d \geq 5$. 
Thus, in \cite{MR3964275}, 
it was shown that 
\[
\lim_{n \to \infty} \widetilde{u}_{n}^{(i)} = W\; (i = 1, 2), 
\qquad 
\mbox{strongly in $L^{2}(\R^{d})$} 
\] 
when $d \geq 5$. 
Namely, we have the $L^{2}$-convergence. 
Then, once we obtain the identity \eqref{eqU-13}, 
we can derive a contradiction just by taking a limit in $n \in \N$. 
On the other hand, since $W, \Lambda W \not\in L^{2}(\R^{d})$ 
when $d = 3$, 
we cannot do the same way. 
To overcome the difficulty, we use the resolvent estimate 
of Lemma \ref{LemR-16}, which comes from the resolvent 
expansion by Jensen and Kato~\cite{MR544248} 
(see Lemma \ref{LemR-2}). 
\end{remark}

\subsubsection{Proof of Proposition \ref{propR-2}}
As we see before, 
we need to obtain Proposition \ref{propR-2} 
in order to show Theorem \ref{thm-bl} \textrm{(i)}. 
Here, we shall give 
the proof of Proposition \ref{propR-2}. 
To this end, we put 
\[
U_{n} := 
\frac{\widetilde{u}_{n}^{(1)} + 
\nu_{n} \widetilde{u}_{n}^{(2)}
(\nu_{n}^{2} \cdot)}{2}.
\]
We have 
\[
\sqrt{\alpha_{n}^{(1)}} 
\langle U_{n}, \widetilde{z}_{n} \rangle
=
\frac{\sqrt{\alpha_{n}}}{2} 
\langle \widetilde{u}_{n}^{(1)}(x) + 
\nu_{n} \widetilde{u}_{n}^{(2)}
(\nu_{n}^{2} \cdot), \widetilde{z}_{n} \rangle
\]
Then, \eqref{eqU-15} can be written by 
\begin{equation} \label{eqU-18}
\liminf_{n \to \infty} 
\sqrt{\alpha_{n}^{(1)}} 
\langle U_{n}, \widetilde{z}_{n} \rangle 
\geq - 3 \pi \kappa. 
\end{equation}
Thus, it suffices to show \eqref{eqU-18} holds. 
In order to prove \eqref{eqU-18}, 
we use the resolvent estimate \eqref{EqB-45} of Lemma \ref{LemR-16}, 
which comes from the resolvent expansion by 
Jensen and Kato~\cite{MR544248}. 
Based on the resolvent estimate \eqref{EqB-45}, 
we decompose $U_{n}$ as follows: 
\begin{equation} \label{eqU-19}
U_{n} = 
C_{n} W^{4} \Lambda W + g_{n}, 
\end{equation}
where 
\[
C_{n} = \frac{\langle 
(- \Delta + 
\alpha_{n}^{(1)})^{-1}
U_{n}, 
W^{4} \Lambda W \rangle}{
\langle (- \Delta + \alpha_{n}^{(1)})^{-1} 
W^{4} \Lambda W, W^{4} 
\Lambda W 
\rangle}.
\] 
Observe that
\begin{equation} \label{eqU-20}
\sqrt{\alpha_{n}^{(1)}} 
\langle U_{n}, \widetilde{z}_{n} \rangle 
= \sqrt{\alpha_{n}^{(1)}} 
C_{n} \langle W^{4} \Lambda W, \widetilde{z}_{n} \rangle 
+ \sqrt{\alpha_{n}^{(1)}} 
\langle g_{n}, \widetilde{z}_{n} \rangle.
\end{equation}
From the decomposition \eqref{eqU-19}, we have 
\begin{equation} \label{eqU-21}
\langle (- \Delta + \alpha_{n}^{(1)})^{-1} 
g_n, W^{4}
\Lambda W \rangle = 0, 
\end{equation}
which implies that we can apply the resolvent estimate 
(Lemma \ref{LemR-16}) for $ (- \Delta + \alpha_{n}^{(1)})^{-1} g_{n}$. 
This enables us to estimate the second term of \eqref{eqU-20}. 
More precisely, we will show the following lemma: 
\begin{lemma} \label{lem-uni3}
Let $2 \leq p < 3$. 
One has 
\begin{equation} \label{eqU-22} 
\lim_{n \to \infty} 
\sqrt{\alpha_{n}^{(1)}} \langle g_{n}, 
\widetilde{z}_{n} \rangle = 0. 
\end{equation}
\end{lemma}
To prove Lemma \ref{lem-uni3}, we need the following lemma: 
\begin{lemma}[Claim 4 of \cite{MR4572464}]
\label{lem-uni4}
\begin{equation}\label{eqU-23}
\lim_{n \to \infty} 
\langle (- \Delta + \alpha_{n}^{(1)})^{-1} 
W^{4} \Lambda W, 
W^{4} \Lambda W 
\rangle = 
\frac{1}{5}
\langle \Lambda W, 
W^{4} \Lambda W \rangle 
\qquad \mbox{as $n \to \infty$}.
\end{equation}
\end{lemma}
Using Lemmas \ref{LemR-14} and \ref{lem-uni4}, 
we give the proof of Lemma \ref{lem-uni3}. 
\begin{proof}[Proof of Lemma \ref{lem-uni3}]
Observe that \eqref{eqU-3} can be rewritten by 
\begin{equation*} 
(- \Delta + \alpha_{n}^{(1)} - 5 W^{4}) \widetilde{z}_{n} 
= R_{n} \widetilde{z}_{n}, 
\end{equation*}
where 
\begin{equation} \label{eqU-24}
R_{n}
:= p \beta_{n}^{(1)} \int_{0}^{1} 
V_{n}^{p-1}(x, \theta) d \theta 
+ 5 \left(\int_{0}^{1} V_{n}^{4} (x, \theta) 
d\theta - W^{4} \right). 
\end{equation}
Then, we have 
\begin{equation} \label{eqU-25}
\begin{split}
\widetilde{z}_{n}
& = 
(- \Delta + \alpha_{n}^{(1)} - 5 
W^{4})^{-1} R_{n} 
\widetilde{z}_{n} \\[6pt]
& = 
\left((1 - 5 
W^{4}(- \Delta + \alpha_{n}^{(1)})^{-1}) (- \Delta + \alpha_{n}^{(1)}) \right)^{-1} R_{n} 
\widetilde{z}_{n} \\[6pt]
& = (- \Delta + \alpha_{n}^{(1)})^{-1} 
(1 - 5 
W^{4}(- \Delta + \alpha_{n}^{(1)})^{-1})^{-1}
R_{n} 
\widetilde{z}_{n}. 
\end{split}
\end{equation}
This yields that 
\begin{equation} \label{eqU-26}
\begin{split}
\langle g_{n}, \widetilde{z}_{n} \rangle 
& = \langle g_{n}, 
(- \Delta + \alpha_{n}^{(1)})^{-1} 
(1 - 5 
W^{4}(- \Delta + \alpha_{n}^{(1)})^{-1})^{-1}
R_{n} 
\widetilde{z}_{n} \rangle \\[6pt]
& = 
\langle 
(1 - 5 (- \Delta + \alpha_{n}^{(1)})^{-1}
W^{4})^{-1}
(- \Delta + \alpha_{n}^{(1)})^{-1} 
g_{n}, 
R_{n}
\widetilde{z}_n \rangle. 
\end{split}
\end{equation}
Note that $(- \Delta + \alpha_{n}^{(1)})^{-1} 
g_{n}$ satisfies the orthogonal condition 
\eqref{eqU-21}. 
Then, applying Lemma \ref{LemR-16} with $f = (- \Delta + \alpha_{n}^{(1)})^{-1} 
g_{n}$, 
we obtain 
\begin{equation} \label{eqU-27}
\begin{split}
\langle g_{n}, \widetilde{z}_{n} \rangle 
& = \biggl| 
\langle 
(1 - 5 (- \Delta + \alpha_{n}^{(1)})^{-1}
W^{4})^{-1}
(- \Delta + \alpha_{n}^{(1)})^{-1} 
g_{n}, 
R_{n}
\widetilde{z}_n \rangle
\biggl| \\[6pt]
& \leq 
\|(1 - 5 (- \Delta + \alpha_{n}^{(1)})^{-1}
W^{4})^{-1}
(- \Delta + \alpha_{n}^{(1)})^{-1} 
g_{n}\|_{L^{\infty}} 
\|R_{n}
\widetilde{z}_n
\|_{L^{1}} \\[6pt]
& \lesssim 
\|(- \Delta + \alpha_{n}^{(1)})^{-1} 
g_{n}\|_{L^{\infty}} 
\|R_{n}
\widetilde{z}_n
\|_{L^{1}}. 
\end{split}
\end{equation}
We shall show that $
\lim_{n \to \infty}
\|R_{n} \widetilde{z}_n\|_{L^{1}} = 0$. 
It follows from 
$U_{n} = O(|x|^{-1})$
(see \eqref{EqL-1} and \eqref{eqU-19}), 
Lemmas \ref{lem2-5}, \ref{lem-uni4} and \ref{LemR-14} that 
\begin{equation} \label{eqU-28}
\begin{split} 
|C_{n}| \lesssim 
|\langle 
(- \Delta + \alpha_{n}^{(1)})^{-1} 
U_{n}, W^{4} \Lambda W
\rangle | 
& 
= 
\frac{1}{5} \|(- \Delta + \alpha_{n}^{(1)})^{-1} 
U_{n}\|_{L^{\infty}} 
\|W^{4} \Lambda W \|_{L^{1}} \\[6pt]
& 
\lesssim 
\|(- \Delta + \alpha_{n}^{(1)})^{-1} 
|x|^{-1}\|_{L^{\infty}} 
\|W^{4} \Lambda W \|_{L^{1}}
\lesssim 
(\alpha_{n}^{(1)})^{- \frac{1}{2}}. 
\end{split}
\end{equation}
Since $g_{n} = U_{n} - C_{n} W^{4} 
\Lambda V$, 
one has by \eqref{eqU-28}, Lemmas \ref{LemR-12}--\ref{LemR-14} and 
$U_{n} = O(|x|^{-1})$
(see \eqref{EqL-1}) that 

\begin{equation} \label{eqU-29}
\begin{split}
\|(- \Delta + \alpha_{n}^{(1)})^{-1} 
g_{n}\|_{L^{\infty}} 
& \leq |C_{n}| 
\|(- \Delta + \alpha_{n}^{(1)})^{-1} W^{4} 
\Lambda W\|_{L^{\infty}} 
+ 
\|(- \Delta + \alpha_{n}^{(1)})^{-1}
U_{n}\|
_{L^{\infty}} \\[6pt]
& \lesssim 
(\alpha_{n}^{(1)})^{- \frac{1}{2}} 
\|W^{4} \Lambda W\|
_{L^{\frac{3}{2} + \e} \cap L^{\frac{3}{2} - \e}} 
+ \|(- \Delta + \alpha_{n}^{(1)})^{-1} 
|x|^{-1}\|_{L^{\infty}}\\[6pt]
& \lesssim 
(\alpha_{n}^{(1)})^{- \frac{1}{2}} 
+ (\alpha_{n}^{(1)})^{- \frac{1}{2}} 
\lesssim (\alpha_{n}^{(1)})^{- \frac{1}{2}}. 
\end{split}
\end{equation}
Using 
\eqref{EqL-1} for $1 \leq r 
\leq (\alpha_{n}^{(1)})^{- \frac{1}{2}}$ 
and \eqref{eqU-6-3} for 
$r \geq (\alpha_{n}^{(1)})^{- \frac{1}{2}}$, 
we have by \eqref{eqU-56} that 
\begin{equation} \label{eqU-29-2}
\begin{split}
\sup_{\theta \in [0, 1]} 
\|V_{n}^{p-1}(\cdot, \theta) 
\widetilde{z}_{n}\|
_{L^{1}} 
& \lesssim \int_{0}^{1} \, dr 
+ \int_{1}^{(\alpha_{n}^{(1)})^{- \frac{1}{2}}} 
\frac{1}{r^{p-2}} \, dr + 
\int_{(\alpha_{n}^{(1)})^{- \frac{1}{2}}}^{\infty} 
\frac{e^{- (p-1)(1 - \e) \sqrt{\alpha_{1, n}} r}}{r^{p-2}} \, dr \\[6pt]
& \lesssim 1 + 
(\alpha_{n}^{(1)})^{\frac{p - 3}{2}} 
+ (\alpha_{n}^{(1)})^{\frac{p - 3}{2}} \int_{1}^{\infty} 
\frac{e^{- (p-1)(1 - \e)s}}{s^{p-2}} \, ds 
\lesssim (\alpha_{n}^{(1)})^{\frac{p - 3}{2}}. 
\end{split}
\end{equation}
We put $\eta_{n}(x, \theta) := V_{n}(x, \theta) - W$.
Then, it follows from \eqref{eqU-50} and \eqref{eqU-4} that 
\begin{equation} \label{eqU-30}
\lim_{n \to \infty} 
\sup_{\theta \in [0, 1]} 
\|\eta_{n}(\cdot, \theta)\|_{\dot{H}^{1} \cap L^{q}} 
= 0. 
\end{equation}
for any $q > 3$. 
This together with \eqref{eqU-24}, 
the H\"{o}lder inequality, 
\eqref{EqL-1} and \eqref{eqU-29-2} yields that 
\begin{equation} \label{eqU-31}
\begin{split}
\|R_{n} \widetilde{z}_{n}\|_{L^{1}} 
& \lesssim 
\sup_{\theta \in [0, 1]}\|(V_{n}^{3}(\cdot, \theta) + W^{3}) 
\eta_{n}(\cdot, \theta) 
\widetilde{z}_{n}\|_{L^{1}} 
+ \beta_{1, n}
\sup_{\theta \in [0, 1]}
\|V_{n}^{p-1}(\cdot, \theta) 
\widetilde{z}_{n}\|
_{L^{1}} \\[6pt]
& \lesssim 
\left(
\sup_{\theta \in [0, 1]}\|V_{n}(\cdot, \theta)
\|_{L^{\frac{9}{2}}}^{3} + 
\|W\|_{L^{\frac{9}{2}}}^{3} \right)
\|\widetilde{z}_{n}\|_{L^{6}} 
\sup_{\theta \in [0, 1]}\|\eta_{n}(\cdot, \theta)\|_{L^{6}} 
+ \beta_{1, n} \alpha_{1, n}^{\frac{p - 3}{2}} \\[6pt]
& \lesssim \sup_{\theta \in [0, 1]}\|\eta_{n}(\cdot, \theta)\|_{L^{6}} 
+ \beta_{1, n}\alpha_{1, n}^{\frac{p - 3}{2}}. 
\end{split}
\end{equation}
By \eqref{EqB-1} and \eqref{EqB-2}, 
we see that 
$\lim_{n \to \infty} \beta_{1, n}\alpha_{1, n}^{\frac{p - 3}{2}} = 0$ for $2 \leq p < 3$. 
This together with \eqref{eqU-30} and 
\eqref{eqU-31} yields that 
$\lim_{n \to \infty} \|R_{n} \widetilde{z}_{n}\|_{L^{1}} = 0$. 
Thus, by \eqref{eqU-27}, \eqref{eqU-28}, 
\eqref{eqU-29} and \eqref{eqU-31}, we obtain 
\[
\begin{split}
|(\alpha_{1, n})^{\frac{1}{2}} 
\langle g_{n}, \widetilde{z}_{n} 
\rangle| 
\lesssim (\alpha_{1, n})^{\frac{1}{2}} 
(\alpha_{1, n})^{- \frac{1}{2}}
\|R_{n} \widetilde{z}_{n}\|_{L^{1}}
= \|R_{n} \widetilde{z}_{n}\|_{L^{1}}
\lesssim \sup_{0 < \theta < 1}\|\eta_{n}(\cdot, \theta)\|_{L^{6}} 
+ \beta_{1, n}\alpha_{1, n}^{\frac{p - 3}{2}}
\to 0 \qquad 
\mbox{as $n \to \infty$}. 
\end{split}
\] 
Thus, we see that \eqref{eqU-22} holds. 
This completes the proof. 
\end{proof}

Next, we pay our attention to the first term 
of the right-hand side of \eqref{eqU-20}. 
Note that $W^{4} \Lambda W$ has a good decay, that is, 
$|W^{4} \Lambda W| \sim |x|^{-5}$ 
as $|x| \to \infty$. 
Using this and $\lim_{n \to \infty}\widetilde{z}_{n} = 
\kappa \Lambda W$ weakly in 
$H^{1}(\mathbb{R}^{d})$, we obtain 
\begin{equation}\label{eqU-32}
\lim_{n \to \infty} 
\langle W^{4} \Lambda W, 
\widetilde{z}_{n} \rangle = 
\kappa \langle W^{4} \Lambda W, 
\Lambda W \rangle. 
\end{equation}
In order to estimate of the coefficient $\sqrt{\alpha_{n}^{(1)}} C_{n}$, 
we will use the followings: 

\begin{lemma}\label{lem-uni5}
Let $2 \leq p < 3$. We have 
\begin{equation} \label{eqU-33}
\liminf_{n \to \infty}
\sqrt{\alpha_{n}^{(1)}}
\langle 
(- \Delta + \alpha_{n}^{(1)})^{-1}
U_{n}, 
W^{4} \Lambda W\rangle 
\geq 
- \frac{3 \pi}{5}. 
\end{equation}
\end{lemma}
\begin{proof}[Proof of Lemma \ref{lem-uni5}]
As in \eqref{EqB-40}, we obtain 
\begin{equation*}
\sqrt{\alpha_{n}^{(1)}}
\langle (- \Delta + \alpha_{n}^{(1)})^{-1}
U_{n}, 
W^{4} \Lambda W \rangle
= 
- \frac{\sqrt{3}}{10}
\sqrt{\alpha_{n}^{(1)}} 
\int_{\mathbb{R}^{3}} 
\frac{e^{- \sqrt{\alpha_{n}^{(1)}}|x|}}{|x|} 
U_{n} \, dx
+ O_{n}(\sqrt{\alpha_{n}^{(1)}}). 
\end{equation*} 
Using \eqref{EqL-1} for $0 \leq |x| 
\leq \e (\alpha_{n}^{(1)})^{- \frac{1}{2}}$ and 
\eqref{eqU-6-3} for $|x| \geq 
\e (\alpha_{n}^{(1)})^{- \frac{1}{2}}$, we obtain 
\begin{equation} \label{eqU-34}
\begin{split}
& \quad 
\liminf_{n \to \infty}
\sqrt{\alpha_{n}^{(1)}}
\langle 
(- \Delta + \alpha_{n}^{(1)})^{-1}
U_{n}, W^{4} \Lambda W\rangle 
\\[6pt]
& 
= 
- \frac{\sqrt{3}}{10} 
\liminf_{n \to \infty} 
\sqrt{\alpha_{n}^{(1)}}
\int_{\mathbb{R}^{3}} 
\frac{e^{- \sqrt{\alpha_{n}^{(1)}}|x|}}{|x|} 
U_{n} \, dx
\\[6pt]
& \geq 
- \frac{\sqrt{3}}{10} \times 
4\pi 
\limsup_{n \to \infty}
\sqrt{\alpha_{n}^{(1)}}
\left\{ 
2 \sqrt{3}
\int_{0}^{\e (\alpha_{n}^{(1)})^{- \frac{1}{2}}}
e^{- \sqrt{\alpha_{n}^{(1)}}r} \, dr + 
\sqrt{3} (1 + \e)
\int_{\e (\alpha_{n}^{(1)})^{- \frac{1}{2}}}^{\infty} 
e^{- (2 - \e) \sqrt{\alpha_{n}^{(1)}} r} \, dr
\right\}
\\[6pt]
& 
= - \frac{2 \sqrt{3} \pi}{5}
\left\{ 
2 \sqrt{3} (1 - e^{- \e}) 
+ \sqrt{3} \frac{1 + \e}{2 - \e} e^{- (2 - \e) \e}
\right\} \\[6pt]
& = - \frac{12 \pi}{5}(1 - e^{- \e}) 
- \frac{6 \pi}{5} \times \frac{1 + \e}{2 - \e} e^{- (2 - \e) \e} \\
& \geq \frac{6 \pi}{5} \e 
- \frac{6 \pi}{5} \times \frac{1 + \e}{2 - \e} e^{- (2 - \e) \e}. 
\end{split}
\end{equation}
for $0 < \e \ll 1$. 
Since $\e>0$ is arbitrary small, we have 
\[
\liminf_{n \to \infty} 
\sqrt{\alpha_{n}^{(1)}}
\langle (- \Delta + \alpha_{n}^{(1)})^{-1}
U_{n}, W^{4} \Lambda W\rangle \geq
- \frac{3 \pi}{5}. 
\]
This concludes the proof. 
\end{proof}
We are now in a position to prove Proposition \ref{propR-2}. 
\begin{proof}[Proof of Proposition \ref{propR-2}]
It follows from \eqref{eqU-20}, \eqref{eqU-22}, 
\eqref{eqU-23}, \eqref{eqU-32} and 
\eqref{eqU-33} that 
\[
\begin{split}
& \quad \liminf_{n \to \infty} 
\sqrt{\alpha_{n}^{(1)}} 
\langle U_{n}, \widetilde{z}_{n} \rangle \\[6pt]
& = \liminf_{n \to \infty} \sqrt{\alpha_{n}^{(1)}} 
C_{n} \langle W^{4} \Lambda W, \widetilde{z}_{n} \rangle \\[6pt]
& = \liminf_{n \to \infty} 
\sqrt{\alpha_{n}^{(1)}} \langle 
(- \Delta + \alpha_{n}^{(1)})^{-1}
U_{n}, 
W^{4} \Lambda W\rangle 
\times 
\frac{1}{\langle (- \Delta + \alpha_{n}^{(1)})^{-1} 
W^{4} \Lambda W, 
W^{4} \Lambda W 
\rangle}
\times 
\langle W^{4} \Lambda W, 
\widetilde{z}_{n} \rangle
\\[6pt]
& \geq- \frac{3 \pi}{5} \times \frac{5}{\langle \Lambda W, 
W^{4} \Lambda W \rangle} \times \kappa 
\langle W^{4} \Lambda W, 
\Lambda W \rangle = - 3 \pi \kappa, 
\end{split}
\]
which yields \eqref{eqU-18}.

\end{proof}


\subsubsection{Proof of Proposition \ref{propR-3}}
In this subsection, we shall prove Proposition \ref{propR-3}. 
We first give the proof in 
case of $2 < p < 3$.
\begin{proof}[Proof of 
Proposition \ref{propR-3} in the case of 
$2 < p < 3$]
Since 
\[
\lim_{n \to \infty} \int_{0}^{1} 
V_{n}^{p}(\cdot, \theta) d\theta = W^{p} 
\qquad \mbox{strongly in $L^{\frac{p+1}{p}}(\R^{3})$}, 
\]
$\lim_{n \to \infty} \widetilde{z}_{n} 
= \kappa \Lambda W$ weakly in 
$\dot{H}^{1}(\R^{3})$
and $\Lambda W = \frac{1}{2} W + x \cdot W$, we have 
\begin{equation} \label{eqU-36}
\begin{split}
\lim_{n \to \infty}
\left \langle \int_{0}^{1} 
V_{n}^{p}(\cdot, \theta) d\theta, \widetilde{z}_{n} \right\rangle 
= \int_{\R^{3}} W^{p}(x) \kappa \Lambda W(x) \, dx 
= - \kappa \frac{5 - p}{2(p + 1)} \|W\|_{L^{p+1}}^{p+1}. 
\end{split}
\end{equation}
By \eqref{EqB-1} and \eqref{eqU-36}, we have 
\begin{equation*}
\begin{split}
\lim_{n \to \infty} 
\frac{\beta_{n}^{(1)}}{ \sqrt{\alpha_{n}^{(1)}}} 
\left \langle \int_{0}^{1} 
V_{n}^{p}(\cdot, \theta) 
d\theta, \widetilde{z}_{n} \right\rangle 
= \frac{12 \pi(p+1)}{(5 - p)\|W\|_{L^{p+1}}^{p+1}} 
\times \left( - \kappa \frac{5 - p}{2(p + 1)} \|W\|_{L^{p+1}}^{p+1} \right) 
= - 6 \pi \kappa. 
\end{split}
\end{equation*}
Thus, we see that \eqref{eqU-16} holds. 
\end{proof}
Next, we consider the 
case of $p = 2$. 
We remark that 
since $W \notin L^{p+1}(\R^{3})$ when $p = 2$, 
\eqref{eqU-36} does not hold in this case. 
Thus, it becomes a subtle problem. 
To overcome the difficulty, we need the following lemma:
\begin{lemma}\label{lem-uni8}
Let $p = 2$. 
We have the following: 
\begin{equation} \label{eqU-35}
\begin{split}
\limsup_{n \to \infty}
|\log \alpha_{n}^{(1)}|^{-1}
\int_{\R^{3}} \left(\int_{0}^{1} 
V_{n}^{p}(x, \theta) d\theta \right) 
\widetilde{z}_{n}(x) \, dx 
= - 3 \sqrt{3} \pi \kappa
\end{split}
\end{equation}
\end{lemma}
We will give the proof of Lemma \ref{lem-uni8} later. 
Admitting it, we shall give the proof of Proposition 
\ref{propR-2}. 
\begin{proof}[Proof of Proposition \ref{propR-3} in the case of $p = 2$]
By \eqref{EqB-2} and \eqref{eqU-35}, we have 
\begin{equation*}
\begin{split}
\limsup_{n \to \infty} 
\frac{\beta_{n}^{(1)}}{ \sqrt{\alpha_{n}^{(1)}}} 
\int_{\R^{3}} 
\left(\int_{0}^{1} 
V_{n}^{p}(x, \theta) d\theta \right)
\widetilde{z}_{n}(x) \, dx 
& = \limsup_{n \to \infty} 
\frac{\beta_{n}^{(1)}|\log \alpha_{n}^{(1)}|}
{ \sqrt{\alpha_{n}^{(1)}}} 
|\log \alpha_{n}^{(1)}|^{-1}
\int_{\R^{3}} 
\left(\int_{0}^{1} 
V_{n}^{p}(x, \theta) d\theta \right)
\widetilde{z}_{n}(x) \, dx \\[6pt]
& = \frac{2}{\sqrt{3}} 
\times (- 3 \sqrt{3} \pi \kappa) 
= -6 \pi \kappa. 
\end{split}
\end{equation*}
Therefore, we find that \eqref{eqU-16} holds. 
\end{proof}
We now give the proof of Lemma \ref{lem-uni8}. 
As in the proof of Proposition \ref{propR-2}, 
based on Lemma \ref{LemR-16}, 
we decompose 
$\int_{0}^{1} V_{n}^{p}(x, \theta) d\theta$ as follows:
\begin{equation} \label{eqU-37-3}
\int_{0}^{1} V_{n}^{p}(x, \theta) d\theta 
= C_{n} W^{4}\Lambda W(x) + g_{n}(x),
\end{equation}
where 
\begin{equation} \label{eqU-44-3} 
C_{n}:= 
\frac{\langle (- \Delta + \alpha_{n})^{-1} 
\int_{0}^{1} V_{n}^{p}(\cdot, \theta) d\theta, 
W^{4} \Lambda W \rangle}
{\langle (- \Delta + \alpha_{n})^{-1} 
W^{4} \Lambda W, 
W^{4} \Lambda W \rangle}. 
\end{equation}
It follows from \eqref{eqU-37-3} that 
\begin{equation} \label{eqU-38} 
|\log \alpha_{n}^{(1)}|^{-1} 
\int_{\R^{3}} 
\int_{0}^{1} V_{n}^{p}(x, \theta) d\theta 
z_{n}(x) \, dx 
=C_{n} 
|\log \alpha_{n}^{(1)}|^{-1} 
\langle W^{4}\Lambda W, \widetilde{z}_{n} \rangle 
+ |\log \alpha_{n}^{(1)}|^{-1} 
\langle g_{n}, \widetilde{z}_{n} \rangle. 
\end{equation}
From the decomposition \eqref{eqU-37-3}, 
we can easily verify that 
\begin{equation*} 
\langle (- \Delta + \alpha_{n})^{-1} 
g_{n}, W^{4} \Lambda W \rangle = 0. 
\end{equation*}
As in Proposition \ref{LemR-6}, 
we obtain the following: 
\begin{lemma}\label{prop4-2-2}
\begin{equation}\label{eqU-48}
\liminf_{n \to \infty} 
|\log \alpha_{n}^{(1)}|^{-1} 
\langle (- \Delta + \alpha_{n})^{-1} 
\int_{0}^{1} V_{n}^{p}(\cdot, \theta) d\theta, 
W^{4} \Lambda W \rangle 
= - \frac{3\sqrt{3}}{5} \pi. 
\end{equation}
\end{lemma}
Since the proof of Lemma \ref{prop4-2-2} is 
almost same as that of Proposition \ref{LemR-6}, 
we omit the proof. 

Next, we shall show the following: 
\begin{lemma}\label{pro4-1-3}
Let $p = 2$. Then, we have 
\begin{equation} \label{eqU-46}
\|(- \Delta + \alpha_{n}^{(1)})^{-1} \int_{0}^{1} 
V_{n}^{p}(\cdot, \theta) d\theta\|_{L^{\infty}} 
\lesssim |\log \alpha_{n}^{(1)}|.
\end{equation}
\end{lemma}
\begin{proof}

We use the formula \eqref{EqB-37} and divide the integral as follows:
\begin{equation*}
\begin{split}
(- \Delta + \alpha_{n}^{(1)})^{-1} 
\int_{0}^{1} V_{n}^{p}(x, \theta) d\theta
& =
\sum_{i = 1}^{3}
\int_{A_{i}} \frac{e^{- \sqrt{\alpha_{n}^{(1)}} |x-y|}}{|x - y|} \int_{0}^{1} V_{n}^{p}(y, \theta) d\theta dy, 
\end{split}
\end{equation*}
where
\begin{equation*}
\begin{split}
& A_{1} := \left\{y \in \mathbb{R}^{3} \colon 
|x - y| \leq \frac{|x|}{2}\right\}, \\[6pt]
& A_{2} := \left\{y \in \mathbb{R}^{3} \colon 
|x - y| > \frac{|x|}{2}, \; |y| \leq 2|x| \right\}, \\[6pt]
& A_{3} := \left\{y \in \mathbb{R}^{3} \colon 
|x - y| > \frac{|x|}{2}, \; |y| > 2|x| \right\}. 
\end{split}
\end{equation*}
We first consider the region $A_{1}$. 
Observe from \eqref{EqL-1} and \eqref{eqU-4} that 
\begin{equation} \label{eqU-42}
\biggl|\int_{0}^{1} V_{n}^{p}(x, \theta) 
d\theta \biggl| \lesssim |x|^{-p}.
\end{equation} 
Note that 
$|y| \geq |x| - |x-y| \geq |x|/2$ for $y \in A_{1}$. 
This together with \eqref{eqU-42} 
and $p = 2$ implies that 
\begin{equation*}
\begin{split}
& \quad 
\int_{A_{1}} \frac{e^{- \sqrt{\alpha_{n}^{(1)}} |x-y|}}{|x - y|} \int_{0}^{1} V_{n}^{p}(y, \theta) d\theta dy \lesssim 
\int_{A_{1}} 
\frac{e^{- \sqrt{\alpha_{n}^{(1)}} |x-y|}}{|x - y||y|^{p}} dy 
= \int_{B(x, |x|/2)} 
\frac{e^{- \sqrt{\alpha_{n}^{(1)}} |x - y|}}{|x - y||y|^{p}}dy \\[6pt] 
& \lesssim \frac{1}{|x|^{p}} \int_{B(0, |x|/2)} 
\frac{e^{- \sqrt{\alpha_{n}^{(1)}}|y^{\prime}|}}{|y^{\prime}|} 
d y^{\prime} 
\lesssim \frac{1}{|x|^{p}} 
\int_{0}^{|x|/2} \frac{e^{- \sqrt{\alpha_{n}^{(1)}} r}}{r} r^{2} \, dr 
\lesssim \frac{1}{|x|^{p}} \int_{0}^{|x|/2} 
e^{- \sqrt{\alpha_{n}^{(1)}} r} r \, dr \\[6pt]
& \lesssim \frac{1}{|x|^{p}} \int_{0}^{|x|/2} r \, dr 
\lesssim 1. 
\end{split}
\end{equation*}
Observe that $|x - y| > |x|/2 $ for $y \in A_{2}$ 
and the function $\frac{e^{- \sqrt{\alpha_{n}^{(1)}} r}}{r}$ 
is decreasing in $r>0$. 
We have by \eqref{eqU-42} and $p = 2$ that 
\begin{equation*}
\begin{split}
& \quad 
\int_{A_{2}} 
\frac{e^{- \sqrt{\alpha_{n}^{(1)}} |x-y|}}{|x - y|}
\int_{0}^{1} V_{n}^{p}(y, \theta) d\theta dy \lesssim 
\int_{A_{2}} 
\frac{e^{- \sqrt{\alpha_{n}^{(1)}} |x-y|}}{|x - y||y|^{p}} dy
\lesssim \frac{e^{- \frac{\sqrt{\alpha_{n}^{(1)}} |x|}{2}}}{|x|}
\int_{A_{2}} 
\frac{1}{|y|^{p}} dy \\[6pt]
& 
\lesssim 
\frac{e^{- \frac{\sqrt{\alpha_{n}^{(1)}} |x|}{2}}}{|x|}
\int_{|y| \leq 2|x|} 
\frac{1}{|y|^{p}} dy
\lesssim \frac{e^{- \frac{\sqrt{\alpha_{n}^{(1)}} |x|}{2}}}{|x|} 
\int_{0}^{2|x|} \frac{r^{2}}{r^{p}} \, dr 
\lesssim e^{- \frac{\sqrt{\alpha_{n}^{(1)}}}{2} |x|} \lesssim 1. 
\end{split}
\end{equation*}
Finally, we shall estimate 
$\int_{A_{3}} \frac{e^{- \sqrt{\alpha_{n}^{(1)}} 
|x-y|}}{|x - y|} \int_{0}^{1} 
V_{n}^{p}(y, \theta) d\theta dy$. 
We first consider the case of $|x| \geq 1$. 
We note that 
$|x - y| \geq |y| - |x| \geq |y| - |y|/2 \geq |y|/2$ for $y \in A_{3}$. 
This together with \eqref{eqU-42}, 
the monotonicity of  the function $\frac{e^{- \sqrt{\alpha_{n}^{(1)}} r}}{r}$ 
and $p = 2$ yields that 
\begin{equation*}
\begin{split}
& \quad 
\int_{A_{3}} \frac{e^{- \sqrt{\alpha_{n}^{(1)}} 
|x-y|}}{|x - y|} \int_{0}^{1} V_{n}^{p}(y, \theta) d\theta dy 
\lesssim 
\int_{A_{3}} 
\frac{e^{- \sqrt{\alpha_{n}^{(1)}} |x-y|}}{|x - y||y|^{p}} dy 
\lesssim \int_{A_{3}} 
\frac{e^{- \frac{\sqrt{\alpha_{n}^{(1)}}}{2} |y|}}{|y|^{p + 1}} dy 
\\[6pt] 
& \lesssim \int_{2|x|}^{\infty} \frac{e^{- \frac{\sqrt{\alpha_{n}^{(1)}}}{2} r}}{r^{p + 1}} 
r^{2} \, dr 
\lesssim \int_{2|x|}^{\infty} 
\frac{e^{-\frac{\sqrt{\alpha_{n}^{(1)}}}{2} r}}{r^{p - 1}}\, dr 
= \int_{2|x|}^{\infty} 
\frac{e^{-\frac{\sqrt{\alpha_{n}^{(1)}}}{2} r}}{r}\, dr. 
\end{split}
\end{equation*}
We estimate $\int_{2|x|}^{\infty} 
\frac{e^{-\frac{\sqrt{\alpha_{n}^{(1)}}}{2} r}}{r}\, dr$ 
dividing into 3 cases

\textbf{(Case 1) $1 \leq |x| \leq 
(\alpha_{n}^{(1)})^{- \frac{1}{2}}/2$.
} 

One has 
\[
\begin{split}
\int_{2|x|}^{\infty} 
\frac{e^{-\frac{\sqrt{\alpha_{n}^{(1)}}}{2} r}}{r}\, dr 
\leq 
\int_{2|x|}^{(\alpha_{n}^{(1)})^{- \frac{1}{2}}} 
\frac{e^{-\frac{\sqrt{\alpha_{n}^{(1)}}}{2} r}}{r}\, dr
+ \int_{(\alpha_{n}^{(1)})^{- \frac{1}{2}}}^{\infty} 
\frac{e^{-\frac{\sqrt{\alpha_{n}^{(1)}}}{2} r}}{r}\, dr 
& \leq \int_{1}^{(\alpha_{n}^{(1)})^{- \frac{1}{2}}} 
\frac{dr}{r} + \int_{1}^{\infty} 
\frac{e^{- \frac{s}{2}}}{s}\, ds \\[6pt] 
& \lesssim |\log \alpha_{n}^{(1)}| + 1 
\lesssim |\log \alpha_{n}^{(1)}|. 
\end{split}
\]

\textbf{(Case 2) $|x| \geq (\alpha_{n}^{(1)})^{- \frac{1}{2}}/2$. 
}
 
We see that 
\[
\int_{2|x|}^{\infty} 
\frac{e^{-\frac{\sqrt{\alpha_{n}^{(1)}}}{2} r}}{r}\, dr \leq 
\int_{(\alpha_{n}^{(1)})^{- \frac{1}{2}}}^{\infty} 
\frac{e^{-\frac{\sqrt{\alpha_{n}^{(1)}}}{2} r}}{r}\, dr 
= \int_{1}^{\infty} 
\frac{e^{-\frac{1}{2} s}}{s} \, ds 
\lesssim 1. 
\]

\textbf{(Case 3) $|x| < 1$. 
}

It follows that  
\begin{equation*}
\begin{split}
\int_{A_{3}} \frac{e^{- \sqrt{\alpha_{n}^{(1)}} 
|x-y|}}{|x - y|} \int_{0}^{1} V_{n}^{p}(y, \theta) d\theta dy 
& = \int_{A_{3} \cap |x - y| \geq 2} 
\frac{e^{- \sqrt{\alpha_{n}^{(1)}} 
|x-y|}}{|x - y|} \int_{0}^{1} V_{n}^{p}(y, \theta) d\theta dy \\
& \quad + \int_{A_{3} \cap |x - y| < 2} 
\frac{e^{- \sqrt{\alpha_{n}^{(1)}} 
|x-y|}}{|x - y|} \int_{0}^{1} V_{n}^{p}
(y, \theta) d\theta dy. 
\end{split}
\end{equation*}
Since $\left\|\int_{0}^{1} V_{n}^{p}(y, \theta) 
d\theta \right\|_{L^{\infty}} \leq 1$, we have 
\[
\begin{split}
& \quad \int_{A_{3} \cap |x - y| < 2} 
\frac{e^{- \sqrt{\alpha_{n}^{(1)}} 
|x-y|}}{|x - y|} \int_{0}^{1} V_{n}^{p}(y, \theta) 
d\theta dy 
\leq \int_{A_{3} \cap |x - y| < 2} 
\frac{e^{- \sqrt{\alpha_{n}^{(1)}} 
|x-y|}}{|x - y|} dy \\[6pt] 
& \leq \int_{|y^{\prime}| < 2} 
\frac{e^{- \sqrt{\alpha_{n}^{(1)}} 
|y^{\prime}|}}{|y^{\prime}|} dy^{\prime} 
\leq \int_{0}^{2} e^{- \sqrt{\alpha_{n}^{(1)}} 
|r|}r \, dr \leq \int_{0}^{2} r \, dr \lesssim 1. 
\end{split}
\]
For $|x - y| \geq 2$ and $|x| < 1$, we have 
$|y| \geq |x - y| - |x| \geq 1$. 
Using this together with \eqref{eqU-42}, 
$|x - y| \geq |y| - |x| \geq |y| - |y|/2 \geq |y|/2$ for $y \in A_{3}$, 
the monotonicity of  the function $\frac{e^{- \sqrt{\alpha_{n}^{(1)}} r}}{r}$ and $p = 2$, 
we obtain 
\[
\begin{split}
& \quad 
\int_{A_{3}\cap |x - y| \geq 2} 
\frac{e^{- \sqrt{\alpha_{n}^{(1)}} 
|x-y|}}{|x - y|} \int_{0}^{1} V_{n}^{p}(y, \theta) 
d\theta dy 
\lesssim 
\int_{A_{3} \cap |x - y| \geq 2} 
\frac{e^{- \sqrt{\alpha_{n}^{(1)}} |x-y|}}{|x - y||y|^{p}} dy 
\lesssim \int_{A_{3} \cap |x - y| \geq 2} 
\frac{e^{- \frac{\sqrt{\alpha_{n}^{(1)}}}{2} |y|}}{|y|^{p + 1}} dy \\[6pt]
& \lesssim 
\int_{|y| \geq 1} 
\frac{e^{- \frac{\sqrt{\alpha_{n}^{(1)}}}{2} |y|}}{|y|^{p + 1}} dy 
\lesssim \int_{1}^{\infty} 
\frac{e^{- \frac{\sqrt{\alpha_{n}^{(1)}}}{2} r}}{r^{p - 1}} \, dr 
\leq \int_{1}^{ (\alpha_{n}^{(1)})^{- \frac{1}{2}} } 
\frac{e^{- \frac{\sqrt{\alpha_{n}^{(1)}}}{2} r}}{r^{p - 1}} \, dr 
+ \int_{(\alpha_{n}^{(1)})^{- \frac{1}{2}}}^{\infty} 
\frac{e^{- \frac{\sqrt{\alpha_{n}^{(1)}}}{2} r}}{r^{p - 1}} \, dr \\[6pt]
& \lesssim 
\int_{1}^{ (\alpha_{n}^{(1)})^{- \frac{1}{2}} } 
\frac{1}{r^{p - 1}} \, dr 
+ \int_{1}^{\infty} 
\frac{e^{- \frac{s}{2}}}{s^{p - 1}} \, ds 
\lesssim |\log \alpha_{n}^{(1)}| + 1 
\lesssim |\log \alpha_{n}^{(1)}|. 
\end{split}
\]
This completes the proof. 
\end{proof}
Next, we study the second term of \eqref{eqU-38}. 
Following Lemma \ref{lem-uni3}, we can obtain the following: 
\begin{lemma} \label{pro4-1-2}
\begin{equation} \label{eqU-39}
\lim_{n \to \infty} 
|\log \alpha_{n}^{(1)}|^{-1} 
\langle g_{n}, \widetilde{z}_{n} \rangle = 0. 
\end{equation}
\end{lemma} 
\begin{proof}

We shall show that 
\begin{equation} \label{eqU-40}
\|(- \Delta + \alpha_{n}^{(1)})g_{n}\|_{L^{\infty}} 
\lesssim |\log \alpha_{n}^{(1)}|. 
\end{equation}
Once \eqref{eqU-40} holds, we can prove Lemma \ref{pro4-1-2}
by a similar argument in the proof of Lemma \ref{lem-uni3}. 
Observe from \eqref{eqU-37-3} that 
\begin{equation} \label{eqU-41}
\|(- \Delta + \alpha_{n}^{(1)})^{-1}g_{n}\|_{L^{\infty}} 
\leq |C_{n}| \|(- \Delta + \alpha_{n}^{(1)})^{-1}W^{4} 
\Lambda W\|_{L^{\infty}} + 
\|(- \Delta + \alpha_{n}^{(1)})^{-1}
\int_{0}^{1} V_{n}^{p}(\cdot, \theta) d\theta\|_{L^{\infty}}. 
\end{equation}
As in \eqref{EqB-40}, we have
\begin{equation*} 
\langle (- \Delta + \alpha_{\omega})^{-1}
\int_{0}^{1} V_{n}^{p}(\cdot, \theta) d\theta, W^{4} \Lambda W \rangle
= 
- \frac{\sqrt{3}}{10} 
\int_{\mathbb{R}^{3}} 
\frac{e^{- \sqrt{\alpha_{\omega}}|x|}}{|x|} 
\int_{0}^{1} V_{n}^{p}(x, \theta) d\theta \, dx 
+ O(1). 
\end{equation*}
This together with \eqref{eqU-42}   
and $p = 2$ yields that 
\begin{equation}\label{eqU-43}
\begin{split}
\biggl|\langle (- \Delta + \alpha_{n}^{(1)})^{-1}
\int_{0}^{1} V_{n}^{p}(\cdot, \theta) d\theta, W^{4} \Lambda W \rangle 
\biggl| 
& = 
\biggl|- \frac{\sqrt{3}}{10} 
\int_{\mathbb{R}^{3}} 
\frac{e^{- \sqrt{\alpha_{n}^{(1)}}|x|}}{|x|} 
\int_{0}^{1} V_{n}^{p}(x, \theta) d\theta \, dx
+ O_{n}(1) \biggl| \\[6pt]
& \lesssim \int_{0}^{1} \, dr 
+ \int_{1}^{(\alpha_{n}^{(1)})^{- \frac{1}{2}}}
\frac{dr}{r^{p - 1}} + 
\int_{(\alpha_{n}^{(1)})^{- \frac{1}{2}}}^{\infty} 
\frac{e^{- \sqrt{\alpha_{n}^{(1)}} r}}{r^{p - 1}} 
\, dr 
+ O_{n}(1) \\[6pt]
& \lesssim |\log \alpha_{n}^{(1)}| + 
\int_{1}^{\infty} 
\frac{e^{- s}}{s^{p - 1}} 
\, dr 
+ O_{n}(1) \\
& \lesssim |\log \alpha_{n}^{(1)}|. 
\end{split}
\end{equation} 
By \eqref{eqU-44-3}, \eqref{eqU-43} and 
Lemma \ref{lem-uni4}, we have 
\begin{equation}\label{eqU-49}
|C_{n}| \lesssim |\log \alpha_{n}^{(1)}|. 
\end{equation}
Thus, we see from \eqref{eqU-41}, 
\eqref{eqU-49} and \eqref{eqU-46} that 
\[
\begin{split}
\|(- \Delta + \alpha_{n}^{(1)})^{-1}g_{n}\|_{L^{\infty}} 
& 
\leq |C_{n}| \|(- \Delta + \alpha_{n}^{(1)})^{-1}W^{4} 
\Lambda W\|_{L^{\infty}} + 
\|(- \Delta + \alpha_{n}^{(1)})^{-1}
\int_{0}^{1} V_{n}^{p}(\cdot, \theta) d\theta\|_{L^{\infty}} \\
& \lesssim 
|\log \alpha_{n}^{(1)}| \|W^{4} \Lambda W\|_{L^{\frac{3 + \varepsilon}{2}} \cap L^{\frac{3 - \varepsilon}{2}}} 
+ |\log \alpha_{n}^{(1)}| \\
& \lesssim |\log \alpha_{n}^{(1)}|. 
\end{split}
\]
This completes the proof. 
\end{proof}
We now give the proof of Lemma \ref{lem-uni8}. 
\begin{proof}[Proof of Lemma \ref{lem-uni8}]
Note that 
since $\lim_{n \to \infty} \widetilde{z}_{n} = 
\kappa \Lambda W$ weakly in $\dot{H}^{1}(\R^{3})$ 
and $W^{4}\Lambda W \sim |x|^{-5}$ as $|x| \to \infty$, 
we see that 
\begin{equation}\label{eqU-47}
\lim_{n \to \infty} 
\langle W^{4} \Lambda W, 
\widetilde{z}_{n} \rangle = 
\kappa \langle W^{4} \Lambda W, 
\Lambda W \rangle. 
\end{equation} 
Then, it follows from \eqref{eqU-38}, \eqref{eqU-39}, \eqref{eqU-47}, 
\eqref{eqU-44-3}, \eqref{eqU-23} and 
\eqref{eqU-48} that 
\begin{equation*}
\begin{split}
& \quad 
\limsup_{n \to \infty}
|\log \alpha_{n}^{(1)}|^{-1}
\int_{\R^{3}} \int_{0}^{1} V_{n}^{p}(x, \theta) d\theta \widetilde{z}_{n}(x) \, dx \\[6pt]
& = \liminf_{n \to \infty} \left(
C_{n} 
|\log \alpha_{n}^{(1)}|^{-1} \langle W^{4}\Lambda W, 
\widetilde{z}_{n} \rangle 
+ |\log \alpha_{n}^{(1)}|^{-1} \langle g_{n}, 
\widetilde{z}_{n} \rangle \right)\\[6pt]
& =
\kappa \langle W^{4}\Lambda W, \Lambda W \rangle
\liminf_{n \to \infty}
C_{n} |\log \alpha_{n}^{(1)}|^{-1} \\[6pt]
& = \kappa \langle W^{4}\Lambda W, \Lambda W \rangle
\liminf_{n \to \infty}
|\log \alpha_{n}^{(1)}|^{-1} 
\frac{\langle (- \Delta + 
\alpha_{n})^{-1} 
\int_{0}^{1} V_{n}^{p}(\cdot, \theta) d\theta, 
W^{4} \Lambda W \rangle}
{\langle (- \Delta + \alpha_{n})^{-1} 
W^{4} \Lambda W, 
W^{4} \Lambda W \rangle} \\[6pt]
& = 5 \kappa \liminf_{n \to \infty}
|\log \alpha_{n}^{(1)}|^{-1} 
\langle (- \Delta + 
\alpha_{n})^{-1} 
\int_{0}^{1} V_{n}^{p}(\cdot, \theta) d\theta, 
W^{4} \Lambda W \rangle
\\[6pt]
& = - 3 \sqrt{3} \pi \kappa. 
\end{split}
\end{equation*}
Thus, we see that \eqref{eqU-35} holds. 
\end{proof}

\subsection{Case of $1 < p < 2$}
In this subsection, we consider the case of 
$1 < p < 2$.
Here, we will show the following:
\begin{proposition} \label{propR-12} 
Let $1 < p < 2$. 
We obtain the following: 
\begin{equation} \label{eqU-15-2}
\liminf_{n \to \infty} 
\sqrt{\alpha_{n}^{(1)}} 
\langle \widetilde{u}_{n}^{(1)}(x) + 
\nu_{n} \widetilde{u}_{n}^{(2)}
(\nu_{n}^{2} \cdot), \widetilde{z}_{n} \rangle
= 0. 
\end{equation}
\end{proposition}

\begin{proposition}\label{propR-3-2} 
Let $1 < p < 2$. 
We obtain the following: 
\begin{equation} \label{eqU-16-2}
\limsup_{n \to \infty}
\frac{\beta_{n}^{(1)}}{\sqrt{\alpha_{n}^{(1)}}} 
\left\langle \int_{0}^{1} 
V_{n}^{p}(x, \theta) d\theta, 
\widetilde{z}_{n} \right\rangle 
< - \frac{2 \pi}{5} 3^{\frac{p + 1}{2}} 
\int_{0}^{\infty} 
\frac{e^{- (p + 1) s}}{s^{p -1}} \, ds. 
\end{equation}
\end{proposition}
From Propositions \ref{propR-12} and 
\ref{propR-3-2}, we can prove Theorem 
\textrm{(ii)} in the case of 
$1 < p < 2$ immediately. 
\begin{proof}[Proof of Theorem 
\ref{thm-bl} \textrm{(i)} in the case of 
$1 < p < 2$]

Suppose to the contrary that 
there exists $\{\omega_{n}\}$ in $(0, \infty)$ 
with $\lim_{n \to \infty} \omega_{n} = 0$ such that 
for each $n \in \mathbb{N}$, 
$u_{n}^{(i)}\; (i = 1, 2)$ is a solution 
to \eqref{sp} with $\omega = \omega_{n}$ satisfying 
$u_{n}^{(1)} \neq u_{n}^{(2)}$ and 
$\lim_{n \to \infty} \|u_{n}^{(i)}\|_{L^{\infty}} 
= \infty$. 
Then, we can derive a contradiction from 
\eqref{eqU-51}, \eqref{eqU-15-2} and 
\eqref{eqU-16-2}. 
This completes the proof. 
\end{proof}

\subsubsection{Proof of Proposition 
\ref{propR-12}}
In this subsection, we will prove Proposition 
\ref{propR-12}. 
To this end, we need the following lemma:
\begin{lemma}\label{eqU-Lem10}
We take $\delta$ and $L > 0$ arbitrary and fix it. 
For any $\varepsilon > 0$, 
there exists $N_{0} = N_{0}(\delta, L , \varepsilon) 
\in \mathbb{N}$ such that for $n \geq N_{0}$, 
we obtain 
\begin{equation} \label{eqU-60}
|\widetilde{z}_{n}(r)| < \frac{\varepsilon}{r} 
\qquad \mbox{for any $r \in [\delta (\alpha_{n}^{(1)})^{-\frac{1}{2}}, 
L (\alpha_{n}^{(1)})^{-\frac{1}{2}}]$}. 
\end{equation}
\end{lemma}
\begin{proof}
We put $\widehat{z}_{n}(s) = \widetilde{z}_{n}(r)$ 
and $s = (\alpha_{n}^{(1)})^{\frac{1}{2}} r$.  
Then, we see that $\widehat{z}_{n}$ satisfies 
	\begin{equation}\label{eqU-61}
	- \p_{ss} \widehat{z}_{n}
	- \frac{2}{s} \p_{s} \widehat{z}_{n}
+ \widehat{z}_{n} 
= p \beta_{n}^{(1)} 
(\alpha_{n}^{(1)})^{\frac{p - 3}{2}}
\int_{0}^{1} \widehat{V}_{n}^{p-1}
(s, \theta) d\theta 
\overline{z}_{n} + 
5 \alpha_{n}^{(1)} 
\int_{0}^{1} 
\widehat{V}_{n}^{4}(s, \theta) d\theta 
\widehat{z}_{n}. 
	\end{equation}
where 
\begin{equation}\label{eqU-55}
\widehat{V}_{n}(s, \theta) 
:= \theta \widehat{u}_{n}^{(1)}(s) 
+ (1 - \theta) \nu_{n}^{-1} 
\widehat{u}_{n}^{(2)}(\nu_{n}^{4} s) 
\end{equation}
and $\widehat{u}_{\omega}^{(i)}(s)
= (\alpha_{n}^{(i)})^{-\frac{1}{2}} 
\widetilde{u}_{n}^{(i)}(r)$
for $i = 1$ and $2$. 
In addition, we put 
\begin{equation} \label{eqU-59-2}
\overline{z}_{n} 
(s) = s \widehat{z}_{n}(s). 
\end{equation}
Then, we see from \eqref{eqU-61} that $\overline{z}_{n}$ 
satisfies 
\begin{equation} \label{eqU-54}
- \p_{ss} \overline{z}_{n}
+ \overline{z}_{n} 
= p \beta_{n}^{(1)} 
(\alpha_{n}^{(1)})^{\frac{p - 3}{2}}
s^{- (p-1)} 
\int_{0}^{1} \overline{V}_{n}^{p-1}
(s, \theta) d\theta 
\overline{z}_{n} + 
5 \alpha_{n}^{(1)} s^{-4}
\int_{0}^{1} 
\overline{V}_{n}^{4}(s, \theta) d\theta 
\overline{z}_{n}, 
\end{equation}
where 
\begin{equation}\label{eqU-62}
\overline{V}_{n}(s, \theta) 
:= s \widehat{V}_{n}(s, \theta) 
= \theta \overline{u}_{n}^{(1)}(s) 
+ (1 - \theta) \nu_{n}^{-1} 
\overline{u}_{n}^{(2)}(\nu_{n}^{4} s). 
\end{equation}
Here $\overline{u}_{n}^{(i)}(s) 
: = s \widehat{u}_{n}^{(i)}(s)$ 
for $i = 1$ and $2$. 
From Lemma \ref{lem-uni2}, 
we see that 
$\sup_{n \in \mathbb{N}} \|\overline{z}_{n}\|_{L^{\infty}} < C_{1}$ 
for some $C_{1} > 0$. 
In addition, it follows from \eqref{EqL-1} 
and Lemma \ref{lem-conv2} that 
\[
|\overline{u}_{n}^{(i)}(s)| \lesssim 
1 \qquad \mbox{for all $s > 0$}. 
\] 
and 
\[
\lim_{n \to \infty} \overline{u}_{n}^{(i)}(s)
= \overline{u}_{0}(s) 
\qquad 
\lim_{n \to \infty} 
\frac{d \overline{u}_{\omega_{n}}}{d s}(s) 
= \frac{d \overline{u}_{0}}{d s}(s) 
\qquad 
\mbox{uniformly in $s \in [\delta, L]$}, 
\] 
where 
$\overline{u}_{0}$ is 
the unique positive solution to 
\eqref{EqB-4}. 
Then, as in the similar argument in the 
proof of Lemma \ref{lem-conv2}, 
we can apply the Ascoli-Arzela theorem 
and find that 
there exist 
a subsequence of $\{\overline{z}_{n}\}$ 
(we still denote it by the same letter) and 
a function $\overline{z}_{0}$ 
such that $\lim_{n \to \infty} 
\overline{z}_{n}(s) = \overline{z}_{0}(s), 
\lim_{n \to \infty} 
\frac{d \overline{z}_{n}}{d s}(s) = 
\frac{d \overline{z}_{0}}{d s}(s)$ 
uniformly in $s \in [\delta, L]$. 
Then,  we see that $\overline{z}_{0}$ satisfies 
\[
- \p_{ss} \overline{z}_{0} 
+ \overline{z}_{0} = p \theta_{0} 
\overline{u}_{0}^{p-1} \overline{z}_{0} 
\qquad \mbox{in $s \in [\delta, L]$}. 
\]
Since $\delta > 0$ and $L > 0$ are arbitrary, 
we see that $\overline{z}_{0}$ satisfies 
\[
- \p_{ss} \overline{z}_{0} 
+ \overline{z}_{0} = p \theta_{0} 
\overline{u}_{0}^{p-1} \overline{z}_{0} 
\qquad \mbox{in $s \in (-\infty, \infty)$}. 
\]
It follows from non-degeneracy of 
$\overline{u}_{0}$ (see Proposition 
\ref{thm-scud}) that 
$\overline{z}_{0} = 0$. 
Thus, for any $\e > 0$, there exists $N_{\e} \in \mathbb{N}$ such that 
for $n \geq N_{\e}$, we have 
\[
|\overline{z}_{n}(s)| < \varepsilon 
\qquad \mbox{for all $s \in [\delta, L]$}. 
\]
This together with 
$\overline{z}_{n} 
(s) = r \widetilde{z}_{n}(r)$ and $s = (\alpha_{n}^{(1)})^{\frac{1}{2}} r$ 
yields that 
\[
|\widetilde{z}_{n}(r)| 
= \frac{|\overline{z}_{n}(s)|}{r} 
< \frac{\varepsilon}{r} 
\qquad \mbox{for all $r \in [\delta (\alpha_{n}^{(1)})^{- \frac{1}{2}}, 
L (\alpha_{n}^{(1)})^{- \frac{1}{2}}]$}. 
\]
This completes the proof. 
\end{proof}

We are now in a position to prove 
Proposition \ref{propR-12}. 
\begin{proof}[Proof of Proposition \ref{propR-12}]
We take $\delta$ and $L > 0$ arbitrary and fix it. 
It follows that 
\begin{equation} \label{eqU-55-1}
\begin{split}
& \quad |\sqrt{\alpha_{n}^{(1)}} 
\langle \widetilde{u}_{n}^{(1)}(x) + 
\nu_{n} \widetilde{u}_{n}^{(2)}
(\nu_{n}^{2} \cdot), \widetilde{z}_{n} \rangle| \\
& \leq \sqrt{\alpha_{n}^{(1)}} 
4 \pi \int_{0}^{\delta (\alpha_{n}^{(1)})^{- \frac{1}{2}}}
(\widetilde{u}_{n}^{(1)}(r) + 
\nu_{n} \widetilde{u}_{n}^{(2)})|\widetilde{z}_{n}| r^{2}\, dr \\
& \quad + \sqrt{\alpha_{n}^{(1)}}
4 \pi \int_{\delta (\alpha_{n}^{(1)})^{- \frac{1}{2}}}
^{L (\alpha_{n}^{(1)})^{- \frac{1}{2}}}
(\widetilde{u}_{n}^{(1)}(r) + 
\nu_{n} \widetilde{u}_{n}^{(2)})|\widetilde{z}_{n}| r^{2} \, dr \\
& \quad 
+ \sqrt{\alpha_{n}^{(1)}}
4 \pi \int_{L (\alpha_{n}^{(1)})^{- \frac{1}{2}}}^{\infty}
(\widetilde{u}_{n}^{(1)}(r) + 
\nu_{n} \widetilde{u}_{n}^{(2)})|\widetilde{z}_{n}| r^{2} \, dr \\
& =: I_{n} + II_{n} + III_{n}
\end{split}
\end{equation}
From \eqref{EqL-1}, \eqref{eqU-4-1}, 
\eqref{eqU-6-3} and \eqref{eqU-60}, we see that 
there exists $L_{\delta} > 0$ such that 
if $L > L_{\delta}$, we obtain 
\begin{equation} \label{eqU-55-2}
I_{n} 
\lesssim 
\sqrt{\alpha_{n}^{(1)}} 
\int_{0}^{\delta (\alpha_{n}^{(1)})^{- \frac{1}{2}}}
\frac{1}{r^{2}} r^{2} \, dr 
\lesssim \sqrt{\alpha_{n}^{(1)}} \times 
\delta (\alpha_{n}^{(1)})^{- \frac{1}{2}} 
= \delta. 
\end{equation}
\begin{equation} \label{eqU-55-3}
\begin{split}
III_{n} 
& \lesssim 
\sqrt{\alpha_{n}^{(1)}}\int_{L (\alpha_{n}^{(1)})^{- \frac{1}{2}}}^{\infty}
e^{- \frac{\sqrt{\alpha_{\omega}^{(1)}}}{2} r} \, dr 
\lesssim e^{- \frac{L}{2} } < \delta. 
\end{split}
\end{equation}
Applying Lemma \ref{eqU-Lem10} with $\e = \delta/ L$,  
there exists $N_{0} \in \N$ such that for $n \geq N_{0}$, 
we have 
\[
|\widetilde{z}_{n}(r)| < 
\frac{\delta}{L r} 
\qquad \mbox{for all $r \in [\delta, L]$}. 
\] 
This together with \eqref{EqL-1} yields that   
\begin{equation} \label{eqU-55-4}
II_{n} \lesssim 
\sqrt{\alpha_{n}^{(1)}}
\int_{\delta (\alpha_{n}^{(1)})^{- \frac{1}{2}}}
^{L (\alpha_{n}^{(1)})^{- \frac{1}{2}}}
\frac{\delta}{L r^{2}} r^{2} \, dr \leq 
\sqrt{\alpha_{n}^{(1)}}
\int_{0}
^{L (\alpha_{n}^{(1)})^{- \frac{1}{2}}}
\frac{\delta}{L r^{2}} r^{2} \, dr =  
\sqrt{\alpha_{n}^{(1)}} \times L (\alpha_{n}^{(1)})^{- \frac{1}{2}} 
\times \frac{\delta}{L} = \delta. 
\end{equation} 
Since $\delta > 0$ are arbitrary, we see from 
\eqref{eqU-55-1}--\eqref{eqU-55-4} that 
\eqref{eqU-15-2} holds. 
This completes the proof.
\end{proof}

\subsubsection{Proof of Proposition \ref{propR-3-2}}
Here, we shall show Proposition \ref{propR-3-2}. 
As in the proof of Proposition \ref{propR-2}, 
based on Lemma \ref{LemR-16}, 
we decompose 
$\int_{0}^{1} V_{n}^{p}(x, \theta) d\theta$ as follows:
\begin{equation} \label{eqU-37-4}
\int_{0}^{1} V_{n}^{p}(x, \theta) d\theta 
= C_{n} W^{4}\Lambda W(x) + g_{n}(x),
\end{equation}
where 
\begin{equation} \label{eqU-44-3-2} 
C_{n}:= 
\frac{\langle (- \Delta + 
\alpha_{n}^{(1)})^{-1} 
\int_{0}^{1} V_{n}^{p}(\cdot, \theta) d\theta, 
W^{4} \Lambda W \rangle}
{\langle (- \Delta + \alpha_{n}^{(1)})^{-1} 
W^{4} \Lambda W, 
W^{4} \Lambda W \rangle}. 
\end{equation}
It follows from \eqref{eqU-37-4} that 
\begin{equation} \label{eqU-38-4} 
\frac{\beta_{n}^{(1)}}{\sqrt{\alpha_{n}^{(1)}}} 
\left\langle
\int_{0}^{1} V_{n}^{p}(x, \theta) d\theta, 
z_{n}(x) \right\rangle 
=C_{n} 
\frac{\beta_{n}^{(1)}}{\sqrt{\alpha_{n}^{(1)}}} 
\langle W^{4}\Lambda W, \widetilde{z}_{n} \rangle 
+ \frac{\beta_{n}^{(1)}}{\sqrt{\alpha_{n}^{(1)}}} 
\langle g_{n}, \widetilde{z}_{n} \rangle. 
\end{equation}
From the decomposition \eqref{eqU-37-4} and \eqref{eqU-44-3-2}, 
we can easily verify that 
\begin{equation*} 
\langle (- \Delta + \alpha_{n}^{(1)})^{-1} 
g_{n}, W^{4} \Lambda W \rangle = 0. 
\end{equation*}
Using the resolvent estimate \eqref{EqB-45}, 
we shall show the following: 
\begin{proposition}\label{lem-3-3-1}
Let $1 < p < 2$. 
We obtain 
\[
\lim_{n \to \infty} 
\left(\alpha_{n}^{(1)}\right)^{\frac{2 - p}{2}}
|\left\langle g_{n}, \widetilde{z}_{n} 
\right\rangle| = 0. 
\]
\end{proposition}
To prove Proposition \ref{lem-3-3-1} , 
we need the following:
\begin{lemma}\label{lem-uni4-4} 
Let $1 < p < 2$. 
One has 
\begin{equation} \label{eqU-46-2}
\|(- \Delta + \alpha_{n}^{(1)})^{-1} \int_{0}^{1} 
V_{n}^{p}(\cdot, \theta) d\theta\|_{L^{\infty}} 
\lesssim 
(\alpha_{n}^{(1)})^{- \frac{2 - p}{2}}.
\end{equation}
\end{lemma}
\begin{proof}
We use the formula \eqref{EqB-37} and divide the integral as follows:
\begin{equation*}
\begin{split}
|(- \Delta + \alpha_{n}^{(1)})^{-1} 
\int_{0}^{1} 
V_{n}^{p}(\cdot, \theta) d\theta| 
& \leq
\sum_{i = 1}^{3}
\int_{A_{i}} \frac{e^{- \sqrt{\alpha_{n}^{(1)}} 
|x-y|}}{|x - y|} \int_{0}^{1} 
|V_{n}^{p}(y, \theta)| d\theta dy, 
\end{split}
\end{equation*}
where
\begin{equation*}
\begin{split}
& A_{1} := \left\{y \in \mathbb{R}^{3} \colon 
|x - y| \leq \frac{|x|}{2}\right\}, \\[6pt]
& A_{2} := \left\{y \in \mathbb{R}^{3} \colon 
|x - y| > \frac{|x|}{2}, \; |y| \leq 2|x| \right\}, \\[6pt]
& A_{3} := \left\{y \in \mathbb{R}^{3} \colon 
|x - y| > \frac{|x|}{2}, \; |y| > 2|x| \right\}. 
\end{split}
\end{equation*}
Note that 
\begin{equation}\label{eq-nd46}
|y| \geq |x| - |x-y| \geq |x|/2 
\qquad \mbox{for $y \in A_{1}$}. 
\end{equation}
For $|x| \leq \alpha_{n}^{- \frac{1}{2}}$, we obtain 
\begin{equation}
\label{eq-nd46-1}
\begin{split}
& \quad 
\int_{A_{1}} \frac{e^{- \sqrt{\alpha_{n}^{(1)}} 
|x-y|}}{|x - y|} \int_{0}^{1} 
|V_{n}^{p}(y, \theta)| d\theta dy 
\lesssim 
\int_{A_{1}} 
\frac{e^{- \sqrt{\alpha_{n}^{(1)}} 
|x-y|}}{|x - y||y|^{p}} dy 
\lesssim \int_{B(x, |x|/2)} 
\frac{e^{- \sqrt{\alpha_{n}^{(1)}} 
|x - y|}}{|x - y||x|^{p}}dy \\[6pt] 
& 
\lesssim \frac{1}{|x|^{p}} \int_{B(0, |x|/2)} 
\frac{e^{- \sqrt{\alpha_{n}^{(1)}}
|y^{\prime}|}}{|y^{\prime}|} 
d y^{\prime} 
\lesssim \frac{1}{|x|^{p}} 
\int_{0}^{|x|/2} \frac{e^{- 
\sqrt{\alpha_{n}^{(1)}} r}}{r} r^{2} \, dr 
\lesssim \frac{1}{|x|^{p}} \int_{0}^{|x|/2} 
e^{- \sqrt{\alpha_{n}^{(1)}} r} r \, dr \\[6pt]
& \leq \frac{1}{|x|^{p}} \int_{0}^{|x|/2} r \, dr 
\lesssim |x|^{2 - p} \leq 
\left(\alpha_{n}^{(1)}\right)^{\frac{p - 2}{2}}. 
\end{split}
\end{equation}
Moreover, since $|x - y| \geq |x|/2 $ for $y \in A_{2}$ 
and the function $\frac{e^{- \sqrt{\alpha_{n}} r}}{r}$ 
is decreasing in $r>0$, for $|x| \leq \alpha_{n}^{- \frac{1}{2}}$, 
we have by \eqref{eqU-42} that 
\begin{equation}
\label{eq-nd46-3}
\begin{split}
& \quad 
\int_{A_{2}} 
\frac{e^{- \sqrt{\alpha_{n}} |x-y|}}{|x - y|}
\int_{0}^{1} |V_{n}^{p}(y, \theta)| d\theta dy \lesssim 
\int_{A_{2}} 
\frac{e^{- \sqrt{\alpha_{n}^{(1)}} 
|x-y|}}{|x - y||y|^{p}} dy
\leq 2 \frac{e^{- \frac{
\sqrt{\alpha_{n}^{(1)}} |x|}{2}}}{|x|}
\int_{A_{2}} 
\frac{1}{|y|^{p}} dy \\[6pt]
& 
\lesssim 
\frac{e^{- \frac{\sqrt{\alpha_{n}^{(1)}} 
|x|}{2}}}{|x|}
\int_{|y| \leq 2|x|} 
\frac{1}{|y|^{p}} dy
\lesssim \frac{e^{- \frac{\sqrt{
\alpha_{n}^{(1)}} |x|}{2}}}{|x|} 
\int_{0}^{2|x|} \frac{1}{r^{p}} r^{2} \, dr 
\lesssim e^{- \frac{\sqrt{\alpha_{n}^{(1)}}}
{2} |x|} 
|x|^{2 - p}\lesssim 
\left(\alpha_{n}^{(1)} \right)^{\frac{p - 2}{2}}. 
\end{split}
\end{equation}
Finally, we shall estimate 
$\int_{A_{3}} 
\frac{e^{- \sqrt{\alpha_{n}^{(1)}} 
|x-y|}}{|x - y|} 
\int_{0}^{1} 
|V_{n}^{p}(y, \theta)| d\theta dy$. 
We note that 
$|x - y| \geq |y| - |x| \geq |y| - |y|/2 \geq |y|/2$ for $y \in A_{3}$. 
This together with \eqref{eqU-42} 
yields that 
\begin{equation}
\label{eq-nd46-4}
\begin{split}
& \quad 
\int_{A_{3}} \frac{e^{- \sqrt{\alpha_{n}^{(1)}} 
|x-y|}}{|x - y|} 
\int_{0}^{1} |V_{n}^{p}(y, \theta)| d\theta dy 
\lesssim 
\int_{A_{3}} 
\frac{e^{- \sqrt{\alpha_{n}^{(1)}} 
|x-y|}}{|x - y||y|^{p}} dy 
\lesssim \int_{A_{3}} 
\frac{e^{- \frac{\sqrt{\alpha_{n}^{(1)}}}{2} 
|y|}}{|y|^{p+1}} dy \\[6pt]
& \lesssim \int_{2|x|}^{\infty} 
\frac{e^{- \frac{\sqrt{\alpha_{n}^{(1)}}}{2} r}}{r^{p+1}} r^{2} \, dr 
\lesssim \int_{2|x|}^{\infty} 
\frac{e^{-\frac{\sqrt{\alpha_{n}^{(1)}}}{2} r}}{r^{p-1}}\, dr. 
\end{split}
\end{equation}
It follows that 
\begin{equation} 
\label{eq-nd46-5}
\begin{split}
\int_{2|x|}^{\infty} 
\frac{e^{-\frac{\sqrt{\alpha_{n}^{(1)}}}{2} r}} {r^{p-1}}\, dr 
\leq 
\int_{0}^{\infty} 
\frac{e^{-\frac{\sqrt{\alpha_{n}^{(1)}}}{2} r}} {r^{p-1}}\, dr 
& \leq 
\int_{0}^{(\alpha_{n}^{(1)})^{- \frac{1}{2}}} 
\frac{e^{-\frac{\sqrt{\alpha_{n}^{(1)}}}{2} r}} {r^{p-1}}\, dr
+ \int_{(\alpha_{n}^{(1)}
)^{- \frac{1}{2}}}^{\infty} 
\frac{e^{-\frac{\sqrt{\alpha_{n}^{(1)}}}{2} r}} {r^{p-1}}\, dr \\[6pt]
& \leq \int_{0}^{(\alpha_{n}^{(1)}
)^{- \frac{1}{2}}} 
\frac{dr}{r^{p-1}}
+ \left(\alpha_{n}^{(1)}\right)
^{\frac{p-2}{2}} \int_{1}^{\infty} 
\frac{e^{- \frac{s}{2}}}{s^{p-1}}\, ds \\[6pt]
& \lesssim \left(\alpha_{n}^{(1)}
\right)^{\frac{p-2}{2}} 
+ \left(\alpha_{n}^{(1)}\right)
^{\frac{p-2}{2}} 
\lesssim \left(\alpha_{n}^{(1)}
\right)^{\frac{p-2}{2}}. 
\end{split}
\end{equation}
From \eqref{eq-nd46-1}--
\eqref{eq-nd46-5}, we see that 
\eqref{eqU-46-2} holds. 
\end{proof}
Using Lemma \ref{lem-uni4-4}, we will show 
Proposition \ref{lem-3-3-1}. 
\begin{proof}[Proof of Proposition 
\ref{lem-3-3-1}]
The proof is similar to that of \eqref{eqU-22}. 
However, we need several modifications. 
By \eqref{eqU-26}, we have  
\[
\begin{split}
\langle g_{n}, \widetilde{z}_{n} \rangle 
= 
\langle 
(1 - 5 (- \Delta + \alpha_{n}^{(1)})^{-1}
W^{4})^{-1}
(- \Delta + \alpha_{n}^{(1)})^{-1} 
g_{n}, 
R_{n}
\widetilde{z}_n \rangle,  
\end{split}
\]
where $R_{n}$ is given by \eqref{eqU-24}. 
As in \eqref{eqU-27}, we obtain 
\begin{equation} \label{EqND-25-2} 
\begin{split}
\big|\langle g_{n}, \widetilde{z}_{n} \rangle \big|
\lesssim 
\|(- \Delta + \alpha_{n}^{(1)})^{-1} 
g_{n}\|_{L^{\infty}} 
\|R_{n}
\widetilde{z}_n
\|_{L^{1}}. 
\end{split}
\end{equation}
Similarly to \eqref{eqU-28}, we observe from 
by Lemmas \ref{lem2-5}, 
\ref{lem-uni4-4} and \ref{LemR-14}, 
that 
\[
\begin{split}
& \quad 
|C_{n}| \lesssim 
|\langle (- \Delta + \alpha_{n}^{(1)})^{-1} 
\int_{0}^{1} V_{n}^{p}(\cdot, \theta) d\theta, 
W^{4} \Lambda W \rangle| 
\lesssim \left(\alpha_{n}^{(1)}\right)
^{\frac{p - 2}{2}}. 
\end{split}
\]
Using this, 
by a similar argument to \eqref{eqU-29}, 
we have 
\begin{equation} \label{EqND-14-2}
\begin{split}
\|(- \Delta + \alpha_{n}^{(1)})^{-1} 
g_{n}\|_{L^{\infty}} 
& \leq |C_{n}| 
\|(- \Delta + \alpha_{n}^{(1)})^{-1} W^{4} 
\Lambda W\|_{L^{\infty}} 
+ 
\|(- \Delta + \alpha_{n}^{(1)})^{-1}
\int_{0}^{1} V_{n}^{p} (x, \theta) d \theta 
\|_{L^{\infty}} \\[6pt]
& \lesssim 
(\alpha_{n}^{(1)})^{\frac{p - 2}{2}} 
\|W^{4} \Lambda W\|
_{L^{\frac{3}{2} + \e} \cap L^{\frac{3}{2} - \e}} 
+ (\alpha_{n}^{(1)})^{\frac{p - 2}{2}}  \\[6pt]
& \lesssim \left(\alpha_{n}^{(1)}
\right)^{\frac{p - 2}{2}}. 
\end{split}
\end{equation}
We shall show that $
\lim_{n \to \infty}
\|R_{n} \widetilde{z}_n\|_{L^{1}} = 0$. 
Using 
\eqref{EqL-1} and Lemma \ref{lem-uni2} 
for $0< r \leq \delta 
\alpha_{1, n}^{- \frac{1}{2}}$, 
\eqref{EqL-1} and Lemma \ref{eqU-Lem10} 
with $\e = \frac{\delta}{L^{3 - p}}$ for $\delta 
\left(\alpha_{1, n}^{(1)}\right)^{- \frac{1}{2}} 
< r < L \left(\alpha_{1, n}^{(1)}
\right)^{- \frac{1}{2}}$, 
and \eqref{eqU-6-3} and 
\eqref{eqU-4-1}
for $r \geq L 
\left(\alpha_{1, n}^{(1)}\right)^{- \frac{1}{2}}$, 
we obtain 
\[
\begin{split}
\|\widetilde{u}_{n}^{p-1} 
\widetilde{z}_{n}\|_{L^{1}} 
& \lesssim
\int_{0}^{\delta 
\left(\alpha_{1, n}^{(1)}\right)^{- \frac{1}{2}}} 
\frac{1}{r^{p-2}} \, dr 
+ \int_{\delta 
\left(\alpha_{1, n}^{(1)}\right)^{- \frac{1}{2}}}
^{L \left(\alpha_{1, n}^{(1)}\right)^{- \frac{1}{2}}
} 
\frac{\delta}{L^{3 - p} r^{p-2}} \, dr 
+ \int_{L \left(\alpha_{1, n}^{(1)}
\right)^{- \frac{1}{2}}}^{\infty} 
\frac{e^{- \frac{(p-1)}{2}
\sqrt{\alpha_{1, n}^{(1)}} r}
}{r^{p-2}} \, dr \\[6pt]
& \lesssim \delta^{3 - p} 
\left(\alpha_{1, n}^{(1)}\right)^{\frac{p - 3}{2}} 
+ \delta \left(\alpha_{1, n}^{(1)}
\right)^{\frac{p - 3}{2}} 
+ \left(\alpha_{1, n}^{(1)}
\right)^{\frac{p - 3}{2}} 
\int_{L}^{\infty} 
\frac{e^{- \frac{(p-1)}{2}s}}{s^{p-2}} \, ds 
\lesssim \delta \left(
\alpha_{1, n}^{(1)} \right)^{\frac{p - 3}{2}}. 
\end{split}
\]
We put $\eta_{n}(x, \theta) := V_{n}(x, \theta) - W$. 
Then, as in \eqref{eqU-31}, 
we have by \eqref{eqU-24} and 
the H\"{o}lder inequality that 
\begin{equation} \label{EqND-16}
\begin{split}
\|R_{n} \widetilde{z}_{n}\|_{L^{1}} 
& \lesssim 
\sup_{0 < \theta < 1}
\|(\widetilde{u}_{n}^{3} + W^{3}) 
\eta_{n}(\cdot, \theta) 
\widetilde{z}_{n}\|_{L^{1}} 
+ \delta \beta_{1, n}^{(1)}
\|\widetilde{u}_{n}^{p-1}
\widetilde{z}_{n}\|
_{L^{1}} \\[6pt]
& \lesssim 
\left(\|\widetilde{u}_{n}\|_{L^{\frac{9}{2}}}^{3} + 
\|W\|_{L^{\frac{9}{2}}}^{3} \right)
\|\widetilde{z}_{n}\|_{L^{6}} 
\sup_{0 < \theta < 1}\|\eta_{n}
(\cdot, \theta)\|_{L^{6}} 
+ \delta \beta_{1, n}^{(1)} 
\left(\alpha_{1, n}^{(1)}\right)
^{\frac{p - 3}{2}} \\[6pt]
& \lesssim 
\sup_{0 < \theta < 1} 
\|\eta_{n}(\cdot, \theta)\|_{L^{6}} 
+ \delta \beta_{1, n}^{(1)}
\left(\alpha_{1, n}^{(1)}\right)^{\frac{p - 3}{2}}. 
\end{split}
\end{equation}
This together with \eqref{EqB-3} and \eqref{eqU-30} 
yields that 
\begin{equation}\label{EqND-26-2}
\limsup_{n \to \infty} \|R_{n} \widetilde{z}_{n}\|_{L^{1}} < \delta.
\end{equation}
Thus, by \eqref{EqND-25-2}, 
\eqref{EqND-14-2} and \eqref{EqND-26-2}, 
we obtain 
\[
\begin{split}
\limsup_{n \to \infty}
|(\alpha_{1, n}^{(1)})^{\frac{2 - p}{2}} 
\langle g_{n}, \widetilde{z}_{n} 
\rangle| 
\lesssim 
\limsup_{n \to \infty}
(\alpha_{1, n}^{(1)})^{\frac{2 - p}{2}} 
(\alpha_{1, n}^{(1)})^{\frac{p - 2}{2}}
\|R_{n} \widetilde{z}_{n}\|_{L^{1}}
= 
\limsup_{n \to \infty}
\|R_{n} \widetilde{z}_{n}\|_{L^{1}}
< \delta. 
\end{split}
\] 
Since $\delta > 0$ is arbitrary, 
we have obtained the desired result. 
\end{proof}

\begin{proof}[Proof of Proposition \ref{propR-3-2}]
From \eqref{eqU-37-4} 
\eqref{eqU-44-3-2}, Proposition 
\ref{lem-3-3-1}, \eqref{eqU-23} and 
$\lim_{n \to \infty} 
\widetilde{z}_{n} = 
\kappa \Lambda W$ 
weakly in $\dot{H}^{1}(\R^{3})$, we obtain 
\begin{equation} \label{eqU-52}
\begin{split}
& \quad \limsup_{n \to \infty}
\frac{\beta_{n}^{(1)}}{\sqrt{\alpha_{n}^{(1)}}} 
\left\langle \int_{0}^{1} 
V_{n}^{p}(x, \theta) d\theta, 
\widetilde{z}_{n} \right\rangle \\
& 
= \limsup_{n \to \infty}
\frac{\beta_{n}^{(1)}}{\sqrt{\alpha_{n}^{(1)}}} 
\left\langle C_{n} W^{4} \Lambda W + g_{n}, \widetilde{z}_{n} 
\right\rangle \\
& 
= \limsup_{n \to \infty}
\frac{\beta_{n}^{(1)}}{\sqrt{\alpha_{n}^{(1)}}} 
\frac{\langle (- \Delta + 
\alpha_{n}^{(1)})^{-1} 
\int_{0}^{1} V_{n}^{p}(\cdot, \theta) d\theta, 
W^{4} \Lambda W \rangle}
{\langle (- \Delta + \alpha_{n}^{(1)})^{-1} 
W^{4} \Lambda W, 
W^{4} \Lambda W \rangle} 
\left\langle W^{4} \Lambda W, \widetilde{z}_{n} 
\right\rangle \\
& = 5 \kappa \limsup_{n \to \infty}
\frac{\beta_{n}^{(1)}}{(\alpha_{n}^{(1)})^{\frac{3 - p}{2}}} 
\times (\alpha_{n}^{(1)})^{\frac{2 - p}{2}}
\langle (- \Delta + \alpha_{n})^{-1} 
\int_{0}^{1} V_{n}^{p}(\cdot, \theta) d\theta, 
W^{4} \Lambda W \rangle. 
\end{split}
\end{equation}
We put $Z:= \Lambda W - 
\left( - \frac{\sqrt{3}}{2} |x|^{-1}\right)$. 
Then, as in the proof of Proposition 
Proposition \ref{LemR-6}, we have 
\begin{equation} \label{EqB-36-2}
\begin{split}
\langle 
(- \Delta + 
\alpha_{\omega}^{(1)})^{-1}
\int_{0}^{1} V_{n}^{p}(\cdot, \theta) d\theta, 
W^{4} \Lambda W \rangle 
& = 
\frac{1}{5}
\langle 
(- \Delta + 
\alpha_{\omega}^{(1)})^{-1}
\int_{0}^{1} V_{n}^{p}(\cdot, \theta) d\theta, -\Delta \Lambda W \rangle \\[6pt]
& = 
- \frac{\sqrt{3}}{10}
\langle 
(- \Delta + 
\alpha_{\omega}^{(1)})^{-1}
\int_{0}^{1} V_{n}^{p}(\cdot, \theta) d\theta, 
- \Delta |x|^{-1} \rangle 
\\[6pt] & \quad 
+ \frac{1}{5}
\langle 
(- \Delta + 
\alpha_{\omega}^{(1)})^{-1}
\int_{0}^{1} V_{n}^{p}(\cdot, \theta) d\theta, 
- \Delta Z \rangle. 
\end{split}
\end{equation}
In addition, note that $- \Delta |x|^{-1} = 
4 \pi \delta_{0}$ in 
$\mathcal{D}^{\prime}
(\mathbb{R}^{3})$ (see e.g. \cite[Page 156]{MR1817225}). 
Thus, it follows from \eqref{EqB-37} that 
\begin{equation} \label{EqB-38-2}
\begin{split}
& \quad - \frac{\sqrt{3}}{10}
\langle 
(- \Delta + 
\alpha_{\omega}^{(1)})^{-1}
\int_{0}^{1} 
V_{n}^{p}(x, \theta) d\theta, 
(- \Delta)|x|^{-1} \rangle \\
& = 
- \frac{2\sqrt{3} \pi}{5} 
\left\{(- \Delta + 
\alpha_{\omega}^{(1)})^{-1} 
\int_{0}^{1} 
V_{n}^{p}(x, \theta) d\theta
\right\}(0) \\
& = - \frac{\sqrt{3}}{10} 
\int_{\mathbb{R}^{3}} 
\frac{e^{- \sqrt{\alpha_{\omega}^{(1)}}|x|}}
{|x|} 
\int_{0}^{1} 
V_{n}^{p}(x, \theta) d\theta \, dx. 
\end{split}
\end{equation}
By the H\"older inequality and Lemma \ref{LemR-11},
$\int_{0}^{1} 
V_{n}^{p}(x, \theta) d\theta = O(|x|^{-p})$ 
and 
$Z = O(|x|^{-3})$ as $|x| \to \infty$, we have 
\begin{equation} \label{EqB-39-2}
\begin{split}
& \quad 
\biggl|
\langle (- \Delta + 
\alpha_{\omega}^{(1)})^{-1} 
\int_{0}^{1} 
V_{n}^{p}(x, \theta) d\theta, 
\Delta Z \rangle
\biggl| = 
\biggl|
\langle (- \Delta + 
\alpha_{\omega}^{(1)})^{-1} 
\int_{0}^{1} 
V_{n}^{p}(x, \theta) d\theta, \left\{\alpha_{\omega} 
- (- \Delta + 
\alpha_{\omega}^{(1)}) \right\}Z \rangle
\biggl| \\[6pt]
& \leq \alpha_{\omega}
\|(- \Delta + \alpha_{\omega}^{(1)})^{-1} 
\int_{0}^{1} 
V_{n}^{p}(x, \theta) d\theta\|
_{L^{6}}\|Z\|_{L^{\frac{6}{5}}} 
+ \left\|\int_{0}^{1} 
V_{n}^{p}(x, \theta) d\theta
\right\|_{L^{6}}\|Z\|_{L^{\frac{6}{5}}} \\[6pt]
& \lesssim \alpha_{\omega}^{(1)} 
\times \left(\alpha_{2, 
\omega}^{(1)}\right)^{\frac{3}{8}p - \frac{5}{4}}
\|\int_{0}^{1} 
V_{n}^{p}(x, \theta) d\theta\|
_{L^{\frac{4}{p}}}\|Z\|_{L^{\frac{6}{5}}} 
+ 1
\lesssim 1. 
\end{split}
\end{equation}
From \eqref{EqB-36}, \eqref{EqB-38-2} and \eqref{EqB-39-2}, 
we obtain 
\begin{equation}\label{EqB-40-2} 
\langle (- \Delta + 
\alpha_{\omega}^{(1)})^{-1}
\int_{0}^{1} 
V_{n}^{p}(x, \theta) d\theta, 
W^{4} \Lambda W \rangle
= 
- \frac{\sqrt{3}}{10} 
\int_{\mathbb{R}^{3}} 
\frac{e^{- \sqrt{\alpha_{\omega}^{(1)}}|x|}}
{|x|} 
\int_{0}^{1} 
V_{n}^{p}(x, \theta) d\theta \, dx + O(1). 
\end{equation}

From \eqref{EqB-40-2} 
and \eqref{EqB-35}, we obtain 
\begin{equation*}
\begin{split}
& \quad 
\langle 
(- \Delta + \alpha_{\omega}^{(1)})^{-1}
\int_{0}^{1} V_{n}^{p}(x, \theta) d\theta, 
W^{4} \Lambda W\rangle
\\[6pt]
& 
= 
- \frac{\sqrt{3}}{10} 
\int_{\mathbb{R}^{3}} 
\frac{e^{- \sqrt{\alpha_{\omega}^{(1)}}|x|}}
{|x|} 
\int_{0}^{1} 
V_{n}^{p}(x, \theta) d\theta(x) \, dx + O(1) \\[6pt]
& \leq 
- \frac{\sqrt{3}(1 - \varepsilon)^{p}}{10} \times 
4 \pi
\int_{R_{\e}}^{\infty} 
\frac{e^{- \sqrt{\alpha_{\omega}^{(1)}}r}}{r} 
Y_{\omega}^{p}(r)\times r^{2}\, dr + O(1) \\[6pt]
& = - \frac{2(1 - \varepsilon)^{p} \pi}{5} 3^{\frac{p + 1}{2}}
\int_{R_{\e}}^{\infty} 
\frac{e^{- (p + 1) 
\sqrt{\alpha_{\omega}^{(1)}}r}}{r^{p -1}} \, dr 
+ O(1)
\\[6pt]
& = - \frac{2(1 - \varepsilon)^{p} \pi}{5} 3^{\frac{p + 1}{2}} 
(\alpha_{\omega}^{(1)})^{\frac{p - 2}{2}}
\int_{R_{\e} 
\sqrt{\alpha_{\omega}^{(1)}}}^{\infty} 
\frac{e^{- (p + 1) s}}{s^{p -1}} \, ds 
+ O(1).
\end{split}
\end{equation*}
\[
(\alpha_{n}^{(1)})^{\frac{2 - p}{2}}
\langle (- \Delta + \alpha_{n}^{(1)})^{-1} 
\int_{0}^{1} V_{n}^{p}(\cdot, \theta) d\theta, 
W^{4} \Lambda W \rangle 
\leq - \frac{2(1 - \varepsilon)^{p} \pi}{5} 3^{\frac{p + 1}{2}} 
\int_{R_{\e} 
\sqrt{\alpha_{\omega}^{(1)}}}^{\infty} 
\frac{e^{- (p + 1) s}}{s^{p -1}} \, ds 
+ O((\alpha_{\omega}^{(1)}
)^{\frac{2 - p}{2}}).
\]
Thus, we have 
\[
\limsup_{n \to \infty} 
(\alpha_{n}^{(1)})^{\frac{2 - p}{2}}
\langle (- \Delta + \alpha_{n}^{(1)})^{-1} 
\int_{0}^{1} V_{n}^{p}(\cdot, \theta) d\theta, 
W^{4} \Lambda W \rangle 
\leq - \frac{2(1 - \varepsilon)^{p} \pi}{5} 3^{\frac{p + 1}{2}} 
\int_{0}^{\infty} 
\frac{e^{- (p + 1) s}}{s^{p -1}} \, ds. 
\]
Since $\e> 0$ is arbitrary, 
we see that \eqref{eqU-16-2} holds. 
This concludes the proof. 
\end{proof}

\subsection{Proof of Theorem \ref{class}}
By using Theorems \ref{thm-0} \textrm{(i)} and \ref{thm-bl} \textrm{(i)}, 
we shall show Theorem \ref{class}. 
\begin{proof}[Proof of Theorem \ref{class}]
Suppose to the contrary that 
there exists $\{\omega_{n}\}$ in $(0, \infty)$ 
with $\lim_{n \to \infty} \omega_{n} = 0$ such that 
for each $n \in \mathbb{N}$, 
$u_{n}^{(i)}\; (i = 1, 2, 3)$ is a solution 
to \eqref{sp} with 
$\omega = \omega_{n}$ satisfying 
$u_{n}^{(i)} \neq u_{n}^{(j)}\; (i \neq j)$. 
We first consider the following case:
\begin{quote}
\textbf{(Case 1). 
} 
$\limsup_{n \to \infty} \|u_{n}^{(i)}\|_{L^{\infty}} < \infty$ 
for $i = 1, 2$. 
\end{quote}
We can verify from 
(Case 1) does not occur from Theorem \ref{thm-0} 
\textrm{(i)}. 
Next, we consider the following case: 
\begin{quote}
\textbf{(Case 2). 
} 
There exists a subsequence of $\{u_{n}^{(i)}\}\; (i = 1, 2)$ 
(we still denote it 
by the same symbol) such that 
$\lim_{n \to \infty} \|u_{n}^{(i)}\|_{L^{\infty}} = \infty$ 
for $i = 1, 2$.
\end{quote}
Then, we can deny the possibility of (Case 2) 
by Theorem \ref{thm-bl} \textrm{(i)}. 
Finally, we consider the following case:
\begin{quote}
\textbf{(Case 3). 
} 
There exists a subsequence 
$\{u_{n_{k}}^{(i)}\}$
of $\{u_{n}^{(i)}\}\; (i = 1, 2)$ 
(we still denote it 
by the same symbol) such that 
$\lim_{k \to \infty} \|u_{n_{k}}^{(1)}\|_{L^{\infty}} = \infty$ 
and $\limsup_{k \to \infty} \|u_{n_{k}}^{(2)}\|_{L^{\infty}} < \infty$. 
\end{quote}
Then, we see that either 
$\lim_{k \to \infty} \|u_{n_{k}}^{(3)}\|_{L^{\infty}} = \infty$
or $\limsup_{k \to \infty} \|u_{n_{k}}^{(3)}\|_{L^{\infty}} < \infty$ occurs. 
Then, 
using Theorems \ref{thm-0} \textrm{(i)} and \ref{thm-bl} \textrm{(i)}, 
we can show that neither of the two cases occurs 
by a similar argument as above. 
\end{proof}

\section{Non-degeneracy of the large solution $u_{\omega}$}
\label{sec-nd2}
In this section, we shall show the non-degeneracy of 
the positive solution $u_{\omega}$ 
for sufficiently small $\omega > 0$ 
(Theorem \ref{thm-bl} \textrm{(ii)}). 
The proof is similar to that of the local uniqueness of $u_{\omega}$ 
(Theorem \ref{thm-bl} \textrm{(i)}). 
We shall show the non-degeneracy of 
the positive solution $u_{\omega}$ by contradiction. 
Suppose to the contrary 
that there exist a sequence $\{\omega_{n}\}$ with 
$\lim_{n \to \infty} \omega_{n} = 0$ and 
$z_{n} \in H^{1}_{\text{rad}}(\R^{3}) 
\setminus \{0\}$ such that 
\begin{equation}\label{EqND-11}
(- \Delta + \omega_{n} 
- p u_{\omega_{n}}^{p-1} - 
5 u_{\omega_{n}}^{4}) z_{n} 
= 0 \qquad \mbox{in 
$\R^{3}$}. 
\end{equation}
By the elliptic regularity, we see that 
$z_{n} \in C^{2}(\R^{3})$. 
We put 
\begin{equation} \label{EqND-1}
M_{n} := \|u_{\omega_{n}}\|_{L^{\infty}}, \qquad 
\widetilde{u}_{n}(x) := M_{n}^{-1} 
u_{\omega_{n}}(M_{n}^{-2} x), \qquad 
\widetilde{z}_{n} := z_{n}(M_{n}^{-2} \cdot). 
\end{equation}
Then, $\widetilde{z}_{n}$ 
satisfies
\begin{equation}\label{EqND-2}
(- \Delta + \alpha_{n}
- p \beta_{n}
\widetilde{u}_{n}
^{p-1} - 5 \widetilde{u}_{n}^{4}) 
\widetilde{z}_{n} = 0 
\qquad \mbox{in $\R^{3}$}, 
\end{equation}
where 
\begin{equation*}
\alpha_{n} := \omega_{n} M_{n}^{-4}, \qquad 
\beta_{n} := M_{n}^{p-5}.
\end{equation*}
Since equation \eqref{EqND-2} is linear, 
we may assume that 
$\|\nabla \widetilde{z}_{n}\|_{L^{2}} = 1$ 
for all $n \in \mathbb{N}$. 

Then, as in \cite[Page 22]{MR3964275}, we obtain the following lemma: 
\begin{lemma}\label{nd-lem0}
Let $1 < p < 3$. 
Then, there exists $C_{1}>0$ such that for any 
$x \in \R^{d}$ and $n\in \mathbb{N}$,
\begin{equation*}
|\widetilde{z}_{n}(x)| \leq C_{1} |x|^{-1}. 
\end{equation*} 
\end{lemma}
It follows from $\|\nabla z_{n}\|_{L^{2}} = 1$ for all $n \in \N$ that 
there exists a subsequence of 
$\{\widetilde{z}_{n}\}$ 
(we still denote it by the same letter) 
and a function $\widetilde{z}_{\infty} 
\in \dot{H}^{1}_{\text{rad}}(\R^{3})$ 
such that 
\[
\lim_{n \to \infty} 
\widetilde{z}_{n} = 
\widetilde{z}_{\infty} 
\qquad 
\mbox{weakly in $\dot{H}^{1}
(\R^{3})$ and strongly in $L^{q}(\R^{3})$\; ($3 < q < 6$)}. 
\]
Since $\lim_{n \to \infty} 
\widetilde{u}_{n} = W$ strongly in 
$\dot{H}^{1} \cap L^{q}(\R^{3})\; (3 < q < 6)$ 
and $\lim_{n \to \infty} \alpha_{n} 
= \lim_{n \to \infty} \beta_{n} = 0$, 
we see that 
\begin{equation*}
\begin{split}
\langle (-\Delta - 5 W^{4}) 
\widetilde{z}_{\infty}, \phi \rangle 
& = 
\langle (-\Delta - 5 W^{4}) 
\widetilde{z}_{\infty}, \phi \rangle 
- \langle
(- \Delta + \alpha_{n}
- \beta_{n}
\widetilde{u}_{n}
^{p-1} - 5 \widetilde{u}_{n}^{4}) 
\widetilde{z}_{n}, \phi
\rangle \\[6pt]
& = \left(\langle (-\Delta - 5 W^{4}) 
\widetilde{z}_{\infty}, \phi \rangle
- \langle
(- \Delta - 5 \widetilde{u}_{n}^{4}) 
\widetilde{z}_{n}, \phi
\rangle \right)
+ \langle 
(\alpha_{n}
- \beta_{n}
\widetilde{u}_{n}
^{p-1}) \widetilde{z}_{n}, \phi
\rangle \\[6pt]
& \to 0 \qquad \mbox{as $n \to \infty$} 
\end{split}
\end{equation*}
for any $\phi \in C_{0}^{\infty}
(\R^{3})$.
This implies that 
\begin{equation} \label{EqND-3}
(-\Delta - 5 W^{4}) 
\widetilde{z}_{\infty} = 0 
\qquad 
\mbox{in $H^{-1}(\R^{3})$}. 
\end{equation}

Next, similarly to \cite{MR3964275}, 
we can derive the following identity:
\begin{lemma}[c.f. Page 22 of \cite{MR3964275}]
\label{nd-lem2}
Let $1 < p < 3$, 
$\widetilde{u}_{n}$ and $\widetilde{z}_{n}$ be 
the functions defined by \eqref{EqND-1}. 
Then, we have
\begin{equation}\label{EqND-4}
\frac{\alpha_{n}}
{\beta_{n}} 
\int_{\R^{3}} 
\widetilde{u}_{n} 
\widetilde{z}_{n} \, dx = 
\frac{5 - p}{4}
\int_{\R^{3}} 
\widetilde{u}_{n}^{p} 
\widetilde{z}_{n} \, dx 
\qquad 
\mbox{for all $n \in \mathbb{N}$}. 
\end{equation}
\end{lemma}

We rewrite \eqref{EqND-4} as follows:
\begin{equation}\label{EqND-7}
\sqrt{\alpha_{n}}
\int_{\R^{3}} 
\widetilde{u}_{n} 
\widetilde{z}_{n} \, dx = 
\frac{5 - p}{4}
\frac{\beta_{n}}{\sqrt{\alpha_{n}}} 
\int_{\R^{3}} 
\widetilde{u}_{n}^{p} 
\widetilde{z}_{n} \, dx. 
\end{equation}
\subsection{Case of $2 \leq p < 3$}
In this subsection, we will give the proof of Theorem 
\ref{thm-bl} \textrm{(ii)} in the case of $2 \leq p < 3$.  
By a similar argument in the proof of Lemma \ref{lem-uni7}, 
we can prove the following: 
\begin{lemma}\label{nd-lem1}
Let $2 \leq p < 3$. 
$\widetilde{z}_{\infty} \neq 0$. 
\end{lemma}

It follows from \eqref{EqND-3} and Lemma \ref{nd-lem1} that 
$\widetilde{z}_{\infty} = 
\kappa \Lambda W$ for some 
$\kappa \in \R \setminus \{0\}$. 
Without loss of generality, we may assume that 
$\kappa > 0$.~\footnote{Otherwise, we replace 
$\widetilde{z}_{n}$ by $- \widetilde{z}_{n}$.}
Then, we shall derive a contradiction 
if we obtain the following propositions:
\begin{proposition}\label{prop-nd1}
Let $2 \leq p < 3$ and 
$\widetilde{z}_{n}$ be 
the functions defined by \eqref{EqND-1}. 
Then, we have 
\begin{equation}\label{EqND-8}
\liminf_{n \to \infty} \sqrt{\alpha_{n}} 
\int_{\R^{3}} 
\widetilde{u}_{n} 
\widetilde{z}_{n} \, dx 
\geq - 3 \pi \kappa. 
\end{equation}
\end{proposition}

\begin{proposition} \label{prop-nd2}
Let $2 \leq p < 3$ and 
$\widetilde{z}_{n}$ be 
the functions defined by \eqref{EqND-1}. 
Then, we obtain
\begin{equation} \label{EqND-10}
\lim_{n \to \infty} 
\frac{\beta_{n}}{\sqrt{\alpha_{n}}} 
\int_{\R^{3}} 
\widetilde{u}_{n}^{p} 
\widetilde{z}_{n} \, dx = - 6 \pi\kappa.
\end{equation}
\end{proposition}
We can prove Propositions \ref{prop-nd1} and \ref{prop-nd2} by similar arguments to those 
of Propositions \ref{propR-2} and 
\ref{propR-3}. Thus, we omit the proof. 
We now prove the non-degeneracy of 
$\widetilde{u}_{n}$. 
\begin{proof}[Proof of Theorem \ref{thm-bl} \textrm{(ii)} 
in the case of $2 \leq p < 3$]
It follows from \eqref{EqND-7}, \eqref{EqND-8} and 
\eqref{EqND-10} that 
\[
- 3 \pi \kappa
\leq \lim_{n \to \infty}
\sqrt{\alpha_{n}}
\int_{\R^{3}} 
\widetilde{u}_{n} 
\widetilde{z}_{n} \, dx = 
\frac{5 - p}{4}
\frac{\beta_{n}}{\sqrt{\alpha_{n}}} 
\int_{\R^{3}} 
\widetilde{u}_{n}^{p} 
\widetilde{z}_{n} \, dx
= - \frac{3 (5 - p)}{2} \pi \kappa, 
\]
which implies $p \geq 3$ because $\kappa > 0$.
However, this contradicts 
our assumption $p < 3$. 
This completes the proof. 
\end{proof}
\subsection{Case of $1 <p < 2$}
In this subsection, we will prove Theorem 
\ref{thm-bl} \textrm{(ii)} in the case of $1 < p < 2$.  
We can obtain the following propositions: 
\begin{proposition}\label{prop-nd1-2}
Let $1 < p < 2$ and 
$\widetilde{z}_{n}$ be 
the functions defined by \eqref{EqND-1}. 
\begin{equation}\label{EqND-8-2}
\liminf_{n \to \infty} \sqrt{\alpha_{n}} 
\left\langle \widetilde{u}_{n}, 
\widetilde{z}_{n} \right\rangle = 0. 
\end{equation}
\end{proposition}

\begin{proposition} \label{prop-nd2-2}
Let $1 < p < 2$ and $\widetilde{z}_{n}$ be 
the functions defined by \eqref{EqND-1}. 
We have
\begin{equation} \label{EqND-10-2}
\lim_{n \to \infty} 
\frac{\beta_{n}}{\sqrt{\alpha_{n}}} 
\left\langle 
\widetilde{u}_{n}^{p}, 
\widetilde{z}_{n} \right\rangle 
< - \frac{2 \pi}{5} 3^{\frac{p + 1}{2}} 
\int_{0}^{\infty} 
\frac{e^{- (p + 1) s}}{s^{p -1}} \, ds.
\end{equation}
\end{proposition}
Since the proof of Propositions 
\ref{prop-nd1-2} and \ref{prop-nd2-2} are 
similar to those of Propositions 
\ref{propR-12} and \ref{propR-3-2}, 
we omit them. 

\begin{proof}[Proof of 
Theorem \ref{thm-bl} \textrm{(ii)} 
in the case of $1 < p < 2$]
Suppose to the contrary 
that there exist a sequence $\{\omega_{n}\}$ with 
$\lim_{n \to \infty} \omega_{n} = 0$ and 
$z_{n} \in H^{1}_{\text{rad}}(\R^{3}) 
\setminus \{0\}$ satisfying \eqref{EqND-11}. 
Then, we can derive a contradiction 
from \eqref{EqND-7}, \eqref{EqND-8-2} 
and \eqref{EqND-10-2}. 
This completes the proof. 
\end{proof}


\section{Uniqueness of minimizer for $E_{\min}(m)$ and 
$E_{\min, -}(m)$}
\label{sec-uniR}
In this section, we shall obtain a uniqueness of minimizer 
$R_{\omega(m)}$ for $E_{\min}(m)$ and $R_{\omega(m), -}$ for $E_{\min, -}(m)$, 
where $E_{\min}(m)$ and $E_{\min, -}(m)$ 
are defined by \eqref{eq-mini-e} and \eqref{mini-sec-mp}, 
respectively. 

The uniqueness is required to study the continuity of 
$R_{\omega(m)}$ (see Proposition \ref{lem2-diff} below).
We shall show the following: 
\begin{theorem}\label{thm-muni}
\begin{enumerate}
\item[\textrm{(i)}]
Let $\frac{7}{3} \leq p < 3$. 
There exists $m_{1} > 0$ such that 
the 
minimizer $R_{\omega(m)}$ for $E_{\min}(m)$
is unique in $H^{1}_{\text{rad}}(\R^{d})$ 
for $m \in (0, m_{1})$. 
\item[\textrm{(ii)}] 
Let $1 < p < \frac{7}{3}$. 
There exists $m_{2} > 0$ such that 
the 
minimizer $R_{\omega(m), -}$ for $E_{\min, -}(m)$
is unique in $H^{1}_{\text{rad}}(\R^{d})$ 
for $m \in (0, m_{2})$. 
\end{enumerate}
\end{theorem} 
\begin{remark}\label{re-muni1}
It seems that 
Theorem \ref{thm-muni} does not follow from 
Theorem \ref{thm-bl} \textrm{(i)} 
because it is non-trivial whether 
the Lagrange multiplier $\omega(m)$ is determined uniquely or not. 
\end{remark}
In what follows, when $1 < p < 1 + \frac{4}{d}$, 
we write $R_{\omega(m), -}$ and $\widetilde{\omega}(m)$ 
simply as $R_{\omega(m)}$ and 
$\omega(m)$ for 
the sake of clarity.  
\begin{remark}\label{rem-u2}
Since $\lim_{m \to 0}\omega(m) = 0, \lim_{m \to 0} 
\|R_{\omega(m)}\|_{L^{\infty}} = \infty$ and 
$R_{\omega(m)}$ satisfies \eqref{sp}, we see from 
Theorem \ref{thm-bl} \textrm{(i)} that $R_{\omega(m)} 
= u_{\omega(m)} 
= u_{2, \omega(m)}$ 
for sufficiently small $m>0$. 
\end{remark}
\subsection{Case of $2 < p < 3$}
\subsubsection{Relation between $\|R_{\omega(m)}\|_{L^{q}}$ and $m$}
In this subsection, we shall study the 
relation between $\|R_{\omega(m)}\|_{L^{q}}\; (q>3)$ and 
the parameter $m$ to prove Theorem \ref{thm-muni} in the case of 
$2 < p < 3$. 
More precisely, we shall show the following: 
\begin{proposition}\label{l-muni1}
Let $2 < p < 3$ and 
$\omega(m) > 0$ be the Lagrange multiplier 
given in Theorems \ref{thm-ex-1} and \ref{thm-ex2} \textrm{(ii)}. 
Putting $M_{m} = \|R_{\omega(m)}\|_{L^{\infty}}$, 
we have the following: 
\begin{enumerate}
\item[\textrm{(i)}] 
\begin{equation} \label{e-muni1}
\lim_{m \to 0} m M_{m}^{p-1}
= 6 \pi C_{p}^{-1}, 
\end{equation}
where 
\begin{equation}\label{e-muni2}
C_{p} := \frac{5 - p}{12 \pi (p+1)} \|W\|_{L^{p+1}}^{p+1}.
\end{equation}
\item[\textrm{(ii)}]
\begin{equation} \label{e-muni3}
\lim_{m \to 0} \frac{\omega(m)}{M_{m}^{2p-6}} 
= C_{p}^{2}. 
\end{equation}
\item[\textrm{(iii)}]
Let $q > 3$. 
\begin{equation} \label{e-muni4}
\lim_{m \to 0} \frac{\|R_{\omega(m)}\|_{L^{q}}}
{m^{\frac{6-q}{q(p-1)}}} 
= (6 \pi C_{p}^{-1})^{- \frac{6 - q}{q(p - 1)}} 
\|W\|_{L^{q}}.
\end{equation}
\item[\textrm{(iv)}]
\begin{equation}\label{e-muni5}
\lim_{m \to 0} \frac{\omega(m)}{m^{\frac{6 - 2p}
{p-1}}} = 
(6 \pi)^{\frac{2p-6}{p-1}} 
C_{p}^{\frac{4}{p-1}}.
\end{equation}
\end{enumerate}
\end{proposition}
\begin{proof}[Proof of Proposition \ref{l-muni1}]
\textrm{(i)}
We define $\widetilde{R}_{m}$ by 
\begin{equation} \label{e-muni6} 
\widetilde{R}_{m}(\cdot )
:= M_{m}^{-1} R_{\omega(m)}(M_{m}^{-2}\cdot).
\end{equation}
Then, we can verify that $\widetilde{R}_{m}$ satisfies 
\begin{align}
\label{e-muni7}
- \Delta \widetilde{R}_{m} + \alpha_{m} \widetilde{R}_{m}
- \beta_{m} \widetilde{R}_{m}^{p} 
- \widetilde{R}_{m}^{5}
= 0, 
\qquad 
\|\widetilde{R}_{m}\|_{L^{\infty}} =1, 
\end{align}
where 
\begin{equation}\label{e-muni8}
\alpha_{m} := \alpha_{\omega(m)} = \omega(m) M_{m}^{-4}, \qquad 
\beta_{m} := \beta_{\omega(m)} = M_{m}^{p-5}.
\end{equation}
We see from Remark \ref{rem-u2} and 
Theorem \ref{thm-bl-0} that 
\begin{equation} \label{e-muni9}
\lim_{m \to 0} \widetilde{R}_{m} = W 
\qquad \mbox{in $\dot{H} \cap L^{q}(\R^{3})\; (q > 3)$}. 
\end{equation} 
Applying \eqref{main-eq7} to 
$\widetilde{u}_{\omega(m)} = \widetilde{R}_{m}$ that 
\begin{equation} \label{e-muni10}
\|\widetilde{R}_{m}\|_{L^{2}}^{2} 
= \frac{5 - p}{2(p+1)} 
\frac{\beta_{m}}{\alpha_{m}}
\|\widetilde{R}_{m}\|_{L^{p+1}}^{p+1}. 
\end{equation}
This together with $\beta_{m} = M_{m}^{p-5}$ 
and \eqref{e-muni6}
yields that 
\begin{equation}\label{e-muni11}
\begin{split}
m = \|R_{\omega(m)}\|_{L^{2}}^{2} 
= M_{m}^{-4} \|\widetilde{R}_{m}\|_{L^{2}}^{2} 
= M_{m}^{-4} \times \frac{5 - p}{2(p+1)} 
\frac{\beta_{m}}{\alpha_{m}}
\|\widetilde{R}_{m}\|_{L^{p+1}}^{p+1} 
= \frac{5 - p}{2(p+1)}
M_{m}^{1 - p}
\frac{\beta_{m}^{2}}{\alpha_{m}}
\|\widetilde{R}_{m}\|_{L^{p+1}}^{p+1}. 
\end{split}
\end{equation}
In addition, one has by \eqref{EqB-1} that 
\[
\lim_{m \to 0} \frac{\beta_{m}}
{\sqrt{\alpha_{m}}} = C_{p}^{-1},  
\]
where $C_{p}$ is given by \eqref{e-muni2}. 
This together with \eqref{EqB-1}, \eqref{e-muni2} 
and \eqref{e-muni9} yields 
\[
\lim_{m \to 0} m M_{m}^{p-1}
= \lim_{m \to 0} 
\frac{5 - p}{2(p+1)}
\frac{\beta_{m}^{2}}{\alpha_{m}}
\|\widetilde{R}_{m}\|_{L^{p+1}}^{p+1} 
= \frac{5 - p}{2(p+1)} C_{p}^{- 2} \|W\|_{L^{p+1}}^{p+1} 
= 6 \pi C_{p}^{-1}. 
\]
Thus, \eqref{e-muni1} holds. 

\textrm{(ii)}
It follows from 
\eqref{EqB-1}, \eqref{e-muni2}, 
$\alpha_{m} = \omega(m) M_{m}^{-4}$ and 
$\beta_{m} = M_{m}^{p-5}$ that 
\[
C_{p} = \lim_{m \to 0} 
\frac{\sqrt{\alpha_{m}}}{\beta_{m}} 
= 
\lim_{m \to 0} \frac{\sqrt{\omega(m)}
M_{m}^{-2}}{M_{m}^{p-5}} 
= \lim_{m \to 0} \frac{\sqrt{\omega(m)}}{
M_{m}^{p-3}}. 
\]
Thus, \eqref{e-muni3} immediately holds. 

\textrm{(iii)}
Since $R_{\omega(m)}(x) = M_{m} \widetilde{R}_{m}(M_{m}^{2} x)$ for 
$x \in \R^{3}$, we have 
\[
\|R_{\omega(m)}\|_{L^{q}}^{q} 
= M_{m}^{- (6-q)} 
\|\widetilde{R}_{m}\|_{L^{q}}^{q}. 
\]
This together with \eqref{e-muni1} and 
\eqref{e-muni9} yields that 
\[
\begin{split}
\lim_{m \to 0} 
\frac{\|R_{\omega(m)}\|_{L^{q}}}{m^{\frac{6-q}{q(p-1)}}} 
= \lim_{m \to 0} 
\dfrac{M_{m}^{\frac{- (6-q)}{q}} \|\widetilde{R}_{m}\|_{L^{q}}}
{m^{\frac{6-q}{q(p-1)}}}
= \lim_{m \to 0} (m M_{m}^{p-1})^{\frac{- (6-q)}{q(p-1)}}
\|\widetilde{R}_{m}\|_{L^{q}} 
= (6 \pi C_{p}^{-1})^{- \frac{6 - q}{q(p - 1)}} 
\|W\|_{L^{q}}.
\end{split}
\]

\textrm{(iv)}
We see from \eqref{e-muni1} and 
\eqref{e-muni3} that 
\begin{equation} 
\begin{split}
\lim_{m \to 0} \frac{\omega(m)}{m^{\frac{6 - 2p}
{p-1}}} = \lim_{m \to 0} 
\frac{\omega(m)}{M_{m}^{2p-6}} 
\frac{M_{m}^{2p-6}}
{m^{- \frac{2p - 6}{p-1}}} 
= 
\lim_{m \to 0}\frac{\omega(m)}{M_{m}^{2p-6}} 
\lim_{m \to 0} 
(m M_{m}^{p-1})^{\frac{2p-6}{p-1}} 
= C_{p}^{2} \times (6 \pi 
C_{p}^{-1})^{\frac{2p-6}{p-1}} 
= (6 \pi)^{\frac{2p-6}{p-1}} 
C_{p}^{\frac{4}{p-1}}. 
\end{split}
\end{equation}
This completes the proof. 
\end{proof}
We also show the following:

\begin{lemma}\label{l-muni2}
Let $q > 1$. 
Then, we obtain
$\|R_{\omega(m_{n})}\|_{L^{q}}^{q} 
\lesssim M_{n, j}^{pq - 4q - 3p + 9}$. 
\end{lemma}
\begin{proof}
By \eqref{e-muni3}, we obtain 
\begin{equation} \label{e-muni19}
\alpha_{m} = \omega(m) M_{m}^{-4} 
\sim M_{m}^{2p -10} 
\qquad \mbox{as $m \to 0$}. 
\end{equation}
Since $R_{\omega(m)}(x) = M_{m} \widetilde{R}_{m}(M_{m}^{2} x)$ for 
$x \in \R^{3}$, we have 
\begin{equation} \label{e-muni18}
\|R_{\omega(m)}\|_{L^{q}}^{q} 
= M_{m}^{- (6-q)} 
\|\widetilde{R}_{m}\|_{L^{q}}^{q}. 
\end{equation}
Then, it follows from \eqref{EqL-1}, 
\eqref{eqU-6-3} and \eqref{e-muni19} 
that 
\[
\begin{split}
\|\widetilde{R}_{m}\|_{L^{q}}^{q} 
& = 4 \pi \int_{0}^{1} 
\widetilde{R}_{m}^{q}(r) r^{2} \, dr 
+ 4 \pi \int_{1}^{\alpha_{m}^{- \frac{1}{2}}} 
\widetilde{R}_{m}^{q}(r) r^{2} \, dr 
+ 4 \pi \int_{\alpha_{m}^{- \frac{1}{2}}}
^{\infty} 
\widetilde{R}_{m}^{q}(r) r^{2} \, dr \\
& \lesssim 1 + 
\int_{1}^{\alpha_{m}^{- \frac{1}{2}}} 
r^{2- q} \, dr 
+ \int_{\alpha_{m}^{- \frac{1}{2}}}
^{\infty} 
r^{2- q} e^{- q \sqrt{\alpha_{m}}r}\, dr \\
& \lesssim 1 + \alpha_{m}^{- 
\frac{1}{2}(3 - q)} 
+ \alpha_{m}^{- 
\frac{1}{2}(3 - q)} \int_{1}^{\infty} 
s^{2 - q} e^{- q s} \, ds 
\lesssim \alpha_{m}^{- 
\frac{1}{2}(3 - q)} 
\sim M_{m}^{pq -3 p - 5 q + 15}. 
\end{split}
\] 
Thus, by \eqref{e-muni18}, we obtain 
\[
\|R_{\omega(m)}\|_{L^{q}}^{q} 
= M_{m}^{- (6-q)} 
\|\widetilde{R}_{m}\|_{L^{q}}^{q} 
\lesssim M_{m}^{pq - 4q - 3p + 9}. 
\]
This completes the proof. 
\end{proof}

\subsubsection{Proof of Theorem \ref{thm-muni} 
in the case of $2 < p < 3$}
This subsection is devoted to the proof of 
Theorem \ref{thm-muni}. 
We shall show this by contradiction. 
Suppose to the contrary that there exists a 
sequence $\{m_{n}\}$ in $(0, \infty)$ such that 
$\lim_{n \to \infty} m_{n} = 0$ and for each 
$n \in \mathbb{N}$, 
$E_{\min}(m_{n})$ has two minimizers 
$R_{\omega_{1}(m_{n})}, 
R_{\omega_{2}(m_{n})}$ with 
$R_{\omega_{1}(m_{n})} 
\neq R_{\omega_{2}(m_{n})}$.
For $j = 1, 2$, we set 
\begin{equation*}
\begin{split}
& 
M_{n, j} = R_{\omega_{j}(m_{n})}(0), 
\qquad 
\widetilde{R}_{n, j}(x) = M_{n, j}^{-1} 
R_{\omega_{j}(m_{n})}(M_{n, j}^{-2} x), \\[6pt]
& 
\alpha_{n, j} = \omega_{j}(m_n)M_{n, j}^{-4}, 
\qquad 
\beta_{n, j} = M_{n, j}^{p-5}. 
\end{split}
\end{equation*}
By \eqref{e-muni1} and 
\eqref{e-muni3}, we see that 
\[
\lim_{n \to \infty} \frac{M_{n, 1}}{M_{n, 2}} = 
\lim_{n \to \infty} \frac{\omega_{1}(m_{n})}{
\omega_{2}(m_{n})} 
= 1, 
\]
which implies 
\[
\lim_{n \to \infty} \frac{\alpha_{n, 1}}{\alpha_{n, 2}} 
= \lim_{n \to \infty} \frac{\beta_{n, 1}}{\beta_{n, 2}} = 1.
\]
Let 
\begin{align}
& 
Z_{n} = \dfrac{R_{\omega_{1}(m_{n})} - R_{\omega_{2}(m_{n})}}
{\|R_{\omega_{1}(m_{n})} - R_{\omega_{2}(m_{n})}\|_{L^{\infty}} + |\omega_{1}(m_{n}) 
- \omega_{2}(m_{n})|}, 
\label{e-muni13}\\[6pt] 
& 
\kappa_{n} = 
\dfrac{\omega_{1}(m_{n}) - \omega_{2}(m_{n})}
{\|R_{\omega_{1}(m_{n})} - R_{\omega_{2}(m_{n})}\|_{L^{\infty}} + |\omega_{1}(m_{n}) 
- \omega_{2}(m_{n})|}. 
\label{e-muni14}
\end{align}
Clearly, we have 
\begin{equation} \label{e-muni15}
1 = \|Z_{n}\|_{L^{\infty}} + |\kappa_{n}|
\end{equation}
We can verify that the pair $(Z_{n}, \kappa_{n})$ satisfies 
\begin{equation}\label{e-muni16}
\begin{split}
- \Delta Z_{n} + \omega_{1}(m_{n}) Z_{n} 
- p \int_{0}^{1} V_{n}^{p-1}(x, \tau) d\tau Z_{n} 
- 5 \int_{0}^{1} V_{n}^{4}(x, \tau) d \tau Z_{n} 
= - \kappa_{n} R_{\omega_{2}(m_{n})}, 
\end{split}
\end{equation}
where 
\begin{equation} \label{e-muni17}
V_{n}(x, \tau) = \tau 
R_{\omega_{1}(m_{n})} (x) + (1 - \tau)
R_{\omega_{2}(m_{n})}(x) \qquad 
(\tau \in (0, 1), \; x \in \R^{3}). 
\end{equation}
We shall show the following: 
\begin{lemma}\label{lem-uni2-3-2}
Let $2 < p < 3$. 
Then, There exists a constant $C > 0$ 
such that 
$|\kappa_{n}| \leq 
C M_{n}^{p^{2} - \frac{11}{2} p 
+ \frac{13}{2}} \|\nabla Z_{n}\|_{L^{2}}$ 
for all $n \in \mathbb{N}$. 
\end{lemma}
\begin{proof}
By a similar argument of \eqref{main-eq7}, 
we see from the Pohozaev identity of $R_{\omega_{j}(m_{n})}$ that 
\[
\omega_{j}(m_{n}) m_{n} 
= \omega_{j}(m_{n})\|R_{\omega_{j} 
(m_{n})}\|_{L^{2}}^{2}
= \frac{5 - p}{2(p+1)} \|R_{\omega_{j} 
(m_{n})}\|_{L^{p+1}}^{p+1} \qquad 
\mbox{for $j = 1, 2$}. 
\]
Taking a difference, we obtain 
\[
(\omega_{1}(m_{n}) - \omega_{2}(m_{n})) m_{n} 
= \frac{5 - p}{2(p+1)} \left(\|R_{\omega_{1} 
(m_{n})}\|_{L^{p+1}}^{p+1} - \|R_{\omega_{2} 
(m_{n})}\|_{L^{p+1}}^{p+1}\right). 
\]
This together with 
\eqref{e-muni1} and 
\eqref{e-muni13} yields that 
\begin{equation} \label{e-muni20}
\begin{split}
|\kappa_{n}| 
& = 
\dfrac{|\omega_{1}(m_{n}) - \omega_{2}(m_{n})|}
{\|R_{\omega_{1}(m_{n})} - R_{\omega_{2}(m_{n})}\|_{L^{\infty}} + |\omega_{1}(m_{n}) 
- \omega_{2}(m_{n})|}\\[6pt]
& \lesssim \frac{1}{m_{n}} 
\int_{\R^{3}} (|R_{\omega_{1} (m_{n})}|^{p} 
+ |R_{\omega_{2} (m_{n})}|^{p})|Z_{n}| \, dx \\
& \lesssim 
\frac{1}{m_{n}} \left(
\|R_{\omega_{1} (m_{n})}\|_{L^{\frac{6}{5} p}}^{p} 
+ \|R_{\omega_{2} (m_{n})}\|_{L^{\frac{6}{5} p}}^{p} \right)\|Z_{n}\|_{L^{6}} \\
& \lesssim M_{n}^{p - 1} 
M_{n}^{\frac{5}{6} \left(\frac{6}{5}
p^{2} - \frac{39}{5} p + 9 \right)} 
\|\nabla Z_{n}\|_{L^{2}}\\
& \lesssim M_{n}^{p^{2} - \frac{11}{2} p 
+ \frac{13}{2}} \|\nabla Z_{n}\|_{L^{2}}. 
\end{split}
\end{equation}
\end{proof}

\begin{lemma}\label{lem-uni2-2}
Let $2 < p < 3$. 
Then, there exists a constant $C > 0$ 
such that 
$\|\nabla Z_{n}\|_{L^{2}}^{2} \leq C m_{n}^{\frac{2}{p-1}}$ for all 
$n \in \mathbb{N}$. 
\end{lemma}

\begin{proof}
Multiplying \eqref{e-muni16} by $Z_{n}$ and integrating 
the resulting equation on $\R^{3}$, we have 
\begin{equation} \label{e-muni21}
\begin{split}
\|\nabla Z_{n}\|_{L^{2}}^{2} + 
\omega_{1}(m_{n}) \|Z_{n}\|_{L^{2}}^{2} 
& = p \int_{\R^{3}}\int_{0}^{1} V_{n}^{p-1}(x, \tau) 
d\tau |Z_{n}|^{2} \, dx 
+ 5 \int_{\R^{3}} \int_{0}^{1} V_{n}^{4}(x, \tau) 
d \tau |Z_{n}|^{2} \, dx \\[6pt]
& \quad - \kappa_{n} \int_{\R^{3}} 
R_{\omega_{2}(m_{n})} Z_{n} \, dx
\end{split}
\end{equation}
By the H\"{o}lder and Young inequalities, 
$ \|R_{\omega_{2}(m_{n})}\|_{L^{2}} = m_{n}^{1/2}$, 
\eqref{e-muni1}, Lemma 
\ref{lem-uni2-3-2} and \eqref{e-muni3}, 
we obtain
\begin{equation*} 
\begin{split}
\biggl|\kappa_{n} \int_{\R^{3}} R_{\omega_{2}(m_{n})} Z_{n} \, dx \biggl| 
\leq |\kappa_{n}| \|R_{\omega_{2}(m_{n})}\|_{L^{2}}
\|Z_{n}\|_{L^{2}} 
& \leq m_{n}^{\frac{1}{2}} 
M_{n}^{p^{2} - \frac{11}{2} p 
+ \frac{13}{2}} \|\nabla Z_{n}\|_{L^{2}}
\|Z_{n}\|_{L^{2}} \\[6pt]
& 
\lesssim 
M_{n}^{p^{2} - 6p + 7}
\|\nabla Z_{n}\|_{L^{2}}
\|Z_{n}\|_{L^{2}} \\[6pt]
& \lesssim 
\delta \|\nabla Z_{n}\|_{L^{2}}^{2} 
+ \delta^{-1} M_{n}^{2p^{2} - 12p + 14}
\|Z_{n}\|_{L^{2}}^{2} \\[6pt]
& \lesssim 
\delta \|\nabla Z_{n}\|_{L^{2}}^{2} 
+ \delta^{-1} \omega(m_{n})
^{\frac{p^{2} - 6p + 7}{p - 3}}
\|Z_{n}\|_{L^{2}}^{2}. 
\end{split}
\end{equation*}
for $2 < p < 3$. 
Note that 
$\frac{p^{2} - 6p + 7}{p - 3} 
> 1$ for $2 < p < 3$ and 
$0 < \omega(m_{n}) \ll 1$ for sufficiently large 
$n \in \N$. 
Thus, one has 
\begin{equation*}
\biggl|\kappa_{n} \int_{\R^{3}} R_{\omega_{2}(m_{n})} Z_{n} \, dx \biggl| 
\lesssim \delta \|\nabla Z_{n}\|_{L^{2}}^{2} 
+ \delta \omega(m_{n})
\|Z_{n}\|_{L^{2}}^{2}. 
\end{equation*}
Since $R_{\omega_{j}(m_{n})}(\cdot) = M_{n, j} 
\widetilde{R}_{n, j}(M_{n, j}^{2} \cdot)$, we have 
by \eqref{EqL-1} and \eqref{e-muni17} that 
\[
|V_{n}(x, \tau)| \lesssim 
\frac{1}{\min\{M_{n, 1}, M_{n, 2}\} r}. 
\]
For any $\delta > 0$, 
we observe from \eqref{e-muni3} that 
\begin{equation}\label{e-muni22}
|V_{n}(x, \tau)|^{p-1} \leq \frac{1}{\min\{M_{n, 1}^{p-1}, M_{n, 2}^{p-1}\}} 
\delta^{p-1} M_{n, 1}^{3p - 7}
\lesssim \delta^{p - 1} M_{n, 1}^{2p -6}
\lesssim {2} \delta^{p-1} \omega(m_{n})
\end{equation}
for $|x| \geq \delta^{-1} M_{n, 1}^{\frac{7 - 3p}{p - 1}}$. 
In addition, by the H\"{o}lder inequality and \eqref{e-muni17}, we obtain 
\begin{equation}\label{e-muni23} 
\begin{split}
& \quad \biggl|p \int_{|x| \leq \delta^{-1} M_{n, 1}^{\frac{7 - 3p}{p - 1}}} 
\int_{0}^{1} V_{n}^{p-1}(x, \tau) 
d\tau |Z_{n}|^{2} \, dx \biggl| \\
& \lesssim 
\frac{1}{\min\{M_{n, 1}^{p-1}, M_{n, 2}^{p-1}\}} 
\int_{|x| \leq \delta^{-1} M_{n, 1}^{\frac{7 - 3p}{p - 1}}}
|x|^{- (p-1)} |Z_{n}|^{2} \, dx \\
& \lesssim 
M_{n, 1}^{- (p-1)}
\left(\int_{|x| \leq \delta^{-1} M_{n, 1}^{\frac{7 - 3p}{p - 1}}}
|x|^{- \frac{3}{2}(p-1)} \, dx \right)^{\frac{2}{3}}
\|\nabla Z_{n}\|_{L^{2}}^{2} \\
& \lesssim 
\delta^{p - 3}
M_{n, 1}^{- (p - 1)} M_{n, 1}^{\frac{(7 - 3p)(3 - p)}{p - 1}} 
\|\nabla Z_{n}\|_{L^{2}}^{2} 
\leq C \delta^{p-3} M_{n, 1}^{\frac{2}{p-1}(p-2)(p-5)} \|\nabla Z_{n}\|_{L^{2}}^{2}. 
\end{split}
\end{equation}
Therefore, by \eqref{e-muni22} 
and \eqref{e-muni23}, we obtain 
\begin{equation}\label{e-muni24}
\biggl|p \int_{\R^{3}}\int_{0}^{1} V_{n}^{p-1}(x, \tau) 
d\tau |Z_{n}|^{2} \, dx \biggl| \leq 
4 C_{p}^{2} \delta^{p-1} \omega(m_{n}) \|Z_{n}\|_{L^{2}}^{2} + 
C \delta^{p-3} M_{n, 1}^{\frac{2}{p-1}(p-2)(p-5)} \|\nabla Z_{n}\|_{L^{2}}^{2}. 
\end{equation}
For any $\e>0$, we have by \eqref{e-muni4} that 
\[
\|R_{\omega_{j}(m_{n})}\|_{L^{\frac{24}{4 + \e}}} 
\lesssim m_{n}^{\frac{\e}{4(p-1)}} 
\qquad \mbox{for $j = 1, 2$}. 
\] 
By the H\"{o}lder, Young inequalities and \eqref{e-muni15}, we obtain 
\begin{equation} \label{e-muni25}
\begin{split}
\biggl| 5 \int_{\R^{3}} 
\int_{0}^{1} V_{n}^{4}(x, \tau) d \tau |Z_{n}|^{2} \, dx \biggl|
& \lesssim \int_{\R^{3}} (|R_{\omega_{1}(m_{n})}|^{4} 
+ |R_{\omega_{2}(m_{n})}|^{4}) |Z_{n}|^{2 - \varepsilon} \, dx \\[6pt]
& \leq (\|R_{\omega_{1}(m_{n})}\|_{L^{\frac{24}{4 + 
\varepsilon}}}^{4} + 
\|R_{\omega_{2}(m_{n})}\|_{L^{\frac{24}{4 + 
\varepsilon}}}^{4}) \|Z_{n}\|_{L^{6}}^{2 - \e} \\[6pt]
& \leq C 
m_{n}^{\frac{\e}{p-1}} \|\nabla Z_{n}\|_{L^{2}}^{2 - \e} \\[6pt]
& \leq 
C m_{n}^{\frac{2}{p-1}}
+ \frac{1}{2} \|\nabla Z_{n}\|_{L^{2}}^{2}. 
\end{split}
\end{equation}
From \eqref{e-muni21}, \eqref{e-muni23}--\eqref{e-muni25} and $2 < p < 3$, 
we have obtained the desired result. 
\end{proof}
From Lemma \ref{lem-uni2-2}, 
we have $\lim_{n \to \infty} Z_{n} = 0$ strongly in 
$\dot{H}^{1}(\R^{3})$. 
In addition, we have the following: 
\begin{lemma}
Let $2 < p < 3$. Then, there exists 
a constant $C > 0$ such that 
\begin{equation} \label{e-muni26}
|\kappa_{n}| \leq 
C M_{n}^{p^{2} - \frac{11}{2} p + \frac{11}{2}}
\qquad \mbox{for all $n \in \mathbb{N}$}. 
\end{equation}
\end{lemma} 
\begin{proof}
From Lemmas \ref{l-muni2}, 
\ref{lem-uni2-2} and \eqref{e-muni1}, 
we have 
\[
|\kappa_{n}| \lesssim
M_{n}^{p^{2} - \frac{11}{2} p 
+ \frac{13}{2}} \|\nabla Z_{n}\|_{L^{2}}
\lesssim M_{n}^{p^{2} - \frac{11}{2} p 
+ \frac{13}{2}} m_{n}^{\frac{1}{p - 1}} 
\lesssim M_{n}^{p^{2} - \frac{11}{2} p 
+ \frac{11}{2}}. 
\]
This completes the proof. 
\end{proof}
\begin{remark}
Note that 
$p^{2} - \frac{11}{2} p + \frac{11}{2} < 0$ 
for $2 < p < 3$. 
This together with 
\eqref{e-muni26} yields that 
\begin{equation} \label{e-muni27}
\lim_{n \to \infty} \kappa_{n} = 0. 
\end{equation}
From this and \eqref{e-muni15}, 
we have 
$\lim_{n \to \infty} \|Z_{n}\|_{L^{\infty}} = 1$.
\end{remark}
We are now in a position to prove 
the uniqueness of the minimizer $R_{\omega(m)}$. 
\begin{proof}[Proof of Theorem \ref{thm-muni} in the case of $2 < p < 3$]
Since $Z_{n}$ is radially symmetric, we have by Lemma \ref{lem-uni2-2} 
that $|Z_{n}(x)| \lesssim |x|^{- 1/2}$ for $|x| >0$ (see Berestycki and 
Lions~\cite[Lemma A.III]{MR695535}). 
Using the elliptic regularity argument and 
$\lim_{n \to \infty} Z_{n} = 0$ strongly in $\dot{H}^{1}(\R^{3})$ 
(see Lemma \ref{lem-uni2-2}), 
we have $\lim_{n \to \infty} Z_{n} = 0$ in $C_{\text{loc}}^{2}
(\R^{3})$, which contradicts $\lim_{n \to \infty} \|Z_{n}\|_{L^{\infty}} = 1$. 
\end{proof}

\subsection{Case of $p = 2$}
In this subsection, we will give the proof of Theorem \ref{thm-muni} in the case of $p = 2$. 
The strategy is similar to that of the case of $2 < p < 3$. 
However, since the relation between $\alpha_{\omega(m)}$ and 
$\beta_{\omega(m)}$ is different (see Proposition \ref{PropR-1}), 
we need to investigate individually. 
\subsubsection{Relation between $\|R_{\omega(m)}\|_{L^{q}}$ and $m$}
Here, similarly to the case of $2 < p < 3$, we shall study the 
relation between $\|R_{\omega(m)}\|_{L^{q}}\; (q>3)$ and 
the parameter $m$ to prove Theorem \ref{thm-muni} in the case of $p = 2$. 
We will prove the following: 
\begin{proposition}\label{p-muni6}
Let $p = 2$ and 
$\omega(m) > 0$ be the Lagrange multiplier 
given in Theorem \ref{thm-ex2} \textrm{(ii)} and 
$M_{m} = \|R_{\omega(m)}\|_{L^{\infty}}$. 
Then, we have the following: 
\begin{enumerate}
\item[\textrm{(i)}]
\begin{equation} \label{e-muni28}
\lim_{m \to 0} \frac{\omega(m)}{M_{m}^{-2} (\log M_{m})^{2}} 
= 27. 
\end{equation}
\item[\textrm{(ii)}] 
\begin{equation} \label{e-muni29}
m \sim M_{m}^{- 1} (\log M_{m})^{-1} 
\qquad \mbox{as $m \to 0$}. 
\end{equation}
\item[\textrm{(iii)}]
Let $q > 3$. 
\begin{equation} \label{e-muni30}
\|R_{\omega(m)}\|_{L^{q}}
\sim (m (- \log m))^{\frac{6-q}{q}} 
\qquad \mbox{as $m \to 0$}.
\end{equation}
\item[\textrm{(iv)}]
\begin{equation}\label{e-muni31}
\omega(m) \sim m^{2}(- \log m)^{4} 
\qquad \mbox{as $m \to 0$}.
\end{equation}
\end{enumerate}
\end{proposition}
\begin{proof}[Proof of Proposition \ref{l-muni1}]
\textrm{(i)}
First, we shall show that 
\begin{equation} \label{e-muni32}
\liminf_{m \to 0} \frac{\omega(m) M_{m}^{2}}{(\log M_{m})^{2}} 
> 0. 
\end{equation}
Suppose to the contrary that there exists a sequence 
$\{m_{n}\}$ with $\lim_{n \to \infty} m_{n} = 0$ such that 
\begin{equation} \label{e-muni33}
\lim_{n \to \infty} \frac{\omega(m_{n}) M_{m_{n}}^{2}}{(\log M_{m_{n}})^{2}} 
= 0. 
\end{equation}
Then, it follows from \eqref{EqB-2}, 
\eqref{e-muni8}, 
$p = 2$, $\log \omega(m_{n}) < 0$ and \eqref{e-muni33} that 
\begin{equation} \label{e-muni34}
\begin{split}
\frac{2}{\sqrt{3}} = \lim_{n \to \infty} 
\frac{\beta_{m_{n}} |\log \alpha_{m_{n}}|}
{\sqrt{\alpha_{m_{n}}}} 
& = \lim_{n \to \infty} 
\frac{|\log \omega(m_{n}) - 4 \log M_{m_{n}}|}{M_{m_{n}} \omega^{\frac{1}{2}}(m_{n})} \\
& = \lim_{n \to \infty} 
\frac{\left(- \log \omega(m_{n}) + 4 \log M_{m_{n}} \right)}{M_{m_{n}} 
\omega^{\frac{1}{2}}(m_{n})} \\
& > \liminf_{n \to \infty} 4 
\frac{\log M_{m_{n}}}{\omega^{\frac{1}{2}}(m_{n}) M_{m_{n}}} \\
& = \infty, 
\end{split}
\end{equation} 
which is a contradiction. Thus, \eqref{e-muni32} holds. 

Next, we claim that 
\begin{equation} \label{e-muni35}
\limsup_{m \to 0} \frac{\omega(m) M_{m}^{2}}{(\log M_{m})^{2}} 
< \infty. 
\end{equation}
Suppose to the contrary that there exists a sequence 
$\{m_{n}\}$ with $\lim_{n \to \infty} m_{n} = 0$ such that 
\begin{equation} \label{e-muni36}
\lim_{n \to \infty} \frac{\omega(m_{n}) M_{m_{n}}^{2}}{(\log M_{m_{n}})^{2}} 
= \infty. 
\end{equation}
Thus, for any $L > 1$, there exists $n_{L} \in \mathbb{N}$ such that 
for $n \geq n_{L}$, we have 
\begin{equation} \label{e-muni93}
\frac{\omega(m_{n}) M_{m_{n}}^{2}}{(\log M_{m_{n}})^{2}} 
\geq L. 
\end{equation}
In addition, we see from \eqref{e-muni36} 
that $\lim_{n \to \infty} \omega(m_{n}) M_{m_{n}}^{2} 
= \infty$.  
These together with \eqref{EqB-2}, 
\eqref{e-muni34} and \eqref{e-muni93} 
yield that 
\begin{equation} \label{e-muni37}
\begin{split}
\frac{2}{\sqrt{3}} = \lim_{n \to \infty} 
\frac{\beta_{m_{n}} |\log \alpha_{m_{n}}|}
{\sqrt{\alpha_{m_{n}}}} 
& =  \lim_{n \to \infty} 
\frac{\left(- \log \omega(m_{n}) + 4 \log M_{m_{n}} \right)}{M_{m_{n}} 
\omega^{\frac{1}{2}}(m_{n})} \\
& < \limsup_{n \to \infty} 
\frac{\left(- \log \left(L M_{m_{n}}^{- 2} (\log M_{m_{n}})^{2}\right) 
+ 4 \log M_{m_{n}} \right)}{M_{m_{n}} 
\omega^{\frac{1}{2}}(m_{n})}\\
& = \limsup_{n \to \infty}
\left\{ 
6 \frac{\log M_{m_{n}}}{M_{m_{n}} \omega^{\frac{1}{2}}(m_{n})} 
-2 \frac{\log (\log M_{m_{n}})}{M_{m_{n}} \omega^{\frac{1}{2}}(m_{n})} 
- \frac{\log L}{M_{m_{n}} \omega^{\frac{1}{2}}(m_{n})}
\right\} \\
& \leq \frac{6}{\sqrt{L}} < \frac{1}{\sqrt{3}}
\end{split}
\end{equation} 
for sufficiently large $L > 0$,  
which is absurd. 

Then, we see from \eqref{e-muni32} and 
\eqref{e-muni35} that up to subsequence 
(we still denote it by the same letter), 
there exists $C_{2} >0$ such that 
\begin{equation} \label{e-muni38}
\lim_{n \to 0} \frac{\omega(m_{n}) M_{m_{n}}^{2}}{(\log M_{m_{n}})^{2}} 
= C_{2}. 
\end{equation}
Putting $\ell_{n} = \frac{\omega(m_{n}) M_{m_{n}}^{2}}{(\log M_{m_{n}})^{2}}$, 
one has $\lim_{n \to \infty} \ell_{n} = C_{2}$ and 
	\begin{equation*}
	\frac{- \log \omega(m_{n})}
{M_{m_{n}} \omega^{\frac{1}{2}}(m_{n})} 
= - \frac{\log \ell_{n}}{\sqrt{\ell_{n}} \log M_{m_{n}}} 
+ \frac{2}{\sqrt{\ell_{n}}} - 2 \frac{\log (\log M_{m_{n}})}{
\sqrt{\ell_{n}} \log M_{m_{n}}}, 
	\end{equation*}
which yields that 
	\begin{equation} \label{e-muni39}
	\lim_{n \to \infty} \frac{- \log \omega(m_{n})}
{M_{m_{n}} \omega^{\frac{1}{2}}(m_{n})}  
= \frac{2}{\sqrt{C_{2}}}. 
	\end{equation}
Then, it follows from \eqref{e-muni34} and 
\eqref{e-muni39} that 
	\[
	\frac{2}{\sqrt{3}} = 
	\lim_{n \to \infty} 
\frac{\left(- \log \omega(m_{n}) + 4 \log M_{m_{n}} \right)}
{M_{m_{n}} \omega^{\frac{1}{2}}(m_{n})} 
= \lim_{n \to \infty} 
\frac{- \log \omega(m_{n})}{M_{m_{n}} 
\omega^{\frac{1}{2}}(m_{n})} + \frac{4}{\sqrt{C_{2}}} 
= \frac{6}{\sqrt{C_{2}}}
	\]
This implies that $C_2 = 27$. 
Thus, \textrm(i) holds. 

\textrm{(ii)}
Observe from \eqref{e-muni10} that 
\begin{equation} \label{e-muni40}
\|\widetilde{R}_{\omega(m)}\|_{L^{2}}^{2} 
= \frac{5 - p}{2(p + 1)} \frac{\beta_{m}}{\alpha_{m}} 
\|\widetilde{R}_{\omega(m)}\|_{L^{3}}^{3}. 
\end{equation}
From \eqref{EqL-1} and \eqref{eqU-6-3}, 
we obtain 
\begin{equation}\label{e-muni41}
\begin{split}
\|\widetilde{R}_{m}\|_{L^{3}}^{3} 
& = 4 \pi \int_{0}^{1} 
\widetilde{R}_{m}^{3}(r) r^{2} \, dr 
+ 4 \pi \int_{1}^{\alpha_{m}^{- \frac{1}{2}}} 
\widetilde{R}_{m}^{3}(r) r^{2} \, dr 
+ 4 \pi \int_{\alpha_{m}^{- \frac{1}{2}}}
^{\infty} 
\widetilde{R}_{m}^{3}(r) r^{2} \, dr \\
& \lesssim 1 + 
\int_{1}^{\alpha_{m}^{- \frac{1}{2}}} 
r^{-1} \, dr 
+ \int_{\alpha_{m}^{- \frac{1}{2}}}
^{\infty} 
r^{-1} e^{- 3 \sqrt{\alpha_{m}}r}\, dr \\
& \lesssim 1 
+ |\log \alpha_{m}| 
+ \int_{1}^{\infty} 
s^{-1} e^{- 3 s} \, ds 
\lesssim |\log \alpha_{m}|. 
\end{split}
\end{equation}
Moreover, by \eqref{EqB-89}, we have 
\begin{equation}\label{e-muni42}
\begin{split}
\|\widetilde{R}_{m}\|_{L^{3}}^{3} 
& = 4 \pi \int_{0}^{R_{2}} 
\widetilde{R}_{m}^{3}(r) r^{2} \, dr 
+ 4 \pi \int_{R_{2}}^{\alpha_{m}^{- \frac{1}{2}}} 
\widetilde{R}_{m}^{3}(r) r^{2} \, dr 
+ 4 \pi \int_{\alpha_{m}^{- \frac{1}{2}}}
^{\infty} 
\widetilde{R}_{m}^{3}(r) r^{2} \, dr \\
& \gtrsim \int_{R_{2}}
^{\alpha_{m}^{- \frac{1}{2}}} 
\widetilde{R}_{m}^{3}(r) r^{2} \, dr 
\gtrsim \int_{R_{2}}^{\alpha_{m}^{- \frac{1}{2}}} r^{-1} \, dr 
\gtrsim |\log \alpha_{m}|. 
\end{split}
\end{equation}
It follows from \eqref{e-muni40}, \eqref{e-muni41}, 
\eqref{e-muni42} 
and \eqref{EqB-2} that 
\begin{equation} \label{e-muni43}
\|\widetilde{R}_{\omega(m)}\|_{L^{2}}^{2} \sim 
\frac{\beta_{m} |\log \alpha_{m}|}{\alpha_{m}} 
\sim \alpha_{m}^{- \frac{1}{2}} 
\qquad \mbox{as $m \to 0$}. 
\end{equation}
for sufficiently small $m > 0$. 
By \eqref{e-muni28} and \eqref{e-muni43}, 
we obtain 
\[
m = \|R_{\omega(m)}\|_{L^{2}}^{2} 
= M_{m}^{-4} \|\widetilde{R}_{\omega(m)}\|_{L^{2}}^{2} 
\sim M_{m}^{-4} \alpha_{m}^{- \frac{1}{2}}
= \omega^{- \frac{1}{2}}(m) M_{m}^{-2} 
\sim M_{m}^{-1} (\log M_{m})^{-1}. 
\]

\textrm{(iii)}
Since $R_{\omega(m)}(x) = M_{m} \widetilde{R}_{m}(M_{m}^{2} x)$ 
for $x \in \R^{3}$, 
we have 
\begin{equation} \label{e-muni44}
\|R_{\omega(m)}\|_{L^{q}}^{q} 
= M_{m}^{- (6-q)} 
\|\widetilde{R}_{m}\|_{L^{q}}^{q}. 
\end{equation}
In addition, it follows from \eqref{e-muni29} that 
\begin{equation} \label{e-muni45}
\log m \sim - \log M_{m} - \log (\log M_{m}) 
\sim - \log M_{m} 
\qquad \mbox{as $m \to 0$}. 
\end{equation}
Thus, \eqref{e-muni29} can be written by the following: 
\begin{equation} \label{e-muni94}
m (- \log m) \sim M_{m}^{-1} 
\qquad \mbox{as $m \to 0$}. 
\end{equation}
This together with 
\eqref{e-muni44} yields that 
\[
\begin{split} 
\frac{\|R_{\omega(m)}\|_{L^{q}}}{(m (- \log m))^{\frac{6-q}{q}}} 
\sim 
\dfrac{M_{m}^{\frac{- (6-q)}{q}} \|\widetilde{R}_{m}\|_{L^{q}}}
{(m (- \log m))^{\frac{6-q}{q}}}
\sim (m (- \log m)M_{m})^{- \frac{(6-q)}{q}}
\|\widetilde{R}_{m}\|_{L^{q}} 
\sim 
\|W\|_{L^{q}}.
\end{split}
\]

\textrm{(iv)}
We see from \eqref{e-muni28}, \eqref{e-muni29} and 
\eqref{e-muni45} that 
\begin{equation*} 
\begin{split}
\frac{\omega(m)}{m^{2}(- \log m)^{4}} \sim 
\frac{\omega(m)}{M_{m}^{-2} (\log M_{m})^{2}} 
\frac{M_{m}^{-2} (\log M_{m})^{2}}
{m^{2} (- \log m)^{4}} 
& \sim \frac{\omega(m)}{M_{m}^{-2} (\log M_{m})^{2}} 
\frac{M_{m}^{-2} (\log M_{m})^{-2} (\log M_{m})^{4}}
{m^{2} (- \log m)^{4}} \\
& \sim 
1 \qquad \mbox{as $m \to 0$}. 
\end{split}
\end{equation*}
This completes the proof. 
\end{proof}

\begin{lemma}\label{l-muni7}
Let $p = 2$ and $1 \leq q< 3$. 
Then, there exists a constant $C > 0$ such that 
$\|R_{\omega_{j}(m)}\|_{L^{q}}^{q} 
\leq 
C M_{m}^{3 - 2 q} (\log M_{m})^{- (3 - q)}$ 
for all $n \in \mathbb{N}$. 
\end{lemma}
\begin{proof}
By \eqref{e-muni28}, we obtain 
\begin{equation} \label{e-muni52}
\alpha_{m} \sim M_{m}^{-6} 
(\log M_{m})^{2} 
\qquad \mbox{as $n \to \infty$}. 
\end{equation}
Since $R_{\omega(m)}(x) = M_{m} 
\widetilde{R}_{\omega(m)}(M_{m}^{2} x)$ 
for $x \in \R^{3}$, we have 
\begin{equation} \label{e-muni53}
\|R_{\omega (m)}\|_{L^{q}}^{q} 
= M_{m}^{- (6-q)} 
\|\widetilde{R}_{\omega(m)}\|_{L^{q}}^{q}. 
\end{equation}
Then, it follows from \eqref{EqL-1}, 
\eqref{eqU-6-3}
and \eqref{e-muni52} that 
\[
\begin{split}
\|\widetilde{R}_{\omega(m)}\|_{L^{q}}^{q} 
& = 4 \pi \int_{0}^{1} 
\widetilde{R}_{\omega(m)}^{q}(r) r^{2} \, dr 
+ 4 \pi \int_{1}^{\alpha_{m}^{- \frac{1}{2}}} 
\widetilde{R}_{\omega(m)}^{q}(r) r^{2} \, dr 
+ 4 \pi \int_{\alpha_{m}^{- \frac{1}{2}}}
^{\infty} 
\widetilde{R}_{\omega(m)}^{q}(r) r^{2} \, dr \\
& \lesssim 1 + 
\int_{1}^{\alpha_{m}^{- \frac{1}{2}}} 
r^{2- q} \, dr 
+ \int_{\alpha_{m}^{- \frac{1}{2}}}
^{\infty} 
r^{2- q} e^{- \frac{q}{2} 
\sqrt{\alpha_{m}}r}\, dr \\
& \lesssim 1 + \alpha_{m}^{- 
\frac{1}{2}(3 - q)} 
+ \alpha_{m}^{- \frac{1}{2}(3 - q)} \int_{1}^{\infty} 
s^{2 - q} e^{- s} \, ds 
\lesssim \alpha_{m}^{- 
\frac{1}{2}(3 - q)} 
\lesssim 
M_{m}^{3(3 - q)} (\log M_{m})^{- (3 - q)}. 
\end{split}
\] 
Thus, by \eqref{e-muni53}, we obtain 
\[
\|R_{\omega(m)}\|_{L^{q}}^{q} 
= M_{m}^{- (6-q)} 
\|\widetilde{R}_{\omega(m)}\|_{L^{q}}^{q} 
\lesssim 
M_{m}^{3 - 2 q} (\log M_{m})^{- (3 - q)}. 
\]
\end{proof}
\subsubsection{Proof of Theorem \ref{thm-muni} 
in the case of $p = 2$}
This subsection is devoted to the proof of 
Theorem \ref{thm-muni} in the case of $p = 2$. 
The strategy of the proof is the same as the one in the case of 
$2 < p < 3$.  
We shall show this by contradiction. 
Suppose to the contrary that there exists a 
sequence $\{m_{n}\}$ in $(0, \infty)$ such that 
$\lim_{n \to \infty} m_{n} = 0$ and for each 
$n \in \mathbb{N}$, $R_{\omega_{1}(m_{n})}, 
R_{\omega_{2}(m_{n})}$ and $R_{\omega_{1}(m_{n})} 
\neq R_{\omega_{2}(m_{n})}$.
For $j = 1, 2$, we set 
\begin{equation*}
\begin{split}
& 
M_{n, j} = R_{\omega_{j}(m_{n})}(0), 
\qquad 
\widetilde{R}_{n, j}(x) = M_{n, j}^{-1} 
R_{\omega_{j}(m_{n})}(M_{n, j}^{-2} x), \\[6pt]
& 
\alpha_{n, j} = \omega_{j}(m_n)M_{n, j}^{-4}, 
\qquad 
\beta_{n, j} = M_{n, j}^{p-5}. 
\end{split}
\end{equation*}
By \eqref{e-muni1} and 
\eqref{e-muni3}, we see that 
\begin{equation} \label{e-muni46}
\lim_{n \to \infty} \frac{M_{n, 1}}{M_{n, 2}} = 
\lim_{n \to \infty} \frac{\omega_{1}(m_{n})}{
\omega_{2}(m_{n})} 
= 1, 
\end{equation}
which implies 
\[
\lim_{n \to \infty} \frac{\alpha_{n, 1}}{\alpha_{n, 2}} 
= \lim_{n \to \infty} \frac{\beta_{n, 1}}{\beta_{n, 2}} = 1.
\]
Let 
\begin{align}
& 
Z_{n} = \dfrac{R_{\omega_{1}(m_{n})} - R_{\omega_{2}(m_{n})}}
{\|R_{\omega_{1}(m_{n})} - R_{\omega_{2}(m_{n})}\|_{L^{\infty}} + |\omega_{1}(m_{n}) 
- \omega_{2}(m_{n})|}, 
\label{e-muni47}\\[6pt] 
& 
\kappa_{n} = 
\dfrac{\omega_{1}(m_{n}) - \omega_{2}(m_{n})}
{\|R_{\omega_{1}(m_{n})} - R_{\omega_{2}(m_{n})}\|_{L^{\infty}} + |\omega_{1}(m_{n}) 
- \omega_{2}(m_{n})|}. 
\label{e-muni48}
\end{align}
Clearly, we have 
\begin{equation} \label{e-muni49}
1 = \|Z_{n}\|_{L^{\infty}} + |\kappa_{n}|
\end{equation}
We can verify that the pair $(Z_{n}, \kappa_{n})$ satisfies 
\begin{equation}\label{e-muni50}
\begin{split}
- \Delta Z_{n} + \omega_{1}(m_{n}) Z_{n} 
- p \int_{0}^{1} V_{n}^{p-1}(x, \tau) d\tau Z_{n} 
- 5 \int_{0}^{1} V_{n}^{4}(x, \tau) d \tau Z_{n} 
= - \kappa_{n} R_{\omega_{2}(m_{n})}, 
\end{split}
\end{equation}
where 
\begin{equation} \label{e-muni51}
V_{n}(x, \tau) = \tau 
R_{\omega_{1}(m_{n})} (x) + (1 - \tau)
R_{\omega_{2}(m_{n})}(x) \qquad 
(\tau \in (0, 1), \; x \in \R^{3}). 
\end{equation}
\begin{lemma}\label{l-muni9}
Let $p = 2$. 
Then, there exists a constant $C>0$ such that 
$|\kappa_{n}| \leq 
C M_{n, 1}^{- \frac{1}{2}} 
(\log M_{n, 1})^{\frac{1}{2}}
\|\nabla Z_{n}\|_{L^{2}}$ 
for all $n \in \mathbb{N}$. 
\end{lemma}
\begin{proof}
As in \eqref{e-muni20}, 
we have 
	\[
	|\kappa_{n}|
	\lesssim \frac{1}{m_{n}} \left(
\|R_{\omega_{1} (m_{n})}\|_{L^{\frac{12}{5}}}^{2} 
+ \|R_{\omega_{2} (m_{n})}\|_{L^{\frac{12}{5}}}^{2} \right)\|Z_{n}\|_{L^{6}}. 
	\] 
This together with 
\eqref{e-muni29}, \eqref{e-muni46} 
and Lemma \ref{l-muni7}
 yields that 
\begin{equation*} 
\begin{split}
|\kappa_{n}| 
\lesssim M_{n, 1} (\log M_{n, 1}) \times 
M_{n, 1}^{- \frac{3}{2}} (\log M_{n, 1})^{- \frac{1}{2}} 
\|\nabla Z_{n}\|_{L^{2}}
\lesssim M_{n, 1}^{- \frac{1}{2}} 
(\log M_{n, 1})^{\frac{1}{2}}
\|\nabla Z_{n}\|_{L^{2}}. 
\end{split}
\end{equation*}
This completes the proof. 
\end{proof}

\begin{lemma}\label{l-muni9-2}
Let $p = 2$. 
Then, there exists a constant $C>0$ such that
$\|\nabla Z_{n}\|_{L^{2}}^{2} 
\leq C M_{n, 1}^{- 2}$ for all $n \in 
\mathbb{N}$. 
\end{lemma}
\begin{proof}
Multiplying \eqref{e-muni50} by $Z_{n}$ and integrating 
the resulting equation on $\R^{3}$, 
we have 
\begin{equation} \label{e-muni55}
\begin{split}
\|\nabla Z_{n}\|_{L^{2}}^{2} + 
\omega_{1}(m_{n}) \|Z_{n}\|_{L^{2}}^{2} 
& = 2 \int_{\R^{3}}\int_{0}^{1} V_{n}(x, \tau) 
d\tau |Z_{n}|^{2} \, dx 
+ 5 \int_{\R^{3}} \int_{0}^{1} V_{n}^{4}(x, \tau) 
d \tau |Z_{n}|^{2} \, dx \\[6pt]
& \quad - \kappa_{n} \int_{\R^{3}} 
R_{\omega_{2}(m_{n})} Z_{n} \, dx
\end{split}
\end{equation}
By the H\"{o}lder and Young inequalities, 
$ \|R_{\omega_{2}(m_{n})}\|_{L^{2}} = m_{n}^{1/2}$, 
Lemma \ref{l-muni7}, 
\eqref{e-muni29} and \eqref{e-muni31}, 
we obtain
\begin{equation} \label{e-muni56} 
\begin{split}
\biggl|\kappa_{n} \int_{\R^{3}} R_{\omega_{2}(m_{n})} Z_{n} \, dx \biggl| 
\leq |\kappa_{n}| \|R_{\omega_{2}(m_{n})}\|_{L^{2}}
\|Z_{n}\|_{L^{2}} 
& \leq m_{n}^{\frac{1}{2}} 
M_{n, 1}^{- \frac{1}{2}} (\log M_{n, 1})^{- \frac{1}{2}} 
\|\nabla Z_{n}\|_{L^{2}}
\|Z_{n}\|_{L^{2}} \\[6pt]
& 
\lesssim 
M_{n, 1}^{-1}  (\log M_{n, 1})^{- 1}
\|\nabla Z_{n}\|_{L^{2}}
\|Z_{n}\|_{L^{2}} \\[6pt]
& \lesssim 
\delta \|\nabla Z_{n}\|_{L^{2}}^{2} 
+ \delta^{-1} M_{n, 1}^{-2} (\log M_{n, 1})^{- 2}
\|Z_{n}\|_{L^{2}}^{2} \\[6pt]
& \lesssim 
\delta \|\nabla Z_{n}\|_{L^{2}}^{2} 
+ \delta^{-1} \frac{\omega(m_{n})}{
(\log M_{n, 1})^{4}}
\|Z_{n}\|_{L^{2}}^{2}. 
\end{split}
\end{equation} 
Thus, one has 
\begin{equation*}
\biggl|\kappa_{n} \int_{\R^{3}} R_{\omega_{2}(m_{n})} Z_{n} \, dx \biggl| 
\lesssim \delta \|\nabla Z_{n}\|_{L^{2}}^{2} 
+ \frac{\omega(m_{n})}{2}
\|Z_{n}\|_{L^{2}}^{2} 
\end{equation*}
for sufficiently large $n \in \mathbb{N}$. 
Since $R_{\omega_{j}(m_{n})}(x) = M_{n, j} 
\widetilde{R}_{n, j}(M_{n, j}^{2} x)$ for $x \in \R^{3}$, we have 
by \eqref{EqL-1} and \eqref{e-muni51} that 
\[
|V_{n}(x, \tau)| \lesssim 
\frac{1}{M_{n, 1} r}. 
\]
For any $\delta > 0$, 
we observe from \eqref{e-muni28} that 
\[
|V_{n}(x, \tau)| \leq \frac{1}{M_{n, 1}} 
\delta M_{n, 1}^{-1} (\log M_{n, 1})^{2}
\lesssim \delta M_{n, 1}^{-2} (\log M_{n, 1})^{2}
\lesssim \delta \omega(m_{n})
\]
for $|x| \geq \delta^{-1} M_{n, 1} (\log M_{n, 1})^{-2}$. 
In addition, by the H\"{o}lder inequality and \eqref{e-muni51}, we obtain 
\begin{equation*} 
\begin{split}
& \quad \biggl|2 \int_{|x| \leq \delta^{-1} M_{n, 1} (\log M_{n, 1})^{-2}} 
\int_{0}^{1} V_{n}(x, \tau) 
d\tau |Z_{n}|^{2} \, dx \biggl| \\
& \lesssim 
\frac{1}{M_{n, 1}} 
\int_{|x| \leq \delta^{-1} M_{n, 1} (\log M_{n, 1})^{-2}}
|x|^{- 1} |Z_{n}|^{2} \, dx \\
& \lesssim 
M_{n, 1}^{- 1}
\left(\int_{|x| \leq \delta^{-1} M_{n, 1} (\log M_{n, 1})^{-2}}
|x|^{- \frac{3}{2}} \, dx \right)^{\frac{2}{3}}
\|\nabla Z_{n}\|_{L^{2}}^{2} \\
& \lesssim 
\delta^{-1}
M_{n, 1}^{- 1} \times M_{n, 1} (\log M_{n, 1})^{-2}
\|\nabla Z_{n}\|_{L^{2}}^{2} 
\leq \delta^{-1} (\log M_{n, 1})^{-2} \|\nabla Z_{n}\|_{L^{2}}^{2}. 
\end{split}
\end{equation*}
Therefore, we obtain 
\begin{equation}\label{e-muni57}
\biggl|2 \int_{\R^{3}}\int_{0}^{1} V_{n}(x, \tau) 
d\tau |Z_{n}|^{2} \, dx \biggl| \leq 
4 C_{p}^{2} \delta \omega(m_{n}) \|Z_{n}\|_{L^{2}}^{2} + 
C \delta^{-1} (\log M_{n, 1})^{-2} \|\nabla Z_{n}\|_{L^{2}}^{2}. 
\end{equation}
For any $\e>0$, we have 
by \eqref{e-muni30} and \eqref{e-muni94} 
that 
\[
\|R_{\omega_{j}(m_{n})}\|_{L^{\frac{24}{4 + \e}}} 
\lesssim M_{n, j}^{- \frac{\e}{4}} 
\qquad \mbox{for $j = 1, 2$}. 
\] 
By the H\"{o}lder, Young inequalities and \eqref{e-muni15}, we obtain 
\begin{equation} \label{e-muni58}
\begin{split}
\biggl| 5 \int_{\R^{3}} 
\int_{0}^{1} V_{n}^{4}(x, \tau) d \tau |Z_{n}|^{2} \, dx \biggl|
& \lesssim \int_{\R^{3}} (|R_{\omega_{1}(m_{n})}|^{4} 
+ |R_{\omega_{2}(m_{n})}|^{4}) |Z_{n}|^{2 - \varepsilon} \, dx \\[6pt]
& \leq (\|R_{\omega_{1}(m_{n})}\|_{L^{\frac{24}{4 + 
\varepsilon}}}^{4} + 
\|R_{\omega_{2}(m_{n})}\|_{L^{\frac{24}{4 + 
\varepsilon}}}^{4}) \|Z_{n}\|_{L^{6}}^{2 - \e} \\[6pt]
& \leq C 
M_{n, j}^{- \e} 
\|\nabla Z_{n}\|_{L^{2}}^{2 - \e} \\[6pt]
& \leq 
C M_{n, j}^{- 2} 
+ \frac{1}{2} \|\nabla Z_{n}\|_{L^{2}}^{2}. 
\end{split}
\end{equation}
From \eqref{e-muni55}--\eqref{e-muni58}, 
we have obtained the desired result. 
\end{proof}
From Lemma \ref{lem-uni2-2}, 
we have $\lim_{n \to \infty} Z_{n} = 0$ strongly in 
$\dot{H}^{1}(\R^{3})$. 
In addition, we have the following: 
\begin{lemma}\label{l-muni10}
Let $p = 2$. 
\begin{equation} \label{e-muni59}
|\kappa_{n}| \lesssim 
M_{n}^{- \frac{3}{2}} (\log M_{n})^{\frac{1}{2}}. 
\end{equation}
\end{lemma} 
\begin{proof}
From Lemmas \ref{l-muni9} 
and \ref{l-muni9-2}, 
we have 
\[
|\kappa_{n}| \lesssim
M_{n, 1}^{-\frac{1}{2}} (\log M_{n, 1})^{\frac{1}{2}} \|\nabla Z_{n}\|_{L^{2}}
\lesssim 
M_{n, 1}^{-\frac{1}{2}} (\log M_{n, 1})^{\frac{1}{2}} 
\times M_{n, 1}^{- 1} 
\lesssim M_{n, 1}^{- \frac{3}{2} } (\log M_{n, 1})^{\frac{1}{2}}. 
\]
This completes the proof. 
\end{proof}
From Lemmas \ref{l-muni9} and \ref{l-muni10}, 
we can prove 
the uniqueness of the minimizer $R_{\omega(m)}$ as in the case of 
$2 < p < 3$. 
Thus, we omit it.
\subsection{Case of $1 < p < 2$}
\subsubsection{Relation between $\|R_{\omega(m)}\|_{L^{q}}$ and $m$}
In this subsection, as in the previous cases, 
we shall study the 
relation between $\|R_{\omega(m)}\|_{L^{q}}\; 
(2 < q< 3)$ and 
the parameter $m$ to prove 
Theorem \ref{thm-muni} \textrm{(ii)} in 
the case of $1 < p < 2$. 
More precisely, we shall show the following: 
\begin{proposition}\label{p-muni11}
Let $1 < p < 2$ and 
$\omega(m) > 0$ be the Lagrange multiplier 
given in Theorem \ref{thm-ex2} \textrm{(i)} and 
$M_{m} = \|R_{\omega(m)}\|_{L^{\infty}}$. 
Then, we have the following: 
\begin{enumerate}
\item[\textrm{(i)}] 
\begin{equation} \label{e-muni60}
\lim_{m \to 0} \omega^{\frac{p - 3}{2}}(m)M_{m}^{1 - p} 
= 3^{- \frac{p-1}{2}} V^{p - 1}(0). 
\end{equation}
\item[\textrm{(ii)}]
\begin{equation} \label{e-muni63}
	\omega(m) \sim M_{m}^{- 2 \frac{p-1}{3-p}} 
	\qquad \mbox{as $m \to 0$}. 
	\end{equation}
\item[\textrm{(iii)}]
\begin{equation} \label{e-muni61}
m \sim \omega^{\frac{7 - 3p}{2(p - 1)}}(m) 
\qquad \mbox{as $m \to 0$}. 
\end{equation}
\item[\textrm{(iv)}]
\begin{equation} \label{e-muni62}
m \sim M^{- \frac{7 - 3p}{3 - p}}_{m}
\qquad \mbox{as $m \to 0$}.  
\end{equation}
\end{enumerate}
\end{proposition}
\begin{proof}[Proof of Proposition \ref{p-muni11}]

\textrm{(i)}
It follows from 
\eqref{EqB-3}, $\alpha_{m} = \omega(m) M_{m}^{-4}$ and 
$\beta_{m} = M_{m}^{p-5}$ that 
\[
3^{- \frac{p-1}{2}} V^{p - 1}(0) = \lim_{m \to 0} 
\frac{\beta_{m}}{\alpha_{m}^{\frac{3 - p}{2}}} 
= 
\lim_{m \to 0} \omega^{\frac{p - 3}{2}}(m)
M_{m}^{-p + 1}. 
\]
Thus, \eqref{e-muni60} holds. 

\textrm{(ii)}
We see from \eqref{e-muni60} that 
\eqref{e-muni63} immediately holds. 

\textrm{(iii)} 
By \eqref{EqB-89}, one has  
\begin{equation}\label{e-muni64}
\begin{split}
& \quad \|\widetilde{R}_{m}\|_{L^{2}}^{2} 
= 4 \pi \int_{0}^{\alpha_{m}^{- \frac{1}{2}}} 
\widetilde{R}_{m}^{2}(r) r^{2} \, dr 
+ 4 \pi \int_{\alpha_{m}^{- \frac{1}{2}}}
^{\infty} 
\widetilde{R}_{m}^{2}(r) r^{2} \, dr 
\geq 4 \pi \int_{\alpha_{m}^{- \frac{1}{2}}}
^{\infty} 
\widetilde{R}_{m}^{2}(r) r^{2} \, dr 
\gtrsim \int_{\alpha_{m}^{- \frac{1}{2}}}^{\infty} 
e^{- \sqrt{\alpha_{m}} r} \, dr \\
& = \alpha_{m}^{- \frac{1}{2}} 
\int_{1}^{\infty} e^{- s} \, ds
\gtrsim \alpha_{m}^{- \frac{1}{2}}. 
\end{split}
\end{equation}

It follows from
\eqref{e-muni6}, \eqref{EqL-6} and \eqref{e-muni64} that 
\[
\begin{split}
m = \|R_{\omega(m)}\|_{L^{2}}^{2} 
= M_{m}^{-4} \|\widetilde{R}_{m}\|_{L^{2}}^{2} 
\sim M_{m}^{-4} \alpha_{m}^{- \frac{1}{2}} 
\sim \omega^{- \frac{1}{2}}(m) M_{m}^{-2}. 
\end{split}
\] 
This together with \eqref{e-muni60} yields that 
\[
m \sim \omega^{- \frac{1}{2}}(m)
M_{m}^{-2} \sim \omega^{- \frac{1}{2}}(m) 
\omega^{\frac{3 - p}{p - 1}}(m) 
= \omega^{\frac{7 - 3p}{2(p - 1)}}(m). 
\]

\textrm{(iv)}
Using \eqref{e-muni60} and \eqref{e-muni61}, one has 
\[
m \sim \omega^{\frac{7 - 3p}{2(p - 1)}}(m) 
\sim M_{m}^{- 2 \frac{p - 1}{3 - p} \times 
\frac{7 -3p}{2(p - 1)}} \sim 
M_{m}^{\frac{3p - 7}{3 - p}}.
\]
This completes the proof. 
\end{proof}
Next, we estimate 
the $L^{q}$-norm of $R_{\omega(m)}$. 
More precisely, we obtain the following:
\begin{lemma}\label{l-muni13}
Let $q > 1$. 
Then, we obtain
$\|R_{\omega(m)}\|_{L^{q}}^{q} 
\lesssim M_{n, j}^{\frac{3p - 2q -3}{3 - p}}$. 
\end{lemma}
\begin{proof}
Since $R_{\omega(m)}(x) = M_{m} \widetilde{R}_{\omega(m)}(M_{m}^{2} x)$, 
we have 
\begin{equation} \label{e-muni78}
\|R_{\omega(m)}\|_{L^{q}}^{q} 
= M_{m}^{- (6-q)} 
\|\widetilde{R}_{\omega(m)}\|_{L^{q}}^{q}. 
\end{equation}
By \eqref{e-muni63}, we obtain 
\begin{equation} \label{e-muni77}
\alpha_{m} \sim 
M_{m}^{-2 \frac{5 - p}{3 - p}} \qquad 
\mbox{as $m \to 0$}. 
\end{equation}
Then, it follows from \eqref{EqL-1}, 
\eqref{eqU-6-3} 
and \eqref{e-muni77} that 
\begin{equation}\label{e-muni87}
\begin{split}
\|\widetilde{R}_{\omega(m)}\|_{L^{q}}^{q} 
& = 4 \pi \int_{0}^{1} 
\widetilde{R}_{\omega(m)}^{q}(r) r^{2} \, dr 
+ 4 \pi \int_{1}^{\alpha_{m}^{- \frac{1}{2}}} 
\widetilde{R}_{\omega(m)}^{q}(r) r^{2} \, dr 
+ 4 \pi \int_{\alpha_{m}^{- \frac{1}{2}}}
^{\infty} 
\widetilde{R}_{\omega(m)}^{q}(r) r^{2} \, dr \\
& \lesssim 1 + 
\int_{1}^{\alpha_{m}^{- \frac{1}{2}}} 
r^{2- q} \, dr 
+ \int_{\alpha_{m}^{- \frac{1}{2}}}
^{\infty} 
r^{2- q} e^{- \frac{q}{2} 
\sqrt{\alpha_{m}}r}\, dr \\
& \lesssim 1 + 
\alpha_{m}^{- \frac{1}{2}(3 - q)} 
+ \alpha_{m}^{- 
\frac{1}{2}(3 - q)} \int_{1}^{\infty} 
s^{2 - q} e^{- \frac{q}{2}s} \, ds 
\lesssim \alpha_{m}^{- 
\frac{1}{2}(3 - q)} 
\sim M_{m}^{\frac{(5 - p)(3 - q)}{3 - p}}. 
\end{split}
\end{equation}
Thus, by \eqref{e-muni78}, we obtain 
\[
\|R_{\omega(m)}\|_{L^{q}}^{q} 
= M_{m}^{- (6-q)} 
\|\widetilde{R}_{\omega(m)}\|_{L^{q}}^{q} 
= M_{m}^{\frac{3p - 2q -3}{3 - p}}. 
\]
This completes the proof. 
\end{proof}

\subsubsection{Proof of Theorem \ref{thm-muni} 
in the case of $1< p < 2$}
This subsection is devoted to the proof of 
Theorem \ref{thm-muni} in the case of $1 < p < 2$. 
As in the previous cases, we shall show this by contradiction. 
Suppose to the contrary that there exists a 
sequence $\{m_{n}\}$ in $(0, \infty)$ such that 
$\lim_{n \to \infty} m_{n} = 0$ and for each 
$n \in \mathbb{N}$, $R_{\omega_{1}(m_{n})}, 
R_{\omega_{2}(m_{n})}$ and $R_{\omega_{1}(m_{n})} 
\neq R_{\omega_{2}(m_{n})}$.
For $j = 1, 2$, we set 
\begin{equation*}
\begin{split}
& 
M_{n, j} = R_{\omega_{j}(m_{n})}(0), 
\qquad 
\widetilde{R}_{n, j}(x) = M_{n, j}^{-1} 
R_{\omega_{j}(m_{n})}(M_{n, j}^{-2} x), \\[6pt]
& 
\alpha_{n, j} = \omega_{j}(m_n)M_{n, j}^{-4}, 
\qquad 
\beta_{n, j} = M_{n, j}^{p-5}. 
\end{split}
\end{equation*}
By \eqref{e-muni1} and 
\eqref{e-muni3}, we see that 
\[
\lim_{n \to \infty} \frac{M_{n, 1}}{M_{n, 2}} = 
\lim_{n \to \infty} \frac{\omega_{1}(m_{n})}{
\omega_{2}(m_{n})} 
= 1, 
\]
which implies 
\[
\lim_{n \to \infty} \frac{\alpha_{n, 1}}{\alpha_{n, 2}} 
= \lim_{n \to \infty} \frac{\beta_{n, 1}}{\beta_{n, 2}} = 1.
\]
Let 
\begin{align}
& 
Z_{n} = \dfrac{R_{\omega_{1}(m_{n})} - R_{\omega_{2}(m_{n})}}
{\|R_{\omega_{1}(m_{n})} - R_{\omega_{2}(m_{n})}\|_{L^{\infty}} + |\omega_{1}(m_{n}) 
- \omega_{2}(m_{n})|}, 
\label{e-muni66}\\[6pt] 
& 
\kappa_{n} = 
\dfrac{\omega_{1}(m_{n}) - \omega_{2}(m_{n})}
{\|R_{\omega_{1}(m_{n})} - R_{\omega_{2}(m_{n})}\|_{L^{\infty}} + |\omega_{1}(m_{n}) 
- \omega_{2}(m_{n})|}. 
\label{e-muni65}
\end{align}
Clearly, we have 
\begin{equation} \label{e-muni67}
1 = \|Z_{n}\|_{L^{\infty}} + |\kappa_{n}|
\end{equation}
We can verify that the pair $(Z_{n}, \kappa_{n})$ satisfies 
\begin{equation}\label{e-muni68}
\begin{split}
- \Delta Z_{n} + \omega_{1}(m_{n}) Z_{n} 
- p \int_{0}^{1} V_{n}^{p-1}(x, \tau) d\tau Z_{n} 
- 5 \int_{0}^{1} V_{n}^{4}(x, \tau) d \tau Z_{n} 
= - \kappa_{n} R_{\omega_{2}(m_{n})}, 
\end{split}
\end{equation}
where 
\begin{equation} \label{e-muni69}
V_{n}(x, \tau) = \tau 
R_{\omega_{1}(m_{n})} (x) + (1 - \tau)
R_{\omega_{1}(m_{n})} (x) \qquad 
(\tau \in (0, 1), \; x \in \R^{3}). 
\end{equation}
Since $Z_{n}(r)$ is radially symmetric, 
\eqref{e-muni68} can be written by 
the following:
\begin{equation}\label{e-muni70}
\begin{split}
& \quad - \frac{d^{2} Z_{n}}{d r^{2}}(r) 
- \frac{2}{r} \frac{d Z_{n}}{d r}(r)
+ \omega_{1}(m_{n}) Z_{n}(r) 
- p \int_{0}^{1} V_{n}^{p-1}(r, \tau) d\tau Z_{n}(r) 
 - 5 \int_{0}^{1} V_{n}^{4}(r, \tau) d \tau Z_{n}(r) \\ 
& = - \kappa_{n} R_{\omega_{2}(m_{n})}(r), 
\end{split}
\end{equation}
We put $\widetilde{Z}_{n}(\rho) = Z_{n}(r)$ and $\rho = M_{n, 1}^{2} r$. 
Then, since $R_{\omega_{j}(m_{n})}(r) = M_{n, j} 
\widetilde{R}_{n, j}(M_{n, j}^{2} r) = M_{n, j} \widetilde{R}_{n, j}(\rho)$, 
$\widetilde{Z}_{n}(\rho)$ satisfies the following: 
	\begin{equation}\label{e-muni71}
        \begin{split}
& \quad - \frac{d^{2} \widetilde{Z}_{n}}{d \rho^{2}} (\rho)
- \frac{2}{\rho} \frac{d \widetilde{Z}_{n}}{d \rho}(\rho)
+ \alpha_{n, 1}(m_{n}) \widetilde{Z}_{n} (\rho)
- p \beta_{n, 1}(m_{n}) \int_{0}^{1} \widetilde{V}_{n}^{p-1}(\rho, \tau) d\tau 
\widetilde{Z}_{n}(\rho) 
- 5 \int_{0}^{1} \widetilde{V}_{n}^{4}(\rho, \tau) d \tau \widetilde{Z}_{n}(\rho) \\
& = - \kappa_{n} M_{n, 1}^{-3} \widetilde{R}_{\omega_{2}(m_{n})}
(\rho), 
\end{split}
\end{equation}
where 
\[
\widehat{V}_{n}(\rho, \tau) = \tau 
\widetilde{R}_{\omega_{1}(m_{n})} (\rho) 
+ (1 - \tau)
\nu_{n} \widetilde{R}_{\omega_{2}(m_{n})}(\rho) \qquad 
(\tau \in (0, 1)), 
\qquad 
\nu_{n} = 
\frac{M_{n, 2}}{M_{n, 1}}.
\]
We will obtain a following uniform decay of 
$\{\widetilde{Z}_{n}\}$
\begin{lemma}\label{l-muni16}
Let $1 < p < 2$. 
Suppose that $\lim_{n \to \infty} 
\|\nabla \widetilde{Z}_{n}\|_{L^{2}} = \infty$. 
Then, there exists a constant $C_{1} > 0$ such that 
for any $\rho > 1$ and $n \in \N$, 
we have 
    \begin{equation}\label{eq-muni92}
    |\widetilde{Z}_{n}(\rho)| 
    \leq C_{1} \|\nabla 
    \widetilde{Z_{n}}\|_{L^{2}}\rho^{-1}
    \end{equation}
\end{lemma}
\begin{proof}[Proof of Lemma \ref{l-muni16}]
Note that \eqref{e-muni71} can be written by 
    \begin{equation} \label{e-muni88}
    \begin{split}
    - \Delta \widetilde{Z}_{n} 
    + \alpha_{n, 1}(m_{n}) \widetilde{Z}_{n}
- p \beta_{n, 1}\int_{0}^{1} \widetilde{V}_{n}^{p-1}(x, \tau) d\tau 
\widetilde{Z}_{n}
- 5 \int_{0}^{1} \widetilde{V}_{n}^{4}(x, \tau) d \tau \widetilde{Z}_{n} 
= - \kappa_{n} M_{n, 1}^{-3} \widetilde{R}_{\omega_{2}(m_{n})} 
\qquad \mbox{in $\R^{3}$}. 
\end{split}
    \end{equation}
We denote by $K[u]$ by the Kelvin transform 
of a function $u$ on $\R^{3}$, that is, 
    \[
    K[u](x) := |x|^{-1} u(\frac{x}{|x|^{2}}). 
    \]
Then, we see that $K[\widetilde{Z}_{n}]$ satisfies 
    \begin{equation} \label{e-muni89}
    \begin{split}
    & \quad - \Delta K[\widetilde{Z}_{n}] 
    + \frac{\alpha_{n, 1}(m_{n})}{|x|^{4}} 
    K[\widetilde{Z}_{n}] \\ 
    & = |x|^{-4} 
    \left[ 
    p \beta_{n, 1} \int_{0}^{1} 
    \widetilde{V}_{n}^{p-1} 
    (\frac{x}{|x|^{2}}, \theta) d \theta 
    + 5 \int_{0}^{1} 
    \widetilde{V}_{n}^{4} 
    (\frac{x}{|x|^{2}}, \theta) d \theta
    \right] K[\widetilde{Z}_{n}] 
    - \kappa_{n} M_{n, 1}^{-3} 
    |x|^{-5} \widetilde{R}_{\omega_{2}(m_{n})}
    (\frac{x}{|x|^{2}}) 
    \qquad \mbox{in $\R^{3}$}. 
    \end{split}
    \end{equation}
Put 
    \[
    W_{n} := \frac{K[\widetilde{Z}_{n}]} 
{\|\nabla K[\widetilde{Z}_{n}]\|_{L^{2}}}.
    \]
It follows from \eqref{e-muni89} 
and $\|\nabla K[\widetilde{Z}_{n}]\|_{L^{2}} 
= \|\nabla \widetilde{Z}_{n}\|_{L^{2}}$ that 
$W_{n}$ is a solution to the 
following: 
    \begin{equation} \label{e-muni90}
    \begin{split}
    & \quad - \Delta W_{n} 
    + \frac{\alpha_{n, 1}(m_{n})}{|x|^{4}} 
    W_{n} \\ 
    & = \frac{1}{|x|^{4}} 
    \left[ 
    p \beta_{n, 1}(m_{n}) \int_{0}^{1} 
    \widetilde{V}_{n}^{p-1} 
    (\frac{x}{|x|^{2}}, \theta) d \theta 
    + 5 \int_{0}^{1} 
    \widetilde{V}_{n}^{4} 
    (\frac{x}{|x|^{2}}, \theta) d \theta
    \right] W_{n} 
    - \frac{\kappa_{n} M_{n, 1}^{-3}}
    {\|\nabla \widetilde{Z}_{n}\|_{L^{2}}} 
    |x|^{-5} \widetilde{R}_{\omega_{2}(m_{n})}
    (\frac{x}{|x|^{2}}) 
    \qquad \mbox{in $\R^{3}$}. 
    \end{split}
    \end{equation}
In order to Lemma \ref{l-muni16}, 
we will apply Proposition \ref{proposition:B.1} 
to the equation \eqref{e-muni89}. 
We see that 
    \[
    \int_{B_{4}} \frac{\alpha_{n, 1}}{|x|^{4}} 
    |K[\widetilde{Z}_{n}](x) v(x)| \, dx 
    < \infty \qquad 
    \mbox{for any $v \in H_{0}^{1}(B_{4})$}. 
    \]
From the proof of Lemma 4.3 in \cite{MR3964275}, 
we find that 
    \[
    \sup_{n \in \N} 
    \left\|\frac{\beta_{n, 1}(m_{n})}{
    |x|^{4}} \int_{0}^{1} 
    \widetilde{V}_{n}^{p-1} 
    \left(\frac{x}{|x|^{2}}, 
    \theta\right) d \theta \right\|_{L^{q_{0}} 
    (B_{4})} < \infty
    \]
for any $q_{0} > \frac{3}{2}$ and 
\[
    \sup_{n \in \N} 
    \left\|\frac{1}{|x|^{4}} 
    \int_{0}^{1} 
    \widetilde{V}_{n}^{4}\left(\frac{x}{|x|^{2}}, 
    \theta\right) d \theta \right\|_{L^{\infty} 
    (B_{4})} < \infty. 
    \]
It suffices to show that 
    \begin{equation}\label{e-muni91}
       \sup_{n \in \N}
       \left\| \frac{\kappa_{n} M_{n, 1}^{-3}}
    {\|\nabla \widetilde{Z}_{n}\|_{L^{2}}} 
    |x|^{-5} \widetilde{R}_{\omega_{2}(m_{n})}
    (\frac{x}{|x|^{2}}) 
    \right\|_{W^{-1, 2q_{0}}(B_{4})} < \infty. 
    \end{equation}
We take $r_{0}, s_{0} \geq 1$ so that 
    \begin{equation} \label{e-muni92}
    \frac{3}{2} < r_{0} < 
    \frac{5 - p}{2}, \qquad 
    \frac{1}{r_{0}} + \frac{1}{s_{0}} 
    = 1. 
    \end{equation}
This yields that $s_{0} < 3 < 2q_{0}$. 
By the H\"{o}lder and 
Hardy inequalities, we obtain 
    \[
    \begin{split}
    & \quad 
    \left\| \frac{\kappa_{n} M_{n, 1}^{-3}}
    {\|\nabla \widetilde{Z}_{n}\|_{L^{2}}} 
    |x|^{-5} \widetilde{R}_{\omega_{2}(m_{n})}
    (\frac{x}{|x|^{2}}) 
    \right\|_{W^{-1, 2q_{0}}(B_{4})} \\
    & = \frac{\kappa_{n} M_{n, 1}^{-3}}
    {\|\nabla \widetilde{Z}_{n}\|_{L^{2}}} 
    \sup_{f \in W^{1, 2q_{0}}(B_{4}), 
    \|f\|_{W^{1, 2q_{0}}(B_{4})} = 1} 
    |\langle |x|^{-5} \widetilde{R}_{\omega_{2}(m_{n})}
    (\frac{x}{|x|^{2}}), f\rangle| \\
    & \leq \frac{M_{n, 1}^{-3}}
    {\|\nabla \widetilde{Z}_{n}\|_{L^{2}}} 
    \sup_{f \in W^{1, 2q_{0}}(B_{4}), 
    \|f\|_{W^{1, 2q_{0}}(B_{4})} = 1} 
    \left\| 
    |x|^{-4} \widetilde{R}_{\omega_{2}(m_{n})}
    (\frac{x}{|x|^{2}})
    \right\|_{L^{r_{0}}(B_{4})}
    \left\| |x|^{-1} f
    \right\|_{L^{s_{0}}(B_{4})} \\
    & \leq \frac{M_{n, 1}^{-3}}
    {\|\nabla \widetilde{Z}_{n}\|_{L^{2}}} 
    \sup_{f \in W^{1, 2q_{0}}(B_{4}), 
    \|f\|_{W^{1, 2q_{0}}(B_{4})} = 1} 
    \left\| 
    |x|^{-4} \widetilde{R}_{\omega_{2}(m_{n})}
    (\frac{x}{|x|^{2}})
    \right\|_{L^{r_{0}}(B_{4})}
    \left\| \nabla f
    \right\|_{L^{s_{0}}(B_{4})} \\
    & \leq \frac{M_{n, 1}^{-3}}
    {\|\nabla \widetilde{Z}_{n}\|_{L^{2}}} 
    \left\| 
    |x|^{-4} \widetilde{R}_{\omega_{2}(m_{n})}
    (\frac{x}{|x|^{2}})
    \right\|_{L^{r_{0}}(B_{4})}.  
    \end{split}
    \]
Then, for sufficiently small $\delta > 0$,  
we have 
    \[
    \begin{split}
    & \quad \frac{M_{n, 1}^{-3}}
    {\|\nabla \widetilde{Z}_{n}\|_{L^{2}}} 
    \left\| 
    |x|^{-4} \widetilde{R}_{\omega_{2}(m_{n})}
    (\frac{x}{|x|^{2}})
    \right\|_{L^{r_{0}}(B_{4})} \\
    & = \frac{M_{n, 1}^{-3}}
    {\|\nabla \widetilde{Z}_{n}\|_{L^{2}}} 
    \left\| 
    |x|^{-4} \widetilde{R}_{\omega_{2}(m_{n})}
    (\frac{x}{|x|^{2}})
    \right\|_{L^{r_{0}}(B_{
    \delta \sqrt{\alpha_{n, 1}}})} 
    + \frac{M_{n, 1}^{-3}}
    {\|\nabla \widetilde{Z}_{n}\|_{L^{2}}} 
    \left\| 
    |x|^{-4} \widetilde{R}_{\omega_{2}(m_{n})}
    (\frac{x}{|x|^{2}})
    \right\|_{L^{r_{0}}(B_{4} 
    \setminus B_{\delta 
    \sqrt{\alpha_{n, 1}}})} \\
    & =: I_{n} + II_{n}. 
   \end{split}
    \]
It follows from \eqref{e-muni77} and  
\eqref{e-muni92} that 
    \[
    \begin{split}
    I_{n} 
    & = \frac{M_{n, 1}^{-3}}
    {\|\nabla \widetilde{Z}_{n}\|_{L^{2}}} 
    \left( 
    \int_{B_{ \delta \sqrt{\alpha_{n, 1}}}} 
    \biggl| 
    |x|^{-4} \widetilde{R}_{\omega_{2}(m_{n})}
    (\frac{x}{|x|^{2}})
    \biggl|^{r_{0}} \, dx 
    \right)^{\frac{1}{r_{0}}} \\
    & \lesssim 
    M_{n, 1}^{-3} 
    \left( 
    \int_{0}^{\delta \sqrt{\alpha_{n, 1}}} 
    \biggl| 
    r^{-4} \widetilde{R}_{\omega_{2}(m_{n})}
    (\frac{1}{r})
    \biggl|^{r_{0}} r^{2} \, dr 
    \right)^{\frac{1}{r_{0}}} \\
    & \lesssim 
    M_{n, 1}^{-3} 
    \left( 
    \int_{0}^{\sqrt{\alpha_{n, 1}}}  
    r^{-3 r_{0} + 2} 
    e^{- r_{0}
    \frac{\sqrt{\alpha_{n, 1}}}{2 r}} \, dr 
    \right)^{\frac{1}{r_{0}}} \\
      & \lesssim 
    M_{n, 1}^{-3} 
    \left( 
    \int_{0}^{\sqrt{\alpha_{n, 1}}}  
    r^{-3r_{0} + 2} \, dr 
    \right)^{\frac{1}{r_{0}}} \\
    & \lesssim 
    M_{n, 1}^{-3} 
    \alpha_{n, 1}^{\frac{3}{2r_{0}} - \frac{3}{2}} 
    \lesssim 
    M_{n, 1}^{
    \frac{3(2 r_{0} - 5 + p)}{(3 - p) 
    r_{0}}} 
    \to 0 \qquad (n \to \infty),  
    \end{split}
    \]
    \[
    \begin{split}
    II_{n} 
    & = \frac{M_{n, 1}^{-3}}
    {\|\nabla \widetilde{Z}_{n}\|_{L^{2}}} 
    \left( 
    \int_{B_{4} \setminus B_{L^{-1}
    \sqrt{\alpha_{n, 1}}}} 
    \biggl| 
    |x|^{-4} \widetilde{R}_{\omega_{2}(m_{n})}
    (\frac{x}{|x|^{2}})
    \biggl|^{r_{0}} \, dx 
    \right)^{\frac{1}{r_{0}}} \\
    & \lesssim 
    M_{n, 1}^{-3} 
    \left( 
    \int_{L^{-1}\sqrt{\alpha_{n, 1}}}^{4} 
    \biggl| 
    r^{-4} \widetilde{R}_{\omega_{2}(m_{n})}
    (\frac{1}{r})
    \biggl|^{r_{0}} r^{2} \, dr 
    \right)^{\frac{1}{r_{0}}} \\
    & \lesssim 
    M_{n, 1}^{-3} 
    \left( 
    \int_{0}^{L^{-1} \sqrt{\alpha_{n, 1}}}  
    r^{-3r_{0} + 2} \, dr 
    \right)^{\frac{1}{r_{0}}} \\
      & \lesssim 
    M_{n, 1}^{-3} 
    \left( 
    \int_{0}^{\sqrt{\alpha_{n, 1}}} 
    r^{-3r_{0} + 2} \, dr 
    \right)^{\frac{1}{r_{0}}} \\
    & \lesssim 
    M_{n, 1}^{-3} 
    \alpha_{n, 1}^{\frac{3}{2r_{0}} 
    - \frac{3}{2}} 
    \lesssim 
    M_{n, 1}^{\frac{3(2 r_{0} - 5 + p)}
    {(3 - p) r_{0}}} 
    \to 0 \qquad (n \to \infty).  
    \end{split}
    \]
Thus, we see from Proposition \ref{proposition:B.1} 
that $\|W_{n}\|_{L^{\infty}} \leq C$ for some 
$C > 0$. 
This yields \eqref{eq-muni92}. 
\end{proof}

\begin{lemma}\label{l-muni17}
Let $1 < p < 2$. 
Then, there exists a constant $C>0$ such that 
$|\kappa_{n}| \leq 
C M_{n, 1}^{\frac{- p + 1}{2(3 - p)}}
\|\nabla Z_{n}\|_{L^{2}}$ 
for all $n \in \mathbb{N}$. 
\end{lemma}
\begin{proof}
As in \eqref{e-muni20}, 
we have 
	\[
	|\kappa_{n}|
	\lesssim \frac{1}{m_{n}} \left(
\|R_{\omega_{1} (m_{n})}\|_{L^{\frac{12}{5}}}^{2} 
+ \|R_{\omega_{2} (m_{n})}\|_{L^{\frac{12}{5}}}^{2} \right)\|Z_{n}\|_{L^{6}}. 
	\] 
This together with 
\eqref{e-muni61}, \eqref{e-muni46} 
and Lemma \ref{l-muni13}
 yields that 
\begin{equation*} 
\begin{split}
|\kappa_{n}| 
\lesssim M_{n, 1}^{\frac{7 - 3p}{3 - p}} 
\times M_{n, 1}^{\frac{5 p - 13}{2(3 - p)}} 
\|\nabla Z_{n}\|_{L^{2}}
\lesssim M_{n, 1}^{\frac{- p + 1}{2(3 - p)}} 
\|\nabla Z_{n}\|_{L^{2}}. 
\end{split}
\end{equation*}
This completes the proof. 
\end{proof}

In addition, 
we put $\overline{Z}_{n}(s) =\rho  \widetilde{Z}_{n}(\rho)/ \|\nabla \widetilde{Z}_{n}\|_{L^{2}}$ and 
$s = \sqrt{\alpha_{n, 1}} \rho$. 
Then, we see that $\overline{Z}_{n}(s)$ 
satisfies the following: 
\begin{equation}\label{e-muni72}
\begin{split}
& \quad 
- \frac{d^{2} \overline{Z}_{n}}{d s^{2}}(s) 
+ \overline{Z}_{n}(s) 
- p \frac{\beta_{n, 1}}
{\alpha_{n, 1}^{\frac{3 - p}{2}}} 
s^{- (p-1)}
\int_{0}^{1} \overline{V}_{n}^{p - 1} (s, \tau) d\tau \overline{Z}_{n}(s) 
- 5 \alpha_{n, 1}
s^{-4}
\int_{0}^{1} \overline{V}_{n}^{4}(s, \tau) 
d\tau \overline{Z}_{n}(s) \\ 
& 
= - \frac{\kappa_{n} 
\alpha_{n, 1}^{-1}  M_{n, 1}^{-3}}
{\|\nabla \widetilde{Z}_{n}\|_{L^{2}}} 
\overline{R}_{\omega_{2}(m_{n})}(s), 
\end{split}
\end{equation} 
where 
\[
\widehat{V}_{n}(r, \tau) = \tau 
\overline{R}_{\omega_{1}(m_{n})} (s) 
+ (1 - \tau)
\nu_{n} \overline{R}_{\omega_{2}(m_{n})}(s) 
\qquad 
\overline{R}_{\omega_{i}(m_{n})} (s) 
= \rho \widetilde{R}_{\omega_{i}(m_{n})} (\rho)
\qquad (\tau \in (0, 1), \; s>0) .
\] 
We see from \eqref{eq-muni92} that 
    \begin{equation} \label{eq-muni93}    
    \|\overline{Z}_{n}\|_{L^{\infty}} \leq C. 
    \end{equation}

\begin{remark}
From Lemma \ref{l-muni16} and $s = \sqrt{\alpha_{n, 1}} \rho 
= \sqrt{\alpha_{n, 1}} M_{n, 1}^{2} r = \omega_{1}^{\frac{1}{2}}(m_{n}) r$, we see that 
	\begin{equation} \label{e-muni73}
	C_{1} > 
    \frac{1}{\|\nabla \widetilde{Z}_{n}\|_{L^{2}}}
    |\rho  \widetilde{Z}_{n}(\rho)| 
	= \frac{M_{n, 1}^{2}}{\|\nabla \widetilde{Z}_{n}\|_{L^{2}}} 
    r |Z_{n}(r)| 
    = \frac{M_{n, 1}}{\|\nabla Z_{n}\|_{L^{2}}} 
    r |Z_{n}(r)| 
	\qquad \mbox{for $r \in (\delta \omega_{1}^{- \frac{1}{2}}(m_{n}), 
	L \omega_{1}^{- \frac{1}{2}}(m_{n}))$}. 
	\end{equation}
\end{remark}

\begin{lemma}\label{l-muni14}
Let $1 < p < 2$. 
Suppose that $\lim_{n \to \infty} 
\|\nabla \widetilde{Z}_{n}\|_{L^{2}} = \infty$.
Then, we have the following:
\begin{align}
& \biggl|\kappa_{n}\int_{\R^{3}} 
R_{\omega_{2}(m_{n})} Z_{n} \, dx\biggl| 
\leq M_{n, 1}^{-1}
\|\nabla Z_{n}\|_{L^{2}}\|Z_{n}\|_{L^{2}}, 
\label{e-muni79} \\
& \biggl|p \int_{\R^{3}} 
\int_{0}^{1} V_{n}^{p-1}(x, \tau) 
d\tau |Z_{n}|^{2} \, dx \biggl| 
\leq 
C (\delta^{-p + 3}
+  M_{n, 1}^{\frac{p - 5}{3 - p}}(m_{n})
\|\nabla Z_{n}\|_{L^{2}}^{2} 
+  \frac{C}{L^{p - 1}}\omega_{1}(m_{n}) 
\|Z_{n}\|_{L^{2}}^{2},   
\label{e-muni80} \\
& \biggl| 5 \int_{\R^{3}} 
\int_{0}^{1} V_{n}^{4}(x, \tau) d \tau |Z_{n}|^{2} \, dx \biggl| \leq 
C M_{n, 1}^{\frac{2(p - 5)}{3 - p}}  
\|\nabla Z_{n}\|_{L^{2}}^{2}.
\label{e-muni81} 
\end{align}
\end{lemma}
\begin{proof}
We first show \eqref{e-muni79}.  
For any $\delta > 0$ and $L > 0$, we have 
\[
\begin{split}
\biggl|\kappa_{n}\int_{\R^{3}} 
R_{\omega_{2}(m_{n})} Z_{n} \, dx\biggl|
& \leq |\kappa_{n}|\int_{\R^{3}} 
R_{\omega_{2}(m_{n})} |Z_{n}| \, dx \\
& \leq 
|\kappa_{n}| 4 \pi
\int_{0}^{L \omega_{2}^{-\frac{1}{2}}(m_{n})} 
R_{\omega_{2}(m_{n})} |Z_{n}| r^{2}\, dr
+ |\kappa_{n}|4 \pi 
\int_{L \omega_{2}^{-\frac{1}{2}}(m_{n})}^{\infty} 
R_{\omega_{2}(m_{n})} |Z_{n}| r^{2}\, dr \\
& =: I_{n} + II_{n}. 
\end{split}
\]
First, we estimate $I_{n}$. 
By the H\"{o}lder inequality, \eqref{EqL-1},  
\eqref{e-muni63} and Lemma \ref{l-muni7},
we obtain 
\[
\begin{split}
I_{n} 
& \lesssim 
M_{n, 1}^{\frac{- p + 1}{2(3 - p)}} 
\|\nabla Z_{n}\|_{L^{2}}
\left(\int_{0}^{L \omega_{2}^{-\frac{1}{2}}(m_{n})} 
R_{\omega_{2}(m_{n})}^{2}(r) r^{2}\, dr \right)^{\frac{1}{2}} \|Z_{n}\|_{L^{2}} \\
& \lesssim 
M_{n, 1}^{\frac{- p + 1}{2(3 - p)}} 
\|\nabla Z_{n}\|_{L^{2}}
\left(\int_{0}^{L \omega_{2}^{-\frac{1}{2}}(m_{n})} 
M_{n, 1}^{2}\widetilde{R}_{\omega_{2}
(m_{n})}^{2}(M_{n, 2}^{2} r) r^{2}\, dr \right)^{\frac{1}{2}} \|Z_{n}\|_{L^{2}} \\
& \lesssim 
M_{n, 1}^{\frac{- p + 1}{2(3 - p)} + 1} 
\|\nabla Z_{n}\|_{L^{2}}\|Z_{n}\|_{L^{2}}
\left(\int_{0}^{L} 
\widetilde{R}_{\omega_{2}
(m_{n})}^{2}(\alpha_{n, 2}^{-\frac{1}{2}} s) 
\omega_{2}^{-\frac{3}{2}}(m_{n})s^{2}\, ds
\right)^{\frac{1}{2}} \\
& \lesssim 
M_{n, 1}^{\frac{- p + 1}{2(3 - p)} + 1}
\omega_{2}^{-\frac{3}{4}} (m_{n})
\alpha_{n, 2}^{\frac{1}{2}}
\|\nabla Z_{n}\|_{L^{2}}\|Z_{n}\|_{L^{2}}
\left(\int_{0}^{L} \, ds
\right)^{\frac{1}{2}} \\
& \lesssim 
L^{\frac{1}{2}}
M_{n, 1}^{-1}
\|\nabla Z_{n}\|_{L^{2}}\|Z_{n}\|_{L^{2}}.
\end{split}
\] 
Similarly, one has by \eqref{eqU-6-3} that  
\[
\begin{split}
II_{n} 
& \lesssim 
M_{n, 1}^{\frac{- p + 1}{2(3 - p)}} 
\|\nabla Z_{n}\|_{L^{2}}
\left(\int_{L \omega_{2}^{-\frac{1}{2}}(m_{n})}^{\infty} 
R_{\omega_{2}(m_{n})}^{2}(r) r^{2}\, dr \right)^{\frac{1}{2}} \|Z_{n}\|_{L^{2}} \\
& \lesssim 
M_{n, 1}^{\frac{- p + 1}{2(3 - p)}} 
\|\nabla Z_{n}\|_{L^{2}}
\left(\int_{L \omega_{2}^{-\frac{1}{2}}(m_{n})}^{\infty} 
M_{n, 2}^{2}\widetilde{R}_{\omega_{2}
(m_{n})}^{2}(M_{n, 2}^{2} r) r^{2}\, dr \right)^{\frac{1}{2}} \|Z_{n}\|_{L^{2}} \\
& \lesssim 
M_{n, 1}^{\frac{- p + 1}{2(3 - p)} + 1} 
\|\nabla Z_{n}\|_{L^{2}}\|Z_{n}\|_{L^{2}}
\left(\int_{L}^{\infty} 
\widetilde{R}_{\omega_{2}
(m_{n})}^{2}(\alpha_{n, 2}^{-\frac{1}{2}} s) 
\omega_{2}^{-\frac{3}{2}}(m_{n})s^{2}\, ds
\right)^{\frac{1}{2}} \\
& \lesssim 
M_{n, 1}^{\frac{- p + 1}{2(3 - p)} + 1}
\omega_{2}^{-\frac{3}{4}} (m_{n})
\alpha_{n, 2}^{\frac{1}{2}}
\|\nabla Z_{n}\|_{L^{2}}\|Z_{n}\|_{L^{2}}
\left(\int_{L}^{\infty} 
e^{- \frac{s}{2}} \, ds
\right)^{\frac{1}{2}} \\
& \lesssim 
e^{- \frac{L}{2}}
M_{n, 1}^{-1}
\|\nabla Z_{n}\|_{L^{2}}\|Z_{n}\|_{L^{2}}.
\end{split}
\]

Next, we show \eqref{e-muni80}. 
We have  
    \[
    \begin{split}
    & \quad \biggl|p \int_{\R^{3}} 
\int_{0}^{1} V_{n}^{p-1}(x, \tau) 
d\tau |Z_{n}|^{2} \, dx \biggl| \\
& \lesssim \int_{_{|x| \leq \delta 
\omega_{1}^{- \frac{1}{2}}(m_{n})}} 
\int_{0}^{1} |V_{n}^{p-1}(x, \tau)| 
d\tau |Z_{n}|^{2} \, dx
+ \int_{\delta 
\omega_{1}^{- \frac{1}{2}}(m_{n}) 
\leq |x| \leq L 
\omega_{1}^{- \frac{1}{2}}(m_{n})} 
\int_{0}^{1} |V_{n}^{p-1}(x, \tau)| 
d\tau |Z_{n}|^{2} \, dx \\
& \quad + \int_{|x| \geq L 
\omega_{1}^{- \frac{1}{2}}(m_{n})} 
\int_{0}^{1} |V_{n}^{p-1}(x, \tau)| 
d\tau |Z_{n}|^{2} \, dx 
=:I_{n} + II_{n} + III_{n}. 
   \end{split}
    \]
Note that 
\begin{equation} \label{e-muni83}
0 < R_{\omega_{j}(m_{n})}(x)
= M_{n, j}\widetilde{R}_{n, j}(
M_{n, j}^{2} x) 
\lesssim \frac{M_{n, j}}
{M_{n, j}^{2} |x|} 
= \frac{1}{M_{n, j} |x|}. 
\end{equation}
This implies that 
\[
0 < 
\biggl|\int_{0}^{1} V_{n}^{p-1}(x, \tau) d\tau \biggl| \lesssim 
\left(R_{\omega_{1}(m_{n})}^{p - 1}(x) + 
R_{\omega_{2}(m_{n})}^{p - 1}(x)
\right) 
\lesssim (M_{n, 1}^{- (p-1)} 
+ M_{n, 2}^{-(p - 1)})
|x|^{- (p - 1)}. 
\]
Thus, it follows from \eqref{e-muni63} that 
there exists $C > 0$, 
which is independent of $n \in \N$, 
such that 
\[
\biggl|\int_{0}^{1} V_{n}^{p-1}(x, \tau) d\tau \biggl| \lesssim
M_{n, 1}^{-(p-1)} \frac{\omega_{1}^{\frac{p-1}{2}}(m_{n})}{L^{p - 1}} \lesssim 
\frac{\omega_{1}(m_{n})}{L^{p-1}} 
\qquad \mbox{for $|x| \geq 
L \omega_{1}^{- \frac{1}{2}}(m_{n})$}. 
\]
This yields that 
\[
III_{n} = 
\biggl|\int_{|x| \geq L\delta 
\omega_{1}^{- \frac{1}{2}}(m_{n})} 
\int_{0}^{1} V_{n}^{p-1}(x, \tau) 
d\tau |Z_{n}|^{2} \, dx \biggl|
\leq \frac{C}{L^{p - 1}}\omega_{1}(m_{n}) 
\|Z_{n}\|_{L^{2}}^{2}. 
\]
By the H\"{o}lder inequality, \eqref{EqL-1} and 
\eqref{e-muni63}, 
we obtain 
\[
\begin{split}
I_{n} & = \biggl|\int{|x| \leq \delta 
\omega_{1}^{- \frac{1}{2}}(m_{n})} 
\int_{0}^{1} V_{n}^{p-1}(x, \tau) 
d\tau |Z_{n}|^{2} \, dx \biggl| \\
& \lesssim 
\left(\int_{|x| \leq \delta 
\omega_{1}^{- \frac{1}{2}}(m_{n})}
\left(R_{\omega_{1}(m_{n})}(x) 
+ R_{\omega_{2}(m_{n})}(x) 
\right)^{\frac{3}{2}(p - 1)}
\, dx \right)^{\frac{2}{3}}\|Z_{n}\|_{L^{6}}^{2} \\
& \lesssim 
M_{n, 1}^{p - 1}\left( 
\int_{0}^{\delta 
\omega_{1}^{- \frac{1}{2}}(m_{n})}
\left(\widetilde{R}_{\omega_{1}(m_{n})}
^{\frac{3}{2}(p - 1)}(M_{n, 1}^{2}r) 
+ \widetilde{R}
_{\omega_{2}(m_{n})}^{\frac{3}{2}(p - 1)}
(M_{n, 2}^{2} r) \right)r^{2}
\, dr \right)^{\frac{2}{3}} 
\|\nabla Z_{n}\|_{L^{2}}^{2} \\
& \lesssim M_{n, 1}^{p-1} 
\omega_{1}^{-1}(m_{n})
\left( 
\int_{0}^{\delta}
(\widetilde{R}_{\omega_{1}(m_{n})}
^{\frac{3}{2}(p - 1)}(\alpha_{n, 1}^{- 
\frac{1}{2}}s) 
+ \widetilde{R}
_{\omega_{2}(m_{n})}^{\frac{3}{2}(p - 1)}
(\left(
\frac{\omega_{2}(m_{n})}{\omega_{1}(m_{n})} 
\right)^{\frac{1}{2}} 
\alpha_{n, 2}^{- \frac{1}{2}} s))
s^{2} \, ds \right)^{\frac{2}{3}} 
\|\nabla Z_{n}\|_{L^{2}}^{2} \\
& \lesssim M_{n, 1}^{p-1} 
\omega_{1}^{-1}(m_{n})
\left( 
\int_{0}^{\delta} \alpha_{n, 1}^{\frac{3}{4} 
(p-1)} s^{- \frac{3}{2} p + \frac{7}{2}} \, ds 
\right)^{\frac{2}{3}} 
\|\nabla Z_{n}\|_{L^{2}}^{2}\\
& \lesssim M_{n, 1}^{p - 1} 
\omega_{1}^{-1}(m_{n})
\alpha_{n, 1}^{\frac{p - 1}{2}}
\delta^{- p + 3} \|\nabla Z_{n}\|_{L^{2}}^{2} 
\\
& 
= \delta^{- p + 3} 
\|\nabla Z_{n}\|_{L^{2}}^{2}.
\end{split}
\] 
In addition, by 
\eqref{e-muni83}, \eqref{e-muni73} and \eqref{e-muni63}, 
we have 
 \[
\begin{split}
II_{n}& = \biggl|\int_{\delta 
\omega_{1}^{- \frac{1}{2}}(m_{n}) 
\leq |x| \leq L 
\omega_{1}^{- \frac{1}{2}}(m_{n})} 
\int_{0}^{1} V_{n}^{p-1}(x, \tau) 
d\tau |Z_{n}|^{2} \, dx \biggl| \\
& \lesssim 
\int_{\delta \omega_{1}^{- \frac{1}{2}}(m_{n}) 
\leq |x| \leq L 
\omega_{1}^{- \frac{1}{2}}(m_{n})} 
\left(R_{\omega_{1}(m_{n})}^{p - 1}(x) 
+ R_{\omega_{2}(m_{n})}^{p - 1}(x)\right)
|Z_{n}|^{2} \, dx \\
& \lesssim 
\int_{\delta \omega_{1}^{- \frac{1}{2}}(m_{n})}
^{L 
\omega_{1}^{- \frac{1}{2}}(m_{n})}
M_{n, 1}^{- (p-1)}  r^{- (p-1)} 
M_{n, 1}^{-2}
\frac{ \|\nabla Z_{n}\|_{L^{2}}^{2}}
{r^{2}} r^{2} \, dr \\
& \lesssim M_{n, 1}^{- p - 1}
\|\nabla Z_{n}\|_{L^{2}}^{2}
\int_{\delta \omega_{1}^{- \frac{1}{2}}(m_{n})}
^{L 
\omega_{1}^{- \frac{1}{2}}(m_{n})} 
r^{-p + 1} \, dr \\
& \lesssim M_{n, 1}^{-p - 1}
\|\nabla Z_{n}\|_{L^{2}}^{2}
\omega_{1}^{\frac{p - 2}{2}}(m_{n}) 
\leq C M_{n, 1}^{\frac{p - 5}
{3 - p}}
\|\nabla Z_{n}\|_{L^{2}}^{2}. 
\end{split}
\] 

Finally, we give the proof of \eqref{e-muni81}. 
By the H\"{o}lder inequality and Lemma \ref{l-muni13}, we obtain 
\begin{equation} \label{e-muni84}
\begin{split}
\biggl| 5 \int_{\R^{3}} 
\int_{0}^{1} V_{n}^{4}(x, \tau) d \tau |Z_{n}|^{2} \, dx \biggl|
& \lesssim \int_{\R^{3}} (|R_{\omega_{1}(m_{n})}|^{4} 
+ |R_{\omega_{2}(m_{n})}|^{4}) |Z_{n}|^{2} \, dx \\[6pt]
& \leq (\|R_{\omega_{1}(m_{n})}\|_{L^{6}}^{4} 
+ \|R_{\omega_{2}(m_{n})}\|_{L^{6}}^{4}) \|Z_{n}\|_{L^{6}}^{2} \\[6pt]
& \leq C 
M_{n, 1}^{\frac{2(p - 5)}{3 - p}} 
\|\nabla Z_{n}\|_{L^{2}}^{2}. 
\end{split}
\end{equation}
This completes the proof. 
\end{proof}
\begin{lemma}\label{l-muni15}
Let $1 < p < 2$. 
We have $\sup_{n \in \N} 
\|\nabla \widetilde{Z}_{n}\|_{L^{2}} < \infty$.
\end{lemma}

\begin{proof}
Multiplying \eqref{e-muni68} by $Z_{n}$ and integrating the 
resulting equation, we obtain 
\begin{equation} \label{e-muni85}
\begin{split}
& \quad 
\|\nabla Z_{n}\|_{L^{2}}^{2}
+ \omega_{1}(m_{n}) 
\|Z_{n}\|_{L^{2}}^{2} 
- p \int_{\R^{3}} 
\int_{0}^{1} V_{n}^{p-1}(x, \tau) 
d\tau |Z_{n}|^{2} \, dx 
- 5 \int_{\R^{3}} 
\int_{0}^{1} V_{n}^{4}(x, \tau) d \tau 
|Z_{n}|^{2} \, dx \\
& = - \kappa_{n} 
\int_{\R^{3}} R_{\omega_{2}(m_{n})} Z_{n} \, dx.
\end{split}
\end{equation}
Suppose that $\lim_{n \to \infty} 
\|\nabla \widetilde{Z}_{n}\|_{L^{2}} = \infty$.
From Lemmas \ref{l-muni14}, we obtain 
\[
\begin{split}
\|\nabla Z_{n}\|_{L^{2}}^{2} + \omega_{1}(m_{n}) 
\|Z_{n}\|_{L^{2}}^{2} 
& \leq 
C (\delta^{-p + 3}
+  \e^{2} M_{n, 1}^{\frac{p - 5}{3 - p}}(m_{n})
\|\nabla Z_{n}\|_{L^{2}}^{2} 
+  \frac{C}{L^{p - 1}}\omega_{1}(m_{n}) 
\|Z_{n}\|_{L^{2}}^{2}
+ C M_{n, 1}^{\frac{2(p - 5)}{3 - p}}  
\|\nabla Z_{n}\|_{L^{2}}^{2} \\
& \quad 
+ \e L M_{n, 2}^{\frac{p - 5}{2(3 - p)}} 
\|\nabla Z_{n}\|_{L^{2}}^{2}
+  M_{n, 1}^{-1}
\|\nabla Z_{n}\|_{L^{2}}\|Z_{n}\|_{L^{2}}. 
\end{split}
\] 
This implies that 
$\|\nabla Z_{n}\|_{L^{2}}^{2} + \omega_{1}(m_{n}) 
\|Z_{n}\|_{L^{2}}^{2} < 0$, which is absurd.
This completes the proof. 
\end{proof}
We see from Lemma \ref{l-muni15}, 
$\widetilde{Z}_{n}(\rho) = Z_{n}(r)$ and $\rho = M_{n, 1}^{2} r$
that 
    \begin{equation} \label{e-muni95}
\|\nabla Z_{n}\|_{L^{2}} \leq C M_{n, 2}^{-2}
    \end{equation}
for some $C > 0$, which yields that  
$\lim_{n \to \infty} Z_{n} = 0$ strongly 
in 
$\dot{H}^{1}(\R^{3})$. 
In addition, we have the following: 
\begin{lemma}
Let $1 < p < 2$. 
\begin{equation} \label{e-muni86}
\lim_{n \to \infty}\kappa_{n} = 0. 
\end{equation}
\end{lemma} 
\begin{proof}
It follows from Lemma \ref{l-muni17} 
and \eqref{e-muni95} that 
\[
|\kappa_{n}| \leq 
C M_{n, 1}^{\frac{- p + 1}{2(3 - p)}} 
\|\nabla Z_{n}\|_{L^{2}}
\leq C M_{n, 1}^{\frac{p - 5}{2(3 - p)}} 
\to 0 \qquad \mbox{as $n \to \infty$}. 
\]
This completes the proof. 
\end{proof}
Since the rest of the proof is 
similar to that of the case of $2 < p < 3$, 
we omit it. 
\section{Differentiability of the minimizer $R_{\omega}$ 
with respect to $\omega$
in the $L^{2}$-supercritical case}
\label{sec-reg}
To prove Theorem \ref{thm-bl} \textrm{(ii)}, 
we apply the abstract theory of Grillakis, Shatah and 
Strauss~\cite{MR901236}. 
To this end, we need to show that 
$R_{\omega}|_{\omega = \omega(m)}$ is 
differentiable with respect to 
the frequency $\omega > 0$ for sufficiently small $m>0$. 
\subsection{Continuity of the minimization value $E_{\min}(m)$}
In this subsection, we shall obtain the continuity of $E_{\min}(m)$ 
with respect to $m$. 
\begin{proposition}\label{lem1-Diff}
Let $7/3 < p < 6$.
$E_{\min}(m)$ is continuous on $(0, \infty)$. 
\end{proposition}
To prove Proposition \ref{lem1-Diff}, 
we collect several basic properties of 
the minimization problem $E_{\min}(m)$. 
By an elementary computation, we can obtain the following: 
\begin{lemma}\label{lem2-Diff}
Let $7/3 < p < 6$ and $u \in H^{1}(\R^3) 
\setminus \{0\}$. 
Then, there exists a unique $\lambda(u) > 0$ 
such that 
\begin{equation*}
\mathcal{K}(\lambda^{\frac{3}{2}} 
u(\lambda \cdot))
\begin{cases}
> 0 & \qquad \mbox{if $\lambda \in (0, 
\lambda(u))$}, \\[6pt]
= 0 & \qquad \mbox{if $\lambda = \lambda(u)$}, \\[6pt]
< 0 & \qquad \mbox{if $\lambda \in (\lambda(u), \infty)$}. 
\end{cases}
\end{equation*}
\end{lemma}
\begin{lemma}\label{lem3-Diff}
Let $7/3 < p < 6$, 
$m > 0$ and $u \in H^{1}(\R^{3})$ with 
$\mathcal{K}(u) = 0$ and 
$\|u\|_{L^{2}}^{2} = n$
for $0 < n < m$. 
There exist $\rho = \rho(m, n) \in (0, 1)$ and 
$v \in H^{1}(\R^{3})$ with $\|v\|_{L^{2}}^{2} = m$ 
and $\mathcal{K}(v) = 0$ such that 
\begin{equation} \label{eq1-Diff}
\mathcal{E}(v) = 
\mathcal{E}(u) 
- \frac{\|\nabla u\|_{L^{2}}^{2}}{6(p-1)} 
\left\{ 
(3p-7)
(1 - \rho(m, n, u)) 
+ (5-p)(1 - \rho^{3}(m, n, u))
\beta(u) \right\}, 
\end{equation}
where 
\[
\beta(u) = \dfrac{\|u\|_{L^{6}}^{6}}
{\|\nabla u\|_{L^{2}}^{2}} >0. 
\]
\end{lemma}
\begin{remark}\label{rem1-sec2}
It follows from 
\eqref{eq1-Diff} that 
$E_{\min}(m) \leq \E(v) \leq \E(u)$. 
Taking the infimum over 
$u \in H^{1}(\R^{3})$ with 
$\|u\|_{L^{2}}^{2} = n$ and $\K(u) = 0$, 
we have $E_{\min}(m) \leq E_{\min}(n)$. 
Thus, 
$E_{\min}(m)$ is non-increasing 
function of $m$. 
\end{remark}
\begin{proof}[Proof of Lemma \ref{lem3-Diff}]
It follows from $\mathcal{K}(u) = 0$ that 
\begin{equation} \label{eq2-Diff}
\frac{3(p-1)}{2(p+1)} \|u\|_{L^{p+1}}^{p+1} 
= \|\nabla u\|_{L^{2}}^{2} (1 - \beta(u)). 
\end{equation}
We see from \eqref{eq2-Diff} 
that $0 < \beta(u) < 1$. 
Putting 
\[
v (\cdot) = \sqrt{\rho} \lambda^{\frac{1}{2}} 
u(\lambda \cdot) \qquad \mbox{for $\lambda> 0$ 
and $0 < \rho < 1$}, 
\]
we have 
\begin{equation}\label{eq3-Diff}
\begin{split}
& \|\nabla v\|_{L^{2}}^{2} 
= \rho \|\nabla u\|_{L^{2}}^{2}, \qquad 
\|v\|_{L^{6}}^{6} 
= \rho^{3} \|u\|_{L^{6}}^{6}, \\[6pt]
& \|v\|_{L^{2}}^{2} 
= \rho \lambda^{-2}
\|u\|_{L^{2}}^{2}, 
\qquad 
\|v\|_{L^{p+1}}^{p+1} 
= \rho^{\frac{p+1}{2}} \lambda^{- \frac{5 - p}{2}}
\|u\|_{L^{p+1}}^{p+1}. 
\end{split}
\end{equation}
This together with \eqref{eq3-Diff} yields that 
\begin{equation*}
\begin{split}
\mathcal{K}(v) 
& = \rho \|\nabla u\|_{L^{2}}^{2} 
- \frac{3(p-1)}{2(p+1)} \rho^{\frac{p+1}{2}} 
\lambda^{- \frac{5 - p}{2}} \|u\|_{L^{p+1}}^{p+1} 
- \rho^{3} \|u\|_{L^{6}}^{6} \\[6pt]
& = \rho \|\nabla u\|_{L^{2}}^{2} 
\left( 
1 - \rho^{\frac{p-1}{2}} \lambda^{- \frac{5 - p}{2}}
(1 - \beta(u)) - 
\rho^{2} \beta(u)
\right). 
\end{split}
\end{equation*}
Note that $0 < \beta(u)$ and $\rho < 1$. 
Then, choosing 
\begin{equation} \label{eq4-Diff}
\lambda(u, \rho) 
= \left(\frac{\rho^{\frac{p-1}{2}}(1 - \beta(u))}
{1 - \rho^{2} \beta(u)}
\right)^{\frac{2}{5 - p}}, 
\end{equation}
we see that $\lambda(u, \rho)$ is positive and satisfies 
\[
1 - \rho^{\frac{p-1}{2}} \lambda^{- \frac{5 - p}{2}}(u, \rho)
(1 - \beta(u)) - \rho^{2} \beta(u) = 0. 
\]
Then, we obtain $\K(v) = 0$.
We claim that there exists 
$\rho(m, n) \in (0, 1)$ so that
\begin{equation} \label{eq5-Diff}
\|v\|_{L^{2}}^{2}
= \rho(m, n) \lambda^{-2}(u, \rho) 
\|u\|_{L^{2}}^{2} = m. 
\end{equation} 
By \eqref{eq4-Diff}, $\|u\|_{L^{2}}^{2} = n$ and 
$n < m$, 
\eqref{eq5-Diff} is equivalent to 
\[
\rho^{- \frac{3p - 7}{5 - p}} (m, n)
\left(\frac{1 - \rho^{2}(m, n) \beta(u)}{1 - \beta(u)}
\right)^{\frac{4}{5 - p}} 
= \frac{m}{n} (> 1). 
\]
We put 
\[
f(\rho) := \rho^{- \frac{3p - 7}{5 - p}}
\left(\frac{1 - \rho^{2} \beta(u)}{1 - \beta(u)}
\right)^{\frac{4}{5 - p}}. 
\] 
Clearly, we have $f(1) = 1$.
In addition, since $p > \frac{7}{3}$, we see that 
$\lim_{\rho \to 0} f(\rho) = + \infty$. 
Thus, by the intermediate theorem, there exists $\rho(m, n) 
\in (0, 1)$ such that 
$f(\rho(m, n)) = \frac{m}{n}$. 
Thus, \eqref{eq5-Diff} holds. 

Since $\mathcal{K}(v) = 0$, we have by \eqref{eq3-Diff} that 
\begin{equation}\label{eq6-Diff}
\begin{split}
\mathcal{E}(v) 
& 
= \mathcal{E}(v) - 
\frac{2}{3(p - 1)} \mathcal{K} 
(v) \\[6pt]
& 
= \frac{3p-7}{6(p-1)}
\|\nabla v\|_{L^{2}}^{2} 
+ \frac{5-p}{6(p-1)} 
\|v\|_{L^{6}}^{6} \\[6pt]
& = 
\frac{3p-7}{6(p-1)} \rho(m, n) 
\|\nabla u \|_{L^{2}}^{2} + 
\frac{5-p}{6(p-1)} \rho^{3}(m, n) 
\|u\|_{L^{6}}^{6}. 
\end{split}
\end{equation}
Similarly, by $\mathcal{K}(u) = 0$, we obtain 
\begin{equation*}
\mathcal{E}(u) = \mathcal{E}(u) 
- \frac{2}{3(p - 1)}\mathcal{K}(u) 
= \frac{3p-7}{6(p-1)}
\|\nabla u\|_{L^{2}}^{2} 
+ \frac{5-p}{6(p-1)} 
\|u\|_{L^{6}}^{6}. 
\end{equation*}
This together with \eqref{eq6-Diff} yields that 
\begin{equation*}
\begin{split}
\mathcal{E}(v) 
& = 
\mathcal{E}(u) - 
\left\{ 
\frac{3p-7}{6(p-1)}
(1 - \rho(m, n))
\|\nabla u\|_{L^{2}}^{2} + 
\frac{5-p}{6(p-1)}
(1 - \rho^{3}(m, n))
\|u\|_{L^{6}}^{6} 
\right\}
\\[6pt]
& = \mathcal{E}(u) 
- \frac{\|\nabla u\|_{L^{2}}^{2}}{6(p - 1)} 
\left\{ 
(3p-7)(1 - \rho(m, n)) + 
(5-p)(1 - \rho^{3}(m, n))
\beta(u) \right\}. 
\end{split}
\end{equation*}
Thus, \eqref{eq1-Diff} holds. 
\end{proof}
We are now in a position to prove Proposition \ref{lem1-Diff}. 
\begin{proof}[Proof of Proposition \ref{lem1-Diff}]
Let $0 < n < m$. 
It follows from Remark \ref{rem1-sec2} that 
\begin{equation}\label{eq7-Diff}
E_{\min}(m) \leq E_{\min}(n).
\end{equation}
In addition, note that 
$\|(n/m)^{\frac{1}{2}}R_{\omega(m)}\|_{L^{2}}^{2} = n$. 
Since $\mathcal{K}(R_{\omega(m)}) = 0$ and $n/m \in (0, 1)$, 
we see that 
\[
\begin{split}
\mathcal{K} ((n/m)^{\frac{1}{2}}R_{\omega(m)}) 
& = \frac{n}{m} \left\{ 
\|\nabla R_{\omega(m)}\|_{L^{2}}^{2} 
- \frac{3(p-1)}{2(p+1)} \left(\frac{n}{m}\right)^{\frac{p - 1}{2}}
\|R_{\omega(m)}\|_{L^{p+1}}^{p+1}
- \left(\frac{n}{m}\right)^{2} \|R_{\omega(m)}\|_{L^{6}}^{6}
\right\} \\[6pt]
& > \frac{n}{m} \mathcal{K} (R_{\omega(m)}) = 0. 
\end{split}
\]
Then, it follows from Lemma \ref{lem2-Diff}
that there exists $\lambda_{m, n} > 1$ such that 
\[
\mathcal{K}(\lambda_{m, n}^{\frac{3}{2}}
((n/m)^{\frac{1}{2}}R_{\omega(m)})( \lambda_{m, n} \cdot)) = 0.
\]
We can easily verify that $\lim_{n \to m} \lambda_{m, n} 
= 1$ and $\| \lambda_{m, n}^{\frac{3}{2}} 
((n/m)^{\frac{1}{2}}R_{\omega(m)})(\lambda_{m, n} \cdot)\|_{L^{2}}^{2} 
= n$. 
This together with the definition 
of $E_{\min}(n)$ yields that 
\begin{equation} \label{eq8-Diff}
\begin{split}
E_{\min}(n) 
& \leq \E(\lambda_{m, n}^{\frac{3}{2}}
((n/m)^{\frac{1}{2}}R_{\omega(m)})( \lambda_{m, n} \cdot)) \\[6pt]
& = \frac{\lambda_{m, n}^{2}}{2} 
\frac{n}{m} 
\|\nabla R_{\omega(m)}\|_{L^{2}}^{2}
- \frac{\lambda_{m, n}^{\frac{3}{2}(p-1)}}
{p+1} 
\left(\frac{n}{m} \right)^{\frac{p+1}{2}}
\|R_{\omega(m)}\|_{L^{p+1}}^{p+1} 
- \frac{\lambda_{m, n}^{6}}{6} 
\left(\frac{n}{m} \right)^{3} 
\|R_{\omega(m)}\|_{L^{6}}^{6} \\[6pt]
& = E_{\min}(m) 
+ \frac{1}{2} 
\left(\lambda_{m, n}^{2}
\left(\frac{n}{m}\right) - 1\right) 
\|\nabla R_{\omega(m)}\|_{L^{2}}^{2}
- \frac{1}{p+1} 
\left(
\lambda_{m, n}^{\frac{3}{2}(p-1)} 
\left(\frac{n}{m} \right)^{\frac{p+1}{2}}
- 1
\right)
\|R_{\omega(m)}\|_{L^{p+1}}^{p+1} \\[6pt]
&
\quad - \frac{1}{6} 
\left(\lambda_{m, n}^{6}
\left(\frac{n}{m} \right)^{3} - 1
\right)
\|R_{\omega(m)}\|_{L^{6}}^{6}. 
\end{split}
\end{equation}
By $\mathcal{K}(R_{\omega(m)}) = 0$ and \eqref{eq7-Diff}, we obtain 
\begin{equation} \label{eq9-Diff}
\begin{split}
E_{\min}(n) 
\geq E_{\min}(m) 
= \mathcal{E}(R_{\omega(m)}) 
& = \mathcal{E}(R_{\omega(m)}) 
- \frac{2}{3(p - 1)}\mathcal{K}(R_{\omega(m)}) \\[6pt]
& = \frac{3p-7}{6(p-1)}
\|\nabla R_{\omega(m)}\|_{L^{2}}^{2} 
+ \frac{5-p}{6(p-1)} 
\|R_{\omega(m)}\|_{L^{6}}^{6}. 
\end{split}
\end{equation}
This together with \eqref{eq8-Diff}, 
$\|R_{\omega(m)}\|_{L^{2}}^{2} = m$ 
and the H\"{o}lder inequality 
yields that 
there exists a constant $C>0$ such that 
\begin{equation}
\begin{split}
E_{\min}(n) 
& \leq E_{\min}(m) 
+ CE_{\min}(n) \biggl|\lambda_{m, n}^{2} 
\left(\frac{n}{m}\right) - 1\biggl| 
+ C m^{\frac{5-p}{4}} E_{\min}^{\frac{p-1}{4}}(n) 
\biggl| \lambda_{m, n}^{\frac{3}{2}(p-1)}
\left(\frac{n}{m} \right)^{\frac{p+1}{2}}
- 1 \biggl| \\[6pt]
& \quad 
+ CE_{\min}(n)\biggl|\lambda_{m, n}^{6} 
\left(\frac{n}{m} \right)^{3} - 1\biggl|
\end{split}
\end{equation}
Then, it follows from \eqref{eq8-Diff} that 
for any $\varepsilon > 0$, there exists 
$\delta > 0$ such that 
for $n < m < n + \delta$, we have 
\begin{equation}\label{eq10-Diff}
E_{\min}(n) < E_{\min}(m) 
+ \varepsilon.
\end{equation}
From \eqref{eq7-Diff} and \eqref{eq10-Diff}, 
we have obtained the desired result. 
\end{proof}
\subsection{Continuity of the mapping 
$m \in (0, \infty) \mapsto R_{\omega(m)} \in H_{\text{rad}}^{1}(\R^{3})$}
This subsection is devoted to the following: 
\begin{proposition}\label{lem2-diff}
Let $7/3 < p < 6$ and 
$m_{1} > 0$ be the number given in 
Theorem \ref{thm-muni}. 
The map $m \in (0, m_{1}) \mapsto (\omega(m), R_{\omega(m)}) 
\in (0, \infty) \times H_{\text{rad}}^{1}(\R^{3})$ 
is continuous. 
\end{proposition}
To prove Proposition \ref{lem2-diff}, 
we need several preparations. 
First, we obtain the following lower bound of 
$E_{\min}(m)$: 
\begin{lemma} \label{lem-pe}
Let $7/3 \leq p < 6$.
For each $m >0$, there exists a 
constant $C = C(m) > 0$ such that 
$E_{\min}(m) > C$. 
\end{lemma}
\begin{proof}
Let $u \in H^{1}(\R^{3}) \setminus \{0\}$ satisfy 
$\|u\|_{L^{2}}^{2} = m$ and $\K(u) = 0$. 
Since $\K(u) = 0$, 
we see that 
\begin{equation*}
\begin{split}
\mathcal{E}(u) 
= \mathcal{E}(u) - 
\frac{1}{2} \K(u)
= \frac{3p-7}{4(p+1)}
\|u\|_{L^{p+1}}^{p+1} + 
\frac{1}{3} \|u\|_{L^{6}}^{6}. 
\end{split}
\end{equation*}
In addition, it follows from 
$\K(u) = 0$, the Gagliardo-Nirenberg and the Sobolev inequalities 
and $\|u\|_{L^{2}}^{2} = m$ that 
\begin{equation} \label{eq11-Diff}
\begin{split}
\|\nabla u\|_{L^{2}}^{2} 
= \frac{3(p-1)}{2(p+1)}
\|u\|_{L^{p+1}}^{p+1} 
+ \|u\|_{L^{6}}^{6} 
\leq C m^{\frac{5-p}{4}} 
\|\nabla u
\|_{L^{2}}^{\frac{3(p-1)}{2}} 
+ C \|\nabla u\|_{L^{2}}^{6} 
\end{split}
\end{equation}
for some $C > 0$. 
Observe that 
$\frac{3(p-1)}{2} \geq 2$ if 
$p \geq \frac{7}{3}$. 
We see from \eqref{eq11-Diff} that 
there exists a constant $C = C(m) >0$ such that 
$\|\nabla u\|_{L^{2}} \geq C$ for all $u \in H^{1}(\R^3)$ 
satisfying $\|u\|_{L^{2}}^{2} = m$ and $\K(u) = 0$. 
We have by $\K(u) = 0$ that 
\[
\begin{split}
\E(u) = \E(u) - \frac{\K(u)}{6} 
= 
\frac{1}{3} 
\|\nabla u\|_{L^{2}}^{2} 
+ \frac{3p-7}{6(p+1)} 
\|u\|_{L^{p+1}}^{p+1} 
\geq \frac{1}{3}
\|\nabla u\|_{L^{2}}^{2}
\geq C. 
\end{split}
\]
This completes the proof. 
\end{proof}
\begin{lemma}
\label{vc4-3}
Let $7/3 \leq p < 6$.
For each $m > 0$, we have
\begin{equation} \label{eq12-Diff}
E_{\min}(m) 
= \inf\left\{ \mathcal{J}(u) 
\colon u \in H^{1}(\mathbb{R}^{3}) 
\setminus \{0\}, \ \|u\|_{L^{2}}^{2} = m, 
\ \mathcal{K}(u) \leq 0 \right\}, 
\end{equation}
where 
\begin{equation} \label{eq-J}
\J(u) = \E(u) - 
\frac{1}{2} \K(u) 
= \frac{3p-7}{4(p+1)}
\|u\|_{L^{p+1}}^{p+1} + 
\frac{1}{3} \|u\|_{L^{6}}^{6}. 
\end{equation}
\end{lemma}

\begin{proof}[Proof of Lemma \ref{vc4-3}]
We denote the right-hand side of 
\eqref{eq12-Diff} by $\widehat{E}_{\min}(m)$, that is, 
\[
\widehat{E}_{\min}(m) 
:= \inf\left\{ \mathcal{J}(u) 
\colon u \in H^{1}(\mathbb{R}^{3}) 
\setminus \{0\}, \ \|u\|_{L^{2}}^{2} = m, \ \mathcal{K}(u) \leq 0 \right\}
\]
From the definitions $E_{\min}(m), \widehat{E}_{\min}(m)$ 
and \eqref{eq-J}, 
we have $\widehat{E}_{\min}(m) \leq E_{\min}(m)$. 
Let $u \in H^{1}(\R^{3})$ satisfy 
$\|u\|_{L^{2}}^{2} = m$ and 
$\K(u) < 0$. 
Then, it follows form Lemma 
\ref{lem2-Diff} that 
there exists $\lambda(u) \in (0, 1)$ 
such that $\K(\lambda^{\frac{3}{2}}(u)
u(\lambda(u) \cdot)) 
= 0$. 
Clearly, we have 
$\|\lambda^{\frac{3}{2}}(u)
u(\lambda(u) \cdot)\|_{L^{2}}^{2} 
= \|u\|_{L^{2}}^{2} = m$. 
From the definition of 
$E_{\min}$ and $0 < \lambda(u) < 1$, 
we obtain 
\[
\begin{split}
E_{\min}(m) 
\leq \E(\lambda^{\frac{3}{2}}(u)
u(\lambda(u) \cdot)) 
& 
= \E(\lambda^{\frac{3}{2}}(u)
u(\lambda(u) \cdot)) 
- \frac{1}{2} \K 
(\lambda^{\frac{3}{2}}(u)
u(\lambda(u) \cdot)) \\[6pt]
& 
= \J(\lambda^{\frac{3}{2}}(u)
u(\lambda(u) \cdot)) \\[6pt]
& 
= \frac{3p-7}{4(p+1)} 
\lambda^{\frac{3}{2}(p-1)}(u)
\|u\|_{L^{p+1}}^{p+1} + 
\frac{\lambda^{6}(u)}{3} \|u\|_{L^{6}}^{6}
< \J(u). 
\end{split}
\]
Taking an infimum on $u$, we have 
$E_{\min} (m) \leq 
\widehat{E}_{\min}(m)$. 
This completes the proof. 
\end{proof}
In order to prove Proposition \ref{lem2-diff}, 
we use the following result obtained by 
Soave~\cite{MR4096725}: 
\begin{lemma}[Proposition 1.5 (2) 
of Soave~\cite{MR4096725}]\label{lem3-4}
Let $7/3 \leq p < 6$.
For each $m > 0$, we have 
\begin{equation} \label{eq13-Diff}
E_{\min}(m) < \frac{1}{3} \sigma^{\frac{3}{2}}, 
\end{equation}
where $\sigma$ is the constant defined by \eqref{EqL-18}. 
\end{lemma}
We are now in a position to prove 
Proposition \ref{lem2-diff}: 
\begin{proof}[Proof of Proposition \ref{lem2-diff}]
We divide the proof into 4 steps. 

\textbf{(Step 1).}~
We take $m_{*} \in (0, m_{1})$ arbitrary and fix it. 
Let $\{m_{n}\}$ in $(0, \infty)$ be a sequence 
with $\lim_{n \to \infty} m_{n} = m_{*}$. 
By \eqref{eq9-Diff}, \eqref{eq13-Diff} and 
$\|R_{\omega(m_{n})}\|_{L^{2}}^{2} 
= m_{n}$, 
we have $\sup_{n \in \mathbb{N}} 
\|R_{\omega(m_{n})}\|_{H^{1}} < + \infty$. 
In addition, similarly to \eqref{main-eq7}, 
we see from $\mathcal{N}_{\omega(m)}
(R_{\omega(m)}) = 0$ 
and $\K(R_{\omega(m)}) = 0$ that 
\begin{equation}\label{eq14-Diff}
\omega(m) \|R_{\omega(m)}\|_{L^{2}}^{2} = 
\frac{5-p}{2(p+1)} \|R_{\omega(m)}\|_{L^{p+1}}^{p+1}. 
\end{equation}
From \eqref{eq14-Diff}, $\|R_{\omega(m_{n})}\|_{L^{2}}^{2} 
= m_{n}, \sup_{n \in \mathbb{N}} 
\|R_{\omega(m_{n})}\|_{H^{1}} < + \infty$ 
and $\lim_{n\to \infty} m_{n} = m_{*}$, 
we see that 
\[
\sup_{n \in \mathbb{N}} 
|\omega(m_{n})| < + \infty.
\]
Up to a subsequence, there exists 
$(\omega_{\infty}, R_{\infty}) \in (0, \infty) \times 
H^{1}(\R^{3})$ such that 
$\lim_{n \to \infty} \omega(m_{n}) = \omega_{\infty}$ 
and 
\[
\lim_{n \to \infty} 
R_{\omega(m_{n})} = R_{\infty}
\qquad \mbox{weakly in $H^{1}(\R^{3})$ and 
strongly in $L^{q}(\R^{3})$ for $2 < q < 6$}. 
\]

\textbf{(Step 2).}~
We claim that $R_{\infty} \not\equiv 0$. 
Suppose to the contrary that 
$R_{\infty} \equiv 0$. 
Then, it follows from $\K(R_{\omega(m_{n})}) = 0$ that, 
passing to some subsequence, we have 
\begin{equation}\label{eq15-Diff}
0 = \lim_{n\to \infty}\mathcal{K}(R_{\omega(m_{n})})
= \lim_{n\to \infty}
\left\| \nabla R_{\omega(m_{n})}\right\|_{L^{2}}^{2}
- \lim_{n \to \infty}\left\|R_{\omega(m_{n})}\right\|_{L^{6}}^{6}.
\end{equation}
Suppose that $\lim_{n \to \infty} 
\| \nabla R_{\omega(m_{n})} \|_{L^2} = 0$. 
It follows from the boundedness of $\{R_{\omega(m_{n})}\}$ 
in $H^{1}(\mathbb{R}^{3})$
and the H\"{o}lder and the Sobolev inequalities that 
$\lim_{n \to \infty}\|R_{\omega(m_{n})}
\|_{L^q} = 0$ for all $2<q \leq 6$. 
This together with Proposition \ref{lem1-Diff} and 
Lemma \ref{vc4-3} yields that 
\begin{equation*}
E_{\min}(m_{*}) = 
\lim_{n \to \infty}E_{\min}(m_{n}) = \lim_{n \to \infty} \mathcal{J}(R_{\omega(m_{n})}) 
= 0.
\end{equation*} 
However, this contradicts 
the result of Lemma \ref{lem-pe}. 
Therefore, by taking a subsequence, 
we may assume $\lim_{n\to\infty} 
\| \nabla R_{\omega(m_{n})} \|_{L^2} > 0$.

Now, \eqref{eq15-Diff} with the definition of $\sigma$ gives us
\begin{equation}\label{eq16-Diff}
\lim_{n\to \infty}
\left\| \nabla R_{\omega(m_{n})} \right\|_{L^{2}}^{2}
\ge 
\sigma 
\lim_{n\to \infty}\left\| R_{\omega(m_{n})} \right\|_{L^{6}}^{2}
\ge 
\sigma \lim_{n\to \infty}
\left\| \nabla R_{\omega(m_{n})} \right\|_{L^{2}}^{\frac{2}{3}}.
\end{equation}
From this, 
$\lim_{n\to\infty} \| \nabla R_{\omega(m_{n})} \|_{L^2} >0$ 
and \eqref{eq16-Diff}, one has 
\begin{equation}\label{eq17-Diff}
\sigma^{\frac{3}{2}}
\le \lim_{n\to \infty}\left\| 
\nabla R_{\omega(m_{n})} \right\|_{L^{2}}^2 \leq 
\lim_{n \to \infty} \left\|R_{\omega(m_{n})} 
\right\|_{L^{6}}^{6}.
\end{equation}
Let $\J$ be the functional defined by \eqref{eq-J}.
We see from \eqref{eq-J} and \eqref{eq17-Diff} that 
\[
\begin{split}
\lim_{n \to \infty}
E_{\min}(m_{n})
&=
\lim_{n\to \infty}\mathcal{J}(R_{\omega(m_{n})})
\\[6pt]
&\ge 
\lim_{n\to \infty}
\left\{ 
\frac{3p-7}{4(p+1)}
\|R_{\omega(m_{n})}\|_{L^{p+1}}^{p+1} + 
\frac{1}{3} \|R_{\omega(m_{n})}\|_{L^{6}}^{6}
\right\}
\\[6pt]
&\ge \frac{1}{3}\lim_{n\to \infty}
\left\|R_{\omega(m_{n})}\right\|_{L^{6}}^{6}
\ge \frac{1}{3} \sigma^{\frac{3}{2}}.
\end{split}
\]
However, this contradicts the fact that 
$E_{\min}(m) < \frac{1}{3} \sigma^{\frac{3}{2}}$ 
(see Lemma \ref{lem3-4}). 
Thus, $R_{\infty} \not\equiv 0$.

\textbf{(Step 3).}~
We shall show $R_{\infty} = R_{\omega(m_{*})}$. 
First, we claim that 
$\mathcal{K}(R_{\infty}) \leq 0$. 
It follows from MR699419 lemma 
(see Lemma \ref{lemm-bl}) that 
\begin{align}
& 
\lim_{n \to \infty} 
\left\{ 
\|R_{\omega(m_{n})}\|_{L^{2}}^{2} - 
\|R_{\omega(m_{n})} - R_{\infty}\|_{L^{2}}^{2}
- \|R_{\infty}\|_{L^{2}}^{2}
\right\} = 0, \label{eq18-Diff}\\[6pt]
& 
\lim_{n \to \infty} 
\left\{ 
\mathcal{K}(R_{\omega(m_{n})}) - 
\mathcal{K}(R_{\omega(m_{n})} - R_{\infty}) 
- \mathcal{K} (R_{\infty})
\right\} = 0, \label{eq19-Diff}\\[6pt]
& 
\lim_{n \to \infty} 
\left\{ 
\mathcal{J}(R_{\omega(m_{n})}) - 
\mathcal{J}(R_{\omega(m_{n})}) - R_{\infty}) 
- \mathcal{J} (R_{\infty})
\right\} = 0. \label{eq20-Diff}
\end{align}
Suppose to the contrary that 
$\mathcal{K}(R_{\infty}) > 0$. 
Then, it follows from \eqref{eq19-Diff} and 
$\mathcal{K}(R_{\omega(m_{n})}) = 0$ that 
$\mathcal{K}(R_{\omega(m_{n})} - R_{\infty}) < 0$. 
We see from \eqref{eq18-Diff} that $\|R_{\omega(m_{n})} - 
R_{\infty}\|_{L^{2}}^{2} \leq \|R_{\omega(m_{n})}\|_{L^{2}}^{2} = 
m_{n}$. 
This together with Remark \ref{rem1-sec2} and 
Lemma \ref{vc4-3} yields that 
\[
E_{\min} (m_{n}) \leq 
E_{\min}\left(\|R_{\omega(m_{n})} 
- R_{\infty}\|_{L^{2}}^{2}\right) \leq 
\mathcal{J} (R_{\omega(m_{n})} - R_{\infty}). 
\]
Since 
$
\mathcal{J}(R_{\omega(m_{n})})
= \mathcal{E}(R_{\omega(m_{n})})
= E_{\min} (m_{n})$, 
and $ \mathcal{J}(R_{\omega(m_{n})}) - R_{\infty}) 
\geq E_{\min} (m_{n})$, 
we have by \eqref{eq20-Diff} that 
$\mathcal{J} (R_{\infty}) \leq 0$, which contradicts 
$R_{\infty} \neq 0$. Thus, we obtain $\mathcal{K}(R_{\infty}) 
\leq 0$. 

By the weak lower semicontinuity
and Proposition \ref{lem1-Diff}, we have 
\begin{equation}\label{eq21-Diff}
\begin{split}
& 
\|R_{\infty}\|_{L^{2}}^{2} 
\leq \liminf_{n \to \infty} 
\|R_{\omega(m_{n})}\|_{L^{2}}^{2} = m_{*}, \\[6pt]
&
\mathcal{J}(R_{\infty}) 
\leq \liminf_{n \to \infty} 
\mathcal{J}(R_{\omega(m_{n})}) 
= \liminf_{n \to \infty}E_{\min}(m_{n}) 
= E_{\min}(m_{*}). 
\end{split}
\end{equation}
Then, by \eqref{eq21-Diff}, 
Remark \ref{rem1-sec2}, Lemma \ref{vc4-3} and 
$\K(R_{\infty}) \leq 0$, 
we have 
\[
E_{\min}(m_{*}) \leq E_{\min} (\|R_{\infty}\|_{L^{2}}^{2})
\leq \J(R_{\infty}) 
\leq E_{\min}(m_{*}). 
\]
This implies that 
$\J(R_{\infty}) = E_{\min}(m_{*})$, so that $R_{\infty}$ is 
a minimizer for $E_{\min}(m_{*})$. 
In addition, by the uniqueness of the minimizer 
(see Theorem \ref{thm-muni}), 
we have $R_{\infty} = R_{\omega(m_{*})}$. 
By \eqref{eq20-Diff}, 
$\lim_{n \to \infty} $ and $\mathcal{J}(R_{\infty}) 
= \mathcal{J}(R_{\omega(m_{*})}) = E_{\min}(m_{*})$, we obtain 
\[
\lim_{n \to \infty} 
\mathcal{J}(R_{\omega(m_{n})} - 
R_{\omega(m_{*})}) = 0, 
\]
which together with the boundedness of $\{R_{\omega(m_{n})}\}$ 
in $H^{1}(\mathbb{R}^{3})$
implies that $ \lim_{n \to \infty} R_{\omega(m_{n})}
= R_{\omega(m_{*})}$ strongly in $L^{q}(\R^{3})\; (q \in (2, 6])$. 
Moreover, since $\mathcal{K} (R_{\omega(m_{n})}) 
= \mathcal{K} (R_{\omega(m_{*})}) 
= 0$, we have 
\[
\begin{split}
\lim_{n \to \infty} \|\nabla R_{\omega(m_{n})}\|_{L^{2}}^{2} 
& 
= \lim_{n \to \infty} \left\{
\frac{3(p-1)}{2(p+1)} \|R_{\omega(m_{n})}\|_{L^{p+1}}^{p+1} 
+ \|R_{\omega(m_{n})}\|_{L^{6}}^{6}
\right\} \\[6pt]
& 
= \frac{3(p-1)}{2(p+1)} \|R_{\omega(m_{*})}\|_{L^{p+1}}^{p+1} 
+ \|R_{\omega(m_{*})}\|_{L^{6}}^{6} 
= \|\nabla R_{\omega(m_{*})}\|_{L^{2}}^{2}. 
\end{split}
\]
In addition, since $\lim_{n \to \infty} 
\|R_{\omega(m_{n})}\|_{L^{2}}^{2} = \lim_{n \to \infty} m_{n} 
= m_{*} = \|R_{\omega(m_{*})}\|_{L^{2}}^{2}$, 
we see that $ \lim_{n \to \infty} R_{\omega(m_{n})}
= R_{\omega_{*}}$ strongly in $H^{1}(\R^{3})$.

\textbf{(Step 4).}~We derive a conclusion. 
It follows from \eqref{eq14-Diff} and 
$\lim_{n \to \infty} R_{\omega(m_{n})} = R_{\omega(m_{*})}$ 
strongly in $H^{1}(\R^{3})$ that 
\[
\omega_{\infty} = \lim_{n \to \infty} \omega(m_{n})
= \frac{5 - p}{2(p+1)} 
\lim_{n \to \infty} 
\dfrac{\|R_{\omega(m_{n})}\|_{L^{p+1}}^{p+1}}
{m_{n}} 
= \frac{5 - p}{2(p+1)}
\frac{\|R_{\omega(m_{*})}
\|_{L^{p+1}}^{p+1}}{m_{*}} = \omega(m_{*}). 
\]
Therefore, we obtain 
$\lim_{n \to \infty} 
\omega(m_{n}) = \omega_{\infty} 
= \omega(m_{*})$. 
Since $\{m_{n}\}$ is an arbitrary sequence with 
$\lim_{n \to \infty} m_{n} = m_{*}$, 
we obtain the desired result. 
\end{proof}
\subsection{Differentiability of the mapping 
$\omega \in (0, \infty) \mapsto R_{\omega} \in H_{\text{rad}}^{1}(\R^{3})$}
In this subsection, we prove that 
the mapping 
$\omega \in (0, \infty) \mapsto R_{\omega} \in H_{\text{rad}}^{1}(\R^{3})$ is differentiable. 
To prove this, we use the implicit function theorem. 
Thus, it is convenient that 
we regard \eqref{sp} as a nonlinear eigenvalue 
problem, that is, we say that a pair 
$(\omega_{*}, R_{\omega_{*}})$ satisfies 
\eqref{sp} if $R_{\omega_{*}}$ is a solution to 
\eqref{sp} with $\omega = \omega_{*}$. 

Observe that $\lim_{m \to 0}\omega(m) = 0$. 
Thus, by Theorem \ref{thm-bl} \textrm{(ii)}, 
we can take $m_{2} \in (0, m_{1})$ such that 
$R_{\omega(m)}$ is non-degenerate for $m \in (0, m_{2})$. 
We first show the following: 
\begin{proposition}\label{lem3-diff}
Let $m_{*} \in (0, m_{1})$, where 
$m_{1} > 0$ is given by Theorem \ref{thm-muni} 
and $R_{\omega(m_{*})}$ be the minimizer 
for $E_{\min}(m_{*})$.
There exist $\delta > 0$ and $r > 0$ satisfying the following: 
\begin{enumerate}
\item[\textrm{(i)}]
If $(\omega, u_{\omega}) \in (\omega(m_{*})- \delta, 
\omega(m_{*}) + \delta) \times 
B(u_{\omega(m_{*})}, r)$ 
satisfies \eqref{sp}, we have $u_{\omega} = u_{\omega}$, 
where 
\begin{equation*} 
B(u_{\omega(m_{*})}, r) = \left\{
u \in H_{\text{rad}}^{1}(\R^{3}) \colon 
\|u - u_{\omega(m_{*})}\|_{H^{1}} 
< r
\right\}. 
\end{equation*}
\item[\textrm{(ii)}]
$\omega \in (\omega(m_{*}) - \delta, \omega(m_{*}) + \delta) 
\mapsto u_{\omega} \in H_{\text{rad}}^{1}(\R^{3})$ is a 
$C^{1}$ mapping. 
\item[\textrm{(iii)}] 
There exists $\e > 0$ such that if 
$|m - m_{*}| < \e$, we have 
$\omega(m) \in (\omega(m_{*}) - \delta, \omega(m_{*}) + \delta)$, 
$R_{\omega(m)} = u_{\omega(m)}$ and 
the mapping $\omega \in 
(\omega(m_{*}) - \delta, \omega(m_{*}) + \delta) 
\mapsto R_{\omega}|_{\omega = \omega(m)} 
\in H_{\text{rad}}^{1}(\R^{3})$ is 
$C^{1}$. 
\end{enumerate}
\end{proposition} 
\begin{proof}[Proof of Proposition \ref{lem3-diff}]
We first show \textrm{(i)} and \textrm{(ii)}. 
We put $g(u) = u^{p} + u^{5}$ and define a mapping 
$\mathcal{F}:(0, \infty) \times H^{1}_{\text{rad}}(\R^{3}) 
\to H_{\text{rad}}^{1}(\R^{3})$ by 
\[
\mathcal{F}(\omega, u) = 
u - (- \Delta + \omega)^{-1} g(u)
\qquad (\omega > 0, u \in H_{\text{rad}}^{1}(\R^{3})). 
\]
It follows that $(\omega, u_{\omega}) 
\in (0, \infty) \times H^{1}_{\text{rad}}(\R^{3})$ 
is a solution to \eqref{sp} if and only if 
$\mathcal{F}(\omega, u) = 0$. 
We can easily verify that 
$\mathcal{F}(\omega(m_{*}), u_{\omega(m_{*})}) = 0$.
It follows from Theorem \ref{thm-bl} \textrm{(ii)} that 
\[
\frac{\partial \mathcal{F}}{\partial u} 
(\omega(m_{*}), u_{\omega(m_{*})}) 
= I + (- \Delta + \omega(m_{*}))^{-1}
g^{\prime} (u_{\omega(m_{*})}) 
\]
is injective from $H^{1}_{\text{rad}}(\R^{3})$ to itself. 
In addition, we observe from 
$\lim_{|x| \to \infty} u_{\omega(m_{*})}(x) = 0$ that 
$(- \Delta + \omega(m_{*}))^{-1} 
g^{\prime} (u_{\omega(m_{*})})$ 
is a compact operator on 
$H^{1}_{\text{rad}}(\R^{3})$. 
This together with the Fredholm alternative theorem 
yields that $R\left(\frac{\partial \mathcal{F}}{\partial u} 
(\omega(m_{*}), R_{\omega(m_{*})})\right) = H^{1}_{\text{rad}}(\R^{3})$, that is, 
$\frac{\partial \mathcal{F}}{\partial u} 
(\omega(m_{*}), R_{\omega(m_{*})})$ is surjective. 
Thus, we can apply the implicit function theorem, 
which prove \textrm{(i)} and \textrm{(ii)}. 

Finally, we give the proof of \textrm{(iii)}. 
From Proposition \ref{lem2-diff}, 
there exists sufficiently small $\e > 0$ such that 
if $|m - m_{*}| < \e$, we have 
$(\omega(m), R_{\omega(m)}) \in (\omega(m_{*})- \delta, 
\omega(m_{*}) + \delta) \times B(R_{\omega(m_{*})}, r)$. 
By Proposition \ref{lem3-diff} \textrm{(i)}, \textrm{(ii)} 
and $u_{\omega(m)} 
= R_{\omega(m)}$ (see Remark \ref{rem-u2}), 
the mapping $\omega \in 
(\omega(m_{*}) - \delta, \omega(m_{*}) + \delta) 
\mapsto R_{\omega}|_{\omega = \omega(m)} 
\in H_{\text{rad}}^{1}(\R^{3})$ is 
$C^{1}$. 
\end{proof}
\begin{proposition}\label{lem4-diff}
$\omega(m)$ is a strictly increasing function on $(0, m_{2})$. 
\end{proposition}
\begin{proof}
We shall show that there is no open 
interval $I \subset (0, m_{2})$ 
such that $\omega(m)$ is constant on $I$ by contradiction.
Suppose to the contrary 
that $\omega(m)$ is constant on $I_{0}$ for some 
open interval $I_{0} \subset (0, m_{2})$.
Let $m_{0} \in I_{0}$. Then we have 
\[
\omega(m) = \omega_{m_{0}} \qquad \mbox{for all 
$m \in I_{0}$}. 
\]
Then, by Proposition \ref{lem3-diff} \textrm{(iii)}, 
we have $R_{\omega(m)} = R_{\omega(m_{0})}$ 
for all $m \in I_{0}$. 
Note that if $m \in I_{0} \setminus \{0\}$, we have 
$\|R_{\omega(m)}\|_{L^{2}}^{2} 
= m \neq m_{0} = \|R_{\omega(m_{0})}\|_{L^{2}}^{2}$, 
which is absurd. 

Next, 
we shall derive the conclusion. 
Suppose to the contrary 
that $\omega(m)$ is not an increasing function of $m$. 
Then, since $\lim_{m \to 0}\omega(m) = 0$, 
$\omega(m)$ is continuous on $m$ 
(see Proposition \ref{lem2-diff}) and 
there is no open interval such that $\omega(m)$ is constant, 
there exists $m_{*} \in (0, m_{2})$ where 
$\omega(m_{*})$ is a strict local maximum. 
Then, there exists $\mu_{1}, \mu_{2} \in (0, m_{2})$ 
such that 
\[
\mu_{1} < \mu_{2}, \qquad 
|\mu_{i} - m_{*}| < \e \quad (i = 1, 2), \qquad
\omega(\mu_{1}) = \omega(\mu_{2}), 
\] 
where $\e > 0$ is a positive number given in 
Proposition \ref{lem3-diff} \textrm{(iii)}. 
Then, using Proposition \ref{lem3-diff} \textrm{(iii)} 
again, 
we have $R_{\omega(\mu_{1})} = R_{\omega(\mu_{2})}$. 
However, this cannot happen because 
$\|R_{\omega(\mu_{2})}\|_{L^{2}}^{2} 
= \mu_{2} > \mu_{1} = \|R_{\omega(\mu_{2})}\|_{L^{2}}^{2}$. 
This completes the proof. 
\end{proof}
\section{Morse index of $R_{\omega(m)}$}\label{sec-mi}
This section is devoted to the proof of Theorem 
C \textrm{(iv)}. 
We study the Morse index of $R_{\omega(m)}$. 
\subsection{The Morse index is either $1$ or $2$}
In this subsection, we shall show for each $m > 0$, 
the Morse index of $R_{\omega(m)}$ is either $1$ or $2$. 
We put 
\begin{equation} \label{MI-eq1}
g_{m} := - 2 \Delta R_{\omega(m)} 
- \frac{3(p - 1)}{2}R_{\omega(m)}^{p} 
- 6 R_{\omega(m)}^{5}.
\end{equation}
Note that 
$
\mathcal{K}^{\prime}(R_{\omega(m)})u
= (g_{m}, u)_{L^{2}}$.
By the standard elliptic regularity (see e.g. 
Gilberg and Trudinger~\cite{GiTr01}), we see that $g_{m} \in H^{1}(\R^{3})$. 
First, we shall show the following: 
\begin{lemma}\label{lem-mor1}
Let $\frac{7}{3} \leq p < 3$. 
For each $m > 0$, 
$\K^{\prime} (R_{\omega(m)})$ 
and $(\|R_{\omega(m)}\|_{L^{2}}^{2})^{\prime}$ 
are linearly independent. 
\end{lemma}
\begin{proof}
We shall prove this lemma by contradiction. 
Suppose to the contrary that 
there exists a 
constant $C \in \R$ such that 
$\K^{\prime} (R_{\omega(m)}) = C (\|R_{\omega(m)}\|_{L^{2}}^{2})^{\prime}/2$. 
This can be written as 
\begin{equation} \label{MI-eq2}
- 2 \Delta R_{\omega(m)} - 
\frac{3(p-1)}{2} R_{\omega(m)}^p 
- 6 R_{\omega(m)}^{5} = C R_{\omega(m)} 
\qquad \mbox{in $\R^{3}$}. 
\end{equation}
Multiplying \eqref{MI-eq2} by $R_{\omega(m)}$
and integrating the resulting equation, 
we have 
\[
2 \|\nabla R_{\omega(m)}\|_{L^{2}}^{2} 
- \frac{3(p-1)}{2} 
\|R_{\omega(m)}\|_{L^{p+1}}^{p+1} 
- 6 \|R_{\omega(m)}\|_{L^{6}}^{6} 
= C \|R_{\omega(m)}\|_{L^{2}}^{2}. 
\]
This together with $\K(R_{\omega(m)}) = 0$ 
and $p > 1$ yields that 
\[
\begin{split}
C \|R_{\omega(m)}\|_{L^{2}}^{2} 
& = \left( 
\frac{3(p-1)}{p+1} - 
\frac{3(p-1)}{2}
\right)\|R_{\omega(m)}
\|_{L^{p+1}}^{p+1} - 4 
\|R_{\omega(m)}\|_{L^{6}}^{6} \\[6pt]
& = \frac{3(p-1)}{2} 
\left(\frac{2}{p+1} - 1 
\right) \|R_{\omega(m)}\|_{L^{p+1}}^{p+1} 
- 4 \|R_{\omega(m)}\|_{L^{6}}^{6} < 0. 
\end{split}
\]
Thus, we see that $C < 0$. 
Multiplying \eqref{MI-eq2} by $x \cdot \nabla R_{\omega(m)} 
\in H^{1}(\R^{3})$ and integrating 
the resulting equation, we obtain 
\begin{equation} \label{MI-eq3}
\|\nabla R_{\omega(m)}\|_{L^{2}}^{2} 
- \frac{9(p-1)}{2(p+1)} 
\|R_{\omega(m)}\|_{L^{p+1}}^{p+1} 
- 3 \|R_{\omega(m)}\|_{L^{6}}^{6} 
= \frac{3}{2} C 
\|R_{\omega(m)}\|_{L^{2}}^{2}. 
\end{equation}
Multiplying \eqref{MI-eq2} by 
$\frac{3}{2} R_{\omega(m)} \in H^{1}(\R^{3})$ 
and integrating the resulting equation, 
we have 
\begin{equation} \label{MI-eq4}
3 \|\nabla R_{\omega(m)}\|_{L^{2}}^{2} 
- \frac{9(p-1)}{4} 
\|R_{\omega(m)}\|_{L^{p+1}}^{p+1} 
- 9 \|R_{\omega(m)}\|_{L^{6}}^{6} 
= \frac{3}{2} C \|R_{\omega(m)}\|_{L^{2}}^{2}. 
\end{equation}
Subtracting \eqref{MI-eq3} from \eqref{MI-eq4} gives us that 
\[
2 \|\nabla R_{\omega(m)}\|_{L^{2}}^{2} 
- \frac{9(p-1)^2}{4(p+1)} 
\|R_{\omega(m)}\|_{L^{p+1}}^{p+1} 
- 6 \|R_{\omega(m)}\|_{L^{6}}^{6} 
= 0
\] 
Using $\K(R_{\omega(m)}) = 0$ and $p > \frac{7}{3}$, 
we obtain 
\[
0 = - \frac{3(p-1)(3p-7)}{4 (p+1)}
\|R_{\omega(m)}\|_{L^{p+1}}^{p+1} 
- 4 \|R_{\omega(m)}\|_{L^{6}}^{6} < 0, 
\]
which is absurd. 
Thus, we see that 
$\K^{\prime} (R_{\omega(m)})$ 
and $(\|R_{\omega(m)}\|_{L^{2}}^{2})^{\prime}$ 
are linearly independent. 
\end{proof}

It follows from Lemma \ref{lem-mor1} that 
$g_{m}$ and $R_{\omega(m)}$ are linearly 
independent. 
We define the linearized operator associated with 
the equation \eqref{sp} around $R_{\omega(m)}$ on 
$L^{2}(\R^{3})$ with the domain $H^{2}(\R^{3})$ by 
\[
L_{m, +} := - \Delta + \omega(m) 
- p R_{\omega(m)}^{p-1} - 5 R_{\omega(m)}^{4}. 
\] 
We show the following: 
\begin{proposition}\label{mi-thm1}
Let $\frac{7}{3} \leq p < 3$ and $m > 0$.
If $u \in H^{1}(\R^{3})$ satisfies 
\begin{equation} \label{MI-eq5}
(u, g_{m})_{L^{2}} = (u, R_{\omega(m)})_{L^{2}} = 0,
\end{equation} 
then we have $\langle L_{m, +}u, u \rangle \geq 0$. 
\end{proposition}
\begin{proof}[Proof of Proposition \ref{mi-thm1}]
Let $u \in H^{1}(\R^{3})$ satisfy \eqref{MI-eq5}.
Define a function $Z:\R^{3} \to H^{1}(\R^{d})$ by 
\[
Z(a, b, c) := R_{\omega(m)} + au 
+ b g_{m} 
+ c R_{\omega(m)}. 
\]
It is obvious that 
\begin{equation}\label{MI-eq6}
\frac{\p Z}{\p a} (a, b, c) = u, \qquad 
\frac{\p Z}{\p b} (a, b, c) = g_{m}, \qquad 
\frac{\p Z}{\p c} (a, b, c) = R_{\omega(m)}. 
\end{equation}
Define $\Phi:\R^{3} \to \R^{2}$ by 
\[
\Phi(a, b, c) := 
\begin{pmatrix}
\mathcal{K}(Z(a, b, c)) \\[6pt]
\|Z(a, b, c)\|_{L^{2}}^{2} - m
\end{pmatrix}. 
\]
Note that 
\[
\begin{split}
\Phi(a, b, c)\biggl|_{(a, b, c) = (0, 0, 0)} = 
\begin{pmatrix}
\mathcal{K}(R_{\omega(m)}) \\[6pt]
\|R_{\omega(m)}\|_{L^{2}}^{2} - m 
\end{pmatrix} 
= 
\begin{pmatrix}
0 \\[6pt]
0
\end{pmatrix}.
\end{split}
\]
By \eqref{MI-eq5}, \eqref{MI-eq6}, 
$g_{m} = \mathcal{K^{\prime}}(R_{\omega(m)})$, 
\eqref{MI-eq1}
and $\K(R_{\omega(m)}) = 0$, we see that 
\begin{equation} \label{MI-eq7} 
\frac{\p}{\p b} \mathcal{K}(Z(a, b, c))\biggl|_{(a, b, c) = (0, 0, 0)} 
= \langle \mathcal{K^{\prime}} 
(R_{\omega(m)}), g_{m} \rangle 
= \left\|g_{m} \right\|_{L^{2}}^{2} > 0, 
\end{equation}
\begin{equation}\label{MI-eq8}
\begin{split}
\frac{\p}{\p c} \mathcal{K}(Z(a, b, c))\biggl|_{(a, b, c) = (0, 0, 0)} 
& = \langle \mathcal{K^{\prime}} 
(R_{\omega(m)}), R_{\omega(m)} \rangle \\[6pt]
& 
= \langle g_{m}, R_{\omega(m)} 
\rangle \\[6pt]
& = 2 \|\nabla R_{\omega(m)}\|_{L^{2}}^{2} 
- \frac{3(p-1)}{2} \|R_{\omega(m)}\|_{L^{p+1}}^{p+1} 
- 6 \|R_{\omega(m)}\|_{L^{6}}^{6} \\[6pt]
& = - \frac{3(p-1)^{2}}
{2(p+1)}\|R_{\omega(m)}\|_{L^{p+1}}^{p+1} 
- 4 \|R_{\omega(m)}\|_{L^{2^{*}}}^{2^{*}} < 0. 
\end{split}
\end{equation}
Moreover, by \eqref{MI-eq6} and \eqref{MI-eq8}, we have 
\begin{equation} \label{MI-eq10}
\begin{split}
\frac{\p}{\p b} \|Z(a, b, c)\|_{L^{2}}^{2}
\biggl|_{(a, b, c) = (0, 0, 0)} 
& = 2 \langle R_{\omega(m)}, g_{m} \rangle \\[6pt]
& = 
- \frac{3(p-1)^{2}}
{p+1}\|R_{\omega(m)}\|_{L^{p+1}}^{p+1} 
- 8 \|R_{\omega(m)}\|_{L^{2^{*}}}^{2^{*}} < 0, 
\end{split}
\end{equation}
\begin{equation} \label{MI-eq9}
\begin{split}
\frac{\p}{\p c} \|Z(a, b, c)\|_{L^{2}}^{2}
\biggl|_{(a, b, c) = (0, 0, 0)} 
= 2 \langle R_{\omega(m)}, R_{\omega(m)} \rangle
= 2 \|R_{\omega(m)}\|_{L^{2}}^{2} > 0. 
\end{split}
\end{equation} 
Put 
\[
\begin{split}
J(a, b, c) 
& := 
\begin{pmatrix}
\frac{\p}{\p b} \mathcal{K}(Z(a, b, c)) 
& \frac{\p}{\p c} \mathcal{K}(Z(a, b, c)) \\[6pt]
\frac{\p}{\p b} \|Z(a, b, c))\|_{L^{2}}^{2} 
& \frac{\p}{\p c} \|Z(a, b, c))\|_{L^{2}}^{2}
\end{pmatrix}. 
\end{split}
\] 
By \eqref{MI-eq7}--\eqref{MI-eq9} and 
the Cauchy-Schwarz inequality , 
we obtain 
\[
\begin{split}
\det J(0, 0, 0) 
& 
= \det 
\begin{pmatrix}
\|g_{m}\|_{L^{2}}^{2} 
& \langle g_{m}, 
R_{\omega(m)} \rangle 
\\[6pt]
2 \langle g_{m}, 
R_{\omega(m)} \rangle 
&
2 \|R_{\omega(m)}\|_{L^{2}}^{2} 
\end{pmatrix} 
\\[6pt]
& 
= 2 \|g_{m}\|_{L^{2}}^{2}
\|R_{\omega(m)}\|_{L^{2}}^{2} 
- 2 \left(\langle g_{m}, 
R_{\omega(m)} \rangle\right)^{2} \\[6pt]
& > 2 \|g_{m}\|_{L^{2}}^{2}
\|R_{\omega(m)}\|_{L^{2}}^{2} 
- 2 \|g_{m}\|_{L^{2}}^{2}
\|R_{\omega(m)}\|_{L^{2}}^{2} 
=0. 
\end{split}
\] 
Note that since $g_{m}$ 
and $R_{\omega(m)}$ are linearly independent 
(see Lemma \ref{lem-mor1}), 
the Cauchy-Schwarz inequality becomes strict. 
Hence, the implicit function theorem shows that 
there exist $\e_{0} > 0$ and a $C^{2}$-function 
$\textbf{h}:(- \e_{0}, \e_{0}) \to \R^{2}$ such that 
\begin{equation}\label{MI-eq11}
\begin{split}
\Phi(a, \textbf{h}(a)) = 
\begin{pmatrix}
\mathcal{K}(Z(a, \textbf{h}(a))) \\[6pt]
\|Z(a, \textbf{h}(a))\|_{L^{2}}^{2} - m
\end{pmatrix} 
= 
\begin{pmatrix}
0 \\[6pt]
0
\end{pmatrix}
\quad 
\mbox{for all $a \in (- \e_{0}, \e_{0})$}, 
\qquad
\textbf{h}(0) 
= 
\begin{pmatrix}
0 \\[6pt]
0
\end{pmatrix}. 
\end{split}
\end{equation}
Note that we can write 
\begin{equation} \label{MI-eq12}
Z(a, \textbf{h}(a)) 
= R_{\omega(m)} + a u + 
\begin{pmatrix}
g_{m} \\[6pt]
R_{\omega(m)}
\end{pmatrix}
\cdot \textbf{h}(a) 
\end{equation}
and 
\[
\begin{split}
J(a, \textbf{h}(a)) 
& := 
\begin{pmatrix}
\frac{\p}{\p b} \mathcal{K}(Z(a, \textbf{h}(a))) 
& \frac{\p}{\p c} \mathcal{K}(Z(a, \textbf{h}(a))) \\[6pt]
\frac{\p}{\p b} \|Z(a, \textbf{h}(a)))\|_{L^{2}}^{2} 
& \frac{\p}{\p c} \|Z(a, \textbf{h}(a)))\|_{L^{2}}^{2}
\end{pmatrix}. 
\\[6pt]
& = 
\begin{pmatrix}
\langle 
\mathcal{K}^{\prime}(Z(a, \textbf{h}(a)), \mathcal{K}^{\prime} 
(R_{\omega(m)}) \rangle 
& 
\langle 
\mathcal{K}^{\prime}(Z(a, \textbf{h}(a))), 
R_{\omega(m)} \rangle \\[6pt]
\langle 2 Z(a, \textbf{h}(a)), 
\mathcal{K}^{\prime} 
(R_{\omega(m)}) \rangle 
& 2 \langle Z(a, \textbf{h}(a)), 
R_{\omega(m)} \rangle 
\end{pmatrix} \\[6pt]
& = \begin{pmatrix}
\langle 
\mathcal{K}^{\prime}(Z(a, \textbf{h}(a))), g_{m} \rangle 
& 
\langle 
\mathcal{K}^{\prime}(Z(a, \textbf{h}(a))), 
R_{\omega(m)} \rangle \\[6pt]
2 \langle Z(a, \textbf{h}(a)), g_{m} \rangle 
& 2 \langle Z(a, \textbf{h}(a)), 
R_{\omega(m)} \rangle 
\end{pmatrix}. 
\end{split}
\] 
By \eqref{MI-eq11} and \eqref{MI-eq12}, we see that 
\[
\begin{split}
\begin{pmatrix}
0 \\[6pt]
0
\end{pmatrix}
& = 
\frac{d}{d a} \Phi(a, \textbf{h}(a)) 
= 
\begin{pmatrix}
\langle \mathcal{K}^{\prime}(Z(a, \textbf{h}(a))), 
\frac{d}{d a} Z(a, \textbf{h}(a)) \rangle 
\\[6pt]
\langle 2 Z(a, \textbf{h}(a)), 
\frac{d}{d a} Z(a, \textbf{h}(a)) \rangle 
\end{pmatrix} 
\\[6pt]
& = 
\begin{pmatrix}
\left\langle \mathcal{K}^{\prime}(Z(a, \textbf{h}(a))), 
u + 
\begin{pmatrix}
g_{m} \\[6pt]
R_{\omega(m)}
\end{pmatrix}
\cdot 
\frac{d \textbf{h}}{d a}(a) \right\rangle \\[6pt]
\left\langle 
2 Z(a, \textbf{h}(a)), 
u + 
\begin{pmatrix}
g_{m} \\[6pt]
R_{\omega(m)}
\end{pmatrix}
\cdot 
\frac{d \textbf{h}}{d a}(a)
\right\rangle 
\end{pmatrix} 
\\[6pt]
& = 
\begin{pmatrix}
\langle \mathcal{K}^{\prime}(Z(a, \textbf{h}(a))), 
u \rangle \\[6pt]
\langle 2 Z(a, \textbf{h}(a)), 
u \rangle 
\end{pmatrix}
+ 
\begin{pmatrix}
\langle 
\mathcal{K}^{\prime}(Z(a, \textbf{h}(a))), g_{m} \rangle 
& 
\langle 
\mathcal{K}^{\prime}(Z(a, \textbf{h}(a))), 
R_{\omega(m)} \rangle \\[6pt]
2 \langle Z(a, \textbf{h}(a)), g_{m} \rangle 
& 2 \langle Z(a, \textbf{h}(a)), 
R_{\omega(m)} \rangle 
\end{pmatrix}
\frac{d \textbf{h}}{d a}(a) \\[6pt]
& = 
\begin{pmatrix}
\langle \mathcal{K}^{\prime}(Z(a, \textbf{h}(a))), 
u \rangle \\[6pt]
\langle 2 Z(a, \textbf{h}(a)), 
u \rangle 
\end{pmatrix}
+ J(a, \textbf{h}(a)) 
\frac{d \textbf{h}}{d a}(a), 
\end{split}
\] 
so that 
\[
\begin{split}
& \quad \frac{d \textbf{h}}{d a}(a) 
= - J(a, \textbf{h}(a))^{-1} 
\begin{pmatrix}
\langle \mathcal{K}^{\prime}(Z(a, \textbf{h}(a))), 
u \rangle \\[6pt]
\langle 2 Z(a, \textbf{h}(a)), 
u \rangle 
\end{pmatrix} \\[6pt]
& = 
\frac{-1}{\det J(a, \textbf{h}(a))} 
\begin{pmatrix}
\langle 2Z(a, \textbf{h}(a)), R_{\omega(m)} \rangle 
& - \langle \mathcal{K}^{\prime}(Z(a, \textbf{h}(a))), 
R_{\omega(m)} \rangle \\[6pt]
- \langle 2 Z(a, \textbf{h}(a)), {K}^{\prime}
(R_{\omega(m)}) \rangle 
& 
\langle 
\mathcal{K}^{\prime}(Z(a, \textbf{h}(a))), 
g_{m} \rangle 
\end{pmatrix}
\begin{pmatrix}
\langle \mathcal{K}^{\prime}(Z(a, \textbf{h}(a))), u \rangle \\[6pt]
\langle 2 Z(a, \textbf{h}(a)), u \rangle 
\end{pmatrix}. 
\end{split}
\]
It follows from $\textbf{h}(0) = 0$ (see \eqref{MI-eq11}), 
$Z(0, 0, 0) = R_{\omega(m)}$, 
$g_{m} = \mathcal{K^{\prime}}(R_{\omega(m)})$ 
and the assumption \eqref{MI-eq5} that 
\begin{equation}\label{MI-eq13} 
\begin{split}
\frac{d \textbf{h}}{d a} (0) 
= 
\frac{-1}{\det J(0, 0, 0)} 
\begin{pmatrix}
\langle 2 R_{\omega(m)}, R_{\omega(m)} \rangle 
& - \langle g_{m}, 
R_{\omega(m)} \rangle \\[6pt]
- \langle 2 R_{\omega(m)}, 
g_{m} \rangle 
& 
\langle 
g_{m}, 
g_{m} \rangle 
\end{pmatrix}
\begin{pmatrix}
(g_{m}, u)_{L^{2}} \\[6pt]
2 (R_{\omega(m)}, u)_{L^{2}} 
\end{pmatrix} 
= 
\begin{pmatrix}
0 \\[6pt]
0
\end{pmatrix}.
\end{split}
\end{equation}
By \eqref{MI-eq12}, we have 
\begin{equation} \label{MI-eq14}
\begin{split}
\frac{d}{d a} Z(a, \textbf{h}(a)) 
= 
u + 
\begin{pmatrix}
g_{m} \\[6pt]
R_{\omega(m)}
\end{pmatrix}
\cdot \frac{d \textbf{h}}{d a}(a). 
\end{split}
\end{equation}
It follows from $\|Z(a), \textbf{h}(a)\|_{L^{2}}^{2} = m$ 
(see \eqref{MI-eq11}) and \eqref{MI-eq14} that 
\begin{equation} \label{MI-eq15}
\begin{split}
\frac{d}{d a} \mathcal{E}(Z(a, 
\textbf{h}(a))) 
& = \frac{d}{d a} 
\left(\frac{\omega(m)}{2}
\|Z(a, \textbf{h}(a))\|_{L^{2}}^{2} 
+ \mathcal{E}(Z(a, \textbf{h}(a))) \right) \\[6pt]
& = 
\left\langle \omega(m) Z(a, \textbf{h}(a)) 
+ \mathcal{E}^{\prime}(Z(a, \textbf{h}(a)), 
u + 
\begin{pmatrix}
g_{m} \\[6pt]
R_{\omega(m)}
\end{pmatrix}
\cdot \frac{d \textbf{h}}{d a}(a)
\right\rangle. 
\end{split}
\end{equation}
Furthermore, by \eqref{MI-eq15} and \eqref{MI-eq14}, 
one has 
\begin{equation} \label{MI-eq16}
\begin{split}
& \quad \frac{d^{2}}{d a^{2}} 
\mathcal{E}(Z(a, 
\textbf{h}(a))) \\[6pt] 
& = \frac{d}{d a}
\left(
\left\langle \omega(m) Z(a, 
\textbf{h}(a) + \mathcal{E}^{\prime}(Z(a, \textbf{h}(a)), 
u + 
\begin{pmatrix}
g_{m} \\[6pt]
R_{\omega(m)}
\end{pmatrix}
\cdot \frac{d \textbf{h}}{d a}(a)
\right\rangle
\right) \\[6pt]
& = 
\left\langle \left( \omega(m) 
+ 
\mathcal{E}^{\prime \prime}(Z(a, \textbf{h}(a))\right)
\left\{ 
u + 
\begin{pmatrix}
g_{m} \\[6pt]
R_{\omega(m)}
\end{pmatrix}
\cdot \frac{d \textbf{h}}{d a}(a)
\right\}, 
u + 
\begin{pmatrix}
g_{m} \\[6pt]
R_{\omega(m)}
\end{pmatrix}
\cdot \frac{d \textbf{h}}{d a}(a)
\right\rangle \\[6pt]
& \quad 
+ \left\langle \omega(m) Z(a, \textbf{h}(a)) 
+ \mathcal{E}^{\prime}(Z(a, \textbf{h}(a)), 
\begin{pmatrix}
g_{m} \\[6pt]
R_{\omega(m)}
\end{pmatrix}
\cdot \frac{d^{2} \textbf{h}}{d a^{2}}(a) 
\right\rangle. 
\end{split}
\end{equation}
Since 
$R_{\omega(m)}$ is a minimizer for $E_{\min}(m)$ and 
\[
Z(0, \textbf{h}(0)) = R_{\omega(m)}, \qquad 
\mathcal{K}(Z(a, \textbf{h}(a))) = 0, \qquad 
\|Z(a, \textbf{h}(a))\|_{L^{2}}^{2} = m
\]
(see \eqref{MI-eq11}), 
$\mathcal{E}(Z(a, \textbf{h}(a)))$ takes a minimum at 
$a = 0$, which implies 
\[
\frac{d}{d a} \mathcal{E}(Z(a, 
\textbf{h}(a)))|_{a = 0} = 0, \qquad 
\frac{d^{2}}{d a^{2}} \mathcal{E}(Z(a, 
\textbf{h}(a)))|_{a = 0} \geq 0. 
\] 
This together with \eqref{MI-eq16}, 
$\mathcal{S}^{\prime}_{\omega(m)}(R_{\omega(m)}) 
= \omega(m) R_{\omega(m)}
+ \mathcal{E}^{\prime}(R_{\omega(m)}) 
= 0$~\footnote{$R_{\omega(m)}$ satisfies 
\eqref{sp} with $\omega = \omega(m)$}
and \eqref{MI-eq13}
yields that 
\[
\begin{split}
0 
& \leq \frac{d^{2}}{d a^{2}} 
\mathcal{E}(Z(a, 
\textbf{h}(a)))|_{a = 0} \\[6pt]
& = 
\left\langle \left(\omega(m) 
R_{\omega(m)} + 
\mathcal{E}^{\prime \prime}(R_{\omega(m)}\right)
\left\{ 
u + 
\begin{pmatrix}
g_{m} \\[6pt]
R_{\omega(m)}
\end{pmatrix}
\cdot \frac{d \textbf{h}}{d a}(0)
\right\}, 
\left\{ 
u + 
\begin{pmatrix}
g_{m} \\[6pt]
R_{\omega(m)}
\end{pmatrix}
\cdot \frac{d \textbf{h}}{d a}(0)
\right\} \right\rangle \\[6pt]
& = \langle L_{m, +} u, u \rangle. 
\end{split}
\] 
This completes the proof. 
\end{proof}
Let $I(m)$ be the Morse index of $R_{\omega(m)}$, that is, 
\[
I(m) := 
\left\{
\dim H \colon 
\mbox{$H$ is a subspace 
of $H^{1}(\R^{d})$ satisfying $\langle L_{m, +} h, h \rangle 
< 0$ for all $h \in H \setminus 
\{0\}$}
\right\}. 
\]
Then, we obtain the following: 
\begin{lemma}\label{mor-lem1}
Let $7/3 <  p < 5$. 
Then, for each $m > 0$, 
we have $I(m) = 1$ or $2$. 
\end{lemma}
\begin{proof}
From the result of Proposition \ref{mi-thm1}, 
we can verify that $I(m) \leq 2$ for each $m > 0$. 
Suppose to the contrary that 
$I(m_{*}) \geq 3$ for some $m_{*} > 0$. 
Then, there exists a function $v_{*}$ such that 
$(v_{*}, g_{m_{*}})_{L^{2}} 
= (v_{*}, R_{\omega(m_{*})})_{L^{2}} = 0$ 
and $\langle L_{m_{*}, +} v_{*}, v_{*} \rangle < 0$, 
which contradicts the result of Proposition \ref{mi-thm1}. 

Since 
$\mathcal{S}^{\prime}_{\omega(m)}(R_{\omega(m)}) = 0$, 
we can easily find that 
\begin{equation} \label{MI-eq17}
\langle L_{m, +} R_{\omega(m)}, R_{\omega(m)} \rangle
= - (p-1) \|R_{\omega(m)}\|_{L^{p+1}}^{p+1} 
- 4 \|R_{\omega(m)}\|_{L^{6}}^{6} < 0, 
\end{equation}
which implies that $I(m) \geq 1$. 
This completes the proof. 
\end{proof}
\subsection{Proof of Theorem \ref{thm-bl} \textrm{(iii)}}
In this subsection, we conclude the proof of Theorem \ref{thm-bl} \textrm{(iii)}. 
We see from Lemma \ref{mor-lem1} that for each $m > 0$, 
$I(m) = 1$ or $2$. 
Thus, 
to prove Theorem \ref{thm-bl} \textrm{(iii)}, 
it suffices to show that 
$I(m) \neq 1$ for sufficiently small $m>0$. 
To this end, we use ideas from Grillakis, Shatah and Strauss
~\cite{MR901236}. 
For any $\e> 0$, we put 
\begin{equation} \label{MI-eq18}
U_{\e} := \left\{u \in H_{\text{rad}}^{1}(\R^{3}, \R) \colon 
\|u - R_{\omega(m)}\|_{H^{1}}^{2} < \varepsilon, 
\|u\|_{L^{2}} = \|R_{\omega(m)}\|_{L^{2}}
\right\}.
\end{equation}
Then, we shall show the following: 
\begin{lemma}\label{mi-lem6}
Let $7/3 \leq p < 3$, 
$m \in (0, m_{2})$ and 
$R_{\omega(m)}$ be the minimizer for $E_{\min}(m)$. 
Suppose that 
\begin{enumerate}
\item[\textrm{(i)}] $\frac{\p \|R_{\omega}\|_{L^{2}}^{2}}{\p \omega}
\big|_{\omega = \omega(m)}>0$. 
\item[\textrm{(ii)}] $I(m) = 1$. 
\end{enumerate}
Then, there exists constant $C>0$ and $\e>0$ such that 
\begin{equation*}
\mathcal{E}(u) - \mathcal{E}(R_{\omega(m)}) 
\geq C \|u - R_{\omega(m)}\|_{H^{1}}^{2}
\end{equation*}
for all $u \in U_{\e}$. 
\end{lemma}
\begin{proof}
We divide the proof into 3 steps. 

\textbf{(Step 1).}\; 
First, we claim that 
\begin{equation}\label{MI-eq19}
\langle L_{m, +} u, u \rangle > 0
\end{equation} 
for all $u \in H^{1}_{\text{rad}}(\R^{3})$ with 
$(u, R_{\omega(m)})_{L^{2}} = 0$. 
We denote by $\lambda_{1}(m)$ and $\chi_{1}(m)$ a negative 
eigenvalue and eigenfunction of 
$L_{m, +}$ satisfying $\|\chi_{1}(m)\|_{L^{2}} = 1$. 
We decompose 
\[
\p_{\omega} R_{\omega}|_{\omega = \omega(m)} 
= a_{0} \chi_{1}(m) + p_{0}, 
\]
where $a_{0} \in \R$ and $p _{0} \in H^{1}_{\text{rad}}(\R^{3})$ 
satisfying $(\chi_{1}(m), p_{0})_{L^{2}} = 0$. 
By the Weyl's essential spectrum theorem, we have 
$\sigma_{\text{ess}}(L_{m, +}) = [\omega(m), \infty)$. 
It follows from Theorem \ref{thm-bl} \textrm{(ii)} and 
Remark \ref{rem-u2} that $\mathrm{Ker} \; 
L_{m, +}|_{H^{1}_{\text{rad}}} = \{0\}$. 
These together with our assumption $I(m) = 1$ 
imply that 
there exists a constant $C_{m} > 0$ 
which is independent of $u$ such that 
\[
\langle L_{m, +} p_{0}, p_{0} \rangle \geq 
C_{m}\|p_{0}\|_{H^{1}}^{2}. 
\]
Similarly, for $u \in H^{1}_{\text{rad}}(\R^{3})$ with 
$(u, R_{\omega(m)})_{L^{2}} = 0$, we decompose 
\[
u = a \chi_{1}(m) + p, 
\]
where $a \in \R$ and $p \in H^{1}_{\text{rad}}(\R^{3})$ 
satisfying $(\chi_{1}(m), p)_{L^{2}} = 0$. 
It follows that 
\[
\langle L_{m, +} p, p \rangle \geq 
C_{m}\|p\|_{H^{1}}^{2}. 
\]
Note that $R_{\omega(m)}$ satisfies \eqref{sp} with 
$\omega = \omega(m)$. 
Then, differentiating \eqref{sp} with respect to 
$\omega$, we have 
\[
L_{m, +} \p_{\omega} R_{\omega}|_{\omega = \omega(m)} 
= - R_{\omega(m)}. 
\]
Thus, from the assumption 
$\frac{\p \|R_{\omega}\|_{L^{2}}^{2}}{\p \omega}
\big|_{\omega = \omega(m)}>0$, we obtain 
\begin{equation} \label{eq-mi24}
\begin{split}
0 > - \frac{1}{2}\frac{\p \|R_{\omega}\|_{L^{2}}^{2}}{\p \omega}
\big|_{\omega = \omega(m)} 
= - \left( R_{\omega(m)}, \p_{\omega} R_{\omega}|_{\omega = \omega(m)} 
\right)_{L^{2}} 
& =2 \langle L_{m, +} \p_{\omega} R_{\omega}|_{\omega = \omega(m)}, 
\p_{\omega} R_{\omega}|_{\omega = \omega(m)} 
\rangle \\[6pt]
& = a_{0}^{2}\lambda_{1}(m)
+ \langle L_{m, +} p_{0}, p_{0} \rangle. 
\end{split}
\end{equation}
In addition, one has 
\begin{equation*} 
\begin{split}
0 = (u, R_{\omega(m)})_{L^{2}} 
= - \langle u, L_{m, +} \p_{\omega} R_{\omega}|_{\omega = \omega(m)}
\rangle 
& 
= - \langle a \chi_{1}(m) + p, L_{m, +} (a_{0} \chi_{1}(m) + p_{0}) \rangle \\[6pt]
& 
= - a a_{0} \lambda_{1}(m) - \langle L_{m +} p_{0}, p \rangle. 
\end{split}
\end{equation*}
In addition, we can find that 
\begin{equation} \label{eq-mi26}
(\langle L_{m, +} p_{0}, p \rangle)^{2} 
\leq \langle L_{m, +} p_{0}, p_{0} \rangle 
\langle L_{m, +} p, p \rangle. 
\end{equation}
Therefore, by \eqref{eq-mi24}--\eqref{eq-mi26}, one has 
\[
\begin{split}
\langle L_{m, +} u, u \rangle 
= - a^{2} \lambda_{1}(m)^{2} + 
\langle L_{m, +}p, p \rangle 
\geq - a^{2} \lambda_{1}(m)^{2} + 
\frac{(\langle L_{m, +} p_{0}, p \rangle)^{2}}{
\langle L_{m, +} p_{0}, p_{0} \rangle} 
> a^{2} \lambda_{1}(m)^{2} - 
\frac{(- a_{0} a \lambda_{1}(m))^{2}}{a_{0}^{2} \lambda_{1}(m)^{2}} 
= 0. 
\end{split}
\]
Thus, \eqref{MI-eq19} holds.

\textbf{(Step 2).}\; 
We claim that there exists $C_{0} > 0$ such that 
\begin{equation}\label{eq-mi28}
\langle L_{m, +} u, u \rangle \geq C_{0} 
\|u\|_{H^{1}}^{2}. 
\end{equation} 
for all $u \in H^{1}_{\text{rad}}(\R^{3})$ with 
$(u, R_{\omega(m)})_{L^{2}} = 0$. 
We shall show \eqref{eq-mi28} by contradiction. 
Suppose to the contrary that \eqref{eq-mi28} fails. 
It follows from \eqref{MI-eq19} that 
there exists a sequence $\{u_{n}\}$ with
$\|u_{n}\|_{H^{1}} = 1\; (n \in \mathbb{N})$ such that 
$\lim_{n \to \infty} \langle L_{m, +} u_{n}, u_{n} \rangle = 0$. 
Since the sequence $\{u_{n}\}$ is bounded, there exists a subsequence 
of $\{u_{n}\}$ (we still denote it by the same symbol) and $u_{\infty} \in 
H_{\text{rad}}^{1}(\R^{3})$ such that 
$\lim_{n \to \infty} u_{n} = u_{\infty}$ weakly in $H_{\text{rad}}^{1}(\R^{3})$. 
We claim that $u_{\infty} \neq 0$. 
Suppose to the contrary that $u_{\infty} = 0$. 
Since $\lim_{|x| \to \infty} R_{\omega(m)} (x) = 0$, 
$\|u_{n}\|_{H^{1}} = 1$ and $u_{\infty} = 0$, we have 
\[
0 = \lim_{n \to \infty} \langle L_{m, +} u_{n}, u_{n} \rangle 
\gtrsim 1 - \lim_{n \to \infty} 
\int_{\R^{3}} (p R_{\omega(m)}^{p-1} + 5 R_{\omega(m)}^{4}) |u_{n}|^{2}\, dx 
\geq 1, 
\]
which is a contradiction. 
Thus, we see that $u_{\infty} \neq 0$. 
Then, by the lower semicontinuity, $u_{\infty} \neq 0$ 
and \eqref{MI-eq19}, we have 
\[
0 \geq \lim_{n \to \infty} \langle L_{m, +} u_{n}, u_{n} \rangle 
\geq \langle L_{m, +} u_{\infty}, u_{\infty} \rangle > 0, 
\]
which is a contradiction. 
Thus, we infer that \eqref{eq-mi28} holds. 

\textbf{(Step 3).}\; 
We conclude the proof. 
For each $u \in U_{\e}$, 
we put $\eta = u - R_{\omega(m)}$ and 
decompose
\[
\eta = a R_{\omega(m)} + p, 
\] 
where $a \in \R$ and $(R_{\omega(m)}, p)_{L^{2}} = 0$. 
Then, since $\|u\|_{L^{2}}^{2} = \|R_{\omega(m)}\|_{L^{2}}^{2}$ 
(see \eqref{MI-eq18}), we obtain 
\[
\begin{split}
\|R_{\omega(m)}\|_{L^{2}}^{2} 
& = \|u\|_{L^{2}}^{2} \\[6pt]
& = \|R_{\omega(m)} + \eta\|_{L^{2}}^{2} \\[6pt]
& = \|R_{\omega(m)}\|_{L^{2}}^{2} 
+ 2(R_{\omega(m)}, \eta)_{L^{2}} 
+ \|\eta\|_{L^{2}}^{2} \\[6pt]
& = \|R_{\omega(m)}\|_{L^{2}}^{2} 
+ 2(R_{\omega(m)}, a R_{\omega(m)} + p)_{L^{2}} 
+ \|\eta\|_{L^{2}}^{2} \\[6pt]
& = \|R_{\omega(m)}\|_{L^{2}}^{2} 
+ 2a \|R_{\omega(m)}\|_{L^{2}}^{2} 
+ \|\eta\|_{L^{2}}^{2}. 
\end{split}
\]
This 
implies $a = O(\|\eta\|_{L^{2}}^{2})$. 
Observe that $\mathcal{S}_{\omega(m)}^{\prime}(R_{\omega(m)}) = 0$. 
Thus, the Taylor expansion gives that 
\[
\begin{split}
\mathcal{S}_{\omega(m)}(u) 
= \mathcal{S}_{\omega(m)}(R_{\omega(m)} + \eta)
& = \mathcal{S}_{\omega(m)}(R_{\omega(m)}) 
+ \langle \mathcal{S}_{\omega(m)}^{\prime}(R_{\omega(m)}), 
\eta \rangle + \frac{1}{2} 
\langle \mathcal{S}_{\omega(m)}^{\prime \prime}
(R_{\omega(m)}) \eta, 
\eta\rangle + o(\|\eta\|_{H^{1}}^{2}) \\[6pt]
& = \mathcal{S}_{\omega(m)}(R_{\omega(m)}) 
+ \frac{1}{2} \langle L_{m, +}\eta, \eta\rangle + o(\|\eta\|_{H^{1}}^{2}) 
\end{split}
\]
Since $\|u\|_{L^{2}}^{2} = \|R_{\omega(m)}\|_{L^{2}}^{2}$, 
$\mathcal{S}_{\omega(m)} (u)= \E(u) + \omega\|u\|_{L^{2}}^{2}/2$, 
$\eta = a R_{\omega(m)} + p$
and $a = O(\|\eta\|_{L^{2}}^{2})$, 
one has 
\[
\begin{split}
\E(u) - \E(R_{\omega(m)}) 
= \frac{1}{2} \langle L_{m, +}\eta, \eta\rangle + o(\|\eta\|_{H^{1}}^{2}) 
& = \frac{1}{2} \langle L_{m, +}p, p\rangle 
+ O(a \|p\|_{H^{1}}) + O(a^{2})
+ o(\|\eta\|_{H^{1}}^{2}) \\[6pt]
& = \frac{1}{2} \langle L_{m, +}p, p\rangle 
+ o(\|\eta\|_{H^{1}}^{2}). 
\end{split}
\]
Then, it follows from \eqref{eq-mi28} that 
\begin{equation} \label{MI-eq21}
\begin{split}
\E(u) - \E(R_{\omega(m)}) 
\geq \frac{C_{0}}{2} \|p\|_{H^{1}}^{2} 
+ o(\|\eta\|_{H^{1}}^{2}).
\end{split}
\end{equation} 
In addition, since $a = O(\|\eta\|_{L^{2}}^{2}) $ and 
$\|\eta\|_{L^{2}} = \e > 0$ is small, we have 
\[
\|p\|_{H^{1}} = \|\eta - a R_{\omega(m)}|_{H^{1}} 
\geq \|\eta\|_{H^{1}} - |a| \|R_{\omega(m)}\|_{H^{1}} 
\geq \|\eta\|_{H^{1}} + O(\|\eta\|_{H^{1}}^{2})
\] 
This together with \eqref{MI-eq21} implies that 
\[
\E(u) - \E(R_{\omega(m)}) 
\geq \frac{C_{0}}{4} \|\eta\|_{H^{1}}^{2}. 
\]
This completes the proof. 
\end{proof}

\begin{proposition}\label{mi-prop4}
Let $7/3 \leq p < 3$, 
$m \in (0, m_{2})$ and 
$R_{\omega(m)}$ be the minimizer for $E_{\min}(m)$. 
Assume that $\frac{\p \|R_{\omega}\|_{L^{2}}^{2}}{\p \omega}
\big|_{\omega = \omega(m)}>0$. 
Then, we have $I(m) = 2$. 
\end{proposition}

\begin{proof}
It follows from 
Lemma \ref{mor-lem1} that 
$I(m) = 1$ or $2$ for each $m>0$. 
Thus, it suffices to eliminate the possibility that $I(m) = 1$. 
Suppose to the contrary that $I(m) = 1$. 
Then, it follows from \eqref{mi-lem6-1} that 
$R_{\omega(m)}$ is the unique minimizer of 
the energy $E$ on $U_{\e}$. 
This implies that 
\begin{equation} \label{MI-eq22}
\frac{d^{2}}{d \lambda^{2}} 
\E(\lambda^{\frac{3}{2}}
R_{\omega(m)}(\lambda \cdot))\big|_{\lambda = 1} \geq 0.
\end{equation}
On the other hand, 
observe from Remark \ref{rem-u2}, 
Theorem \ref{thm-bl-0} and 
$\lim_{m \to 0} \omega(m) = 0$ that 
\[
\lim_{m \to 0} \|R_{\omega(m)}\|_{L^{6}}^{6} 
= \lim_{m \to 0} \|u_{\omega(m)}\|_{L^{6}}^{6} = 
\|W\|_{L^{6}}^{6} = \sigma^{\frac{3}{2}}, 
\]
where $\sigma>0$ is defined by \eqref{EqL-18}. 
Using $\K(R_{\omega(m)}) = 0$ and 
$7/3 \leq p < 5$, 
we can compute 
\begin{equation} \label{MI-eq23}
\begin{split}
\frac{d^{2}}{d \lambda^{2}} 
\E(\lambda^{\frac{3}{2}}
R_{\omega(m)}(\lambda \cdot))\big|_{\lambda = 1}
& = 2 \|\nabla R_{\omega(m)}\|_{L^{2}}^{2} 
- \frac{9(p-1)^{2}}{4(p+1)} 
\|R_{\omega(m)}\|_{L^{p+1}}^{p+1} 
- 6 \|R_{\omega(m)}\|_{L^{6}}^{6} \\[6pt]
& = - \frac{3(p-1)}{4(p+1)} 
\left(3p-7\right)
\|R_{\omega(m)}\|_{L^{p+1}}^{p+1} 
- 4 \|R_{\omega(m)}\|_{L^{6}}^{6} \\[6pt]
&\leq 
- 4 \|R_{\omega(m)}\|_{L^{6}}^{6} 
\leq - 2 \sigma^{\frac{3}{2}}< 0
\end{split}
\end{equation}
for sufficiently small $m > 0$. 
Since $\|\lambda^{\frac{3}{2}}
R_{\omega(m)}(\lambda \cdot)\|_{L^{2}} 
= \|R_{\omega(m)}\|_{L^{2}}$ for all 
$\lambda > 0$, 
\eqref{MI-eq23} contradicts \eqref{MI-eq22}. 
This completes the proof. 
\end{proof}
We see from Proposition \ref{mi-prop4} that 
$\frac{\p \|R_{\omega}\|_{L^{2}}^{2}}{\p \omega}
\big|_{\omega = \omega(m)}>0$ implies that 
$I(m) = 2$. 
Thus, we need to compute the sign of 
$\frac{\p \|R_{\omega}\|_{L^{2}}^{2}}{\p \omega}
\big|_{\omega = \omega(m)}$. 
Concerning this, we obtain the following: 
\begin{lemma}\label{mi-lem5}
Let $7/3 < p < 3$ and 
$R_{\omega(m)}$ be the minimizer for $E_{\min}(m)$. 
We have
\begin{equation}\label{MI-eq24}
\frac{\p \|R_{\omega}\|_{L^{2}}^{2}}{\p \omega}
\big|_{\omega = \omega(m)}>0 
\qquad \mbox{almost all $m \in (0, m_{2})$}. 
\end{equation}
\end{lemma}
\begin{proof}
It follows from Proposition \ref{lem4-diff} 
that $\p \omega(m)/\p m > 0$ for almost all $m \in (0, m_{2})$. 
For such a point, we have by $\|R_{\omega(m)}\|_{L^{2}}^{2} = m$ 
that 
\[
0 < \frac{\p \omega(m)}{\p m}
= 
\dfrac{1}{\frac{\p \|R_{\omega}\|_{L^{2}}^{2}}{\p \omega}
\big|_{\omega = \omega(m)}}. 
\]
Therefore, we see that \eqref{MI-eq24} holds. 
\end{proof}
We see from \eqref{MI-eq17} and the min-max theorem 
that $\lambda_{1}(m)$ is written by 
\begin{equation}\label{MI-eq25}
\lambda_{1}(m) = 
\inf_{u \in H^{1}(\R^{3})} 
\frac{\langle L_{m, +} u, u \rangle}{\|u\|_{L^{2}}} 
\end{equation} 
and $\lambda_{1}(m) < 0$. 
We shall show the following: 
\begin{lemma}\label{mi-lem6-1}
The maps
$\lambda_{1}(m): (0, \infty) \to \R$ and 
$\chi_{1}(m):(0, \infty) \to H_{\text{rad}}^{1}(\R^{3})$ 
are continuous on $m >0$. 
\end{lemma}
\begin{proof}
Let $\{m_{n}\}$ be a sequence in $(0, \infty)$ 
with $\lim_{n \to \infty} m_{n} = m_{*}$ for some 
$m_{*}>0$. 
We shall show that $\lim_{n \to \infty} \lambda_{1}(m_{n}) 
= \lambda_{1}(m_{*})$. 
Note that 
$\sup_{n \in \N} \|R_{\omega(m_{n})}\|_{H^{1}} < + \infty$.
Then, by the elliptic regularity 
(see e.g. Soave~\cite[Proposition B.1]{MR4096725}), we have 
$\sup_{n \in \N} \|R_{\omega(m_{n})}\|_{L^{\infty}} < + \infty$. 
This together with \eqref{MI-eq25} and 
$\|\chi_{1}(m_{n})\|_{L^{2}} = 1$ yields that 
\[
\begin{split}
0
> 
\lambda_{1}(m_{n}) 
= \langle L_{m_{n}, +} \chi_{1}(m_{n}), \chi_{1}(m_{n}) \rangle
& 
\geq \|\nabla \chi_{1}(m_{n})\|_{L^{2}}^{2} 
+ \omega(m_{n}) 
- p \|R_{\omega(m_{n})}\|_{L^{\infty}}^{p-1} 
- 5 \|R_{\omega(m_{n})}\|_{L^{\infty}}^{4} \\[6pt]
& \geq \|\nabla \chi_{1}(m_{n})\|_{L^{2}}^{2} - C, 
\end{split}
\]
where the constant $C > 0$ is independent of $n \in \N$. 
This together with $\|\chi_{1}(m_{n})\|_{L^{2}} = 1$
implies that $\sup_{n \in \N} 
\|\chi_{1}(m_{n})\|_{H^{1}} < + \infty$. 
Since $\lambda_{1}(m_{n}) = 
\langle L_{m_{n}, +} \chi_{1}(m_{n}), \chi_{1}(m_{n}) 
\rangle$, one has 
$\sup_{n \in \N}|\lambda_{1}(m_{n})| < + \infty$. 
Then, up to a subsequence, there exist 
$\chi_{1, \infty} \in H^{1}(\R^{3})$ and 
$\lambda_{1, \infty} \in \R$ such that 
$\lim_{n \to \infty} \chi_{1}(m_{n}) = \chi_{1, \infty}$ 
weakly in $H^{1}(\R^{3})$ and 
$\lim_{n \to \infty} \lambda_{1}(m_{n}) = \lambda_{1, \infty}$. 

We shall show that $\chi_{1, \infty} \neq 0$. 
Suppose to the contrary that $\chi_{1, \infty} = 0$. 
Then, by the weak lower semicontinuity, 
$\lim_{|x| \to 0}R_{\omega(m_{n})}(x) = 0$
and $\lim_{n \to \infty} 
\omega(m_{n}) = \omega(m_{*})$
(see Proposition \ref{lem2-diff}), we see that 
\[
\begin{split}
0 
& \geq \lim_{n \to \infty} \lambda_{1}(m_{n}) \\[6pt]
& \geq \lim_{n \to \infty} 
\biggl\{
\|\nabla \chi_{1}(m_{n})\|_{L^{2}}^{2} 
+ \omega(m_{n}) \|\chi_{1}(m_{n})\|_{L^{2}}^{2} \\[6pt] 
& \hspace{1.5cm}
- p \int_{\R^{3}} R_{\omega(m_{n})}^{p-1} |\chi_{1}(m_{n})|^{2} \, dx 
- 5 \int_{\R^{3}} R_{\omega(m_{n})}^{4} |\chi_{1}(m_{n})|^{2} \, dx
\biggl\} \\[6pt]
& \geq \omega(m_{*}) > 0, 
\end{split}
\]
which is absurd. 
Thus, we find that $\chi_{1, \infty} \neq 0$. 

By Proposition \ref{lem2-diff}, we have 
$\lim_{n \to \infty} \omega(m_{n}) = \omega(m_{*})$ 
and $\lim_{n \to \infty} R_{\omega(m_{n})} 
= R_{\omega(m_{*})}$ strongly in $H^{1}(\R^{3})$, which 
yields that 
\[
\begin{split}
\lambda_{1, \infty} 
& = \lim_{n \to \infty} \lambda_{1}(m_{n}) \\[6pt]
& \leq \lim_{n \to \infty} 
\langle L_{m_{n}, +} \chi_{1}(m_{*}), \chi_{1}(m_{*}) \rangle \\[6pt]
& \leq \langle L_{m_{*}, +} \chi_{1}(m_{*}), \chi_{1}(m_{*}) \rangle 
+ \lim_{n \to \infty} 
\biggl\{ 
|\omega(m_{n}) - \omega(m_{*})| \\[6pt]
& \quad 
+ p \int_{\R^{3}} |R_{\omega(m_{n})}^{p-1} - R_{\omega_{*}}^{p-1}| 
|\chi_{1}(m_{*})|^{2} \, dx 
+ 5 \int_{\R^{3}} 
||R_{\omega(m_{n})}^{4} - R_{\omega_{*}}^{4}| 
|\chi_{1}(m_{*})|^{2} \, dx 
\biggl\} \\[6pt]
& = \langle L_{m_{*}, +} \chi_{1}(m_{*}), \chi_{1}(m_{*}) \rangle 
= \lambda_{1}(m_{*}). 
\end{split}
\]
Thus, we have $\lambda_{1, \infty} 
\leq \lambda_{1}(m_{*})$.

Next, we shall show that $\|\chi_{1, \infty}\|_{L^{2}} = 1$. 
By the weak lower semicontinuity, we have 
\[
\|\chi_{1, \infty}\|_{L^{2}} \leq 
\liminf_{n \to \infty} \|\chi_{1}(m_{n})\|_{L^{2}} 
= 1.
\]
Suppose to the contrary that 
$\|\chi_{1, \infty}\|_{L^{2}} < 1$. 
Then, there exists $a > 1$ such that 
$\|a \chi_{1, \infty}\|_{L^{2}} = 1$. 
Then, it follows from \eqref{MI-eq25}, 
$\lim_{n \to \infty} \lambda_{1}(m_{n}) = 
\lambda_{1, \infty} \leq \lambda_{1}(m_{*}) < 0$ that 
\[
\begin{split}
& \quad 
\lambda_{1}(m_{*}) 
\leq a^{2} \langle L_{m_{*}, +} \chi_{1, \infty}, 
\chi_{1, \infty} \rangle 
\leq a^{2} 
\liminf_{n \to \infty}
\langle L_{m_{n}, +} \chi_{1}(m_{n}), 
\chi_{1}(m_{n}) \rangle \\[6pt]
& 
= a^{2} \liminf_{n \to \infty} \lambda_{1}(m_{n}) 
= a^{2} \lambda_{1, \infty}
\leq a^{2} \lambda_{1} (m_{*}) 
< \lambda_{1} (m_{*}), 
\end{split}
\]
which is a contradiction. 
Thus, we see that $\|\chi_{1, \infty}\|_{L^{2}} = 1$. 

Since $\|\chi_{1, \infty}\|_{L^{2}} = 1$, we obtain 
\[
\begin{split}
\lambda_{1}(m_{*}) 
& \leq \langle L_{m_{*}, +} \chi_{1, \infty}, 
\chi_{1, \infty} \rangle 
= \liminf_{n \to \infty} 
\langle L_{m_{n}, +} \chi_{1}(m_{n}), 
\chi_{1}(m_{n}) \rangle 
= \lim_{n \to \infty} \lambda_{1}(m_{n}) 
= \lambda_{1, \infty} \leq \lambda_{1}(m_{*}). 
\end{split}
\]
This implies that $\lim_{n \to \infty}\lambda_{1}(m_{n}) 
= \lambda_{1, \infty} = \lambda_{1}(m_{*}) 
= \langle L_{m_{*}, +} \chi_{1, \infty}, 
\chi_{1, \infty} \rangle$. 
Namely, $\chi_{1, \infty}$ is the first eigenfunction of 
$L_{m_{*}, +}$. 
From the simplicity of $\lambda_{1}(m_{*})$, 
we have $\chi_{1, \infty} = \chi_{1}(m_{*})$. 

Finally, we claim that 
$\lim_{n \to \infty} \chi_{1}(m_{n}) 
= \chi_{1}(m_{*})$ strongly in 
$H^{1}(\R^{3})$. 
Since $\lim_{n \to \infty} \chi_{1}(m_{n}) 
= \chi_{1}(m_{*})$ weakly in 
$H^{1}(\R^{3}), \lim_{n \to \infty} 
\lambda_{1}(m_{n}) = \lambda_{1, \infty} 
= \lambda_{1}(m_{*})$ 
and 
$\lim_{n \to \infty} 
\omega(m_{n}) = \omega(m_{*})$, we obtain
\begin{equation*}
\begin{split}
\|\nabla \chi_{1}(m_{*})\|_{L^{2}}^{2}
& = 
\langle L_{m_{*}, +} \chi_{1}(m_{*}), 
\chi_{1}(m_{*}) \rangle 
- \omega(m_{*}) \\[6pt]
&\quad 
+ p \int_{\R^{3}} R_{\omega(m_{*})}^{p-1} 
|\chi_{1}(m_{*})|^{2} \, dx 
+ 5 \int_{\R^{3}} R_{\omega(m_{*})}^{4} 
|\chi_{1}(m_{*})|^{2} \, dx \\[6pt]
& = 
\lim_{n \to \infty} 
\biggl\{
\lambda_{1}(m_{n})
- \omega(m_{n}) + 
p \int_{\R^{3}} R_{\omega(m_{*})}^{p-1} 
|\chi_{1}(m_{n})|^{2} \, dx 
+ 5 \int_{\R^{3}} R_{\omega(m_{*})}^{4} 
|\chi_{1}(m_{n})|^{2} \, dx \biggl\}\\[6pt]
& = \lim_{n \to \infty}
\biggl\{
\|\nabla \chi_{1}(m_{n})\|_{L^{2}}^{2} 
+ (\omega(m_{n}) - \omega(m_{*})) \\[6pt]
& \quad 
- p \int_{\R^{3}} (R_{\omega(m_{n})}^{p-1} 
- R_{\omega(m_{*})}^{p-1} )|\chi_{1}(m_{n})|^{2} \, dx 
- 5\int_{\R^{3}} (R_{\omega(m_{n})}^{4} 
- R_{\omega(m_{*})}^{4} )|\chi_{1}(m_{n})|^{2} \, dx 
\biggl\} \\[6pt]
& = \lim_{n \to \infty} \|\nabla \chi_{1}(m_{n})\|_{L^{2}}^{2}. 
\end{split}
\end{equation*}
Namely, one has 
$\lim_{n \to \infty} \|\nabla \chi_{1}(m_{n})\|_{L^{2}}^{2} = \|\nabla \chi_{1}(m_{*})\|_{L^{2}}^{2}$. 
This together with $\lim_{n \to \infty} 
\chi_{1}(m_{n}) = \chi_{1}(m_{*})$ weakly in 
$H^{1}(\R^{3})$ concludes the proof. 
\end{proof}
We are now in a position to prove 
Theorem \ref{thm-bl} \textrm{(iii)}. 
\begin{proof}[Proof of Theorem 
\ref{thm-bl} \textrm{(iii)}]
By Proposition \ref{mi-prop4} and Lemma \ref{mi-lem5}, 
we have $I(m) = 2$ for almost all $m \in (0, m_{2})$. 
Thus, for any $m_{*} \in (0, m_{2})$, 
we can take a sequence $\{m_{n}\}$ in $(0, m_{2})$
such that $\lim_{n \to \infty} m_{n} = m_{*}$ with 
$I(m_{n}) = 2$ for all $n \in \N$. 
Let $\lambda_{2}(m_{n})$ be the second eigenvalue of the 
operator $L_{m_{n}, +}$ and 
$\chi_{2}(m_{n}) \in H^{1}(\R^{3})$ 
be the corresponding eigenfunction with 
$\|\chi_{2}(m_{n})\|_{L^{2}} = 1$. 
Thus, it follows that 
\[
\langle L_{m_{*}, +} \chi_{2}(m_{n}), 
\chi_{2}(m_{n}) \rangle < 0 
\qquad \mbox{for all $n \in \N$}. 
\]
Then, by a similar argument of the proof of 
Lemma \ref{mi-lem6-1}, we can prove that 
$\sup_{n \in \N} 
\|\chi_{2}(m_{n})\|_{H^{1}} < + \infty$ 
and $\sup_{n \in \N}|\lambda_{2}(m_{n})| < + \infty$. 
Let $\varphi_{0} \in C_{0}^{\infty}(\R^{3}) \setminus \{0\}$ 
be a cut-off 
function satisfying $\varphi_{0}(x) \geq 0$ for all $x \in \R^{3}$. 
We put 
\begin{equation} \label{MI-eq26}
f_{n} = 
\frac{\chi_{2}(m_{n}) + a_{n} \varphi_{0}}
{\|\chi_{2}(m_{n}) + a_{n} \varphi_{0}\|_{L^{2}}}, 
\qquad 
a_{n} = - \frac{(\chi_{2}(m_{n}), \chi_{1}(m_{*}))_{L^{2}}}
{(\varphi_{0}, \chi_{1}(m_{*}))_{L^{2}}}. 
\end{equation}
Then, we see that 
\[
(f_{n}, \chi_{1}(m_{*}))_{L^{2}} = 0, 
\qquad 
\|f_{n}\|_{L^{2}} = 1. 
\]
This yields that 
\begin{equation} \label{MI-eq27}
\begin{split}
\inf_{u \in H^{1}(\R^{3}), 
(u, \chi_{1}(m_{*}))_{L^{2}} = 0} 
\frac{\langle L_{m_{*}, +} u, u \rangle}{\|u\|_{L^{2}}^{2}}
& \leq \langle L_{m_{*}, +} f_{n}, f_{n} \rangle \\[6pt]
& \leq \frac{\langle L_{m_{*}, +} \chi_{2}(m_{n}), 
\chi_{2}(m_{n}) \rangle}
{\|\chi_{2}(m_{n}) + a_{n} \varphi_{0}\|_{L^{2}}^{2}}
+ 2 a_{n}\frac{\langle L_{m_{*}, +} \chi_{2}(m_{n}), 
\varphi_{0} \rangle}
{\|\chi_{2}(m_{n}) + a_{n} \varphi_{0}\|_{L^{2}}^{2}} \\[6pt]
& \quad 
+ a_{n}^{2} \frac{\langle L_{m_{*}, +} \varphi_{0}, 
\varphi_{0} \rangle}
{\|\chi_{2}(m_{n}) + a_{n} \varphi_{0}\|_{L^{2}}^{2}}. 
\end{split}
\end{equation}
By Lemma \ref{mi-lem6-1}, 
$\sup_{n \in \N} 
\|\chi_{2}(m_{n})\|_{H^{1}} < + \infty$, \eqref{MI-eq26}, 
$(\chi_{2}(m_{n}), \chi_{1}(m_{n}))_{L^{2}} = 0$ and 
$\lim_{n \to \infty} \chi_{1}(m_{n}) = 
\chi_{1}(m_{*})$ strongly in $H^{1}(\R^{3})$ 
(see Lemma \ref{mi-lem6-1}), 
one has 
\begin{equation}\label{MI-eq28}
\begin{split}
\lim_{n \to \infty} a_{n} 
& = - \lim_{n \to \infty} \frac{(\chi_{2}(m_{n}), \chi_{1}(m_{*}))_{L^{2}}}
{(\varphi_{0}, \chi_{1}(m_{*}))_{L^{2}}} \\[6pt]
& = - \lim_{n \to \infty} 
\left\{
\frac{(\chi_{2}(m_{n}), \chi_{1}(m_{n}))_{L^{2}}}
{(\varphi_{0}, \chi_{1}(m_{*}))_{L^{2}}} 
+ 
\frac{(\chi_{2}(m_{n}), \chi_{1}(m_{*}) - \chi_{1}(m_{n}))_{L^{2}}}
{(\varphi_{0}, \chi_{1}(m_{*}))_{L^{2}}}
\right\} \\[6pt]
& = 0. 
\end{split}
\end{equation}
In addition, we have 
\[
\begin{split}
\langle L_{m_{*}, +} \chi_{2}(m_{n}), 
\chi_{2}(m_{n}) \rangle 
& = 
\langle L_{m_{n}, +} \chi_{2}(m_{n}), 
\chi_{2}(m_{n}) \rangle + (\omega_{*} - \omega(m_{n})) \\[6pt]
& \quad 
- p \int_{\R^{3}} (R_{\omega(m_{*})}^{p-1} 
- R_{\omega(m_{n})}^{p-1} )|\chi_{2}(m_{n})|^{2} \, dx
- 5\int_{\R^{3}} (R_{\omega(m_{*})}^{4} 
- R_{\omega(m_{n})}^{4} )|\chi_{2}(m_{n})|^{2} \, dx, 
\end{split}
\]
which implies 
\[
\limsup_{n \to \infty} \langle L_{m_{*}, +} \chi_{2}(m_{n}), 
\chi_{2}(m_{n}) \rangle \leq 0. 
\]
In addition, by \eqref{MI-eq28}, we obtain 
\[
\begin{split}
\|\chi_{2}(m_{n}) + a_{n} \varphi_{0}\|_{L^{2}}^{2} 
= \|\chi_{2}(m_{n})\|_{L^{2}}^{2} 
+ 2 a_{n} (\chi_{2}(m_{n}), \varphi_{0})_{L^{2}} 
+ a_{n}^{2}\|\varphi_{0}\|_{L^{2}}^{2}
& = 1 + 2 a_{n} (\chi_{2}(m_{n}), \varphi_{0})_{L^{2}} 
+ a_{n}^{2}\|\varphi_{0}\|_{L^{2}}^{2} \\
& \to 1 \qquad \mbox{as $n \to \infty$}. 
\end{split}
\]
These together with \eqref{MI-eq27} yield that 
\[
\inf_{u \in H^{1}(\R^{3}), 
(u, \chi_{1}(m_{*}))_{L^{2}} = 0} 
\frac{\langle L_{m_{*}, +} u, u \rangle}{\|u\|_{L^{2}}^{2}} 
\leq \limsup_{n \to \infty}
\langle L_{m_{*}, +} f_{n}, f_{n} \rangle
=
\limsup_{n \to \infty} 
\frac{\langle L_{m_{*}, +} \chi_{2}(m_{n}), 
\chi_{2}(m_{n}) \rangle}
{\|\chi_{2}(m_{n}) + a_{n} \varphi_{0}\|_{L^{2}}^{2}}
\leq 0, 
\]
which implies that $\lambda_{2}(m_{*})$ exists and 
$\lambda_{2}(m_{*}) \leq 0$. 
We see 
from Theorem \ref{thm-bl} \textrm{(ii)} and 
$R_{\omega(m)} = u_{\omega(m)}$ 
that $\lambda_{2}(m_{*}) \neq 0$. 
Thus, we obtain $\lambda_{2}(m_{*}) < 0$, which concludes the proof. 
\end{proof}

\appendix
\section{Non-existence of 
positive solution} \label{sec-a}
In this appendix, we shall show the 
there exists no positive solution to \eqref{sp} 
for sufficiently large $\omega > 0$ when $1 < p < 3$ 
by contradiction. 
Suppose to the contrary that there exists a sequence $\{\omega_{n}\}$ 
with $\lim_{n \to \infty} \omega_{n} = \infty$ such that 
for each $n \in \N$, 
\eqref{sp} has a positive solution $u_{\omega_{n}}$ at $\omega = \omega_{n}$. 
Then, by the result of \cite{MR3964275}, 
we obtain the following: 
	\begin{lemma}[Lemma 2.3 of \cite{MR3964275}]
	\label{lem-a1}
	Let $1 < p < 3$. For each $n \in \N$, 
	we put $M_{n}:= \|u_{\omega_{n}}\|_{L^{\infty}}$. 
	Then, we have 
		\begin{equation}\label{eq-a1}
		\lim_{n \to \infty} M_{n} = \infty. 
		\end{equation}
	\end{lemma}
As in Section \ref{sec-cat}, we put 
	\begin{equation} \label{eq-a10}
	\widetilde{u}_{n} := M_{n}^{-1} u_{\omega_{n}} (M_{n}^{-2} \cdot). 
	\end{equation}	
Then, we see that $\widetilde{u}_{n}$ satisfies 
\begin{equation}\label{eq-a2}
- \Delta \widetilde{u}_{n} 
+ \alpha_{n} \widetilde{u}_{n}
- \beta_{n} \widetilde{u}_{n}^{p} 
- \widetilde{u}_{n}^{5}
= 0 
\qquad 
\text{in $\R^{3}$}, 
\end{equation}
where $\alpha_{n} := \omega_{n} M_{n}^{-4}$ 
and $\beta_{n} := M_{n}^{p - 5}$. 
Then,  by a similar argument to the proof of Theorem 
\ref{thm-bl-0}, 
we obtain the following: 
	\begin{proposition}\label{pro-a2}
	Let $1 < p < 3$ and 
    $\{\widetilde{u}_{n}\}$ be the sequence of positive solutions to 
	\eqref{eq-a2} defined by \eqref{eq-a10}. 
	we have 
\begin{equation}\label{eq-a9}
\widetilde{u}_{n} \to W 
\qquad \mbox{strongly in 
$\dot{H}^{1} \cap L^{q}(\R^{3})\; 
(q > 3)$ as $n \to \infty$}. 
\end{equation}
	\end{proposition}
\begin{remark}
As we mentioned in Remark \ref{rem-thm-bl}, 
it seems difficult to obtain 
the boundedness of $\|\widetilde{u}_{n}\|_{H^{1}}$. 
due to this, we cannot obtain the non-existence of the positive solution 
to \eqref{sp} for large $\omega > 0$ before. 
\end{remark}
Once we obtain Proposition \ref{pro-a2}, 
we can show the following lemmas: 
	\begin{lemma}\label{lem-a2}
	Let $1 < p < 3$ and 
    $\{\widetilde{u}_{n}\}$ be the sequence of positive solutions to 
	\eqref{eq-a2} defined by \eqref{eq-a10}. 
	Then, there exists a constant $C > 0$
	\begin{equation}\label{eq-a3}
	 \liminf_{n \to \infty} \sqrt{\alpha_{n}} 
	 \|\widetilde{u}_{n}\|_{L^{2}}^{2} \geq C. 
	\end{equation}
	\end{lemma}
	\begin{proof}
	We can show that $\widetilde{u}_{n}$ satisfies  
	\eqref{e-muni64}, which yields \eqref{eq-a3}. 
	\end{proof}
	\begin{lemma}\label{lem-a3}
	Let $1 < p < 3$ and 
    $\{\widetilde{u}_{n}\}$ be the sequence of positive solutions to 
	\eqref{eq-a2} defined by \eqref{eq-a10}. 
	Then, there exists a constant $C_{i} > 0\; (i = 1, 2, 3)$
	such that 
	\begin{align}
	 & \limsup_{n \to \infty} 
	 \|\widetilde{u}_{n}\|_{L^{p + 1}}^{p + 1} \leq C_{1} 
	 \qquad \mbox{when $2 < p < 3$}, \label{eq-a4} \\
	  & \limsup_{n \to \infty} 
	 \|\widetilde{u}_{n}\|_{L^{p + 1}}^{p + 1} \leq C_{2} |\log \alpha_{n}| 
	 \qquad \mbox{when $p = 2$}, \label{eq-a5} \\
	  & \limsup_{n \to \infty} 
	 \|\widetilde{u}_{n}\|_{L^{p + 1}}^{p + 1} \leq C_{3} \alpha_{n}^{\frac{p - 2}{2}}
	 \qquad \mbox{when $1 < p < 2$}. \label{eq-a6}
	\end{align}
	\end{lemma}
\begin{proof}
We see that the estimate \eqref{EqL-1}, \eqref{e-muni41}
and  \eqref{e-muni87} also hold for $\widetilde{u}_{n}$.  
Then, from these, we see that \eqref{eq-a4}--\eqref{eq-a6} holds. 
\end{proof}
We are now in a position to prove Theorem \ref{thm-none}.  

\begin{proof}[Proof of Theorem \ref{thm-none}]
As in \eqref{main-eq7}, we have 
	\begin{equation} \label{eq-a7}
	\alpha_{n} \|\widetilde{u}_{n}\|_{L^{2}}^{2} 
	= \frac{5 - p}{2(p + 1)} \beta_{n} \|\widetilde{u}_{n}\|_{L^{p + 1}}^{p + 1}. 
	\end{equation}
When $2 < p < 3$, it follows from Lemmas 
\ref{lem-a2}, \ref{lem-a3} \textrm{(i)} and \eqref{eq-a7} that 
	\begin{equation} \label{eq-a8}
	\sqrt{\alpha_{n}} \leq C \beta_{n}
	\end{equation}
for some $C > 0$, which does not depend on $n \in \N$. 
Since $\alpha_{n} = \omega_{n} M_{n}^{-4}$ 
and $\beta_{n} = M_{n}^{p - 5}$, we see from 
\eqref{eq-a8} that $\sqrt{\omega_{n}} \leq C M_{n}^{p -3}$, 
which is absurd because $\lim_{n \to \infty} \omega_{n} = \infty$, 
while $\lim_{n \to \infty} M_{n}^{p - 3} = 0$. 
We can derive a contradiction similarly for the case of $1 < p \leq 2$. 	
\end{proof}


\section{Boundedness of 
$L^{\infty}$ norm of solutions}\label{section:B}

Here, we obtain a boundedness of 
$L^{\infty}$ norm of solution to elliptic equations.  
Although
the following result can be proved by a variant of 
Br\'ezis--Kato type argument and the Moser iteration (\cite{MR539217, MR1814364}),  
we give the proof  
because we don't find any proper references. 
In fact, we follow the argument 
of the arxiv version (Version 1) 
of \cite{MR3964275} 
(see \url{https://arxiv.org/abs/1801.08696} 
for the detail). 
	\begin{proposition}\label{proposition:B.1} 
		Let $d\ge 3$, 
        $a(x)$ and $b(x)$ be functions on $B_{4}$, 
		and let $ u \in H^1(B_4)$ be a weak solution to 
			\begin{equation}\label{eq:B.1}
				-\Delta u + a(x) u = b (x) u + f(x) \quad {\rm in} \ B_4.
			\end{equation}
		Suppose that $a(x)$ and $u$ satisfy that 
			\begin{equation}\label{eq:B.2}
				a(x) \geq 0 \quad {\rm for\ a.a.}\ x \in B_4, \quad 
				\int_{B_4} a(x) |u(x)v(x)| d x < \infty \quad 
				{\rm for\ each\ } v \in H^1_0(B_4). 
			\end{equation}
			Let $s>d/2$ and assume that $ b \in L^{s}(B_{4})$ and 
            $f \in W^{-1, 2s} (B_{4})$. 
			Then, there exists a constant $C(d,s, \|b\|_{L^{s}(B_{4})}, \|f\|_{W^{-1, 2s}(B_{4})})$ such that 
			\[
				\| u \|_{L^\infty(B_1)} \le C \left(d,s, \|b\|_{L^{s}(B_{4})}, 
                \|f\|_{W^{-1, 2s}(B_{4})} 
                \right) 
				\| u \|_{L^{2^\ast}(B_4)}.
			\]
	Here, the constant 
    $C(d,s, \|b\|_{L^{s}(B_{4})}, 
    \|f\|_{W^{-1, 2s}(B_{4})})$ in \emph{(i)} and \emph{(ii)} remain bounded as long as $t_q$ and $\|b\|_{L^{s}( B_{4})}$ are bounded. 
	\end{proposition}

	\begin{proof}
We argue as in \cite[Proof of Proposition 2.2]{MR2983045}. 
Let $T>1$ and set 
	\[
		u_T(x) := \left\{ \begin{aligned}
			& T & &{\rm if} \ u(x) > T,\\
			& u(x) & &{\rm if} \ |u(x)| \leq T, \\
			& - T & &{\rm if} \ u(x) < - T. 
		\end{aligned} \right.
	\]
For $1 \leq r < R \leq 4$, choose an $\eta_{r,R} \in C^\infty_0(\R^d)$ so that 
$0 \leq \eta_{r,R} \leq 1$, $\eta_{r,R}(x) = 1$ if $|x| \leq r$ and 
$\eta_{r,R}(x) = 0$ if $|x| \geq R$. 
For $q>0$, we set $\varphi(x) := \eta_{r,R}^2 (x)|u_T(x)|^{2q}u (x) \in H^1_0(B_4)$. 
Since $f \in W^{-1, q}(B_{4})$, there exists $f_{i} \in L^{q}(B_{4})$ 
for $i = 1, 2, \ldots, d$ such that $f = \sum_{i = 1}^{d} \frac{\partial f}{\partial x_{i}}$ 
(see e.g. Brezis~\cite[Proposition 9.20]{MR2759829}). 
Since $a(x) u(x) \varphi(x) \geq 0$ for a.a. $x \in B_4$ due to \eqref{eq:B.2}, 
it follows from \eqref{eq:B.1} that 
	\begin{equation}\label{eq:B.4}
    \begin{split}
		\int_{B_4} \nabla u \cdot \nabla \varphi d x 
        & \leq 
		\int_{B_4} b(x) \eta_{r,R}^2(x) |u_T(x)|^{2q}  |u(x)|^2 d x 
        + \sum_{j = 1}^{d}\int_{B_4} f_{j}(x) 
        \frac{\partial \varphi}{
        \partial x_{j}}(x) d x \\
        & \leq 
		\int_{B_4} b(x) \eta_{r,R}^2(x) |u_T(x)|^{2q}  |u(x)|^2 d x 
        + \sum_{j = 1}^{d}\int_{B_4} f_{j}(x) 
        \frac{\partial \varphi}{\partial x_{j}}(x) d x. 
	\end{split}
    \end{equation}

	Next, it follows from 
    $\varphi(x) = \eta_{r,R}^2 (x)
    |u_T(x)|^{2q}u (x) $ that 
    \begin{equation} \label{eq:B.10}
    \frac{\partial \varphi}{\partial x_{j}} 
		= 2 \eta_{r,R} |u_T|^{2q} u \frac{\partial \eta_{r,R}}{\partial x_{j}} 
		+ 2 q \eta_{r,R}^2 |u_T|^{2q-2} u_T u \frac{\partial u_T}
        {\partial x_{j}} 
		+ \eta_{r,R}^2 |u_T|^{2q} \frac{\partial u}{\partial x_{j}}, 
    \end{equation}
	\begin{equation} \label{eq:B.10-1}
		\nabla \varphi 
		= 2 \eta_{r,R} |u_T|^{2q} u \nabla \eta_{r,R} 
		+ 2 q \eta_{r,R}^2 |u_T|^{2q-2} u_T u \nabla u_T 
		+ \eta_{r,R}^2 |u_T|^{2q} \nabla u
	\end{equation}
and that $\nabla u_T = \nabla u$ a.e. $[|u|<T]$ and 
$\nabla u_T = 0$ a.e. $[ |u| \geq T ]$. 
Thus, by \eqref{eq:B.10-1}, we get 
	\[
		\begin{aligned}
			\int_{B_4} \nabla u \cdot \nabla \varphi \, dx 
			& = 2 \int_{B_4} \eta_{r,R} |u_T|^{2q} u \nabla \eta_{r,R} \cdot \nabla u d x 
			+ 2q \int_{B_4}\eta_{r,R}^2 |u_T|^{2q-2} u_T u \nabla u_T \cdot \nabla u\, dx 
			\\
			& \quad + \int_{B_4} \eta_{r,R}^2 |u_T|^{2q} | \nabla u|^2 \, dx
			\\
			& = 2 \int_{B_4} \eta_{r,R} |u_T|^{2q} u \nabla \eta_{r,R} \cdot \nabla u d x 
			+ 2q \int_{B_4} \eta_{r,R}^2 |u_T|^{2q} |\nabla u_T |^2 \, dx 
			\\		
			& \quad + \int_{B_4} \eta_{r,R}^2 |u_T|^{2q} |\nabla u |^2 d x
		\end{aligned}
	\]
On the other hand, since 
	\[
		\left| 2 \eta_{r,R} |u_T|^{2q} u \nabla \eta_{r,R} \cdot \nabla u \right| 
		\leq \frac{|u_T|^{2q} 
        \eta_{r,R}^2 |\nabla u|^2}{2} 
		+ 2 |u_T|^{2q} u^2 |\nabla \eta_{r,R} |^2, 
	\]
$u = u_{T}$ for $x \in [|u| < T]$ 
and $\frac{\p u_{T}}{\p x_{j}} = 0$ 
a.e. $[ |u| \geq T ]$, we have 
	\begin{equation} \label{eq:B11}
		\begin{aligned}
			&
            \int_{B_4} \nabla u \cdot \nabla \varphi \, dx 
			\\
			\geq \,&  
            \frac{1}{2} \int_{B_4} \eta_{r,R}^2 |u_T|^{2q} |\nabla u |^2 d x 
			- 2 
            \int_{B_4} 
            \left|\nabla \eta_{r,R}\right|^{2}
            |u_T|^{2q} u^{2} d x 
			+
            2q \int_{B_4} \eta_{r,R}^2 |u_T|^{2q} 
            \left|\nabla u_{T}\right|^{2}\, dx.
		\end{aligned}
	\end{equation}
In addition, observe from \eqref{eq:B.10} that 
  \[
		\begin{aligned}
			\sum_{j = 1}^{d}\int_{B_4}  
            f_{j} \frac{\partial \varphi}{\partial x_{j}} \, dx 
			& = 2 \sum_{j = 1}^{d} \int_{B_4} \eta_{r,R} |u_T|^{2q} u 
            \frac{\partial \eta_{r,R}}{\partial x_{j}}  
            f_{j}d x 
			+ 2q \sum_{j = 1}^{d} 
            \int_{B_4}\eta_{r,R}^2 |u_T|^{2q-2} u_T u 
            \frac{\partial u_{T}}{\partial x_{j}} 
            f_{j} \, dx 
			\\
			& \quad + \int_{B_4} \eta_{r,R}^2 |u_T|^{2q} 
            \frac{\partial u}{\partial x_{j}}
            f_{j}\, dx
			\\
			& = 2 \sum_{j = 1}^{d} \int_{B_4} \eta_{r,R} |u_T|^{2q} u 
            \frac{\partial \eta_{r,R}}{\partial x_{j}}  
            f_{j}d x 
			+ 2q \sum_{j = 1}^{d} 
            \int_{B_4}\eta_{r,R}^2 |u_T|^{2q}  
            \frac{\partial u_{T}}{\partial x_{j}}
            f_{j}\, dx 
			\\
			& \quad + \int_{B_4} \eta_{r,R}^2 |u_T|^{2q} 
             \frac{\partial u}{\partial x_{j}}
             f_{j}\, dx
		\end{aligned}
	\]
  We see that 
  \[
  \begin{split}
  2\sum_{j = 1}^{d} \int_{B_4} \eta_{r,R} |u_T|^{2q} u 
            \frac{\partial \eta_{r,R}}{\partial x_{j}}  
            f_{j}d x 
            & \leq 
            \sum_{j = 1}^{d} \int_{B_4} |u_T|^{2q} u^{2} 
            \left|\frac{\partial \eta_{r,R}}{\partial x_{j}}\right|^{2}  
            d x + 
            \sum_{j = 1}^{d} \int_{B_4} \eta_{r,R}^2 |u_T|^{2q}   
            |f_{j}|^{2}\, dx \\
            & \leq 
            \int_{B_4} 
            \left|\nabla \eta_{r,R}\right|^{2}
            |u_T|^{2q} u^{2} d x 
            + \sum_{j = 1}^{d} \int_{B_4} \eta_{r,R}^2 |u_T|^{2q}   
            |f_{j}|^{2}\, dx, 
  \end{split}
  \]
  \[
  \begin{split}
 2q \sum_{j = 1}^{d} 
            \int_{B_4}\eta_{r,R}^2 |u_T|^{2q} f_{j} 
            \frac{\partial u_{T}}{\partial x_{j}}\, dx  
            \leq 
            q \int_{B_4} \eta_{r,R}^2 |u_T|^{2q} 
            \left|\nabla u_{T}\right|^{2}\, dx 
            + q \sum_{j = 1}^{d}
            \int_{B_4} \eta_{r,R}^2 |u_T|^{2q} 
            |f_{j}|^{2} \, dx
    \end{split}
    \]
    \[
  \begin{split}            
 \sum_{j = 1}^{d} 
 \int_{B_4}\eta_{r,R}^2 |u_T|^{2q}  
            \frac{\partial u}{\partial x_{j}} 
            f_{j}\, dx
            & \leq 
            \frac{1}{4} \sum_{j = 1}^{d} 
            \int_{B_4}\eta_{r,R}^2 |u_T|^{2q}  
            \biggl|\frac{\partial u}{\partial x_{j}}\biggl|^{2}\, dx + 
            2 \sum_{j = 1}^{d} 
            \int_{B_4}\eta_{r,R}^2 |u_T|^{2q} |f_{j}|^{2} \, dx \\
            & \leq 
            \frac{1}{4} 
            \int_{B_4}\eta_{r,R}^2 |u_T|^{2q}  
            |\nabla u |^{2}\, dx 
            + 2 \sum_{j = 1}^{d} 
            \int_{B_4}\eta_{r,R}^2 |u_T|^{2q} |f_{j}|^{2} \, dx. 
  \end{split}
  \]
Then, we obtain 
    \begin{equation} \label{eq:B12}
    \begin{split}
    \sum_{j = 1}^{d}\int_{B_4}  
            f_{j} \frac{\partial \varphi}{\partial x_{j}} \, dx 
            & \leq 
            \frac{1}{4} 
            \int_{B_4}\eta_{r,R}^2 |u_T|^{2q}  
            |\nabla u |^{2}\, dx
            + 
            \int_{B_4} 
            \left|\nabla \eta_{r,R}\right|^{2}
            |u_T|^{2q} u^{2} \, dx
            + q \int_{B_4} \eta_{r,R}^2 |u_T|^{2q} 
            \left|\nabla u_{T}\right|^{2}\, dx \\
            & \quad + (q + 3) 
            \sum_{j = 1}^{d} 
            \int_{B_4}\eta_{r,R}^2 |u_T|^{2q} |f_{j}|^{2} \, dx. 
    \end{split}    
    \end{equation}
  
Thus, it follows from \eqref{eq:B.4}, 
\eqref{eq:B11} and \eqref{eq:B12} that 
	\begin{equation}\label{eq:B.5}
		\begin{aligned}
		 & \int_{B_4}\eta_{r,R}^2 |u_T|^{2q}  
            |\nabla u |^{2}\, dx 
		 + 2 q 
            \int_{B_4} \eta_{r,R}^2 |u_T|^{2q} 
            \left|\nabla u_{T}\right|^{2}\, dx 
		 \\
		 \leq \, &  12 
         \int_{B_4} u^2 |u_T|^{2q} |\nabla \eta_{r,R} |^2 d x 
		 + 4 \int_{B_4} b \eta_{r,R}^2  |u_T |^{2q}  u^2 d x + 4 (q + 3) 
            \sum_{j = 1}^{d} 
            \int_{B_4}\eta_{r,R}^2 |u_T|^{2q} |f_{j}|^{2} \, dx.
		\end{aligned}
	\end{equation}

	Next, observe that 
	\[
		\nabla \left( |u_T|^q u \right) 
		= 
		q |u_T|^{q-2} u_T u \nabla u_T + |u_T|^q \nabla u
		= \left\{\begin{aligned}
			& |u_T|^{q} \nabla u & &{\rm if} \ |u(x) | \geq T,\\
			& (q+1) |u_T|^{q} \nabla u_T & &{\rm if} \ |u(x)| < T.			
		\end{aligned}\right.
	\]
Hence, \eqref{eq:B.5} gives
	\[
		\begin{aligned}
			&\int_{B_4} \eta_{r,R}^2 \left| \nabla \left( |u_T|^q u \right) \right|^2 \, dx 
			\\
			\leq \,&
			\frac{(q+1)^2 + 2q}{2q} 
			\left\{\int_{B_4} \eta_{r,R}^2 |u_T|^{2q} |\nabla u |^2 \, dx 
			+ 2q 
            \int_{B_4} \eta_{r,R}^2 |u_T|^{2q} |\nabla u_T |^2 \, dx 
			\right\}
			\\
			\leq \,& \frac{(q+1)^2 + 2 q}{2q} 
			\int_{B_4} 
			\left\{12 u^2 |u_T|^{2q} |\nabla \eta_{r,R}|^2 
            + 4 b \eta_{r,R}^2  |u_T |^{2q}  u^2 
            + 4 (q + 3) 
            \sum_{j = 1}^{d} 
            \eta_{r,R}^2 |u_T|^{2q} 
            |f_{j}|^{2} \right\}\, d x
			\\
			\leq \,& \frac{2(q+1)^2 + 4 q}{q}
            (q + 3)
			\int_{B_4} \left\{ u^2 |u_T|^{2q} |\nabla \eta_{r,R} |^2 
			+ b \eta_{r,R}^2  |u_T |^{2q}  u^2 
            +  
            \sum_{j = 1}^{d} 
            \eta_{r,R}^2 |u_T|^{2q} |f_{j}|^{2} \right\}\, dx .
		\end{aligned}
	\]
In addition, from $\nabla \left( \eta_{r,R} |u_T|^q u \right) 
= \eta_{r,R} \nabla \left( |u_T|^q u \right) + |u_T|^q u \nabla \eta_{r,R}$, 
it follows that 
	\[
		\begin{aligned}
			\int_{B_4} \left| \nabla \left( \eta_{r,R} |u_T|^q u \right)  \right|^2 \, dx 
			& \leq  2 \left\{ \int_{B_4} \eta_{r,R}^2 |\nabla \left( |u_T|^q u \right)|^2   d x 
			+ \int_{B_4} u^2 |u_T|^{2q} |\nabla \eta_{r,R}|^2 d x  \right\}
			\\
			&\leq \frac{4(q+1)^2 + 8 q}{q} 
            (q + 3)
						\int_{B_4} \left\{ u^2 |u_T|^{2q} |\nabla \eta_{r,R} |^2 
						+ |b| \eta_{r,R}^2  
                        |u_T |^{2q}  u^2 
                        + \sum_{j = 1}^{d} 
            \eta_{r,R}^2 |u_T|^{2q} |f_{j}|^{2}
            \right\} \, dx
			\\
			& \quad +
			2 \int_{B_4} u^2 |u_T|^{2q} |\nabla \eta_{r,R} |^2 \, dx 
			\\
			\leq & \, C_{q} \int_{B_{4}}
            \left\{
            u^2 |u_T|^{2q} |\nabla \eta_{r,R} |^2 
			+ |b| \eta_{r,R}^2  |u_T |^{2q}  u^2 
            + 
            \sum_{j = 1}^{d} \eta_{r,R}^2 |u_T|^{2q} |f_{j}|^{2} 
        \right\} \, dx, 
		\end{aligned}
	\]
where 
    \begin{equation} \label{eq:B.12}
    C_{q} = 
    \frac{4(q+1)^2 + 8 q}{q}(q + 3) + 2. 
    \end{equation}
Recalling $ \eta_{r,R} |u_T|^q u \in H^1_0(B_4)$ and using Sobolev's inequality, we see that  
	\begin{equation}\label{eq:B.6}
		\begin{aligned}
			\left\| \eta_{r,R} |u_T|^q u \right\|^2_{L^{2^\ast}} 
			& \leq \sigma^{-1} \left\| \nabla \left( \eta_{r,R} |u_T|^q u \right) \right\|_{L^2}^2 
			\\
			& \leq C_{q} \sigma^{-1} 
			\int_{B_4} \left\{  u^2 |u_T|^{2q} |\nabla \eta_{r,R} |^2 
			+ |b| \eta_{r,R}^2  |u_T |^{2q}  u^2 
           + \sum_{j = 1}^{d} \eta_{r,R}^2 |u_T|^{2q} |f_{j}|^{2} 
        \right\} \, dx.
		\end{aligned}
	\end{equation}

Letting $T \to \infty$ in \eqref{eq:B.6}, we see from the monotone convergence theorem that  
	\[
		\left\| \eta_{r,R} |u|^{q+1} \right\|^2_{L^{2^\ast}} 
		\leq C_{q} \sigma^{-1} 
		\int_{B_4} \left\{ |u|^{2q+2} |\nabla \eta_{r,R} |^2 
		+ |b| \eta_{r,R}^2  |u|^{2q+2} 
        + \sum_{j = 1}^{d} 
        \eta_{r, R}^2 |u|^{2q} |f_{j}|^{2} 
        \right\} \, dx.
	\]
Noting $\eta_{r,R} \equiv 1$ on $B_r$, $\eta_{r,R} \equiv 0$ on $B_{R}^c$ and 
$\| \nabla \eta_{r,R} \|_{L^\infty} \leq C_0/ (R-r)$ where $C_0$ is independent of $r$ and $R$, and using H\"older's inequality, we see that 
	\[
		\begin{aligned}
			\left\| u \right\|_{L^{2^\ast (q+1)} (B_r) }^{2(q+1)}
			& \leq 
			C_{q} \sigma^{-1} 
			\int_{B_4} \left(  |\nabla \eta_{r,R} |^2  |u|^{2q+2} +  |b| \eta_{r,R}^2 |u|^{2q+2}
            + \sum_{j = 1}^{d} 
        \eta_{r, R}^2 |u|^{2q} |f_{j}|^{2}  
           \right) d x \\
            &\leq 
			C_{q} \sigma^{-1} 
			\left[ \frac{C_0^2|B_4|^{1/s}}{(R-r)^2} + \| b \|_{L^s(B_4)} 
            \right]
			\| u \|_{L^{\frac{(2q+2)s}{s-1}}(B_R)}^{2q+2} 
            + C_{q} \sigma^{-1} 
            \sum_{j = 1}^{d} 
            \|f_{j}\|_{L^{2s}(B_4)}^2 
            \| u \|_{L^{\frac{(2q + 2)s}
            {s-1}}(B_R)}^{2q}.
		\end{aligned}
	\]
Writing $\wt{q} := q+1$ and $s^\ast := s/(s-1)$, we have
    \begin{equation} \label{eq:B.8-1}
    \| u \|_{L^{2^\ast \wt{q} } (B_r) } 
		\leq 
        \left\{
        C_{q} \sigma^{-1} 
			\left[ \frac{C_0^2|B_4|^{1/s}}{(R-r)^2} + \| b \|_{L^s(B_4)} 
            \right]
			\| u \|_{L^{\frac{2 
            \wt{q} s}{s-1}}(B_R)}^{2 \wt{q}} 
            + C_{q} \sigma^{-1} 
            \sum_{j = 1}^{d} 
            \|f_{j}\|_{L^{2s}(B_4)}^2 
            \| u \|_{L^{\frac{2 
            \wt{q} s}
            {s-1}}(B_R)}^{2q}
        \right\}^{\frac{1}{2 \wt{q}}}. 
    \end{equation}
For $n \geq 1$, we define 
	\begin{equation} \label{eq:B.14}
		\tau_0 := \frac{2^\ast}{2s^\ast}, \quad 
		\wt{q}_n := \tau_0^n, \quad R_1 = 2, \ r_1 := R_1 - \frac{1}{2}, \quad 
		R_n := r_{n-1} > r_{n-1} - \frac{1}{2^n} =: r_n. 
	\end{equation}
We remark that 
since $s > d/2$, we have 
$\tau_{0} > 1$, so that 
$\wt{q}_n \to \infty$ and $2^\ast \wt{q}_{n-1} = 2 s^\ast \wt{q}_{n}$. 
We claim that 
    \begin{equation}\label{eq:B.11}
      \sup_{n \in \N} 
      \|u\|_{L^{2^{*} \wt{q}_{n}} (B_1)} 
      < + \infty. 
    \end{equation}
Then, we may assume that 
    $\|u\|_{L^{2^{*} \wt{q}_{n}}(B_{2})} \geq 1$ 
    for any $n \in \N$. 
    Then, setting $q=\wt{q}_n$, $R=R_n$ and $r=r_n$ in \eqref{eq:B.8-1}, 
we see that 
	\begin{equation}\label{eq:B.8}
	\begin{split}
        \| u \|_{L^{2^\ast \wt{q}_{n} } (B_r) } 
		& \leq 
        \left\{
		2C_{q_{n}} \sigma^{-1}
		\left[
		\frac{C_0^2|B_4|^{1/s}}{(R_{n}-r_{n})^2} 
        + \| b \|_{L^s(B_4)} 
        + \sum_{j = 1}^{d} 
            \|f_{j}\|_{L^{2s}(B_4)}^2 
        \right]
		\right\}^{1/(2\wt{q}_{n})}
		\| u \|_{L^{2\wt{q}_{n} s^\ast }(B_R)}.
    \end{split} 
    \end{equation}
By \eqref{eq:B.12}, \eqref{eq:B.14} and \eqref{eq:B.8}, 
we find $C_1=C_1(d,s) > 0$ 
such that for each $n \geq 1$,
	\[
		\| u \|_{L^{2^\ast \wt{q}_n} (B_{r_n}) } 
		\leq C_1^{1 / (2\wt{q}_n)} 
		 \left[ \wt{q}_n^2 \left( 4^n + 1 \right) \right]^{1/(2\wt{q}_n)}
		\| u \|_{L^{2 \wt{q}_n s^\ast } (B_{R_n})},
	\]
which implies 
	\[
		\| u \|_{L^{ 2^\ast \wt{q}_n } (B_{r_n}) } 
		\leq C_1^{\sum_{k=1}^n 1/(2 \wt{q}_k)} 
		\left(\prod_{k=1}^n 
		\left[ \wt{q}_k^2 \left( 4^k + 1 \right) \right]^{1/(2\wt{q}_k)} \right)
		\| u \|_{L^{2^\ast}(B_4)}
	\]
for $n \geq 1$. From the definition of $\wt{q}_n$, one sees from $\tau_{0} > 1$ that 
the right hand side is bounded as $n \to \infty$. 
Therefore, we obtain  
	\[
		\| u \|_{L^\infty(B_1)} 
		\leq C_1^{\sum_{k=1}^\infty 1/(2\wt{q}_k)} 
		\left(\prod_{k=1}^\infty 
				\left[ \wt{q}_k^2 \left( 4^k + 1 \right) \right]^{1/(2\wt{q}_k)} \right)
				\| u \|_{L^{2^\ast}(B_4)}.
	\]
Hence, our assertion holds.

	\end{proof}

\section{Table of notation}
\label{section:C}
\begin{center}
\begin{tabular}{|c|c|}
\hline
Symbols 
& Descriptions 
\\[6pt] 
\hline
$\{u_{1, \omega}\}$
& a family of positive 
solutions to \eqref{sp} with 
$\limsup_{\omega \to 0} \|u_{1, \omega}\|_{L^{\infty}} 
< \infty$
\\[6pt]
$\widetilde{u}_{1, \omega}$ 
& $\widetilde{u}_{1, \omega} (\cdot) 
= \omega^{- \frac{1}{p-1}} u_{1, \omega} 
(\omega^{-\frac{1}{2}} \cdot)$
\\[6pt]
$M_{1, \omega}$ 
& $M_{1, \omega} = \|u_{1, \omega}\|_{L^{\infty}}$
\\[6pt]
$\widehat{u}_{1, \omega}$
& $\widehat{u}_{1, \omega}(\cdot) 
= M_{1, \omega}^{-1} u_{1, \omega}(M_{1, \omega}
^{-\frac{p-1}{2}} \cdot)$
\\[6pt]
$\alpha_{1, \omega}$
& $\alpha_{1, \omega} = \omega M_{1, \omega}^{-(p-1)}$
\\[6pt]
$\gamma_{1, \omega}$ 
& 
$\gamma_{1, \omega} 
= - \alpha_{1, \omega} + 1 + M_{1, \omega}^{5 - p}$ 
\\[6pt]
$\{u_{2, \omega}\}$
& a family of positive 
solutions to \eqref{sp} with 
$\limsup_{\omega \to 0} \|u_{2, \omega}\|_{L^{\infty}} 
= \infty$
\\[6pt]
$M_{2, \omega}$ 
& $M_{2, \omega} = \|u_{2, \omega}\|_{L^{\infty}}$
\\[6pt]
$\alpha_{2, \omega}$
& $\alpha_{2, \omega} = \omega 
M_{2, \omega}^{-4}$
\\[6pt]
$\beta_{2, \omega}$
& $\beta_{2, \omega} = 
M_{2, \omega}^{p - 5}$
\\[6pt]
$u_{\omega}, \; M_{\omega}, \;
\alpha_{\omega}, \; \beta_{\omega}$
& 
$u_{\omega} = u_{2, \omega}, \;
M_{\omega} = M_{2, \omega}, \;
\alpha_{\omega} = \alpha_{2, \omega}, \;
\beta_{\omega} = \beta_{2, \omega}$
\\[6pt]
$W$
& 
Aubin-Talenti function 
\qquad
$W(x) = \left(1 + \frac{|x|^{2}}{d(d-2)} \right)^{-\frac{d-2}{2}}$
\\[6pt]
$\mathcal{E}(u)$
&
$\mathcal{E}(u) = 
\frac{1}{2} \|\nabla u\|_{L^{2}}^{2} 
- \frac{1}{p+1} \|u\|_{L^{p+1}}^{p+1} 
- \frac{1}{2^{*}} \|u\|_{L^{2^{*}}}^{2^{*}}$
\\[6pt]
$\mathcal{K}(u)$
& 
$\mathcal{K}(u) = 
\|\nabla u\|_{L^{2}}^{2} 
- \frac{d(p-1)}{2(p+1)}\|u\|_{L^{p+1}}^{p+1} 
- \|u\|_{L^{2^{*}}}^{2^{*}}$
\\[6pt]
$\mathcal{S}_{\omega}(u)$
& 
$\mathcal{S}_{\omega}(u) = 
\frac{1}{2}\|\nabla u \|_{L^{2}}^{2}
+ \frac{\omega}{2}\|u\|_{L^{2}}^{2}
- \frac{1}{p+1}\|u\|_{L^{p+1}}^{p+1}
- \frac{1}{2^{*}}\|u\|_{L^{2^{*}}}^{2^{*}}$
\\[6pt]
$E_{\min}(m)$ 
& 
$E_{\min}(m) = \inf\left\{
\mathcal{E} (u) \colon 
u \in \mathcal{P}(m) \right\}$ 
\\[6pt]
$\mathcal{P}(m)$
&
$\mathcal{P}(m) = 
\left\{ 
u \in S(m) \colon \mathcal{K}(u) = 0
\right\} 
= \left\{u \in H^{1}(\R^{d}) \colon 
\|u\|_{L^{2}}^{2} = m, \; \mathcal{K}(u) = 0
\right\}$ 
\\[6pt]
$S(m)$ 
& 
$S(m) = \left\{ 
u \in H^{1}(\R^{d}) \colon \|u\|_{L^{2}}^{2} = m
\right\} $
\\[6pt]
$E_{\min, \pm}(m)$ 
& $E_{\min, \pm}(m) = \inf\left\{ 
\E(u) \colon u \in \mathcal{P}_{\pm}(m)
\right\}$
\\[6pt]
$\mathcal{P}_{+}(m)$ 
& $\mathcal{P}_{+}(m) = 
\left\{
u \in \mathcal{P}(m) \colon 
2\|\nabla u\|_{L^{2}}^{2} 
> \frac{d^{2}(p-1)^{2}}{4 (p+1)} 
\|u\|_{L^{p+1}}^{p+1}
+ 2^{*} \|u\|_{L^{2^{*}}}^{2^{*}}
\right\}
$
\\[6pt]
$\mathcal{P}_{-}(m)$ 
& $\mathcal{P}_{-}(m) = 
\left\{
u \in \mathcal{P}(m) \colon 
2\|\nabla u\|_{L^{2}}^{2} 
< \frac{d^{2}(p-1)^{2}}{4 (p+1)} 
\|u\|_{L^{p+1}}^{p+1}
+ 2^{*} \|u\|_{L^{2^{*}}}^{2^{*}}
\right\}
$
\\[6pt]
$Q_{\omega(m)}$ 
& ground state 
to \eqref{sp} when $\omega = \omega(m)$
\\[6pt]
$R_{\omega(m)}$
& minimizers for $E_{\min}(m)$
\\[6pt]
$R_{\omega_{-}(m), -}$
& 
minimizers for 
$E_{\min, -}(m)$
\\[6pt]
$L_{\omega, +}$
& 
$L_{\omega, +} = 
- \Delta + \omega - p 
u_{\omega}^{p-1} 
- (2^{*} - 1) 
u_{\omega}^{2^{*} - 2}$
\\[6pt]
\hline
\end{tabular}
\end{center}


\bibliographystyle{plain}
\bibliography{biblio}

\noindent
Takafumi Akahori
\\[6pt]
Department of Engineering,
\\[6pt]
Shizuoka University,
\\[6pt]
3-5-1 Johoku Chuo-ku, Hamamatsu-shi, Shizuoka 432-8011, JAPAN
\\[6pt]
E-mail: akahori.takafumi@shizuoka.ac.jp

\vspace{0.5cm}

\noindent
Slim Ibrahim
\\[6pt]
Department of Mathematics and Statistics,
\\[6pt]
University of Victoria,
\\[6pt]
3800 Finnerty Road, Victoria, BC V8P 5C2, Canada
\\[6pt]
E-mail:ibrahims@uvic.ca

\vspace{0.5cm}

\noindent
Hiroaki Kikuchi
\\[6pt]
Department of Mathematics,
\\[6pt]
Tsuda University,
\\[6pt]
2-1-1 Tsuda-machi, Kodaira-shi, Tokyo 187-8577, JAPAN
\\[6pt]
E-mail: hiroaki@tsuda.ac.jp

\vspace{0.5cm}

\noindent
Masataka Shibata
\\[6pt]
Department of Mathematics,
\\[6pt]
Meijo University,
\\[6pt]
1-501 Shiogamaguchi, Tempaku-ku, Nagoya 468-8502, JAPAN
\\[6pt]
E-mail: mshibata@meijo-u.ac.jp

\vspace{0.5cm}

\noindent
Juncheng Wei 
\\[6pt]
Department of Mathematics, 
\\[6pt]
Chinese University of Hong Kong
\\[6pt]
Shatin, NT, Hong Kong
\\[6pt]
E-mail:jcwei@math.ubc.ca

\end{document}